\documentclass[11pt]{smfart}
\usepackage{amssymb,amsxtra, amsmath, amsfonts}
\usepackage{smfthm}
\theoremstyle{plain}
\usepackage[french,english]{babel}
\makeatletter
\def\amsbb{\use@mathgroup \M@U \symAMSb}
\makeatother

\usepackage[all]{xy}
\SelectTips{cm}{}
\usepackage{tikz}
\usepackage{tikz-cd}
\usepackage[a4paper, marginparwidth=60pt, margin=1.0in]{geometry}
\usepackage{mathrsfs}
\usepackage{marginnote}
\usepackage{stmaryrd}
\usepackage{bookmark}
\usepackage{imakeidx}
\usepackage{listings}
\usepackage{xspace}
\usepackage{etoolbox}
\usepackage{csquotes}
\usepackage{enumitem}
\usepackage{tensor}
\usepackage{mathtools}
\usepackage{longtable}
\usepackage{multirow}
\usepackage{makecell}
\usepackage{bigints}
\usepackage{mathdots}
\usepackage{capt-of}
\usepackage{mathtools}
\usepackage{xcolor}
\usepackage{graphicx}
\usetikzlibrary{matrix,arrows,backgrounds}

\usepackage{hyperref} 

\numberwithin{equation}{section}

\setitemize[0]{leftmargin=*,itemsep=\the\smallskipamount}
\setenumerate[0]{leftmargin=*,itemsep=\the\smallskipamount}

\renewcommand{\to}{%
   \ifbool{@display}{\longrightarrow}{\rightarrow}%
   }
\let\shortmapsto\mapsto
\renewcommand{\mapsto}{%
   \ifbool{@display}{\longmapsto}{\shortmapsto}%
   }

\newcommand{\lra}{%
   \ifbool{@display}{\longleftrightarrow}{\leftrightarrow}%
   }
  
\newcommand{\sC}{\ensuremath{\mathscr{C}}\xspace}
\newcommand{\sD}{\ensuremath{\mathscr{D}}\xspace}
\newcommand{\sE}{\ensuremath{\mathscr{E}}\xspace}
\newcommand{\sF}{\ensuremath{\mathscr{F}}\xspace}
\newcommand{\sG}{\ensuremath{\mathscr{G}}\xspace}

\newcommand{\sL}{\ensuremath{\mathscr{L}}\xspace}
\newcommand{\sM}{\ensuremath{\mathscr{M}}\xspace}

\newcommand{\sW}{\ensuremath{\mathscr{W}}\xspace}

\newcommand{\fkg}{\ensuremath{\mathfrak{g}}\xspace}

\newcommand{\fkm}{\ensuremath{\mathfrak{m}}\xspace}

\newcommand{\fkS}{\ensuremath{\mathfrak{S}}\xspace}

\newcommand{\BB}{\ensuremath{\amsbb{B}}\xspace}

\newcommand{\BD}{\ensuremath{\amsbb{D}}\xspace}

\newcommand{\BF}{\ensuremath{\amsbb{F}}\xspace}
\newcommand{\BG}{\ensuremath{\amsbb{G}}\xspace}
\newcommand{\BH}{\ensuremath{\amsbb{H}}\xspace}

\newcommand{\BL}{\ensuremath{\amsbb{L}}\xspace}

\newcommand{\BN}{\ensuremath{\amsbb{N}}\xspace}

\newcommand{\BP}{\ensuremath{\amsbb{P}}\xspace}
\newcommand{\BQ}{\ensuremath{\amsbb{Q}}\xspace}

\newcommand{\BS}{\ensuremath{\amsbb{S}}\xspace}
\newcommand{\BT}{\ensuremath{\amsbb{T}}\xspace}

\newcommand{\BV}{\ensuremath{\amsbb{V}}\xspace}

\newcommand{\BX}{\ensuremath{\amsbb{X}}\xspace}

\newcommand{\BZ}{\ensuremath{\amsbb{Z}}\xspace}

\newcommand{\bD}{\ensuremath{\mathbf{D}}\xspace}
\newcommand{\bE}{\ensuremath{\mathbf{E}}\xspace}
\newcommand{\bF}{\ensuremath{\mathbf{F}}\xspace}
\newcommand{\bG}{\ensuremath{\mathbf{G}}\xspace}
\newcommand{\bH}{\ensuremath{\mathbf{H}}\xspace}

\newcommand{\bM}{\ensuremath{\mathbf{M}}\xspace}
\newcommand{\bN}{\ensuremath{\mathbf{N}}\xspace}

\newcommand{\bV}{\ensuremath{\mathbf{V}}\xspace}

\newcommand{\CA}{\ensuremath{\mathcal{A}}\xspace}
\newcommand{\CB}{\ensuremath{\mathcal{B}}\xspace}
\newcommand{\CC}{\ensuremath{\mathcal{C}}\xspace}
\newcommand{\CD}{\ensuremath{\mathcal{D}}\xspace}
\newcommand{\CE}{\ensuremath{\mathcal{E}}\xspace}
\newcommand{\CF}{\ensuremath{\mathcal{F}}\xspace}
\newcommand{\CG}{\ensuremath{\mathcal{G}}\xspace}
\newcommand{\CH}{\ensuremath{\mathcal{H}}\xspace}

\newcommand{\CL}{\ensuremath{\mathcal{L}}\xspace}
\newcommand{\CM}{\ensuremath{\mathcal{M}}\xspace}

\newcommand{\CO}{\ensuremath{\mathcal{O}}\xspace}

\newcommand{\CX}{\ensuremath{\mathcal{X}}\xspace}

\newcommand{\rmB}{{\ensuremath{\rm{B}}\xspace}}
\newcommand{\rmC}{{\ensuremath{\rm{C}}\xspace}}
\newcommand{\rmD}{{\ensuremath{\rm{D}}\xspace}}

\newcommand{\rmH}{{\ensuremath{\rm{H}}\xspace}}

\newcommand{\rmL}{{\ensuremath{\rm{L}}\xspace}}
\newcommand{\rmM}{{\ensuremath{\rm{M}}\xspace}}

\newcommand{\rmO}{{\ensuremath{\rm{O}}\xspace}}

\newcommand{\rmR}{{\ensuremath{\rm{R}}\xspace}}

\newcommand{\rmU}{{\ensuremath{\rm{U}}\xspace}}

\newcommand{\rmW}{{\ensuremath{\rm{W}}\xspace}}

\newcommand{\rmb}{{\ensuremath{\rm{b}}\xspace}}

\newcommand{\rmm}{{\ensuremath{\rm{m}}\xspace}}

\newcommand{\cf}{{\it cf.}\ }
\newcommand{\ie}{{\it i.e.}\ }
\newcommand{\wt}{\widetilde}
\newcommand{\wh}{\widehat}

\newcommand{\intn}[1]{\big ( {#1} \big )}

\newcommand{\ov}{\overline}
\newcommand{\ud}{\underline}
\newcommand{\incl}{\hookrightarrow}

\newcommand{\presp}[1]{\prescript{p}{}{#1}}
\newcommand{\pbrv}[1]{{}^{p}\brv{#1}}

\newcommand{\bPi}{{\boldsymbol{\Pi}\xspace}}

\newcommand{\conv}{{ \rm conv}}

\newcommand{\alg}{{ \rm alg}}
\newcommand{\lalg}{{ \rm lalg}}
\newcommand{\abe}{{ \rm ab}}

\newcommand{\cri}{{\rm cr}}
\newcommand{\cris}{{\rm cr}}

\newcommand{\st}{ \rm st}

\newcommand{\pst}{ \rm pst}
\newcommand{\et}{\acute{\mathrm{e}}\mathrm{t}}

\newcommand{\pet}{\mathrm{p}\et}

\newcommand{\an}{{\rm an}}
\newcommand{\lan}{{\rm lan}}
\newcommand{\DR}{{\rm DR}}
\newcommand{\Op}{{\rm Op}}

\newcommand{\dR}{{\rm dR}}

\newcommand{\Dr}{{ \rm Dr}}

\newcommand{\nr}{{\rm nr}}

\newcommand{\HK}{{\rm HK}}
\newcommand{\FF}{{\rm FF}}

\newcommand{\syn}{{\rm syn}}
\newcommand{\rig}{{\rm rig}}

\newcommand{\deR}[1]{\Xunders{#1}{\dR}}
\newcommand{\decris}[1]{\Xunders{#1}{\cris}}
\newcommand{\dest}[1]{\Xunders{#1}{\st}}

\newcommand{\xdecris}[2]{\Xunders{#1}{ \cris,#2}}

\newcommand{\funp}[2]{{#1}^{+}_{#2}}
\newcommand{\decrisp}[1]{\funp{#1}{\cris}}

\newcommand{\destp}[1]{\funp{#1}{\st}}
\newcommand{\deRp}[1]{\funp{#1}{\dR}}

\newcommand{\Aut}{{\rm Aut}}

\DeclareMathOperator{\Coker}{Coker}
\DeclareMathOperator{\Ker}{Ker}
\newcommand{\coim}{\mathrm{coim}}
\newcommand{\coker}{\mathrm{coker}}
\newcommand{\Isom}{\mathrm{Isom}}

\newcommand{\Gr}{{\rm Gr}\xspace}

\newcommand{\Fil}{{\rm Fil}\xspace}
\newcommand{\Filb}{{\rm Fil}^{\bullet}\xspace}
\newcommand{\Ind}{{\rm Ind}}

\DeclareMathOperator{\rang}{rang}

\DeclareMathOperator{\Sym}{Sym}

\newcommand{\End}{{\rm End}}

\newcommand{\Lie}{{\rm Lie}\xspace}

\newcommand{\Pic}{{\rm Pic\xspace}}

\newcommand{\Spf}{{\rm Spf\xspace}}

\newcommand{\Sp}{{\rm Sp\xspace}}

\newcommand{\Spa}{{\rm Spa\xspace}}

\newcommand{\Ext}{{\rm Ext}}
\newcommand{\holim}{{\rm holim}}
\newcommand{\hocolim}{{\rm hocolim}}

\newcommand{\Nilp}{ {\rm Nilp}}

\newcommand{\Vect}{\mathrm{Vect}}

\newcommand{\Loc}{{\rm Loc}}

\newcommand{\Li}{{\rm Litc}}

\newcommand{\bVectd}[1]{\bVectd{#1}}
\newcommand{\bVectdf}[1]{\bVectdf{#1}}

\newcommand{\Iso}{{\rm Iso}}

\newcommand{\Hom}{{\rm Hom}}
\newcommand{\bHom}{{\mathbf{Hom}}}
\newcommand{\CHom}{{\mathcal Hom}}

\newcommand{\Xunders}[2]{{#1}_{#2}}

\newcommand{\Filt}[1]{\ensuremath{\mathrm{Fil^{#1}}}\xspace}
\newcommand{\rmOD}{\rmO_{\rmD}}

\newcommand{\rmODt}{\tim{\rmO}_{\rmD}}

\newcommand{\invp}[1]{\big [ \tfrac 1 {#1}\big ] }
\newcommand{\id}{\mathrm{id}}
\newcommand{\tim}[1]{{#1}^{\times}}

\newcommand{\wotimes}{\wh{\otimes}}

\newcommand{\Rep}{{\rm Rep}}

\newcommand{\Gal}{{\rm Gal}}

\newcommand{\GQp}{\sG_{\BQ_p}}

\newcommand{\XBGm}[1]{\BG_{\rmm,#1}}

\newcommand{\GL}{{\rm GL}}

\newcommand{\PSL}{{\rm PSL}}

\newcommand{\PGL}{{\rm PGL}}

\newcommand{\B}{{\rm B}}

\newcommand{\Bcris}{\decris{\B}}
\newcommand{\Bdr}{\deR \B}
\newcommand{\Bst}{\dest{\B}}
\newcommand{\Bcrisp}{\decrisp{\B}}
\newcommand{\Bdrp}{\deRp {\B}}
\newcommand{\Bstp}{\destp{\B}}
\newcommand{\whBstp}{\destp{\wh{\B}}}

\newcommand{\BBcris}{\decris{\BB}}
\newcommand{\BBdr}{\deR \BB}

\newcommand{\BBdrp}{\deRp {\BB}}

\newcommand{\OBBdr}{\CO\deR \BB}
\newcommand{\OBBdrp}{\CO\BBdrp}

\newcommand{\XBBcris}[1]{\xdecris{\BB}{#1}}

\newcommand{\XBBcrisp}[1]{\funp{\BB}{\cris,#1}}

\newcommand{\XOBBdr}[1]{\CO\BB_{\dR, #1}}

\newcommand{\Harg}[2]{\! \big ( #1\, ; \,  #2 \big ) }
\newcommand{\Hargg}[2]{\! \left ( #1\, ; \,  #2 \right ) }
\newcommand{\intnn}[2]{\! \big ( #1\, , \,  #2 \big) }

\newcommand{\Fpbar}{\bar{\BF}_p}
\newcommand{\Qp}{\BQ_p}
\newcommand{\Qpt}{\Qp^{\times}}
\newcommand{\Qpd}{\BQ_{p^2}}
\newcommand{\Qpbr}{\breve{\BQ}_p}
\newcommand{\Qpbar}{\bar{\BQ}_p}
\newcommand{\Zp}{\BZ_p}
\newcommand{\Zpd}{\BZ_{p^2}}

\newcommand{\Zpbr}{\breve{\BZ}_p}
\newcommand{\Zpt}{\Zp^{\times}}

\newcommand{\JL}{{\rm JL}}
\newcommand{\LL}{{\rm LL}}
\newcommand{\WD}{{\rm WD}}
\newcommand{\sWD}{\sW\!\sD}
\newcommand{\Pilan}{\prescript{\lan}{}{\Pi}}

\newcommand{\nrd}{\mathrm{nrd}}

\newcommand{\Res}{{\rm Res}}

\newcommand{\czG}{\check {G}}

\newcommand{\itemb}{\item[$\bullet$]}

\newcommand{\cz}[1]{{\check{#1}}}

\newcommand{\rec}{{\rm{rec}}}

\newcommand{\inv}{{\rm inv}}

\newcommand{\brv}[1]{\breve{#1}}

\newcommand{\St}{{\rm St }}

\newcommand{\Frac}{\rm Frac\xspace}

\newcommand\restr[2]{{
  \left.\kern-\nulldelimiterspace 
  #1 
  \vphantom{\big|} 
  \right|_{#2} 
  }}

\title[Cohomologie de la tour de Drinfeld]{Cohomologie de systèmes locaux $p$-adiques sur les revêtements du demi-plan de Drinfeld}
\author{Arnaud Vanhaecke}
\address{DMA ENS, 45 rue d'Ulm, 75005, Paris}
\email{arnaud.vanhaecke@ens.fr}
\date{\today}
\begin{document}
\maketitle
\begin{abstract}
Colmez, Dospinescu et Nizio{\l} démontrent que les seules représentations $p$-adiques de $\rm{Gal}(\bar{\mathbb{Q}}_p/\mathbb{Q}_p)$ qui apparaissent dans la cohomologie étale $p$-adique de la tour de revêtements du demi-plan de Drinfeld sont les représentations cuspidales (\ie potentiellement semi-stables, dont la représentation de Weil-Deligne associée est irréductible) de dimension $2$, à poids de Hodge-Tate $0$ et $1$ et que leur multiplicité est donnée par la correspondance de Langlands $p$-adique. Nous étendons ce résultat en poids quelconque, en considérant la cohomologie étale $p$-adique à coefficients dans les puissances symétriques du système local universel sur la tour de Drinfeld. Une différence notable est que toutes les représentations potentiellement semi-stables de dimension $2$ non cristabélines, pas seulement les cuspidales, apparaissent avec les multiplicités attendues. Le point clé est que les systèmes locaux que l'on considère sont particulièrement simples : ce sont des \og opers isotriviaux\fg sur une courbe. Nous donnons une recette pour calculer la cohomologie proétale de tels systèmes locaux à partir de la cohomologie de Hyodo-Kato de la courbe et du complexe de de Rham du fibré plat filtré associé au système local.
\end{abstract}
\begin{altabstract}
Colmez, Dospinescu and Nizio{\l} have shown that the only $p$-adic representations of $\rm{Gal}(\bar{\mathbb{Q}}_p/\mathbb{Q}_p)$ appearing in the $p$-adic étale cohomology of the coverings of Drinfeld's half-plane are the $2$-dimensional cuspidal representations (\ie potentially semi-stable, whose associated Weil-Deligne representation is irreducible) with Hodge-Tate weights $0$ and $1$ and their multiplicities are given by the $p$-adic Langlands correspondence. We generalise this result to arbitrary weights, by considering the $p$-adic étale cohomology with coefficients in the symmetric powers of the universal local system on Drinfeld's tower. A novelty is the appearance of potentially semistable $2$-dimensional non-cristabelian representations, with expected multiplicity. The key point is that the local systems we consider turn out to be particularly simple: they are ``isotrivial opers" on a curve. We develop a recipe to compute the proétale cohomology of such a local system using the Hyodo-Kato cohomology of the curve and the de Rham complex of the flat filtered bundle associated to the local system.
\end{altabstract}
\newpage
\setcounter{tocdepth}{2}
\tableofcontents
\newpage
\section{Introduction}
Soit $p$ un nombre premier, soient $\sG_{\Qp}\coloneqq \Gal(\Qpbar/\Qp)$ le groupe de Galois absolu de $\Qp$, $\sW_{\Qp}\subset \GQp$ le groupe de Weil, $\sWD_{\Qp}$ le groupe de Weil-Deligne et $G\coloneqq \GL_2(\Qp)$. Soit $L$ une extension finie de $\Qp$, que l'on suppose assez grande\footnote{Dans ce texte, ceci signifiera que $L$ contient toutes les extensions quadratiques de $\Qp$.} et qui sera notre corps des coefficients ; on parle d'une \emph{$L$-représentation} pour désigner une représentation sur un $L$-espace vectoriel. La correspondance de Langlands locale $p$-adique fait correspondre les $L$-représentations continues de dimension $2$ de $\sG_{\Qp}$ et les $L$-représentations Banachiques de~$G$. Cet article est consacré à la poursuite du programme de géométrisation de cette correspondance, initié par Colmez, Dospinescu et Niziol dans \cite{codoni}, sur le modèle de la correspondance classique. Plus précisément, on étend les résultats de \cite{codoni} aux représentations de de Rham de poids quelconque en utilisant des coefficients non triviaux.

\Subsection{Les revêtements du demi-plan de Drinfeld et son système local}
On note $C$ le complété d'une cloture algébrique de $\Qp$, $\Qpbr$ le complété de l'extension maximale non ramifiée de $\Qp$ et $\Zpbr$ son anneau des entiers.

Soit $\rmD$ l'algèbre de quaternion non scindée sur $\Qp$, \ie  l'unique corps gauche tel que \hbox{$\inv_{\Qp}\rmD=\frac 1 2$}. On note $\rmOD\subset \rmD$ l'unique ordre maximal de $\rmD$, $\varpi_{\rmD}\in \rmOD$ une uniformisante et $\czG\coloneqq \rmD^{\times}$ le groupe des éléments inversibles de $\rmD$.

Le \emph{demi-plan $p$-adique de Drinfeld} est défini comme 
$$
\BH_{\Qp}\coloneqq \BP^1_{\Qp}\setminus \BP^1(\Qp).
$$ 
Cet espace admet naturellement une structure d'espace analytique rigide sur $\Qp$ et une action de~$G$ par homographie, compatible avec cette structure. Le théorème de représentabilité de Drinfeld (\cf\cite{drin}) donne une interprétation modulaire de cet espace. Plus précisément, il définit un espace formel $\sM_{\Dr}$ sur $\Zpbr$ à partir d'un problème de modules de quasi-isogénies de certains $\rmOD$-modules formels, les \emph{$\rmOD$-modules formels spéciaux}. On note $\brv{\rmM}_{C}\coloneqq \brv{\rmM}_{\Qpbr}\otimes_{\Qpbr}C$ où $\brv{\rmM}_{\Qpbr}$ est la fibre générique  de $\sM_{\Dr}$. D'après le théorème de Drinfeld, les composantes connexes de $\brv{\rmM}_{C}$ s'identifient alors à $\BH_{C}\coloneqq \BH_{\Qp}\otimes_{\Qp}C$  ; plus précisément, $\brv{\rmM}_{C}\cong \BH_{C}\times \BZ$ en décomposant suivant la hauteur de la quasi-isogénie du problème de modules. 

Le problème de modules fournit alors un $\rmOD$-module formel spécial universel sur $\sM_{\Dr}$ dont le module de Tate définit un $\Zp$-système local étale de rang $4$ sur $\brv{\rmM}_{C}$ que l'on note $\BV_{\Dr}^+$. Pour~$n\geqslant 1$ un entier, on définit $\brv{\rmM}_{C}^n$, le $n$-ième étage de la tour de Drinfeld qui trivialise le système local de torsion $\BV_{\Dr}^+/\varpi_D^n$. C'est un revêtement galoisien de $\brv{\rmM}_{C}^0\coloneqq \brv{\rmM}_{C}$ de groupe de Galois $\rmODt/(1+\varpi_D^n\rmOD)$ et donc en particulier c'est un $C$-espace analytique rigide qui est Stein. On note $\brv{\rmM}_{C}^{\infty}$ le système projectif non complété des $\brv{\rmM}_{C}^n$. L'espace $\brv{\rmM}_{C}^n$ est muni d'une action de $\sW_{\Qp}$ compatible avec son action sur $C$ ; il est donc muni d'une action de~$\sW_{\Qp}\times G\times \czG$. De plus, les flèches de transition $\brv{\rmM}_{C}^{n+1}\rightarrow \brv{\rmM}_{C}^{n}$ sont équivariantes sous l'action de ces trois groupes.

On a un morphisme $\nu\colon  \sW_{\Qp}\times G\times\czG\rightarrow \Qpt$ donné par $\rec\otimes{\det}^{-1}\otimes \nrd $ où $\det$ est le déterminant de $G$, $\nrd$ la norme réduite de $\czG$ et $\rec \colon \sW_{\Qp}\rightarrow \Qpt$ le morphisme de réciprocité donné par le quotient $\sW_{\Qp}\rightarrow \sW_{\Qp}^{\abe}$ et l'isomorphisme du corps de classe $\sW_{\Qp}^{\abe}\cong \Qpt$ normalisé de sorte que le Frobenius arithmétique soit envoyé sur $p$. Rappelons que l'ensemble prodiscret $\pi_0(\brv{\rmM}_{C}^{\infty})=\varprojlim_n \pi_0(\brv{\rmM}_{C}^{n})$ des composantes connexes de la tour est un espace homogène principal sous l'action de $\Qpt$ et que l'action de $ \sW_{\Qp}\times G\times\czG$ sur ces composantes connexes est donnée par $\nu$ (\cf le numéro \ref{subsubsec:compco}).

Pour $\rmH^{\bullet}$ une théorie cohomologique (\ie  contravariante) raisonnable, en particulier celle de de Rham ou de Hyodo-Kato, on pose
$$
\rmH^{\bullet}(\brv{\rmM}_{C}^{\infty})\coloneqq \varinjlim_n \rmH^{\bullet}(\brv{\rmM}_{C}^{n}).
$$

En inversant $p$, on définit un $\Qp$-système local étale sur $\brv{\rmM}_{C}$ et ses revêtements, par $\BV_{\Dr}\coloneqq \BV_{\Dr}^+\otimes_{\Zp}\Qp$. Ce système local est $\sW_{\Qp}\times G\times \czG$-équivariant. Le système local qui nous intéresse est la $\Qp$-algèbre symétrique associée à ce système local, soit 
$$
\Sym \BV_{\Dr}\coloneqq \bigoplus_{k\geqslant 0}\Sym^k_{\Qp}\BV_{\Dr}.
$$
Notons qu'il admet un réseau $\Sym \BV^+_{\Dr}$ défini de façon similaire, mais qui n'est pas stable sous l'action de $G$ et $\czG$. On définit finalement $\BV$, puis sa $L$-algèbre symétrique  $\Sym \BV$, en appliquant $L\otimes_{\Qp}\cdot$ à $\BV_{\Dr}$ puis en tordant par un $L$-caractère lisse approprié pour rendre l'action de $p\in \czG$ triviale. Ces systèmes locaux ont naturellement des réseaux que l'on note $\BV^+$ et $\Sym \BV^+$ respectivement. On peut alors définir la cohomologie (pro)étale de $\Sym \BV$, qui est naturellement un $L$-espace vectoriel muni d'une action de $\sW_{\Qp}\times G\times\czG$.

\Subsection{Cohomologie étale du système local $p$-adique et correspondance de Langlands $p$-adique}\label{subsec:intcolang}
\subsubsection{Représentations de de Rham et objets associés}
Rappelons que l'on note $L$ notre corps des coefficients, qui est une extension finie de $\Qp$. On suppose que $L$ est assez gros, en particulier, qu'il contient l'extension quadratique non-ramifiée de $\Qp$ et on fixe une injection $\rmD\incl M_2(L)$ donnant $\czG\incl \GL_2(L)$. On définit 
$$
P_+\coloneqq \{(\lambda_1,\lambda_2)\in \BN^2\mid \lambda_2>\lambda_1\geqslant 0\},
$$
que l'on appelle, dans notre contexte, l'ensemble des \emph{poids positifs réguliers}. À $\lambda \in P_+$ on associe $\lvert \lambda\rvert = \lambda_1+\lambda_2\in \BN$ et $w(\lambda)=\lambda_2-\lambda_1\in \BN$. Si $V$ est une $L$-représentation de de Rham de $\GQp$ de dimension $2$, dont les poids de Hodge-Tate $\lambda_1$ et $\lambda_2$ sont tels que $\lambda =(\lambda_1,\lambda_2)\in P_+$, on dit simplement que $V$ est à poids $\lambda$ et on lui associe les objets suivants :
\begin{itemize}
\itemb Un $L$-$(\varphi,N,\GQp)$ module de rang $2$ sur $L\otimes_{\Qp}\Qp^{\nr}$, défini par
$$
\bD_{\pst}(V)\coloneqq\varinjlim_{{ [K:\Qp]<\infty }}\intn{V\otimes_{\Qp}\Bst}^{\sG_K}.
$$
\itemb Une $L$-représentation $\WD(V)\coloneqq \WD(\bD_{\pst}(V)[1])$ de $\sWD_{\Qp}$ obtenue par la recette de Fontaine (\cf  \cite{fonl}), où $[k]$ signifie que l'on multiplie l'action du Frobenius $\varphi$ par $p^{k}$.
\itemb Une $L$-représentation localement algébrique irréductible $\LL^{\lalg}(V)$ de $G$ définie de la façon suivante : par la correspondance de Langlands locale, on associe à $\WD(V)$ une $L$-représentation lisse irréductible (de dimension infinie) de $G$ notée $\LL^{\infty}(V)\coloneqq \LL(\WD(V))$ puis, à partir du poids $\lambda$, on pose 
$$
W_{\lambda}^*\coloneqq \Sym_L^{w(\lambda)-1}\otimes_L{\det}^{\lambda_1},\quad \LL^{\lalg}(V)\coloneqq \LL^{\infty}(V)\otimes_L W_{\lambda}^*,
$$
où $\Sym_L^{k}$ est la puissance symétrique $k$-ième de la $L$-représentation régulière de~$G$, \ie  celle obtenue en faisant agir naturellement $G$ sur $L^2$.
\itemb Une $L$-représentation localement algébrique irréductible $\JL^{\lalg}(V)$ de $\czG$ définie de la façon suivante : par la correspondance de Jacquet-Langlands locale, on associe à $\LL^{\infty}(V)$ une $L$-représentation lisse irréductible (de dimension finie) de $\czG$ notée $\JL^{\infty}(V)\coloneqq \JL(\LL^{\infty}(V))$ puis, à partir du poids $\lambda$, on pose 
$$
\cz{W}_{\lambda}\coloneqq \Sym_L^{w(\lambda)-1}\otimes_L{\nrd}^{\lambda_1},\quad \JL^{\lalg}(V)\coloneqq \JL^{\infty}(V)\otimes_{L} \cz{W}_{\lambda},
$$
où $\Sym_L^{k}$ est la puissance symétrique $k$-ième de la $L$-représentation régulière de $\czG$, identifiée à un sous-groupe de $\GL_2(L)$.
\itemb Une $L$-représentation unitaire continue $\bPi(V)$ de $G$ obtenue par la correspondance de Langlands $p$-adique dont on note $\bPi(V)'$ le dual (stéréotypique).
\end{itemize}
Notons que les représentations $\LL^{\lalg}(V)$ et $\JL^{\lalg}(V)$ ne dépendent que du $L$-$(\varphi,N,\GQp)$-module $\bD_{\pst}(V)$ et du poids $\lambda$. Un point fondamental est la relation entre $\bPi(V)$ et $\LL^{\lalg}(V)$ : les vecteurs localement algébriques de $\bPi(V)$ sont précisément $\LL^{\lalg}(V)$, \ie 
$$
\bPi(V)^{\lalg}=\LL^{\lalg}(V).
$$
\goodbreak
On continue à supposer que $V$ est à poids réguliers positifs. On dit que  :
\begin{itemize}
\itemb $V$ est \emph{cuspidale} si $\WD(V)$ est absolument irréductible ou, de manière équivalente, si $\LL^{\infty}(V)$ est cuspidale en tant que $L$-représentation lisse. Dans ce cas, $N=0$ sur $\bD_{\pst}(V)$, ce qui signifie que $V$ est potentiellement cristalline, 
\itemb $V$ est \emph{spéciale} si $N\neq 0$ (ce qui signifie que $N$ est d'ordre maximal puisque $\bD_{\pst}(V)$ est de rang $2$) ou, de manière équivalente, si $\LL^{\infty}(V)$ est le tordue par un caractère lisse de la $L$-représentation de Steinberg lisse de $G$. Dans ce cas, $V$ est semi-stable à torsion près, non cristabéline et $\JL^{\infty}(V)$ est un caractère lisse : ces dernières propriétés impliquent elles aussi que $V$ est spéciale.
\end{itemize}
Notons que si $V$ est cuspidale ou spéciale, alors $V$ est absolument irréductible sauf si $V$ est spéciale et $w(\lambda)=1$ .
\subsubsection{Résultat principal}
On peut maintenant énoncer le résultat principal :
\medskip
\begin{theo}\label{thm:princ}
Soit $V$ une $L$-représentation absolument irréductible de $\GQp $ de dimension~\hbox{$\geqslant 2$}. 
\begin{enumerate}
\item Si $V$ est de dimension $2$ et spéciale ou cuspidale alors, en tant que $L$-représentation de $G\times \czG$, on a un isomorphisme topologique
$$
\Hom_{\sW_{\Qp}}\intnn{V}{\rmH^1_{\et}\Harg{\brv{\rmM}_{C}^{\infty}}{\Sym\BV(1)}} \cong \bPi(V)'\otimes_{L}\JL^{\lalg}(V).
$$
\item Dans tous les autres cas
$$
\Hom_{\sW_{\Qp}}\intnn{V}{\rmH^1_{\et}\Harg{\brv{\rmM}_{C}^{\infty}}{\Sym\BV(1)}} = 0.
$$
\end{enumerate}
\end{theo}
\medskip
\begin{rema}
\
\begin{itemize}
\itemb  On définit $\rmH^1_{\et}\Harg{\brv{\rmM}_{C}^{\infty}}{\Sym\BV(1)}\coloneqq \bigoplus_{k\in \BN}\varinjlim_n\rmH^1_{\et}\Harg{\brv{\rmM}_{C}^{n}}{\Sym^k_L\BV(1)}$ plutôt que la cohomologie étale prise directement dans $\Sym\BV(1)$. De plus, on se restreint ici aux poids positifs mais en considérant $\varinjlim_m\rmH^1_{\et}\Harg{\brv{\rmM}_{C}^{\infty}}{\Sym\BV(1-m)}$, on pourrait supprimer cette hypothèse.
\itemb Pour $\Sym^0_{L}\BV(1)\cong \Qp(1)$ (i.e\ pour des coefficients constants) ce résultat a été démontré dans \cite{codoni}. Le théorème \ref{thm:princ} répond à la question posée dans \cite[Remarque 0.3(ii)]{codoni} ; une différence importante est ce qui se passe en niveau $0$ où l'on voit apparaître des représentations spéciales irréductibles au lieu de caractères. Le traitement de ce cas utilise des techniques dérivées comme on l'explique au paragraphe \ref{subsubsec:lecaspeint}.
\itemb On retrouve le fait frappant de \cite[Remarque 0.3(i)]{codoni} que la cohomologie étale du système local $\rmH^1_{\et}\Harg{\brv{\rmM}_{C}^{\infty}}{\Sym\BV(1)}$ est de Hodge-Tate avec les poids attendus, comme si $\brv{\rmM}_{C}^{\infty}$ était une courbe algébrique.
\itemb En suivant \cite{codonifac}, il est probablement possible de déterminer la structure complète du $ \sWD_{\Qp}\times G\times \czG$-module $\rmH^1_{\et}\Harg{\brv{\rmM}_{\Qpbar}^{\infty}}{\Sym\BV(1)}$ sous la forme d'une décomposition factorisée à la Emerton (\cf  \cite{emecomp}). Nous espérons y revenir dans un travail ultérieur.
\end{itemize}
\end{rema}
\smallskip
La structure de la preuve est similaire à celle de \cite{codoni} sauf dans le cas spécial (\cf la section~\ref{sec:caspe})~; une différence essentielle est que l'on doit remplacer les théorèmes de comparaison pour les coefficients triviaux par des théorèmes de comparaison pour les systèmes locaux (nos systèmes locaux sont très particuliers, ce qui nous permet de nous rattacher au cas des coefficients triviaux plutôt que de prouver un théorème de comparaison à partir de zéro). Les étapes sont les suivantes~: 
\begin{itemize}
\itemb La flèche naturelle $\rmH^1_{\et}\Harg{\brv{\rmM}_{C}^{\infty}}{\Sym\BV(1)}\incl\rmH^1_{\pet}\Harg{\brv{\rmM}_{C}^{\infty}}{\Sym\BV(1)}$ identifie la cohomologie étale au sous-espace des vecteurs $G$-bornés ; en particulier, cette flèche est injective. Symboliquement, on écrira 
\begin{equation}\label{eq:Gbd}
\rmH^1_{\et}\Harg{\brv{\rmM}_{C}^{\infty}}{\Sym\BV(1)}\cong \rmH^1_{\pet}\Harg{\brv{\rmM}_{C}^{\infty}}{\Sym\BV(1)}^{G\text{-}\rmb}.
\end{equation}
L'argument est très similaire à celui pour les coefficients triviaux (\cf  \cite[Proposition 2.12]{codoni}), sur lequel il repose, mais une petite différence est que $\rmH^1_{\et}\Harg{\brv{\rmM}_{C}^{\infty}}{\Sym\BV^+(1)}$ contient de la torsion.
\itemb On a une décomposition du $L$-système local
\begin{equation}\label{eq:decint}
\Sym\BV=\bigoplus_{\lambda\in P^+}\BV_{\lambda}\otimes_L \cz{W}_{\lambda},
\end{equation}
obtenue à partir de la combinatoire des foncteurs de Schur. Dans cette décomposition, les $\BV_{\lambda}$ sont des $L$-systèmes locaux $G$-équivariants sur lesquels $\czG$ opère trivialement. On obtient donc une décomposition similaire de la cohomologie proétale.
\itemb Les $L$-systèmes locaux $\BV_{\lambda}$ sont des \emph{opers $p$-adiques isotriviaux}, ce qui permet de décrire leur cohomologie proétale en termes de formes différentielles et de la relier aux cohomologies de de Rham et de Hyodo-Kato de l'espace. Ce \emph{diagramme fondamental} reliant les différentes cohomologies sera décrit plus en détail dans la section \ref{subsec:diagfond} de l'introduction. 
\end{itemize}
Le premier point nous ramène à calculer $\Hom_{\sW_{\Qp}}\intnn{V}{\rmH^1_{\pet}\Harg{\brv{\rmM}_{C}^{\infty}}{\Sym\BV(1)}} $ dont on considère ensuite les vecteurs $G$-bornés pour obtenir le premier point du théorème \ref{thm:princ}. Le calcul de cette multiplicité dans la cohomologie proétale est extrêmement différent dans le cas spécial et cuspidal. Pour le second point du théorème \ref{thm:princ}, on suit l'argument de \cite[Théorème~5.10]{codonifac} qui est un peu moins technique que la preuve de \cite[Théorème~2.15]{codoni}. Cet argument repose sur une variante du théorème \ref{thm:princ} où on calcule une multiplicité $G$-équivariante plutôt que galoisienne (théorème \ref{thm:princ2} ci-dessous). De plus, ce théorème est une étape essentielle dans le calcul de la multiplicité d'une représentation spéciale dans la cohomologie proétale. Si~$\Pi$ est une $L$-représentation unitaire de Banach de $G$ de longueur finie,  on note $\bV(\Pi)$ la représentation de $\GQp$ qui lui est associée par le foncteur $\bV$ de Colmez (\cf \cite{colp}). Supposons que $\Pi$ est telle que ses vecteurs localement algébriques s'écrivent
$$
\Pi^{\lalg}\cong W\otimes_L\Pi^{\infty},
$$
où $W$ est une $L$-représentation algébrique de $G$ et $\Pi^{\infty}$ une $L$-représentation lisse de $G$, alors on dit que 
\begin{itemize}
\itemb $\Pi$ est \emph{cuspidale} (ou cuspidale) si $\Pi^{\infty}$ n'est pas nul et cuspidale ; ceci est équivalent à ce que $\bV(\Pi)$ soit cuspidale,
\itemb $\Pi$ est \emph{spéciale} si $\Pi^{\infty}$ est une représentation de Steinberg lisse tordue par un caractère lisse.
\end{itemize}
En particulier, si $\Pi$ est spéciale ou cuspidale, alors $\Pi^{\lalg}\neq 0$ et donc $\bV(\Pi)$ est de Rham à poids de Hodge-Tate distincts.
\medskip
\begin{theo}\label{thm:princ2}
Soit $\Pi$ une $L$-représentation de Banach unitaire admissible absolument irréductible de $G$.
\begin{enumerate}
\item Si $\Pi$ est spéciale ou cuspidale, on a un isomorphisme de $L$-représentations de $\sW_{\Qp}\times \czG$,
$$
\Hom_{G}\intnn{\Pi'}{\rmH^1_{\et}\Harg{\brv{\rmM}_{C}^{\infty}}{\Sym \BV(1)}}\cong \bV(\Pi)\otimes_{L}\JL^{\lalg}(\bV(\Pi)).
$$
\item Dans tous les autres cas,
$$
\Hom_{G}\intnn{\Pi'}{\rmH^1_{\et}\Harg{\brv{\rmM}_{C}^{\infty}}{\Sym \BV(1)}}=0.
$$
\end{enumerate}
\end{theo}
\medskip

Le calcul dans le cas cuspidal est à peu de choses près le même que celui de \cite{codoni}. Premièrement, la cohomologie de Hyodo-Kato de la tour est décrite dans \cite{codoni}. À partir de \cite{dolb}, on utilise un argument de changement de poids expliqué dans \cite{colw} pour obtenir la conjecture de Breuil-Strauch en poids supérieurs, ce qui permet de décrire le complexe de de Rham associé à $\BV_{\lambda}$ en tant que représentation de $G\times \czG$. Le calcul se fait ensuite à partir du diagramme fondamental. 
\subsubsection{Le cas spécial}\label{subsubsec:lecaspeint}
Le calcul dans le cas spécial est très différent. 

Tout d'abord, on utilise le résultat de la thèse de Schraen (\cf  \cite{sch}) pour obtenir la multiplicité dérivée du dual des vecteurs localement analytiques d'une représentation de la forme $\bPi(V)$, pour~$V$ une représentation spéciale, dans le complexe de de Rham. Définissons cette catégorie dérivée. On note $\sD(G,L)$ l'anneau des distributions localement analytiques sur $G$ à coefficients dans $L$ et on considère la catégorie dérivée des $\sD(G,L)$-modules à caractère central fixé ; on note alors $\bHom\ $ les homorphismes dans cette catégorie dérivée.

Pour $\Pi$ une  $L$-représentation de $G$ on note $\Pi^{\lan}$ la sous-représentation des vecteurs localement analytiques et $(\Pi^{\lan})'$ son dual stéréotypique, qui est un $\sD(G,L)$-module. On note $(\Pi^{\lan})'[-1]$ l'objet de cette catégorie dérivée concentré en degrée $1$ et on note $\rmR\Gamma_{\!\dR}\Harg{\brv{\rmM}^{\infty}_{C}}{\Sym \BV(1)}$ le complexe de de Rham de la tour de Drinfeld associé au système local ; ce complexe définit un élément de la catégorie dérivée en question. La version du résultat de Schraen que l'on démontre est la suivante. 
\medskip
\begin{prop}\label{prop:entrederint}
Soit $V$ une $L$-représentation spéciale de $\GQp$. Alors, on a un isomorphisme de $(\varphi,N)$-modules filtrés munis d'une action de $\sWD_{\Qp}\times\czG$
$$
\bHom\ \intnn{(\bPi(V)^{\lan})'[-1]}{\rmR\Gamma_{\!\dR}\Harg{\brv{\rmM}^{\infty}_{C}}{\Sym \BV(1)}}\cong \bD_{\pst}(V)\otimes_L \JL^{\lalg}(V).
$$
\end{prop}
\medskip
\begin{rema}
\begin{itemize}
\itemb La structure de $(\varphi,N)$-module filtré sur l'entrelacement dérivé provient du complexe de de Rham. L'isomorphisme de Hyodo-Kato munit le complexe de de Rham d'une structure de $(\varphi,N)$-module filtré (dérivé) et la filtration est la filtration de de Rham.
\itemb L'action de $\czG$ apparaît ici naturellement contrairement à \cite{sch}. Par ailleurs, la partie algébrique provient du système local, et considérer la tour $\brv{\rmM}^{\infty}_{C}$ plutôt que le demi-plan permet d'obtenir les torsions par les caractères lisses de $\czG$.
\itemb En réalité, ce ne sont pas vraiment les vecteurs localement analytiques de $\bPi(V)$ qui apparaissent dans les calculs de Schraen, mais une représentation plus petite. Lorsque $V$ est spéciale, $\bPi(V)^{\lan}$ a $4$ composantes de Jordan-Hölder. Schraen considère une $L$-représentation de $G$ faisant intervenir $3$ de ces $4$ composantes et on montre que la dernière composante (une série principale localement analytique) n'intervient pas dans la cohomologie, ce qui permet d'adapter le résultat comme déjà annoncé dans \cite[Remarque 5.5]{sch}. 
\end{itemize}
\end{rema}
On utilise le théorème de comparaison proétale-syntomique à coefficients isotriviaux pour se ramener à un calcul syntomique (le complexe syntomique est construit à partir du complexe de Hyodo-Kato et du complexe de de Rham par une version dérivée de la recette de Fontaine (\cf  \cite{cofo}) pour construire une représentation galoisienne à partir d'un $(\varphi,N)$-module filtré). On déduit de la proposition \ref{prop:entrederint} le résultat suivant : 
\medskip
\begin{theo}\label{thm:speder}
Soit $V$ une $L$-représentation spéciale de $\GQp$, on a un isomorphisme de $L$-représentations de $\sW_{\Qp}\times \czG$
$$
\bHom\ \intnn{\intn{\bPi(V)^{\lan}}'[-1]}{\rmR\Gamma_{\!\pet}\Harg{\brv{\rmM}^{\infty}_{C}}{\Sym \BV(1)}}\cong V\otimes_L \JL^{\lalg}(V).
$$
\end{theo}
\medskip
Enfin, en considérant les vecteurs $G$-bornés et en invoquant (\ref{eq:Gbd}), on obtient le théorème suivant, où $\bHom\ $ est cette fois considéré dans la catégorie dérivée des $L[G]$-modules topologiques à caractère central fixé :
\medskip
\begin{theo}\label{thm:speder}
Soit $V$ une $L$-représentation spéciale de $\GQp$, on a un isomorphisme de $L$-représentations de $\sW_{\Qp}\times \czG$
$$
\bHom\ \intnn{\bPi(V)'[-1]}{\rmR\Gamma_{\!\et}\Harg{\brv{\rmM}^{\infty}_{C}}{\Sym \BV(1)}}\cong V\otimes_L \JL^{\lalg}(V).
$$
\end{theo}
\medskip
\begin{rema}
\begin{itemize}
\itemb Si $w(\lambda)\neq1$, on a
$$
\rmH^0_{\pet}\Harg{\brv{\rmM}^{\infty}_{C}}{\BV_{\lambda}(1)}=0,
$$
et dans ce cas le premier point du théorème \ref{thm:princ2} est équivalent au théorème \ref{thm:speder}.
\itemb Si $w(\lambda)=1$, ce qui correspond au cas des coefficients triviaux (à torsion près par un caractère), le théorème \ref{thm:speder} est un analogue d'un résultat de Dat (\cf \cite{dat}) selon lequel les représentations réductibles sont encodées dans le complexe et non dans un seul groupe de cohomologie. On obtient donc un renforcement de \cite{codoni} puisque les méthodes dérivées font sortir des représentations réductibles de la bonne dimension au lieu des caractères.
\end{itemize}
\end{rema}

\Subsection{Opers $p$-adiques isotriviaux sur les courbes Stein}
Discutons maintenant de la première partie de cet article, qui concerne la cohomologie de certains systèmes locaux $p$-adiques sur les courbes Stein sur $K$ ; \og Stein sur $K$ \fg signifie \og rigide Stein lisse sur $K$\fg. On se restreint à une classe très restrictive de systèmes locaux dont on sait définir une cohomologie de Hyodo-Kato et donc une cohomologie syntomique qui se compare à la cohomologie proétale du système local.

Soit $\rmO_K$ un anneau à valuation discrète complet de caractéristique mixte. On note $K$ son corps des fractions, $k \coloneqq \rmO_K/\fkm_K$ son corps résiduel que l'on suppose parfait et $K_0\subset K$ la sous-extension maximale non-ramifiée. On note $\Bdrp$ l'anneau des périodes de de Rham et $\Bdr\coloneqq\Frac(\Bdrp)$ son corps des fractions. C'est un anneau filtré muni d'une action de $\GQp $. La filtration est donnée par une uniformisante $t\in \Fil^1\Bdrp$ qui est définie à partir de $\varepsilon =(1,\zeta_p,\zeta_{p^2},\dots)$, une suite compatible de racines $p^n$-ièmes de l'unité, par $t=\log([\varepsilon])$. De même, $\Bcrisp$ désigne l'anneau des périodes cristallines : il est muni d'un Frobenius $\varphi$ et d'une action de $\GQp $. On a $\Bstp=\Bcrisp[u]$ et on étend le Frobenius par $\varphi(u)=pu$ et on définit la monodromie par $N(u)=-1$. On définit le morphisme $\iota\colon \Bstp\hookrightarrow \Bdrp$ par $u\mapsto u_{p}=\log([p^{\flat}]/p)$ et ce plongement munit $\Bstp$ d'une action de $\GQp $. Ces anneaux admettent des versions faisceautiques pour la topologie proétale qui sont définies dans \cite{schphdg} et que l'on utilisera librement dans ce texte.
\subsubsection{Systèmes locaux isotriviaux}

Soit $X$ un schéma formel semi-stable sur $\rmO_K$, on note $X_K$ l'espace rigide associé que l'on suppose lisse sur $K$ et $X_k$ la fibre spéciale de $X$. Soit $\BV$ un $\Qp$-système local proétale de Rham sur $X_K$. D'après Scholze et Brinon (\cf  \cite{schphdg} et \cite{bri}), on associe à $\BV$ un \emph{fibré plat filtré} $\CE=(\CE,\nabla,\Filb)$ (\cf définition \ref{def:fibplafilt}).

On note $\BBcris\coloneqq \XBBcris{X_K}$ le faisceau proétale des périodes cristallines sur $X_K$ (\cf \cite{tantong}). On dit que~$\BV$ est \emph{isotrivial cristallin} s'il existe un isocristal $D$ sur $k$ et un isomorphisme $\varphi$-équivariant, de faisceaux proétales sur $X_K$,
$$
\BV\otimes_{\Qp}\BBcris\cong D\otimes_{K_0} \BBcris.
$$
Dans ce cas, $\CE\cong D\otimes_{K_0}\CO_{X_K}$ et donc les seules données qui varient sont la filtration et la connexion sur~$\CE$. Sous l'isomorphisme $\CE\cong D\otimes_{K_0}\CO_{X_K}$, si la connexion prend la forme $\nabla = \id\otimes d$, on dit que le système local est \emph{fortement isotrivial}. L'intérêt est que les systèmes locaux universels sur les espaces de Rapoport-Zink sont fortement isotriviaux : c'est une réinterprétation d'un théorème de Rapoport et Zink (\cf \cite[Proposition 5.15]{razi}). En particulier, ce sont ces exemples qui ont motivé notre définition d'isotrivial. On démontre que les systèmes locaux isotriviaux forment une catégorie Tannakienne pour le foncteur fibre $\BV\rightsquigarrow D$, ce qui permet de considérer des produits tensoriels et des foncteurs de Schur de systèmes locaux isotriviaux. 

Pour associer une suite exacte fondamentale à $\BV$, on définit le faisceau proétale
$$
\BX_{\cris}(D)\coloneqq \intn{D\otimes_{K_0} \BBcris}^{\varphi=1},
$$
que l'on peut voir comme un espace de Banach-Colmez relatif. Pour la partie de Rham, on a le faisceau des périodes $\BBdr$ et on a déjà mentionné la construction de Scholze et Brinon $\BV\rightsquigarrow \CE=(\CE,\nabla, \Filb)$. À partir de là, on peut définir le faisceau proétale
$$
\BX_{\dR}(\CE)\coloneqq D\otimes_{K_0} \BBdr,
$$
qui est muni d'une filtration introduite par celle de $\CE$ et que l'on peut définir à partir de $\OBBdr$, le gros faisceau des périodes. On a une inclusion $\BX_{\cris}(D)\incl \BX_{\dR}(\CE,\nabla,\Filb)$.
\medskip 
\begin{prop}\label{prop:isot1}
Soit $\BV$ un système local étale $p$-adique fortement isotrivial sur $X_K$. Soit $D$ le $\varphi$-module sur $K_0$ associé à $\BV$ et $(\CE,\nabla,\Filb)$ le fibré plat filtré associé à $\BV$. Alors, on a une suite exacte de faisceaux proétales sur $X_K$
$$
0 \rightarrow \BV\rightarrow \BX_{\cris}(D)\rightarrow \BX_{\dR}(\CE)/\Fil^0\rightarrow 0.
$$
\end{prop}
\medskip

\subsubsection{Cohomologie syntomique géométrique}
On garde les notations précédentes et on suppose que $X_K$ est Stein sur $K$. On considère $\BV$ un système local proétale fortement isotrivial sur $X_K$ à poids de Hodge-Tate positifs\footnote{Ceci signifie que les sauts de la filtration sont en degrés négatifs.}, d'isocristal associé $D$ et de fibré plat filtré $(\CE,\nabla,\Filb)$. Dans tous les cas on peut définir un complexe de de Rham $\rmR\Gamma_{\!\dR}\Harg{X_K}{\CE}$ qui est naturellement filtré, puis l'isotrivialité forte donne un isomorphisme 
$$
\rmR\Gamma_{\!\dR}\Harg{X_K}{\CE}\cong\rmR\Gamma_{\!\dR}(X_K)\otimes_{K_0}D.
$$ 
Prenons garde que cet isomorphisme n'est pas filtré. Le \emph{complexe de Hyodo-Kato}, défini comme le complexe de cohomologie rationnelle log-cristallin $\rmR\Gamma_{\!\HK}(X_K)\coloneqq\rmR\Gamma_{\!\cri}(X_k/\rmO^{0}_{K_0})\otimes_{\rmO_{K_0}}K_0$ où $\rmO^{0}_{K_0}$ est le log-schéma formel $\Spf(\rmO_{K_0})$ muni de la structure logarithmique induite par $\BN\rightarrow \rmO_{K_0}$, $1\mapsto 0$. Le complexe $\rmR\Gamma_{\!\HK}(X_K)$ est muni d'un Frobenius, d'un opérateur de monodromie et comme la notation le suggère, ne dépend pas du modèle $X$ (\cf \cite{colniz}). On peut définir le \emph{complexe de Hyodo-Kato isotrivial}
$$
\rmR\Gamma_{\!\HK}\Harg{X_K}{D}\coloneqq\rmR\Gamma_{\!\HK}(X_K)\otimes_{K_0}D,
$$
muni du Frobenius produit et de l'opérateur de monodromie provenant de $\rmR\Gamma_{\!\HK}(X_K)$. On récupère gratuitement un isomorphisme de Hyodo-Kato isotrivial à partir de l'isomorphisme de Hyodo-Kato usuel (cf. \ref{subsubsec:hkisot}), \ie 
$$
\rmR\Gamma_{\!\HK}\Harg{X_K}{D}\otimes_{K_0}K\cong\rmR\Gamma_{\!\dR}\Harg{X_K}{\CE}.
$$
À partir de là, on est en mesure de définir un complexe syntomique en suivant \cite{codoniste} comme la fibre 
$$
\rmR\Gamma_{\!\syn}\Harg{X_C}{\BV}\coloneqq \left [ \left [ \rmR\Gamma_{\!\HK}\Harg{X_K}{D}\wotimes_{K_0}^R\whBstp \right ]^{N=0,\varphi=1}\xrightarrow{\iota_{\HK}\otimes \iota} \intn{\rmR\Gamma_{\!\dR}\Harg{X_K}{\CE}\wotimes_K^R\Bdrp}/\Fil^{0}\right].
$$
Cette cohomologie s'inscrit naturellement dans un diagramme, que l'on décrira dans un cas particulier au paragraphe suivant : on peut la comparer à la cohomologie proétale du système local, ce qui donne une méthode pour la calculer. Le théorème de comparaison que l'on obtient est le suivant :
\medskip
\begin{theo}\label{thm:compint}
Supposons que $X_K$ soit Stein. Soit $i\geqslant 0$ un entier. Il existe, dans\footnote{On note $\rmC_K$ la catégorie des espaces vectoriels topologiques localement convexes sur $K$ et $\sD(\rmC_K)$ la $\infty$-catégorie associée.} $\sD(\rmC_{\Qp})$ un morphisme
$$
\alpha_{\BV(i)}\colon\rmR\Gamma_{\!\syn}\Harg{X_C}{\BV(i)}\rightarrow \rmR\Gamma_{\!\pet}\Harg{X_C}{\BV(i)},
$$
qui est un quasi-isomorphisme strict après troncation par $\tau_{\leqslant i}$. En particulier, on a des isomorphismes topologiques
$$
\rmH^j_{\syn}\Harg{X_C}{\BV(i)}\cong\rmH_{\pet}^j\Harg{X_C}{\BV(i)},
$$
pour tout entier $j$ tel que $0\leqslant j \leqslant i$.
\end{theo}
\medskip
On déduit ce résultat des théorèmes de Bosco (\cf  \cite{bosc1} et \cite{bosc2}). Remarquons que, contrairement à \cite{codoniste}, on ne définit pas la cohomologie syntomique de manière surconvergente ; celle-ci est cachée dans ce que comme $X$ est Stein, $\rmR\Gamma_{\!\HK}(X_K)$ s'identifie à la cohomologie rationnelle surconvergente rigide de $X_k$ (\cf lemme \ref{lemm:hksurconv}).

\subsubsection{Le diagramme fondamental pour les opers}
On suppose maintenant que $X_K$ est une courbe Stein géométriquement connexe. On note 
$$
D_K\coloneqq D\otimes_{K_0}K,\quad \Omega_{\CE}^1\coloneqq \CE\otimes_{\CO}\Omega^1, \quad \rmB_k\coloneqq \Bdrp/t^k,\quad k\geqslant 0.
$$
Notons $\CE_{\lambda}$ le fibré plat filtré associé au système local $\BV_{\lambda}$ apparaissant dans la décomposition~(\ref{eq:decint}). Schneider-Stuhler (\cf \cite[5]{scst}) calculent les connexions induites sur les gradués de $\CE_{\lambda}$ et montrent que ce sont des isomorphismes sur les gradués intermédiaires. Cette condition apparaît dans la littérature sous le nom \emph{d'oper}, formalisée par Beilinson-Drinfeld (\cf\cite{bedr}), fondamentale dans la correspondance de Langlands géométrique : les opers permettent par exemple de décrire le centre de l'algèbre enveloppante universelle d'une algèbre de Kac-Moody.

D'après Beilinson-Drinfeld, si $(\CE,\nabla,\Filb)$ est un oper de poids $(a,b)$ (ces entiers correspondent aux sauts extrémaux de la filtration) on lui associe un couple de fibrés en droites, muni d'un opérateur différentiel d'ordre $b-a+1$, soit\footnote{$\Gr_i$ désigne ici le $(-i)$-ième gradué. On a préféré cette convention pour éviter d'alourdir tous les indices avec un signe $-$ redondant.}$ (\Gr_b\CE,\Gr_a\Omega_{\CE},L_{\nabla})$ ; alors, si on définit
$$
\rmR\Gamma_{\!\Op}\Harg{X_K}{\CE}\coloneqq \big ( \Gr_b\CE\xrightarrow{L_{\nabla}} \Gr_a\Omega_{\CE}\big ) ,
$$
muni de la filtration induite, on a un quasi-isomorphisme filtré $\rmR\Gamma_{\!\Op}\Harg{X_K}{\CE}\cong \rmR\Gamma_{\!\dR}\Harg{X_K}{\CE}$. Pour $\CE_{\lambda}$ ce complexe est introduit dans \cite[5]{scst} sous le nom de \emph{complexe de de Rham réduit} et Schneider-Stuhler démontrent le quasi-isomorphisme filtré dans ce cas. 

Si $(\CE,\nabla,\Filb)$, associé à un système local $\BV$, est un oper on dit simplement que $\BV$ est un \emph{$\Qp$-oper}. Dans cette situation, on obtient la proposition suivante :
\medskip
\begin{prop}\label{prop:diagfondopint}
Soit $\BV$ un $\Qp$-oper de poids $(a,b)$ fortement isotrivial d'isocristal associé~${D}$ et de fibré associé $\CE=(\CE,\nabla,\Filb)$. Supposons que $a\geqslant 0$. Avec les mêmes notations que précédemment, on a une application naturelle entre suites exactes strictes~:
{\small
\begin{center}
\begin{tikzcd}[column sep=small]
t^aX_{\cri}^{1}(D[a])\ar[d]\ar[r]&\Gr_b\CE(X_K)\wotimes_K t^a\rmB_{b+1} \ar[r]\ar[d,equal]& \rmH_{\pet}^1\Harg{X_C}{\BV(1)}\ar[r]\ar[d]&t^aX_{\st}^{1}\intn{\rmH^1_{\HK}\Harg{X_K}{D}[a]}\ar[r]\ar[d,"\gamma"]&0\\
D_K\otimes_Kt^a\rmB_{b+1} \ar[r]& \Gr_b\CE(X_K)\wotimes_Kt^a\rmB_{b+1} \ar[r,"L_{\nabla}"]&\Gr_a\Omega_{\CE}(X_K)\wotimes_Kt^a\rmB_{b+1}\ar[r]&\rmH^1_{\dR}\Harg{X_K}{\CE}\wotimes_K t^a\rmB_{b+1}\ar[r]&0
\end{tikzcd}
\end{center}}
\noindent Ce diagramme est $\sG_K$-équivariant et on a noté :
$$
\begin{gathered}
X_{\st}^{1}\intn{\rmH^1_{\HK}\Harg{X_K}{D}[a]}\coloneqq  \intn{\rmH^1_{\HK}\Harg{X_K}{D}\otimes_{K_0}\Bstp}^{N=0,\varphi=p^{1-a}}, \\
X_{\cri}^{1}(D[a])\coloneqq  \intn{D\otimes_{K_0}\Bcrisp}^{\varphi=p^{1-a}}.
\end{gathered}
$$
De plus, les flèches verticales sont strictes d'images fermées et 
$$
\Ker(\gamma) = t^{a+b+1}\intn{\rmH^1_{\HK}\Harg{X_K}{D}\otimes_{K_0}\Bstp}^{N=0,\varphi=p^{-a-b}}.
$$
\end{prop}
\medskip
\begin{rema}
\begin{itemize}
\itemb Il n'est pas tout à fait clair comment obtenir un diagramme similaire pour les espaces Stein de dimension supérieure par la même méthode. De plus, on peut définir des $\BV_{\lambda}$  pour $\GL_n(\Qp)$ avec $n>2$, mais ce ne sont plus des opers si on recopie naïvement la définition au delà des courbes.
\itemb Une différence avec le cas des coefficients triviaux est que les images de $t^aX_{\cri}^{1}(D[a])$ et $D_K\otimes_Kt^a\rmB_{b+1}$ par les flèches horizontales de gauche ne sont pas forcément les mêmes. C'est cette différence qui explique la différence du traitement du cas spécial.
\end{itemize}
\end{rema}
\subsection{La cohomologie proétale de l'espace de Drinfeld}\label{subsec:intproco}
Pour finir cette introduction on va présenter ce que donnent les calculs de la section précédente dans le cas de l'espace de Drinfeld avant de discuter comment l'invariant $\CL$, qui encode la filtration, apparaît dans ce diagramme. Pour le diagramme, on commence par considérer $\presp{\rmM}_{\Qp}^n$ un modèle sur $\Qp$ du quotient $\brv{\rmM}_{\Qpbr}^n/p^{\BZ}$ de l'espace de Drinfeld; on note $\presp{\rmM}_{C}^n\coloneqq \presp{\rmM}^n_{\Qp}\otimes_{\Qp}C$. On se restreint au système local $\BV_{\lambda}$ sur $\presp{\rmM}_{\Qp}^n$ où on a fixé $\lambda\in P_+$. Comme évoqué plus haut, c'est un $L$-oper isotrivial de poids $(\lambda_1,\lambda_2-1)$. On rappelle que l'on note $\CE_{\lambda}$ le fibré plat filtré associé à $\BV_{\lambda}$, $\Omega_{\lambda}\coloneqq \CE_{\lambda}\otimes_{\CO} \Omega^1_{\presp{\rmM}_{\Qp}^n}$ et de plus
$$
\omega_{\lambda}^{\infty}\coloneqq\Gr_{\lambda_1}{\Omega}_{\lambda}(\presp{\rmM}_{\Qp}^{\infty}),\quad \CO_{\lambda}^{\infty}\coloneqq\Gr_{\lambda_2-1}\CE_{\lambda}(\presp{\rmM}_{\Qp}^{\infty}).
$$
Ces notations sont motivées par l'idée que $\CO_{\lambda}^{\infty}$ et $\omega_{\lambda}^{\infty}$ sont respectivement les fonctions et les formes différentielles rigides analytiques sur la tour mais où l'action de $G$ est tordue par un facteur d'automorphie. Dans le cas spécial, à torsion par un caractère lisse près, seule la restriction de ces faisceaux à $\BH_{\Qp}$ intervient. On note simplement
$$
\omega_{\lambda}\coloneqq\Gr_{\lambda_1}\Omega_{\lambda}(\BH_{\Qp}),\quad \CO_{\lambda}\coloneqq\Gr_{\lambda_2-1}\CE_{\lambda}(\BH_{\Qp}).
$$
De plus, l'isomorphisme de Morita s'écrit $\omega_{\lambda}\cong \intn{{\St}^{\lan}_{\lambda}}'$ où ${\St}^{\lan}_{\lambda}$ est une Steinberg localement analytique. On considère les vecteurs localement algébriques ${\St}^{\lalg}_{\lambda}\incl {\St}^{\lan}_{\lambda}$ dont la surjection duale correspond alors à la surjection sur la cohomologie de de Rham de $\CE_{\lambda}$, \ie
$$
\intn{{\St}^{\lan}_{\lambda}}'\rightarrow \intn{{\St}^{\lalg}_{\lambda}}'\cong \rmH^1_{\dR}\Harg{\BH_{\Qp}}{\CE_{\lambda}}.
$$
\subsubsection{Le diagramme fondamental pour l'espace de Drinfeld}\label{subsec:diagfond}
On fixe $M$ un $L$-$(\varphi,N,\GQp)$-module de rang $2$ qui est spécial ou cuspidal. Au début du numéro \ref{subsec:intcolang}, on a défini $\LL^{\lalg}(V)$ pour $V$ une représentation de $\GQp$ mais cette construction ne dépend que des poids et du $L$-$(\varphi,N,\GQp)$-module $M=\bD_{\pst}(V)[1]$ ; ainsi, on définit la $L$-représentation localement algébrique de $\LL_M^{\lambda}$ de $G$ par la même recette. De même, on a $\JL_M$, la $L$-représentation lisse de $\czG$ obtenue par la correspondance de Jacquet-Langlands ; on ne considère pas la représentation localement algébrique puisque on a déjà traité la partie algébrique à l'aide de (\ref{eq:decint}). On définit un foncteur sur les $L[\czG]$-modules 
$$
X\mapsto X[M]\coloneqq \Hom_{L[\czG]}\intnn{\JL_M}{X},
$$
que l'on va appliquer au diagramme pour traiter la partie $\czG$-équivariante. De plus, on note 
$$
M_{\dR}\coloneqq\intn{M\otimes_{\Qp^{\nr}}C}^{\GQp},\quad X_{\st}^{1}(M[\lambda_1])\coloneqq (M\otimes_{\Qp^{\nr}}\Bstp)^{N=0,\varphi=p^{1-{\lambda_1}}}.
$$
Alors, $M_{\dR}$ est un $L$-module libre de rang $2$. Pour alléger le diagramme, on définit pour $\lambda\in P_+$, les espaces de Banach-Colmez
$$
\rmB_{\lambda}\coloneqq t^{\lambda_1}\Bdrp/t^{\lambda_2}\Bdrp,\quad \rmU_{\lambda}^0\coloneqq t^{\lambda_1}\intn{\Bcrisp}^{\varphi^2=p^{w(\lambda)+1}},\quad \rmU_{\lambda}^1\coloneqq t^{\lambda_1}\intn{\Bcrisp}^{\varphi^2=p^{w(\lambda)-1}}.
$$
Notons que les $\rmU_{\lambda}^0$ et $\rmU_{\lambda}^1$ sont naturellement des $\intn{\Bcrisp}^{\varphi^2=1}\cong \Qpd$-espaces vectoriels où $\Qpd$ est l'extension quadratique non ramifiée de $\Qp$. 
\medskip
\begin{theo}\label{thm:diagfondrint}
Soit $\lambda\in P_+$ et soit $M$ un $L$-$(\varphi,N,\GQp)$-module de rang $2$ absolument indécomposable. Quitte a tordre $M$ par un caractère on peut supposer que 
$$
\rmH^1_{\pet}\Harg{\brv{\rmM}_{C}^{\infty}}{\BV_{\lambda}(1)}[M]=\rmH^1_{\pet}\Harg{\presp{\rmM}_{C}^{\infty}}{\BV_{\lambda}(1)}[M].
$$
\begin{itemize}
\itemb Si $M$ est cuspidal on a un diagramme commutatif à lignes exactes d'espaces de Fréchet $\sW_{\Qp}\times G$-equivariants
{
\begin{center}
\begin{tikzcd}[column sep=small]
0\ar[r]&\rmB_{\lambda}\wotimes_{\Qp}\CO_{\lambda}^{\infty}[M]\ar[r]\ar[d,equal]&\rmH^1_{\pet}\Harg{\presp{\rmM}_{C}^{\infty}}{\BV_{\lambda}(1)}[M]\ar[r]\ar[d,hook]&t^{\lambda_1}X_{\st}^{1}(M[\lambda_1])\wotimes_L{\LL_M^{\lambda}}'\ar[r]\ar[d,hook]&0\\
0\ar[r]&\rmB_{\lambda}\wotimes_{\Qp}\CO_{\lambda}^{\infty}[M]\ar[r]&\rmB_{\lambda}\wotimes_{\Qp}\omega_{\lambda}^{\infty}[M]\ar[r]&(\rmB_{\lambda}\otimes_{\Qp}M_{\dR})\wotimes_L{\LL_M^{\lambda}}'\ar[r]&0
\end{tikzcd}
\end{center}}
et les flèches verticales sont injectives.
\itemb Si $M$ est spécial, alors $\rmH^1_{\pet}\Harg{\presp{\rmM}^{\infty}_{C}}{\BV_{\lambda}(1)}[M]\cong\rmH^1_{\pet}\Harg{\BH_{C}}{\BV_{\lambda}(1)}\otimes \chi_M$, où $\chi_M$ est le caractère de $\sW_{\Qp}\times G$ défini par $M$ et on a
{
\begin{center} 
\begin{tikzcd}[column sep=small]
&\rmU_{\lambda}^0\otimes_{\Qpd}{W}_{\lambda} \ar[d]\ar[r]&\rmB_{\lambda}\wotimes_{\Qp}\CO_{\lambda}\ar[r]\ar[d,equal]&\rmH^1_{\pet}\Harg{\BH_{C}}{\BV_{\lambda}(1)}\ar[r]\ar[d]&\rmU_{\lambda}^1\wotimes_{\Qpd}\intn{{\St}^{\lalg}_{\lambda}}'\ar[r]\ar[d,hook] &0\\
 0\ar[r]&\rmB_{\lambda}\otimes_{\Qp}{W}_{\lambda}\ar[r]&\rmB_{\lambda}\wotimes_{\Qp}\CO_{\lambda}\ar[r] & \rmB_{\lambda}\wotimes_{\Qp}\intn{{\St}^{\lan}_{\lambda}}'\ar[r]& \rmB_{\lambda}\wotimes_{\Qp}\intn{{\St}^{\lalg}_{\lambda}}'\ar[r]& 0
\end{tikzcd}
\end{center}}
et les premières flèches, horizontale et verticale, sont injectives si et seulement si $w(\lambda)\neq1$.
\end{itemize}
De plus, dans les deux diagrammes, les flèches verticales sont d'images fermées.
\end{theo}
\medskip
\subsubsection{L'invariant $\CL$}
À partir du $L$-$(\varphi,N,\GQp)$-module $M$ et d'une filtration sur $M_{\dR}$ on peut reconstruire la représentation $V$ par la recette de Fontaine. Comme on est en rang~$2$, cette filtration est caractérisée par un poids (régulier) et une droite dans $M_{\dR}$ que l'on appelle \emph{l'invariant~$\CL$ galoisien}. Côté automorphe, cet invariant n'apparaît pas directement dans $\bPi(V)$ mais dans~$\bPi(V)^{\lan}$ : cette représentation localement analytique est réductible et les extensions entre les constituants de Jordan-Hölder sont caractérisées par un paramètre que l'on appelle \emph{l'invariant~$\CL$ automorphe}. On peut retrouver $\bPi(V)$ à partir de $\bPi(V)^{\lan}$ en prenant son complété unitaire universel.

Dans le cas cuspidal, $\bPi(V)^{\lan}$ a deux constituants de Jordan-Hölder : le socle est constitué des vecteurs localement algébriques, \ie $\LL_M^{\lambda}$ et le cosocle est le dual de $\CO_{\lambda}^{\infty}[M]$. L'extension entre ces deux composantes est paramétrée par l'invariant $\CL$ automorphe. Dans la ligne du bas du premier diagramme du théorème \ref{thm:diagfondrint} on voit apparaître le dual de l'extension universelle et il s'agit de montrer que lorsqu'on fixe une droite $\CL\subset M_{\dR}$, l'extension obtenue est exactement celle de paramètre $\CL$ (\cf proposition \ref{prop:invlcusp}).

Dans le cas spécial, $\bPi(V)^{\lan}$ a quatre constituants de Jordan-Hölder dont une série principale qui ne joue aucun rôle. La représentation $\bPi(V)^{\lan}$ est complètement déterminée par son invariant $\CL$ automorphe \ie une droite de l'espace $\Ext^1_G\intnn{{W}^*_{\lambda}}{{\St}^{\lan}_{\lambda}}$ qui est de dimension $2$. Dans le second diagramme du théorème \ref{thm:diagfondrint} le dual de l'extension universelle apparaît dans la flèche horizontale du milieu \ie deux copies du dual de ${W}^*_{\lambda}$ apparaissent dans le noyau de l'application 
$$
\rmH^1_{\pet}\Harg{\BH_{C}}{\BV_{\lambda}(1)}\rightarrow \rmB_{\lambda}\wotimes_{\Qp}\intn{{\St}^{\lan}_{\lambda}}'.
$$
Dans cette extension fournie par la cohomologie, il s'agit d'identifier l'invariant $\CL$ automorphe à l'invariant $\CL$ galoisien (\cf théorème \ref{thm:mugalspe}).
\Subsection{Plan de l'article}
Cet article est constituée de deux parties faisant appel à des techniques indépendantes : la première traite des systèmes locaux isotriviaux sur les espaces rigides et leur cohomologie ; la seconde concerne le calcul de la cohomologie étale et proétale de la tour de Drinfeld à coefficients dans le système local universel.

\noindent La première partie est constituée de cinq sections : 
\begin{itemize}
\itemb la section $2.$ est consacrée à des rappels sur les systèmes locaux étales $p$-adiques  sur les espaces rigides,
\itemb dans la section $3.$ on définit les systèmes locaux (fortement) isotriviaux et on montre la proposition \ref{prop:isot1} puis on décrit l'exemple des groupes $p$-divisibles sur les espaces de Rapoport-Zink,
\itemb la section $4.$ est consacrée aux cohomologies de de Rham et de Hyodo-Kato isotriviales,
\itemb dans la section $5.$ on définit la cohomologie syntomique et on démontre le diagramme fondamental de la proposition \ref{prop:diagfondopint},
\itemb dans la section $6.$ on démontre le théorème de comparaison \ref{thm:compint}.
\end{itemize}

\noindent La seconde partie est constituée de sept sections :
\begin{itemize}
\itemb dans la section $7.$ on fait quelques rappels sur la tour de Drinfeld et le système local universel,
\itemb  la section $8.$ est consacrée à des rappels sur les représentations de $G$ et $\czG$ : les représentations algébriques, les représentations de Steinberg et la série spéciale localement analytique,
\itemb dans la section $9.$ on définit les cohomologies qu'on va étudier et on les décompose en des morceaux plus simples ; en particulier on relit les torsions aux composantes connexes et on démontre la décomposition (\ref{eq:decint}) puis réinterprète les calculs de Schneider-Stuhler pour établir, en partie, le théorème \ref{thm:diagfondrint},
\itemb on démontre l'isomorphisme (\ref{eq:Gbd}) dans la section $10.$,
\itemb dans la section $11.$, on démontre le cas cuspidal du théorème \ref{thm:diagfondrint} avant de calculer la multiplicité des représentations cuspidales dans la cohomologie proétale du système local ; on en déduit le cas cuspidal des premiers points des théorèmes \ref{thm:princ} et \ref{thm:princ2},
\itemb dans la section $12.$, on démontre le cas spécial des premiers points théorèmes \ref{thm:princ} et \ref{thm:princ2} en suivant ce qu'on a expliqué dans le paragraphe \ref{subsubsec:lecaspeint},
\itemb la section $13.$ est consacrée à la preuve du second point du théorème \ref{thm:princ}, ce qui complète sa démonstration.
\end{itemize}
\Subsection{Remerciements}
Cet article est issu de ma thèse que j'ai réalisée sous la direction de Pierre Colmez, que je remercie vivement pour son soutien indéfectible et sa patience. Je remercie aussi les rapporteurs, Matthew Emerton et Benjamin Schraen dont les retours éclairés ont grandement amélioré la qualité de ce texte. 
\cleardoublepage
\part{Cohomologie $p$-adique à coefficients isotriviaux}
Soit $\rmO_K$ un anneau à valuation discrète, complet et de caractéristique mixte. On note $\fkm_K\subset \rmO_K$ son idéal maximal, $k \coloneqq \rmO_K/\fkm_K$ son corps résiduel que l'on suppose parfait, $K$ son corps des fractions et $\pi\in \fkm_K$ une uniformisante de $K$. On note de plus $\rmO_{K_0}\coloneqq \rmW(k)$ l'anneau des vecteurs de Witt de $k$. Son idéal maximal est engendré par $p\in\rmO_{K_0}$ et son corps résiduel est $k$. On note $K_0$ son corps des fractions qui est muni d'un automorphisme de Frobenius $\sigma \colon K_0\rightarrow K_0$. On suppose que la valuation sur $K$ est la valuation $p$-adique $v_p\colon \tim{K} \rightarrow \BQ$, normalisée par $v_p(p)=1$. On note $C= \wh{\ov{K}}$ le complété d'une clôture algébrique de $K$, $\rmO_C$ son anneau des entiers et $\fkm_C$ son idéal maximal. Le corps résiduel $k_C=\rmO_C/\fkm_C$ est algébriquement clos.

Soit $X$ un schéma formel sur $\rmO_K$, $X_K$ l'espace rigide associé sur $K$ et $X_k$ la fibre spéciale de $X$. Dans cette partie on suppose toujours que $X$ est semi-stable et que $X_K$ est lisse sur $K$. On note $X_C\coloneqq X \otimes_K C$ l'extension des scalaires à $C$. De plus, $X$ sera souvent une courbe Stein, qui est le cas qui nous intéresse pour la suite.


On note $\rmC_K$ la catégorie des $K$-espaces vectoriels localement convexes, qui est une catégorie semi-abélienne. Rappelons qu'un morphisme est dit \emph{strict} s'il est relativement ouvert. On note $\sD(\rmC_K)$ la $\infty$-catégorie dérivée bornée à gauche associée, construite à partir de la catégorie des complexes $\sC(\rmC_K)$, et on note $\rmD(\rmC_K)$ la catégorie homotopique de $\sD(\rmC_K)$. On renvoit à \cite[2.1]{codoniste} pour plus de détails sur ces objets que l'on utilisera librement, faisons néanmoins quelques rappels. Pour $E^{\bullet}=(\cdots \rightarrow E^{n}\xrightarrow{d_n}E^{n+1}\rightarrow \cdots ) \in \sC(\rmC_K)$, on définit les foncteurs de troncations usuelles et bêtes sur $\sC(C_K)$ par
$$
\begin{gathered}
\tau_{\leqslant n} E^{\bullet} \coloneqq \cdots \rightarrow E^{n-2} \rightarrow E^{n-1}\rightarrow\ker(d_n) \rightarrow 0 \rightarrow \cdots \\
\tau_{\geqslant n} E^{\bullet} \coloneqq \cdots \rightarrow 0 \rightarrow \coim(d_{n-1}) \rightarrow E^{n}\rightarrow E^{n+1} \rightarrow \cdots \\
\sigma_{\leqslant n} E^{\bullet} \coloneqq \cdots \rightarrow E^{n-1} \rightarrow E^{n}\rightarrow 0 \rightarrow 0 \rightarrow \cdots \\
\sigma_{\geqslant n} E^{\bullet} \coloneqq \cdots \rightarrow 0 \rightarrow 0 \rightarrow E^{n}\rightarrow E^{n+1} \rightarrow \cdots.
\end{gathered}
$$
On dit que $E^{\bullet}$ est \emph{strict} si ses différentiels, les $d_n$, sont strictes et on dit qu'un morphisme de $\sD(\rmC_K)$ est un \emph{quasi-isomorphisme strict} si son cône est strict et exacte. On définit la \emph{cohomologie algébrique} de $E^{\bullet}$ comme la cohomologie usuelle $\rmH^n(E^{\bullet})$ dans la catégorie des $K$-espaces vectoriels ; elle sera munie de la topologie sous-quotient si nécessaire. On définit le complexe $\wt{\rmH}^n(E^{\bullet})\coloneqq \tau_{\leqslant n}\tau_{\geqslant n}(E^{\bullet})=(\coim(d_{n-1})\rightarrow \ker(d_n))$. Précisons la catégorie d'arrivé de ce foncteur (\cf \cite[2.1.1]{codoniste}). Les foncteurs de troncations $(\tau_{\leqslant n},\tau_{\geqslant n})$ définissent une $t$-structure sur $\rmD(C_K)$. Le coeur à gauche de cette $t$-structure sera noté $\rmL\rmH(\rmC_K)$ : tout objet de $\rmL\rmH(\rmC_K)$ est représenté, à équivalence près, par un monomorphisme $f\colon E\rightarrow F$  où $F$ est en degré $0$. On a un plongement naturel $I\colon C_K \incl \rmL\rmH(\rmC_K)$ donné par $E\mapsto (0\rightarrow E)$ et qui induit une équivalence $\rmD(\rmC_K)\xrightarrow{\sim}\rmD(\rmL\rmH(\rmC_K))$ compatible à la $t$-structure. Ces $t$-structures se relèvent au niveau des dg catégories dérivées ce qui promeut l'équivalence précédente en une équivalence $\sD(\rmC_K)\xrightarrow{\sim}\sD(\rmL\rmH(\rmC_K))$. Réciproquement, on a le foncteur \emph{partie classique} $C\colon \rmL\rmH(\rmC_K)\rightarrow \rmC_K$ qui envoie un monomorphisme $f\colon E\rightarrow F$ sur $\coker(f)$. Le foncteur $\wt{\rmH}^n$ s'interprète alors comme un foncteur $\sD(\rmC_K)\xrightarrow{\sim}\sD(\rmL\rmH(\rmC_K))$. On remarque que $C\wt{\rmH}^n=\rmH^n$ et on a un épimorphisme naturel $\wt{\rmH}^n\rightarrow I\rmH^n$. Si l'évaluation en $E^{\bullet}$ de cet épimorphisme, \ie $\wt{\rmH}^n(E^{\bullet})\rightarrow I\rmH^n(E^{\bullet})$, est un isomorphisme, alors on dit que la cohomologie $\wt{\rmH}^i(E^{\bullet})$ est \emph{classique}. 

Pour la définition et les propriétés du produit tensoriel complété dérivé à droite $\wotimes_K^R$ on renvoit à \cite[2.1.5]{codoniste}.

\clearpage
\section{Systèmes locaux}

\Subsection{Systèmes locaux}
Dans ce numéro on fait quelques rappels sur les systèmes locaux (pro)étales et on introduit quelques notations. Soit $X_K$ un espace rigide sur $K$ que l'on suppose lisse. On note $X_{K,\et}$ (resp. $X_{K, \pet}$) le site étale (resp. proétale) associé. On va considérer des sous-catégories de topos associés à ces sites.
\subsubsection{Systèmes locaux étales}
On suit la définition de de Jong des systèmes locaux étales. Pour plus de détails on renvoie à \cite{dej} mais aussi à \cite[1.4]{keli1}.
\medskip
\begin{defi}\label{def:syslocet}
\begin{itemize}
 \itemb On fixe $m\geqslant 1$ un entier. Soit $\Loc\Hargg{X_{K,\et}}{\BZ/p^m}$ la catégorie des faisceaux en $\BZ/p^m$-modules qui sont localement constants pour la topologie étale et de rang fini.
 \itemb Soit $\Loc_p^+\intn{X_{K,\et}} \coloneqq \varprojlim_m\Loc\Hargg{X_{K,\et}}{\BZ/p^m}$, où les morphismes de transition sont donnés par la multiplication par $p$, la catégorie des \emph{systèmes locaux étales en $\Zp$-réseaux}. 
 \itemb Soit $\Loc_p\intn{X_{K,\et}}$ le champ étale associé au localisé en $p^{\BZ}$ de $\Loc_p^+\intn{X_{K,\et}}$. Ces objets sont appelés des \emph{$\Qp$-systèmes locaux étales}. 
\end{itemize}
\end{defi}
La catégorie $\Loc_p\intn{X_{K,\et}}$ (resp. $\Loc_p^+\intn{X_{K,\et}}$) est une catégorie $\Qp$-linéaire (resp. $\BZ_p$-linéaire) muni d'un produit tensoriel et d'un $\Hom$ interne. Soit $\bar{x}\colon \Spa(C,C^+)\rightarrow X_K$ un point géométrique. On note $\pi_1^{\et}\intnn{X_K}{\bar x}$ le groupe fondamental étale introduit par de Jong \cite[2]{dej}, qui est un groupe topologique séparé et prodiscret (\cf \cite[Lemma 2.7]{dej}), à l'instar du groupe fondamental algébrique $\pi_1^{\alg}(X_K,\bar x)$, défini à partir des revêtements étales finis, qui est un groupe topologique séparé profini (\cf \cite[Theorem 2.10]{dej}). Rappelons que, d'après de Jong, le foncteur fibre est à valeurs dans la catégorie des $\Qp$-représentations de dimension finie du groupe fondamental, soit 
$$
\omega_{\bar x} \colon \Loc_p(X_{K,\et}) \rightarrow \Rep_{\Qp}\pi_1^{\et}\intnn{X_K}{\bar x}
$$
qui est une équivalence de catégories si $X_K$ est géométriquement connexe. Pour $\BV$ un $\Qp$-système local étale, on note $\BV(\bar x)\coloneqq \omega_{\bar x}\BV$ sa fibre en $\bar x$. On déduit de la construction de de Jong que la catégorie $\Loc_p(X_{K,\et})$ est naturellement une catégorie Tannakienne. De plus, ce foncteur se restreint en une équivalence de catégories 
$$
\omega_{\bar x} \colon \Loc_p^+(X_{K,\et}) \rightarrow \Rep_{\Zp}\pi_1^{\alg}\intnn{X_K}{\bar x}.
$$
\medskip
\begin{rema}\label{rem:nolatt}
On insiste que le localisé en $p^{\BZ}$ de $\Loc_p^+\intn{X_{K,\et}}$ n'est pas un champ, qu'un $\Qp$-système local est en réalité une donné de descente et que tout $\Qp$-système local ne provient pas d'un réseau (\cf \cite[Example 1.4.5]{keli1}). En terme de représentations des groupes fondamentaux, on voit qu'une $\Qp$-représentation d'un groupe prodiscret qui n'est pas profini n'admet pas nécessairement de réseau stable.
\end{rema}
\medskip
\subsubsection{Systèmes locaux proétales}
Le cas proétale est moins délicat. Premièrement, on peut recopier la définition \ref{def:syslocet} en remplaçant $X_{K,\et}$ par $X_{K,\pet}$ pour obtenir $\Loc_p^+(X_{K,\pet})$, la catégorie des \emph{$\Zp$-systèmes locaux proétales}, puis $\Loc_p(X_{K,\pet})$ la catégorie des \emph{$\Qp$-systèmes locaux proétales}. Une différence est que $\Loc_p(X_{K,\pet})$ est bien le localisé en $p^{\BZ}$ de $\Loc_p^+(X_{K,\pet})$, puisque c'est déjà un champ (\cf \ref{lem:keliproetoet}). Le $\BZ_p$-système local proétale constant $\wh{\BZ}_p \coloneqq \varprojlim_m \BZ/p^m$ permet de définir $\Loc_p^+(X_{K,\pet})$ comme la catégorie des $\wh{\BZ}_p$-modules localement constants de rang fini. De même, on définit $\Loc_p(X_{K,\pet})$ comme la catégorie des $\wh{\BQ}_p$-modules localement constant de rang fini, où $\wh{\BQ}_p\coloneqq \wh{\BZ}_p\!\invp{p}$. Ces deux définitions sont équivalentes et on obtient directement que ces catégories sont munis de produits tensoriels et de $\Hom$ internes.

On note $\nu \colon X_{K,\pet}\rightarrow X_{K,\et}$ le morphisme de sites naturel. Ce morphisme induit un foncteur $\Loc_p(X_{K,\et})\rightarrow \Loc_p(X_{K,\pet})$, donné par $\BV\mapsto \wh{\BV}\coloneqq \nu^{*}\BV$. Rappelons un lemme de Kedlaya et Liu (\cf  \cite[Lemma 9.1.11]{keli1}) qui nous assure que sur un espace rigide, les catégories des systèmes locaux étales et proétales coïncident ; notons que cette équivalence de catégories n'est pas exacte. 
\medskip
\begin{lemm}\label{lem:keliproetoet}
Le foncteur naturel $\Loc_p(X_{K,\et})\rightarrow \Loc_p(X_{K,\pet})$ est une équivalence de catégorie.
\end{lemm}
\medskip
\qed

On définit la torsion de Tate pour les systèmes locaux proétales. Posons $\wh{\BZ}_p(1)\coloneqq \varprojlim_n \mu_{p^n}$ qui est un élément de $\Loc^+_p(X_{K,\pet})$ dont le $\Qp$-système local proétale associé est noté $\wh{\BQ}_p(1)$. On note $\wh{\BQ}_p(-1)\coloneqq\wh{\BQ}_p (1)^{\vee}$ le dual et pour $k\in \BZ$, $\wh{\BQ}_p(k)\coloneqq\wh{\BQ}_p(1)^{\otimes k}$ si $k\geqslant 0$ et $\wh{\BQ}_p(k)\coloneqq \wh{\BQ}_p(-1)^{\otimes -k}$ si~$k\leqslant 0$. 
\Subsection{Fibrés}
Dans ce numéro on garde les notations précédentes. En particulier, $X_K$ désigne un espace rigide lisse sur $K$. On renvoit à \cite[Definition 7.4]{schphdg} pour la définition suivante :
\medskip
\begin{defi}\label{def:fibplafilt}
Un \emph{fibré plat filtré} est un triplet $(\CE,\nabla,\Filb)$ où : 
\begin{itemize}
\itemb $\CE$ est un fibré vectoriel sur $X_K$, \ie un $\CO_{X_K}$-modules localement libres et de rang fini,
\itemb $\nabla \colon \CE \rightarrow \CE \otimes_{\CO_{X_K}} \Omega^1_{X_K}$ est une connection plate, \ie une application $K$-linéaire satisfaisant la règle de Leibniz et telle que $\nabla^{(2)}\circ\nabla$ = 0 où $\nabla^{(2)}\colon  \CE \otimes_{\CO_{X_K}} \Omega^1_{X_K}\rightarrow  \CE \otimes_{\CO_{X_K}} \Omega^2_{X_K}$ est l'extension de la connexion,
\itemb $\Filb$ une filtration sur $\CE$ qui est finie, décroissante, exhaustive et séparée, \ie  une suite de sous-$\CO_{X_K}$-modules $\{\Fil^{i}\}_{i\in \BZ}$ d'un fibré $\CE$ telle qu'il existe deux entiers $a\leqslant b$ tels que $\CE=\Fil^{a}\supseteq \dots \supseteq \Fil^{b}=0$ et telle que pour tout $i\in \llbracket a,b\rrbracket$, $\Fil^{i}$ admet Zariski localement un supplémentaire\footnote{Insistons qu'a priori ceci est moins fort que de demander qu'il existe un supplémentaire qui soit un fibré vectoriel.} $\Fil_{i}$ tel que $\CE = \Fil^{i}\oplus \Fil_{i}$,
\end{itemize}
le tout tel que la connexion satisfait à la \emph{transversalité de Griffiths} par rapport à la filtration, \ie  pour tout $i\in \BZ$,
$$
\nabla \Fil^{i}\subseteq\Fil^{i-1}\otimes_{\CO_{X_K}}\Omega_{X_K}^1.
$$
On note $\Vect_{\nabla}^{\bullet}(X_K)$ la catégorie des fibrés plats filtrés.
\end{defi}
\medskip
\begin{rema}
Rappelons que ce qui nous permet d'omettre la topologie en indice, sans ambiguité, est que la catégorie $\Vect_{\nabla}^{\bullet}(X_K)$ ne dépend pas de la topologie (\cf  \cite[Lemma 7.3]{schphdg}), au sens où on a des équivalences naturelles
$$
\Vect_{\nabla}^{\bullet}(X_{K,\an})\xrightarrow{\sim}\Vect_{\nabla}^{\bullet}(X_{K,\et})\xrightarrow{\sim} \Vect_{\nabla}^{\bullet}(X_{K,\pet}).
$$
La seconde équivalence est notée $\CE \rightarrow \wh{\CE}\coloneqq \nu^*\CE$ : on prendra garde, ce n'est pas un foncteur exact. 
\end{rema}

\subsubsection{$\BBdrp$-systèmes locaux}
On renvoit à \cite[Definition 6.1]{schphdg} pour la définition de $\BB_{\dR}^{(+)}\coloneqq\BB_{\dR,X_K}^{(+)}$ et à \cite[Definition 6.8]{schphdg} pour la définition de $\CO\BB_{\dR,X_K}^{(+)}\coloneqq\CO\BB_{\dR,X_K}^{(+)}$ (on rappelle que le produit tensoriel dans la définition doit être complété).
\medskip
\begin{defi}
\
\begin{itemize}
\itemb Un \emph{$\BBdrp$-système local} est un faisceau proétale en $\BBdrp$-modules qui est localement libre et de rang fini.
\itemb Un \emph{$\OBBdrp$-module plat} est un faisceau proétale en $\OBBdrp$-modules libre et de rang fini muni d'une connexion plate.
\end{itemize}
Un $\OBBdrp$-module plat $(\CM,\nabla)$ est dit \emph{associé} à un élément $(\CE,\nabla,\Filb)$ de $\Vect_{\nabla}^{\bullet}(X_K)$ s'il existe un isomorphisme de faisceaux proétales 
\begin{equation}\label{eq:assoc}
\CM\otimes_{\OBBdrp  }\OBBdr \cong \CE \otimes_{\CO_{X_K}}\OBBdr,
\end{equation}
compatible aux filtrations et connexions. Les filtrations sont les filtrations produits où $\CM$ est muni de la filtration triviale, $\CE$ est muni de la filtration $\Filb$ et $\OBBdrp$ de sa filtration canonique.
\end{defi}
Les $\BBdrp$-systèmes locaux et les $\OBBdrp$-modules plats définissent naturellement des catégories et, d'après \cite[Theorem 7.2]{schphdg}, ces catégories sont équivalentes ; cette équivalence rappelle la correspondance de Riemann-Hilbert classique. 
Pour $\CE=(\CE,\nabla,\Filb)$ un fibré plat filtré, on définit un $\BBdrp$-système local par
$$
\BX_{\dR}^+(\CE)\coloneqq \Filt{0}\intn{\CE \otimes_{\CO_{X_K}}\OBBdr }^{\nabla=0}.
$$
Cette formule définie un foncteur pleinement fidèle ; en effet, on a en tant que faisceaux étales $\CE_{\et} = \nu_*\intn{\BX_{\dR}^+(\CE)\otimes_{\BBdrp}\OBBdr }$ qui est muni d'une connexion et d'une filtration induites par $\OBBdr$ : ceci fournit la construction réciproque pour les fibrés associés. C'est essentiellement le théorème \cite[Theorem 7.6]{schphdg}. 

\medskip
\begin{defi}
Soit $\BV$ un $\Qp$-système local (pro)étale. On dit que $\BV$ est \emph{de Rham} si $\wh{\BV}\otimes_{\Qp}\OBBdrp$ est associé à un fibré plat filtré $\CE=(\CE,\nabla,\Filb)$ au sens de (\ref{eq:assoc}). Dans ce cas, les \emph{poids de Hodge-Tate} de $\BV$ sont les opposés des sauts de la filtration $\Filb$ sur $\CE$. Plus précisément, $i$ est un poids de Hodge-Tate si et seulement si 
 $$
 \Fil^{-i}\CE\neq \Fil^{-i+1}\CE.
 $$
\end{defi}
\medskip
\begin{rema}
On pourrait définir les poids de Hodge-Tate de manière différente en utilisant les opérateurs de Sen (\cf \cite{shi}); d'après \cite[Proposition 7.9]{schphdg} ces définitions sont équivalentes.
\end{rema}
\Subsection{Opers}\label{subsec:opers}
En général, les opers sont des fibrés plats filtrés sur les courbes munis de l'action d'une algèbre de Lie satisfaisant une certaine condition de transversalité. On renvoie à \cite{bedr} pour la définition générale et plusieurs propriétés que l'on rappellera. Techniquement, on ne s'intéresse ici qu'aux opers pour $\GL_n$ mais on va s'autoriser à décaler la filtration pour faire apparaître les poids de Hodge-Tate.
\subsubsection{Définition}
\medskip
\begin{defi}\label{def:oper}
Soit $X_K$ une courbe rigide lisse sur $K$. Soit $\CE=(\CE,\nabla,\Filb)$ un fibré plat filtré. On dit que $\CE$ est un \emph{oper} de poids $(a,b)$, avec $a,b\in\BZ$ tels que $b\geqslant a$, s'il est de rang $b-a+1$ et si
\begin{itemize}
\itemb $\CE=\Fil^{-b}\supsetneq\dots\supsetneq\Fil^{-a}\supsetneq 0$,
\itemb pour tout entier $i$ tel que $a\leqslant i \leqslant b$, le quotient $\Gr_i\CE\coloneqq \Fil^{-i}/\Fil^{-i+1}$ est un fibré en droites,
\itemb l'application naturelle\footnote{Cette application est appelée le \emph{champ de Higgs}.} induite par $\nabla$,
$$
\theta_i\colon \Gr_i\CE\rightarrow \Gr_{i+1}\CE\otimes_{\CO_{X_K}}\Omega^1_{X_K},
$$
est un isomorphisme pour tout entier $i$ tel que $a\leqslant i\leqslant b-1$.
\end{itemize}
\end{defi}
\medskip
Ainsi, si $\BV$ est un $\Qp$-système local de de Rham proétale, on dit que $\BV$ est un \emph{$\Qp$-oper proétale} si le fibré plat filtré associé est un oper. Notons que si cet oper est de poids $(a,b)$ alors les poids de Hodge-Tate de $\BV$ sont précisément $a,a+1,\dots, b$. Remarquons que si $a=b$, $\CE$ est un fibré en droite et la filtration est caractérisée par le cran $a\in \BZ$ ; réciproquement, tout fibré plat filtré de rang $1$ est un oper.

Faisons quelques rappels de \cite[\S 2]{bedr}, notamment comment associer à $\nabla$ un opérateur différentiel $L_{\nabla}$ d'ordre $b-a+1$ entre fibrés en droites. Dans la théorie des équations différentielles ordinaires on sait comment associer à une équation linéaire d'ordre $n\geqslant 1$ un système linéaire de $n$ équations différentielles d'ordre $1$ : c'est précisément cette idée que les opers géométrisent.

Notons $\CO=\CO_{X_K}$, $\Omega=\Omega^1_{X_K}$ et $\Omega_{\CE}\coloneqq \CE\otimes_{\CO}\Omega$. Soit $\CD=\CD_{X_K}$ le faisceau des opérateurs différentiels sur $X_K$ et $\CD_k\subset \CD$ le sous-faisceau des opérateurs différentiels d'ordre $\leqslant k$ pour $k\geqslant 0$ un entier. Notons que $\CD$ est naturellement muni d'une structure de faisceau en anneaux filtré par les $\CD_k$. La connexion $\nabla$ définit une structure de $\CD$-module sur $\CE$ et on obtient une application $\CD \otimes_{\CO}\Gr_a\CE\rightarrow \CE$ qui induit un isomorphisme $\CD_{b-a}\otimes_{\CO}\Gr_a\CE\xrightarrow{\sim} \CE$ d'après la troisième condition dans la définition \ref{def:oper}. Ceci donne un diagramme
\begin{equation}\label{eq:splitopdiag}
\begin{tikzcd}
0\ar[r]&\CD_{b-a}\otimes_{\CO}\Gr_a\CE\ar[r]\ar[rd,"\sim"]&\CD_{b-a+1}\otimes_{\CO}\Gr_a\CE\ar[r]\ar[d]&\Omega^{\otimes (a-b-1)}\otimes_{\CO}\Gr_a\CE\ar[r]&0\\
& & \CE & &
\end{tikzcd}
\end{equation}
qui définit un scindage $ \Omega^{\otimes (a-b-1)}\otimes_{\CO}\Gr_a\CE\rightarrow \CD_{b-a+1}\otimes_{\CO}\Gr_a\CE$. Puisque la troisième condition dans la définition \ref{def:oper} fournit un isomorphisme $\Omega^{\otimes (a-b)}\otimes_{\CO}\Gr_a\CE\cong \Gr_b\CE$, on obtient une application $\Omega^{\otimes (-1)}\otimes_{\CO}\Gr_b\CE\rightarrow \CD_{b-a+1}\otimes_{\CO}\Gr_a\CE$ et donc une application $\intn{\Gr_a\Omega_{\CE}}^{\otimes(-1)}\rightarrow \otimes_{\CO}\intn{\Gr_b\CE}^{\otimes(-1)}\otimes_{\CO}\CD_{b-a+1}$ qui définit un opérateur différentiel $\CB^{\otimes(-1)} \rightarrow \CA^{\otimes(-1)}$ où $\CB\coloneqq \Gr_a\Omega_{\CE}$ et $\CA\coloneqq \Gr_b\CE$. Or, comme $\CA$ et $\CB$ sont des faisceaux inversibles, on a un isomorphisme canonique entre les opérateurs différentiels $\CA\rightarrow \CB$ et les opérateurs différentiels $\CB^{\otimes(-1)}\rightarrow \CA^{\otimes(-1)}$. En effet, puisqu'on a un isomorphisme canonique $\CB^{\otimes(-1)}\otimes_{\CO}\CA\cong \CA\otimes_{\CO}\CB^{\otimes(-1)}$, on obtient
$$
\CB\otimes_{\CO}\CD\otimes_{\CO}\CA^{\otimes(-1)}\cong \CHom_{\CO}\intnn{\CB^{\otimes(-1)}\otimes_{\CO}\CA}{\CD}\cong \CHom_{\CO}\intnn{\CA\otimes_{\CO}\CB^{\otimes(-1)}}{\CD}\cong \CA^{\otimes(-1)}\otimes_{\CO}\CD\otimes_{\CO}\CB
$$
Ainsi, on obtient\footnote{La construction de l'opérateur à partir de la connexion diffère légèrement dans ces dernières étapes de celle de \cite[\S 2]{bedr}. En effet, le scindage (\ref{eq:splitopdiag}) nous donne directement un opérateur différentiel $\intn{\Gr_a{\CE}}^{\otimes(-1)}\rightarrow \intn{\Gr_b{\CE}}^{\otimes(-1)}\otimes_{\CO}\Omega$. Il est alors tentant de passer à la transposé pour obtenir un opérateur $\CA\rightarrow \CB$ mais, comme le montrerons les calculs locaux, ce n'est pas l'opérateur que l'on veut considérer : on veut plutôt utiliser l'inversion \og naïve \fg que l'on propose. } finalement un opérateur différentiel 
\begin{equation}\label{eq:opdiffop}
L_{\nabla}\colon \Gr_b\CE \rightarrow \Gr_a\Omega_{\CE}.
\end{equation}
C'est un opérateur différentiel d'ordre $b-a+1$ et $\sigma_{\nabla}^{b-a+1}=1$ d'après \cite[\S 2.1]{bedr}, où $\sigma_{\nabla}^{b-a+1}$ est le \emph{symbole principal d'ordre $b-a+1$} définit naturellement comme la section 
$$
\sigma_{\nabla}^{b-a+1}\in\rmH^0\Harg{X_K}{\CA^{\otimes(-1)}\otimes_{\CO}\CB\otimes_{\CO}\Omega^{\otimes (a-b-1)}}.
$$
Réciproquement, à partir d'un opérateur différentiel d'ordre $b-a+1$, entre des faisceaux inversibles, de symbole principal inversible, on peut construire un oper de poids $(a,b)$. 
Notons que $\Gr_b\CE$ et $\Gr_a\Omega_{\CE}$ sont munis des filtrations à un cran
\begin{equation}\label{eq:filtop}
\Fil^{i}\Gr_b\CE=
\begin{cases}
\Gr_b\CE&\text{ si } i\leqslant -b,\\
0 & \text{ si } i\geqslant -b+1,
\end{cases}
\quad 
\Fil^{i}\Gr_a\Omega_{\CE}=
\begin{cases}
\Gr_a\Omega_{\CE}&\text{ si } i\leqslant -a,\\
0 & \text{ si } i\geqslant -a+1.
\end{cases}
\end{equation}
\subsubsection{Calculs locaux}\label{subsub:locop}
On va détailler la forme local des opers et l'opérateur différentiel associé. Supposons que $X$ est un affinoïde tel qu'on peut trivialiser $\CE$ et $\Omega$, \ie tel que
$$
\begin{gathered}
\CE=\bigoplus_{i=a}^{b}\CO e_i,\quad \Omega = \CO dt, \\
\Fil^{-k}\CE = \bigoplus_{i=a}^{k}\CO e_i,\ a\leqslant k \leqslant b,  \quad \Omega^{\otimes (-1)}= \CO \partial. 
\end{gathered}
$$
Ici, $dt$ est une section de $\Omega$ et $\partial\in\Omega^{\otimes (-1)}$ est la section duale, vue comme dérivation, définie pour $g\in \CO$ par $dg = \partial g dt$. Dans ces conditions, on peut écrire $\nabla$ comme dérivation sous la forme $\nabla = \partial-Q$ avec $Q$ une matrice de taille $(b-a+1)\times (b-a+1)$ ; de plus après transformation de jauge (équivalence de connexions) on peut supposer que $Q$ est la matrice compagnon associée à des fonctions $q_a,\dots, q_b\in \CO$, \ie
$$
Q=
\begin{pmatrix}
 q_b & q_{b-1}&  \cdots & \cdots & q_a\\
 1 & 0 &  \cdots & \cdots & 0\\
 0 & 1 & \ddots & &\vdots\\
 \vdots & \ddots & \ddots &\ddots &\vdots\\
 0 & \cdots & 0 & 1& 0
\end{pmatrix}.
$$
%

Soit $P=\sum_{i=0}^{\infty}a_i\partial^i$ une section de $\CD$, \ie un opérateur différentiel, sous l'application $\CD\otimes_{\CO}\Gr_a\CE\rightarrow \CE$ on obtient
$$
P\otimes e_a\mapsto \sum_{i=0}^{\infty}a_i(-\nabla)^ie_a=\sum_{i=0}^{\infty}a_iQ^ie_a.
$$
Ceci permet d'expliciter la section $\Omega^{\otimes (-1)}\otimes_{\CO}\Gr_b\CE\rightarrow \CD_{b-a+1}\otimes_{\CO}\Gr_a\CE$ de la suite exacte (\ref{eq:splitopdiag}) :
$$
\partial\otimes e_b\mapsto \partial^{b-a+1}\otimes e_a -\sum_{i=a}^{b}q_i\partial^{i-a}\otimes e_a = \colon L\otimes e_a.
$$
En effet, il est clair que $L\otimes e_a$ est bien une section de $\partial\otimes e_b$ et il suffit de justifier que son image dans $\CE$ par l'action $\CD\otimes_{\CO}\Gr_a\CE\rightarrow \CE$ est nulle, or l'image de $L\otimes e_a$ est
$$
Q^{b-a+1}-\sum_{i=a}^{b}q_iQ^{i-a}e_a=0,
$$
puisque l'on reconnait exactement le polynôme caractéristique de $Q$. Finalement, on a montré que $L_{\nabla}(fe_b)=L(f)e_adt$ pour $f\in\CO$ où
\begin{equation}\label{eq:locexpcon}
L =  \partial^{b-a+1} -\sum_{i=a}^{b}q_i\partial^{i-a}.
\end{equation}
\subsubsection{Scindage de la filtration}

\medskip

\begin{lemm}\label{lem:opdec}
Supposons que $X_K$ soit une courbe rigide lisse sur $K$. Soit $\CE=(\CE,\nabla,\Filb)$ un oper de poids $(a,b)$ sur $X_K$. Alors il existe des scindages $K$-linéaires
$$
\begin{gathered}
\CE\cong \Fil^{-b}\oplus \Gr_b\CE,\\
\Omega_{\CE}\cong \Gr_a\Omega_{\CE}\oplus \Omega_{\CE}/\Fil^{-a},
\end{gathered}
$$
tel que les projections/inclusions de ces scindages induisent des diagrammes commutatifs
\begin{center}
\begin{tikzcd}
\CE \ar[r]\ar[d,"\nabla"]& \Gr_b\CE\ar[d,"L_{\nabla}"] & & &\CE \ar[d,"\nabla"]& \Gr_b\CE\ar[l]\ar[d,"L_{\nabla}"] \\
\Omega_{\CE}\ar[r] & \Gr_a\Omega_{\CE}& & &\Omega_{\CE}&\ar[l] \Gr_a\Omega_{\CE}
\end{tikzcd}
\end{center}

\end{lemm}
\begin{proof}
On commence par décrire les scindages $K$-linéaire. La connexion induit un isomorphisme $K$-linéaire 
$$
u\colon \Fil^{-b}\xrightarrow{\sim}\Omega_{\CE}/\Fil^{-a},
$$
que l'on obtient par récurrence à partir du troisième point dans la définition \ref{def:oper}. On a un diagramme commutatif à lignes exactes
\begin{equation}\label{eq:splitopde}
\begin{tikzcd}
0\ar[r] & \Fil^{-b+1}\ar[r]\ar[d,"\cong"]& \CE \ar[r]\ar[d,"\nabla"]& \Gr_b\CE\ar[r]& 0\\
0 &\ar[l] \Omega_{\CE}/\Fil^{-a} &\ar[l] \Omega_{\CE} &\ar[l] \Gr_a\Omega_{\CE}&\ar[l] 0.
\end{tikzcd}
\end{equation}
Notons que la flèche $L_{\nabla}\colon \Gr_b\CE\rightarrow \Gr_a\Omega_{\CE}$ qui vient compléter le diagramme ne commute pas aux autres applications du diagramme (il suffit de remarquer que $\Fil^{-b+1}$ n'est pas dans le noyau de $\nabla$).

La composée
$$
\CE \xrightarrow{\nabla} \Omega_{\CE}\xrightarrow{\mod \Fil^{-a}}\Omega_{\CE}/\Fil^{-a} \xrightarrow{u^{-1}}\Fil^{-b+1},
$$
nous donne une application $r\colon \CE \rightarrow \Fil^{-b+1}$. Comme le diagramme (\ref{eq:splitopde}) commute, $r$ est une rétraction de l'inclusion $\Fil^{-b}\rightarrow \CE$ ce qui scinde la suite exacte du haut. Ceci définit une section $K$-linéaire $s\colon \Gr_b\CE \rightarrow \CE$ et un isomorphisme 
$$
\CE\cong \Fil^{-b+1}\oplus \Gr_b\CE.
$$ 

Pour la suite exacte du bas, la composée
$$
\Omega_{\CE}/\Fil^{-a}\xrightarrow{u^{-1}}\Fil^{-b+1}\xrightarrow{}\CE\xrightarrow{\nabla}\Omega_{\CE}
$$
nous donne une application $\sigma \colon \Omega_{\CE}/\Fil^{-a}\rightarrow \Omega_{\CE}$. Comme le diagramme (\ref{eq:splitopde}) commute, $\sigma$ est une section de la projection $\Omega_{\CE}\rightarrow\Omega_{\CE}/\Fil^{-a}$ et scinde la suite exacte du bas. Ceci définit une rétraction $K$-linéaire $\rho \colon \Omega_{\CE}\rightarrow \Gr_a\Omega_{\CE}$ et donc un isomorphisme 
$$
\Omega_{\CE}\cong \Gr_a\Omega_{\CE}\oplus \Omega_{\CE}/\Fil^{-a}.
$$ 

Pour montrer que les diagrammes du lemme commutent, quitte à se restreindre à un ouvert convenable, on peut supposer que l'on est dans les conditions du numéro \ref{subsub:locop} dont on conserve les notations. On va calculer localement les sections/rétractions que l'on a défini. Pour $i=a,\dots, b$ et $f_{i}\in \CO$, on a (avec la convention $e_{b+1}=0$) :
$$
\nabla(f_ie_i)=(\partial f_i e_i-f_ie_{i+1}-q_{b+a-i}f_ie_a)dt\equiv (\partial f_i e_i-f_ie_{i+1})dt \mod \Fil^{-a}.
$$
Ainsi $\Ker(r) = \{\sum_{i=a}^{b}\partial^{b-i}f e_i \ ; \ f\in \CO\}$ et la composée $\Gr_b\CE\xrightarrow{s}\CE\xrightarrow{\nabla} \Omega_{\CE}$ s'explicite localement pour $f\in \CO$ par 
$$
fe_b\mapsto \sum_{i=a}^{b}\partial^{b-i}f e_i \mapsto \partial^{b-a+1}fe_a-\sum_{i=a}^{b}q_{b+a-i}\partial^{b-i}f e_a.
$$
L'image est bien dans $\Gr_a\Omega_{\CE}$ et en changeant les indices de la somme par $j=a+b-i$ on reconnait l'expression de $L$ (\cf (\ref{eq:locexpcon})) \ie on trouve que la composée précédente s'écrit $fe_b\mapsto L(f)e_a=L_{\nabla}(fe_b)$ comme attendu. Ainsi, on a montré que le second diagramme du lemme commute, \ie
\begin{center}
\begin{tikzcd}
\CE \ar[d,"\nabla"]& \Gr_b\CE\ar[l,"s"]\ar[d,"L_{\nabla}"] \\
\Omega_{\CE}&\ar[l] \Gr_a\Omega_{\CE}.
\end{tikzcd}
\end{center}
On montre maintenant que le premier diagramme du lemme commute. Soit $f=\sum_{i=a}^bf_ie_i\in \CE$ qui s'écrit dans le scindage $\CE\cong \Fil^{-b+1}\oplus \Gr_b\CE$ sous la forme 
$$
f= \underbrace{\sum_{i=a}^b(f_i-\partial^{b-a}f_b)}_{=g\in \Fil^{-b+1}}+\underbrace{\sum_{i=a}^b\partial^{b-i}f_be_i}_{=h\in \Gr_{b}\CE}.
$$
Comme on a défini le scindage $\Omega_{\CE}\cong \Omega_{\CE}/\Fil^{-a}\oplus \Gr_a\Omega_{\CE}$ à partir de l'image par $\nabla$ du scindage $\CE\cong \Fil^{-b+1}\oplus \Gr_b\CE$ le scindage de $\nabla f$ s'écrit
$$
\nabla f = \nabla g + \nabla h\in \Omega_{\CE}/\Fil^{-a}\oplus \Gr_a\Omega_{\CE}.
$$
L'image par $\rho \colon \Omega_{\CE}\rightarrow \Gr_a\Omega_{\CE}$ de $\nabla h$ donne
$$
(\rho\circ \nabla)(f) = \partial^{b-a+1}f_be_a-\sum_{i=a}^{b}q_{b+a-i}\partial^{b-i}f_b e_a = L_{\nabla}(f_be_b).
$$
Ainsi on a montré que le premier diagramme du lemme commute, \ie
\begin{center}
\begin{tikzcd}
\CE \ar[r]\ar[d,"\nabla"]& \Gr_b\CE\ar[d,"L_{\nabla}"] \\
\Omega_{\CE}\ar[r,"\rho"] & \Gr_a\Omega_{\CE}.
\end{tikzcd}
\end{center}

\end{proof}
\begin{rema}
Le lemme \ref{lem:opdec} nous donne une autre façon de construire un opérateur différentiel à partir de la connexion $\nabla$ et montre que cet opérateur différentiel est exactement $\nabla$. Il est important de noter que les scindages du lemme \ref{lem:opdec} ne sont pas $\CO$-linéaires puisqu'ils dépendent de la connexion. Je remercie Y. Zhou de m'avoir aidé à corriger et améliorer le lemme \ref{lem:opdec}.
\end{rema}

\section{Systèmes locaux isotriviaux}
\Subsection{Définitions et foncteurs fibres}
\subsubsection{Définition}
Rappelons qu'un \emph{$\varphi$-module sur $K_0$} est un $K_0$-espace vectoriel muni d'une application $\sigma$-linéaire bijective. Un \emph{isocristal} sur $k$ est alors un $\varphi$-module sur $K_0$ de dimension finie. Les isocristaux sur un corps parfait $k$ forment une catégorie Tannakienne que l'on note~$\Iso_k$. On renvoit à \cite{tantong} pour la définition du \emph{faisceau proétale des périodes cristallines} $\BBcris\coloneqq\XBBcris{X_K}$ (\cf \cite[Definition 2.1, Definition 2.4]{tantong}).
\medskip
\begin{defi}
Un $\Qp$-système local (pro)étale $\BV$ est \emph{isotrivial cristallin} s'il existe un isocristal~$D$ sur $k$ et un isomorphisme de faisceaux proétales,
$$
\psi \colon \wh{\BV}\otimes_{\Qp}\BBcris \cong D\otimes_{K_0}\BBcris,
$$ 
compatible avec l'endomorphisme de Frobenius. On appelle le couple $(D,\psi)$ une \emph{isotrivialisation} de $\BV$. 
\end{defi}
\medskip
\begin{rema}\label{rem:isogen}
\begin{enumerate}
\item Soit $\BV$ un $\Qp$-système local isotrivial (cristallin), on note que l'on a $\rang_{\Qp}\BV=\rang_{K_0}D$, donc $\BV$ est cristallin (au sens de Faltings \cf \cite[1.0.2]{tantong}).
\item On peut prendre une définition légèrement plus générale, en remplaçant ${D}$ par un faisceau de $\varphi$-modules localement constant pour la topologie de Zariski. Tout ce que l'on fera devrait facilement s'adapter à ce cas.
\end{enumerate}
\end{rema}
\subsubsection{Foncteur fibre cristalline}
On note $\Li_p(X_{K,\et})\subset \Loc_p(X_{K,\et})$ la sous-catégorie pleinement fidèle des $\Qp$-systèmes locaux étales qui sont isotriviaux cristallins. Le lemme suivant résume les propriétés de cette sous-catégorie :
\medskip
\begin{lemm}
\
\begin{enumerate}
\item La sous-catégorie $\Li_p(X_{K,\et})$ de  $\Loc_p(X_{K,\et})$ est une sous-catégorie Tannakienne.
\item Le foncteur $\BD \colon \Li_p(X_{K,\et})\rightarrow \Iso_k$, qui associe à $\BV$ d'isotrivialisation $(D,\psi)$ l'isocristal $D$, est un foncteur tensoriel, bien défini comme foncteur fibre. 
\end{enumerate}
\end{lemm}
\medskip
\begin{proof}
La structure de catégorie Tannakienne sur $\Li_p(X_{K,\et})$ est donnée par le foncteur fibre $\BD$. Premièrement, notons que si $\BV$ est un $\Qp$-système local isotrivial cristallin, alors l'isocristal associé est unique à isomorphisme près. En effet, d'après \cite[Corollary 2.26]{tantong}, on obtient $\nu_{*}\BBcris = \ud{K_0}$, le faisceau constant associé à $K_0$ où l'on rappelle que $\nu \colon X_{K,\pet}\rightarrow X_{K,\et}$ est le changement de site. Ainsi, on définit 
$$
\BD\colon \BV\mapsto \nu_*(\BBcris\otimes_{\Qp}\BV)_x,
$$
où $x\in X_K$ est un point quelconque. Ceci donne $\BD(\BV)=\nu_*(\BBcris)_x\otimes_{K_0}D=D$ et montre que l'association $\BV\mapsto D$ est bien définie. Il est clair que $\BD$ est tensoriel puisque 
$$
\BD(\BV_1\otimes_{\Qp}\BV_2)=\nu_*(\BBcris)_x\otimes_{K_0}(\BD(\BV_1)\otimes_{K_0}\BD(\BV_2))=\BD(\BV_1)\otimes_{K_0}\BD(\BV_2).
$$
\end{proof}
\begin{rema}
Cette construction n'est qu'un cas très particulier de  \cite[Definition 3.12]{tantong}.
\end{rema}
\medskip
Un corollaire direct est que les constructions tensorielles appliqué à un système local isotriviales $\BV$ sont aussi isotriviales. Plus précisément, soient $\mu$ une partition d'un entier $n$ et $s_{\mu}$ le projecteur de Schur associé. Alors $s_{\mu}\cdot\BV$ est isotrivial d'isocristal associé $s_{\mu}\cdot{D}$. En particulier,  pour tout entier $k\in \BN$, $\Sym_{\Qp}^k\BV$ est un système local isotrivial de $\varphi$-module associé $\Sym^k_{K_0}{D}$. 

Pour $\BV$ dans $\Li_p(X_{K,\et})$ on appellera simplement $\BD(\BV)$ l'isocristal associé à $\BV$. Soit $D$ un isocristal sur $k$. On définit les faisceaux proétales
$$
\BX_{\cri}(D)\coloneqq\intn{{D}\otimes_{K_0}\BBcris }^{\Phi=\id},\quad \BX_{\cri}^+(D)\coloneqq\intn{{D}\otimes_{K_0}\XBBcrisp {X_K} }^{\Phi=\id},
$$
où le Frobenius $\Phi$ est le produit tensoriel du Frobenius sur $D$ et de celui sur $\BBcris$. On rappelle que la torsion à la Tate pour les isocristaux est donnée par $K_0[-1]$ qui est l'isocristal de dimension~$1$, muni du Frobenius $\varphi=\frac 1 p \sigma$. En d'autres termes, $\Qp(k)$ pour $k\in \BZ$ définit un système local isotrivial et $\BD(\Qp(k))=K_0[-k]$.

\subsubsection{Foncteur fibre de Rham}

Soit $\BV$ un objet de $\Li_p(X_{K,\et})$ et soit $D=\BD(\BV)$ son isocristal associé. On a alors une inclusion $\BV\rightarrow \BX_{\cri}(D)$. On veut décrire le quotient qui fait intervenir la partie de Rham. On commence par un lemme qui explicite le foncteur fibre $\Li_p(X_{K,\et})\rightarrow \Vect_{\nabla}^{\bullet}(X_K)$.
\medskip
\begin{lemm}
Le système local $\BV$ est de Rham et le fibré associé $\CE=(\CE,\nabla,\Filb)$ est muni d'une trivialisation $\CE\cong D_K\otimes_K \CO_{X_K}$, où $D_K\coloneqq D\otimes_{K_0}K$. Cette trivialisation ne dépend que de l'isotrivialisation $(D,\psi)$ de $\BV$. 
\end{lemm}
\medskip
\qed

Rappelons qu'on a défini pour la partie de Rham, à partir d'un élément $\CE=(\CE,\nabla,\Filb)$ de $\Vect_{\nabla}^{\bullet}(X_K)$, des faisceaux proétales 
$$
\BX_{\dR}(\CE)\coloneqq\intn{\wh{\CE}\otimes_{\wh{\CO}_{X_K}}\XOBBdr {X_K}}^{\nabla=0},\quad \BX_{\dR}^+(\CE)\coloneqq\Filt 0 \intn{\wh{\CE}\otimes_{\wh{\CO}_{X_K}}\XOBBdr {X_K}}^{\nabla=0}.
$$
On a alors un morphisme naturel $\BX_{\cri}(D)\rightarrow\BX_{\dR}(\CE)$ qui est injectif. On note de plus le faisceau quotient
$$
 \BX_{\dR}^-(\CE)\coloneqq  \BX_{\dR}(\CE)/\BX_{\dR}^+(\CE).
$$
\Subsection{Fibrés fortements isotriviaux et la suite exacte fondamentale}
\subsubsection{Fibrés fortement isotriviaux}
On donne maintenant une condition plus forte que l'isotrivialité sur un fibré portant sur la connexion. On dit qu'un fibré vectoriel $\CE$ sur $X_K$ est \emph{isotrivial} s'il existe un $K$-espace vectoriel $D_K$, de dimension finie tel que $\CE\cong\CO_{X_K}\otimes_KD_K$. Par abus, on dit alors que $D_K$ est associé à $\CE$.
\medskip
\begin{defi}
Soit $\CE$ un fibré vectoriel isotrivial sur $X_K$ d'espace vectoriel associé $D_K$. Supposons $\CE$ muni d'une connexion $\nabla\colon D_K\otimes_K\CO_{X_K}\rightarrow D_K\otimes_{K}\Omega_{X_K}$. On note
$$
D^{\nabla}_K=\Ker\left [\nabla \colon D_K\otimes\CO(X_K)\rightarrow D_K\otimes_{K}\Omega_{X_K}(X_K)\right ],
$$
le $K$-espace vectoriel des \emph{sections globales horizontales}. On dit que $(\CE,\nabla)$ est \emph{fortement isotrivial} si l'inclusion naturelle $D_K^{\nabla}\incl D_K\otimes \CO_{X_K}$ induit un isomorphisme $D_K\otimes_K\CO_{X_K}\cong D_K^{\nabla}\otimes_K\CO_{X_K}$. De même, on dit d'un $\Qp$-système local isotrivial cristallin $\BV$ qu'il est \emph{fortement isotrivial (cristallin)} si le fibré associé muni de sa connexion, $(\CE,\nabla)$ est fortement isotrivial.
\end{defi}
\medskip
\begin{rema}
Si la connexion $\nabla$ est unipotente, il suffit que $\dim_K D_K=\dim_KD^{\nabla}_K$ pour que le fibré muni de sa connexion soit fortement isotrivial. En effet, dans ce cas, la matrice de passage entre $D_K^{\nabla}\otimes_K\CO_{X_K}$ et $D_K\otimes_K\CO_{X_K}$ d'une base compatible avec $D_K^{\nabla}$ vers une base compatible avec $D_K$, est unipotente, donc de déterminant $1$. L'inclusion naturelle $D_K^{\nabla}\otimes_K\CO_{X_K}\rightarrow D_K\otimes_K\CO_{X_K}$ est alors un isomorphisme de $\CO_{X_K}$-modules.
\end{rema}
\medskip
La condition pour un fibré plat d'être fortement isotrivial est très restrictive : cette condition implique que la connexion se trivialise  :
\medskip
\begin{lemm}
Soit $(\CE,\nabla)$ un fibré plat isotrivial de $K$-espace vectoriel associé $D_K$. Alors $(\CE,\nabla)$ est fortement isotrivial si et seulement si $\nabla=\id\otimes d\colon D_K^{\nabla}\otimes_K\CO_{X_K}\rightarrow D_K^{\nabla}\otimes_K\Omega_{X_K}$.
\end{lemm}
\medskip
\qed

Pour un fibré vectoriel muni d'une connexion fortement isotrivial on note parfois $(\CE,\nabla,\Filb)=(D_K,d\otimes\id,\Filb)$, ce qui sous-entend que $D_K=D_K^{\nabla}$. 

\subsubsection{Suite exacte fondamentale}
La proposition suivante est une reformulation d'un théorème de Kedlaya-Liu (\cf \cite[Corollary 8.7.10]{keli1}) :
\medskip
\begin{prop}\label{prop:reconstr}
Soit $D$ un isocristal sur $k$ et soit $\Filb$ une filtration décroissante finie et exhaustive sur le fibré vectoriel $\CE = D_K\otimes_K\CO_{X_K}$, que l'on munit de la connexion triviale $\id\otimes d$ pour définir un fibré plat filtré $(\CE,\nabla,\Filb)$. Supposons qu'en tout point géométrique de $X_K$ la filtration soit faiblement admissible. Le noyau de l'application naturelle 
$$
\BV \coloneqq \Ker\left [ \BX_{\cris}(D)\rightarrow \BX^-_{\dR}(\CE)\right ]
$$
définit un $\Qp$-système local proétale fortement isotrivial d'isocristal associé $D$. De plus, l'application $\BX_{\cris}(D)\rightarrow \BX^-_{\dR}(\CE)$ est surjective.
\end{prop}
\medskip
\begin{proof}
Comme la question est locale sur le site proétale, on se restreint à $S\rightarrow X_K$ un espace adique perfectoïde sur $X_K$ et on va montrer que $\BV_S$ définit un système local isotrivial. Comme le site proétale de $S$ et de son basculé $S^{\flat}$ sont équivalents on peut supposer que $S$ est perfectoïde de caractéristique $p$.

On va maintenant considérer la courbe de Fargues-Fontaine relative sur $S$, $\CX_{\FF,S}$. Alors, un théorème\footnote{C'est la version relative d'un théorème classique de Fargues et Fontaine (\cf  \cite[Théorème 9.3.1]{fafo}).} de Kedlaya et Liu (\cf  \cite[Corollary 8.7.10]{keli1}) nous assure que les fibrés vectoriels semi-stables, ponctuellement de pente $0$ sur $\CX_{\FF,S}$ sont équivalents aux $\Qp$-systèmes locaux proétales sur $S$. Or, $\BX_{\cris}(D)$ sont les sections du fibré vectoriel $\sE(D)$ associé à $D$ sur $\CX_{\FF,S}$ et $\BX^-_{\dR}(\CE)$ sont les sections d'un faisceau gratte-ciel en la section $\iota_{\infty,S}\colon S\incl \CX_{\FF,S}$. Comme on considère précisément une modification faiblement admissible de $\sE(D)$, le noyau est un $\Qp$-système local proétale et on a directement que
$$
\BV\otimes_{\Qp}\BBcris\cong {D}\otimes_{K_0}\BBcris,
$$
ce qui permet de conclure que $\BV$ est isotrivial cristallin d'isocristal associé $D$. On conclut de même que le fibré associé est bien $(\CE,\nabla,\Filb)=(D_K\otimes_K\CO_{X_K},\id\otimes d,\Filb)$ et donc que $\BV$ est fortement isotrivial. 
\end{proof}
\medskip
\begin{prop}\label{prop:suitexfond}
 Soient $\BV$ un système local fortement isotrivial, $D=\BD(\BV)$ et $(\CE,\nabla,\Filb)$ le fibré plat filtré associé. Alors on a une suite exacte de faisceaux proétales
 $$
 0\rightarrow \BV\rightarrow \BX_{\cri}(D)\rightarrow \BX_{\dR}^-(\CE)\rightarrow 0.
 $$
\end{prop}
\medskip
\begin{proof}
Premièrement, notons que le théorème est vrai si $X_K$ est un point : c'est l'équivalence entre admissible et faiblement admissible (\cf  \cite{cofo}, voir aussi \cite[Chapitre 10]{fafo}). D'après la proposition précédente, on sait que 
$$
\BV'\coloneqq \Ker\left [ \BX_{\cri}(D)\rightarrow \BX_{\dR}^-(\CE)\right ]
$$
définit un $\Qp$-système local proétale isotrivial de $\varphi$-module associé $D$. Elle permet de plus d'obtenir la surjectivité de la dernière application. Il suffit de justifier que $\wh{\BV}\cong \BV'$. On sait qu'on a une inclusion $\wh{\BV}\incl \BV'\otimes_{\Qp}\BBcris$. De plus, comme le théorème est vérifié pour les points, le germe de cette inclusion est un isomorphisme sur le germe de~$\BV'$ en tout point proétale de~$X_K$. On en déduit que l'inclusion induit un isomorphisme $\wh{\BV}\cong \BV$.
\end{proof}
\subsubsection{Groupes $p$-divisibles}\label{subsubsec:grppd} 
Soit $\CG\rightarrow X$ un groupe $p$-divisible obtenu par déformation d'un groupe $p$-divisible, \ie  il existe un groupe $p$-divisible $G_0$ sur $k$ et une quasi-isogénie $\rho \colon G_0\times_k X_k \dashrightarrow \CG\times_X X_k$. Notons ${D}_0^+ \coloneqq {D}^+(G_0)$ le module de Dieudonné covariant de $G_0$ sur $\rmW(k)$ et ${D}_0\coloneqq {D}_0^+\otimes_{\rmW(k)}K_0$ l'isocristal associé de Frobenius $\varphi_0\colon {D}_0\rightarrow {D}_0$. On définit de plus le \emph{module de Tate entier} de $\CG$ par 
$$
\BT_p(\CG)\coloneqq \varprojlim_n \CG[p^n]_K
$$
qui définit un $\Zp$-faisceau (pro)étale sur $X_K$ et $\BV_p(\CG)\coloneqq \BT_p(\CG)\otimes_{\Qp}\BQ_p$ le \emph{module de Tate rationnel} qui définit un $\Qp$-faisceau (pro)étale sur $X_K$. Rappelons que les poids de Hodge-Tate de ce système local sont contenus dans $\{0,1\}$, \ie la filtration est donnée par $0=\Fil^1\CE\subset\Fil^0\CE\subset\Fil^{-1}\CE=\CE$.
\medskip
\begin{lemm}\label{lem:pdivisoc}
Le module de Tate $\BV_p(\CG)$ définit un $\Qp$-système local fortement isotrivial sur $X_K$, d'isocristal associé ${D}_0[-1]$.
\end{lemm}
\medskip
\begin{proof}
C'est une réinterprétation de \cite[Proposition 5.15]{razi}.
\end{proof}
\medskip
\begin{rema}
D'après \cite{scwe}, le faisceau proétale $\BX_{\cri}^+({D}_0(1))$ est le revêtement universel du groupe $p$-divisible \ie $\wt{\CG}\coloneqq \varprojlim_{[p]}\CG\cong \BX_{\cri}^+({D}_0(1))$. Ainsi le foncteur $\BV\mapsto \BX_{\cri}^+(\BD(\BV(1)))$ définit un \emph{revêtement universel} du système local. La suite exacte de \cite[Proposition 3.4.2 $(v)$]{scwe},
$$
0\rightarrow \BV_p(\CG)\rightarrow \wt{\CG}\rightarrow \Lie(\CG)\rightarrow 0,
$$
n'est autre que la partie $(\ \cdot \ )^+$ de la suite exacte fondamentale de la proposition \ref{prop:suitexfond}.
\end{rema}

\section{Cohomologies à coefficients}
\Subsection{Cohomologies (pro)étale des systèmes locaux}
Soit $\BV^+$ un système local étale en $\BZ_p$-réseaux sur $X_K$ un espace rigide lisse sur $K$. Par définition, $\BV^+=\{ \BV^+[p^m]\}_{m\in \BN}$ et on définit 
$$
\rmR\Gamma_{\!\et}\Harg{X_C}{\BV^+}\coloneqq \holim_m\rmR\Gamma_{\!\et}\Harg{X_C}{\BV^+[p^m]},
$$
où les morphismes de transition dans la limite homotopique sont donnés par la multiplication par $p$. Pour le $\Qp$-système local associé, $\BV=\BV^+\otimes_{\BZ_p}\Qp$, on pose,
$$
\rmR\Gamma_{\!\et}\Harg{X_C}{\BV}\coloneqq\rmR\Gamma_{\!\et}\Harg{X_C}{\BV^+}\otimes_{\Zp}\Qp.
$$
En particulier, pour tout entier $n\geqslant 0$,
$$
\rmH^n_{\et}\Harg{X_K}{\BV}\coloneqq \intn{\varprojlim_m\rmH_{\et}^n\Harg{X_K}{\BV^+[p^m]}\otimes_{\Zp}\Qp}.
$$
\begin{rema}
On peut définir la cohomologie étale d'un système local qui n'admet pas de réseau. Soit $\BV$ un $\Qp$-système local sur un espace rigide $X_K$ lisse sur $K$. D'après la définition de de Jong (\cf \cite{dej}), comme on peut trivialiser le système local sur un revêtement étale, on sait qu'il existe un recouvrement étale galoisien $\pi\colon Y_K\rightarrow X_K$, tel que $\BV_{Y_K}\coloneqq\pi^*\BV$ admet un réseau $\BV_{Y_K}^+$. Notons $H=\Aut(Y_K/X_K)$. Il est alors raisonnable de définir
$$
\rmR\Gamma_{\!\et}\Harg{X_K}{\BV}\coloneqq \rmR\Gamma\Harg{H}{\rmR\Gamma_{\!\et}\Harg{Y_K}{\BV_{Y_K}}}.
$$
En effet, on vérifie facilement à l'aide de la suite spectrale de Hochschild-Serre que cette définition est indépendante du recouvrement choisi.
\end{rema}
La définition de la cohomologie proétale est plus directe, puisqu'on peut définir, pour $\BV$ un $\Qp$-système local proétale sur $X_K$, le complexe $\rmR\Gamma_{\!\pet}\Harg{X_K}{\BV}$, qui est un élément de $\sD(\rmC_{\Qp})$. Le foncteur $\BV\mapsto \rmR\Gamma_{\!\pet}\Harg{X_K}{\BV}$ est le foncteur dérivé des sections globales $\BV\mapsto \BV(X_K)$. Rappelons le lemme suivant, qui est une conséquence de \cite[Corollary 3.17 (ii)]{schphdg} :
\medskip
\begin{lemm}
Supposons que $X_K$ soit un espace rigide lisse sur $K$, quasi-compact ou propre et soit $\BV$ un $\Qp$-système local sur $X_K$. On a un quasi-isomorphisme
$$
\rmR\Gamma_{\!\pet}\Harg{X_K}{\BV}\cong\rmR\Gamma_{\!\et}\Harg{X_K}{\BV}.
$$
\end{lemm}
\medskip
\qed
 
\Subsection{Cohomologie de de Rham}
\subsubsection{Le complexe de de Rham à coefficients}
Soit $X_K$ un espace rigide lisse sur $K$. Soit $\CE=(\CE,\nabla,\Filb)\in \Vect_{\nabla}^{\bullet}(X_K)$ un fibré plat filtré. On définit le \emph{complexe de de Rham}  :
$$
\rmR\Gamma_{\!\dR}\Harg{X_K}{\CE}\coloneqq \Omega^{\bullet}_{\CE}= \intn{\CE \xrightarrow{\nabla}\CE \otimes_{\CO_{X_K}}\Omega^1_{X_K}\xrightarrow{\nabla^{(2)}}\dots}.
$$
Ce complexe définit un élément de $\sD(\rmC_K)$ que l'on note toujours $\rmR\Gamma_{\!\dR}\Harg{X_K}{\CE}$ et son hypercohomologie définit des espaces que l'on note, pour tout entier $n\geqslant 0$, $\rmH^n_{\dR}\Harg{X_K}{\CE}$. Si $X_K$ est un espace Stein alors $\rmH^n_{\dR}\Harg{X_K}{\CE}$ est un espace de Fréchet. Le complexe de de Rham est muni d'une filtration
$$
\Fil^k \rmR\Gamma_{\!\dR}\Harg{X_K}{\CE}\coloneqq \intn{\Fil^k \CE \xrightarrow{\nabla}\Fil^{k-1}\CE \otimes_{\CO_{X_K}}\Omega^1_{X_K}\xrightarrow{\nabla^{(2)}}\dots},
$$
qui est bien définie par la transversalité de Griffiths. Cette filtration définit une filtration sur $\rmH^n_{\dR}\Harg{X_K}{\CE}$ pour tout entier $n\geqslant 0$. Par exemple, la filtration sur $\rmH^0_{\dR}\Harg{X_K}{\CE}$ est définie directement à partir de la filtration sur $\CE$. Si $(\CE,\nabla)=(\CO_X,d)$ on note simplement $\rmR\Gamma_{\!\dR}(X_K)$ le complexe de de Rham. Comme la cohomologie cohérente d'un fibré vectoriel s'annule en degré $n>0$ sur un espace Stein, on obtient le lemme suivant : 
\medskip
\begin{lemm}
Supposons que $X_K$ soit un espace rigide Stein lisse sur $K$. On a un quasi-isomorphisme strict
$$
\rmR\Gamma_{\!\dR}\Harg{X_K}{\CE}\cong \Omega^{\bullet}_\CE(X_K)\cong\intn{\CE(X_K)\xrightarrow{\nabla}\Omega^1(X_K)\wotimes_{\CO(X_K)}\CE(X_K)\xrightarrow{\nabla^{(2)}}\cdots}.
$$
De plus, si $\Fil^{1}\CE=0$, on a
$$
\Fil^{1}\rmR\Gamma_{\!\dR}\Harg{X_K}{\CE}\cong \intn{0\rightarrow\Omega^1(X_K)\wotimes_{\CO(X_K)}\Fil^0\CE(X_K)\rightarrow\cdots}.
$$
Les termes de droite sont des complexes d'espaces de Fréchet.
\end{lemm}
\qed

Soit $\Bdrp$ \emph{l'anneau des périodes de de Rham} et $\Bdr\coloneqq\Frac(\Bdrp)$ son corps des fractions. On note $t\in\Bdrp$ l'uniformisante définie par $t\coloneqq \log [\varepsilon]$ où $\varepsilon$ est une suite compatible de racines $p^n$-ièmes de l'unité. L'anneau $\Bdrp$ est un espace de Fréchet sur $K$ et donc, d'après \cite[Lemma 2.3]{codoniste}, la flèche naturelle
$$
\rmR\Gamma_{\!\dR}\Harg{X_K}{\CE}\wotimes_K\Bdrp\rightarrow \rmR\Gamma_{\!\dR}\Harg{X_K}{\CE}\wotimes_K^R\Bdrp,
$$
où le terme de gauche est le complexe de de Rham dont on a tensorisé les termes par $\Bdrp$ au-dessus de $K$, définit un quasi-isomorphisme strict dans $\sD(\rmC_K)$. De plus, sur le membre de gauche la filtration est définie par 
$$
\Fil^k \intn{\rmR\Gamma_{\!\dR}\Harg{X_K}{\CE}\wotimes_K\Bdrp}\coloneqq\hocolim_{i+j\geqslant k}\Fil^j \rmR\Gamma_{\!\dR}\Harg{X_K}{\CE}\wotimes_Kt^i\Bdrp.
$$
\medskip
La preuve du lemme suivant est la même que \cite[Example 3.30]{codoniste} : 
\begin{lemm}\label{lem:calcfilun}
Supposons que $X_K$ soit un espace rigide Stein lisse sur $K$. On a un quasi-isomorphisme strict
{
$$
\Fil^{1}\intn{\rmR\Gamma_{\!\dR}\Harg{X_K}{\CE}\wotimes_K^{R}\Bdrp}\cong\intn{\sum_{i\in\BN}\Fil^{-i}\CE(X_K)\wotimes_{K}t^{i+1}\Bdrp\xrightarrow{\nabla}\sum_{i\in\BN}\Fil^{-i}\Omega_{\CE}^1(X_K)\wotimes_Kt^{i}\Bdrp\rightarrow\cdots}.
$$}
Le terme de gauche est naturellement un complexe d'espaces de Fréchet.
\end{lemm}
\medskip
\qed

\subsubsection{Le petit complexe de de Rham des opers}\label{subsub:lepetitcom}
Supposons que $X_K$ est une courbe sur $K$ et soit $\CE=(\CE,\nabla,\Filb)$ un oper de poids $(a,b)$. On définit le \emph{petit complexe de de Rham} par 
$$
\rmR\Gamma_{\!\Op}\Harg{X_K}{\CE}\coloneqq \big ( \Gr_b\CE\xrightarrow{L_{\nabla}} \Gr_a\Omega_{\CE} \big ),
$$ 
(\cf (\ref{eq:opdiffop})) muni de la filtration induite par la filtration sur les gradués, explicité en (\ref{eq:filtop}). Du lemme \ref{lem:opdec}, on déduit la proposition suivante :
\medskip
\begin{prop}\label{prop:drtoop}
Supposons que $X_K$ soit une courbe rigide lisse sur $K$. Soit $\CE=(\CE,\nabla,\Filb)$ un oper de poids $(a,b)$. On a un quasi-isomorphisme strict filtré
$$
\rmR\Gamma_{\!\Op}\Harg{X_K}{\CE}\cong\rmR\Gamma_{\!\dR}\Harg{X_K}{\CE}.
$$
\end{prop}
\qed
\medskip
Ce résultat et la description de la filtration (\ref{eq:filtop}) permet de calculer la filtration précédente : 
\medskip
\begin{lemm}\label{lem:compfildr}
Supposons que $X_K$ soit une courbe rigide Stein lisse sur $K$. On a un quasi-isomorphisme strict
{
$$
\Fil^{1}\intn{\rmR\Gamma_{\!\Op}\Harg{X_K}{\CE}\wotimes_K\Bdrp}\cong\left (\Gr_b\CE(X_K)\wotimes_Kt^{b+1}\Bdrp\xrightarrow{L_{\nabla}}\Gr_a\Omega_{\CE}(X_K)\wotimes_Kt^{a}\Bdrp\right ).
$$}
\end{lemm}
\medskip
\qed

\subsubsection{Cohomologie de de Rham isotriviale}
Le complexe de de Rham des fibrés plats fortement isotriviaux est particulièrement simple :
\medskip
\begin{lemm}\label{lem:isotdr}
Soit $X_K$ un espace rigide lisse sur $K$ et $(\CE,\nabla,\Filb)=(D_K\otimes_K\CO_{X_K},\id\otimes d,\Filb)$ un fibré plat filtré fortement isotrivial. On a un quasi-isomorphisme entre complexes 
$$
\rmR\Gamma_{\!\dR}\Harg{X_K}{\CE}\cong \rmR\Gamma_{\!\dR}(X_K)\otimes_KD_K.
$$
En particulier, pour tout entier $n\geqslant 0$, on a $\rmH^n_{\dR}\Harg{X_K}{\CE}\cong \rmH^n_{\dR}(X_K)\otimes_KD_K$.
\end{lemm}
\medskip

\Subsection{Cohomologie de Hyodo-Kato}
Soit $X$ un schéma formel et semi-stable sur $\rmO_K$. On définit la cohomologie de Hyodo-Kato comme cohomologie rationnelle log-cristalline de la fibre spéciale $X_k$ sur $\rmO^{0}_{K_0}$ où $\rmO^{0}_{K_0}$ désigne le log-schéma formel $\Spf(\rmO_{K_0})$ muni de la structure logarithmique induite par $\BN\rightarrow \rmO_{K_0}$, $1\mapsto 0$:
\begin{equation}\label{eqn:hkdef}
\rmR\Gamma_{\!\HK}(X_K)\coloneqq \rmR\Gamma_{\!\cris}(X_k/\rmO^{0}_{K_0})\otimes_{\rmO_{K_0}}K_0.
\end{equation}
D'après \cite[Proposition 4.11]{colniz} cette cohomologie ne dépend que de la fibre générique $X_K$ (comme on le suggère la notation) et peut être définit indépendamment de l'existence d'un modèle semi-stable (\cf \cite{colniz}). Le complexe (\ref{eqn:hkdef}) est muni d'un \emph{endomorphisme de Frobenius} $\varphi$ et d'un endomorphisme $N$, \emph{l'opérateur de monodromie}. Ces endomorphismes vérifient la relation $N\varphi=p\varphi N$. 

\subsubsection{Cohomologie de Hyodo-Kato isotriviale} 
Si $D$ un un isocristal sur $k$. On définit la \emph{cohomologie de Hyodo-Kato isotriviale} 
\begin{equation}\label{eq:defhkisot}
\rmR\Gamma_{\!\HK}\Harg{X_K}{D}\coloneqq \rmR\Gamma_{\!\HK}(X_K)\otimes_{K_0}D
\end{equation}
Comme pour (\ref{eqn:hkdef}) cette cohomologie est indépendante du modèle de $X_K$. Pour $n\geqslant 0$ on note 
$$
\wt{\rmH}^n_{\HK}\Harg{X_K}{D}\coloneqq \wt{\rmH}^n\rmR\Gamma_{\!\HK}\Harg{X_K}{D},\ \rmH^n_{\HK}\Harg{X_K}{D}\coloneqq \rmH^n\rmR\Gamma_{\!\HK}\Harg{X_K}{D}.
$$
\subsubsection{Isomorphisme de Hyodo-Kato isotrivial}\label{subsubsec:hkisot}
Pour définir l'isomorphisme de Hyodo-Kato, on doit faire une digression par la géométrie surconvergente. Il s'avère que  la cohomologie de Hyodo-Kato des espaces Stein est naturellement surconvergente. Soit $\wt{X}$ un schéma faiblement formel semi-stable sur $\rmO_K$ et $\wt{X}_K$ l'espace surconvergent associé que l'on suppose lisse sur $K$ (\cf \cite{gkrig}). On définit la \emph{cohomologie de Hyodo-Kato surconvergente} comme la cohomologie surconvergente rigide de la fibre spécial $X_k=\wt{X}\otimes k$ sur $\rmO_{K_0}^0$ (\cf \cite[3.1.2]{codoniste}):
$$
\rmR\Gamma_{\HK}(\wt{X}_K)\coloneqq\rmR\Gamma_{\! \rig}(X_k/\rmO_{K_0}^0). 
$$
Cette cohomologie est indépendante du modèle (\cf \cite[Proposition 5.3]{colniz}), comme le suggère la notation. On peut associer à $\wt{X}_K$ son \emph{complété}, qui un espace rigide et que l'on note $X_K$. 
\medskip
\begin{lemm}\label{lemm:hksurconv}
Soit $\wt{X}_K$ un espace surconvergent Stein lisse sur $K$. On a un quasi-isomorphisme stricte
$$
\rmR\Gamma_{\! \HK}(\wt{X}_K)\rightarrow \rmR\Gamma_{\! \HK}(X_K).
$$
\end{lemm}
\medskip
\begin{proof}
D'après la preuve de \cite[Theorem 4.1]{codoniste}, ce quasi-isomorphisme strict pour les espace Stein découle de la comparaison avec la cohomologie convergente (\cf \cite[Remark 4.4]{codoniste}) ; 
$$
\rmR\Gamma_{\! \rig}(X_k/\rmO_{K_0}^0)\rightarrow\rmR\Gamma_{\! \conv}(X_k/\rmO_{K_0}^0)\cong \rmR\Gamma_{\! \cri}(X_k/\rmO_{K_0}^0)\otimes_{\rmO_{K_0}}K_0.
$$
On montre ensuite que la première flèche est un quasi-isomorphisme stricte en utilisant un recouvrement Stein.
\end{proof}
Comme les espaces Stein sont partiellement propre ils admettent un modèle faiblement formel (\cf \cite[Theorem 2.27]{gkrig}) ; et donc les espaces rigides lisses Stein admettent des modèles surconvergents. Si on se donne $X_K$ un espace rigide Stein lisse sur $K$ alors il existe $\wt{X}_K$ un espace surconvergent Stein lisse sur $K$ d'espace rigide associé $X_K$ et on va supposer qu'ils ont des modèles semi-stables sur $\rmO_K$ (mais ce n'est pas nécessaire).

On construit le morphisme de Hyodo-Kato à partir du morphisme de Hyodo-Kato surconvergent. Soient $D$ un isocristal sur $k$ et $(\CE,\nabla)\coloneqq(D_K\otimes_K\CO_{X_K},d\otimes \id)$ un fibré plat fortement isotrivial où $D_K\coloneqq D\otimes_{K_0}K$. On a un \emph{morphisme de Hyodo-Kato isotrivial} 
$$
\iota_{\HK}\colon\rmR\Gamma_{\!\HK}\Harg{X_K}{D}\rightarrow\rmR\Gamma_{\!\dR}\Harg{X_K}{\CE}.
$$
En effet, on commence par considérer le morphisme de Hyodo-Kato surconvergent (\cf \cite[3.1.3]{codoniste}) :
$$
\wt{\iota}_{0}\colon \rmR\Gamma_{\!\HK}(\wt{X}_K)\rightarrow\rmR\Gamma_{\!\dR}(\wt{X}_K).
$$
où le terme de droite est la cohomologie de de Rham surconvergente de $\wt{X}_K$. Par \cite[Theorem 2.26]{gkrig} (\cf le premier point de la preuve de \cite[Theorem 4.1]{codoniste}), comme $X_K$ est Stein, on a un quasi-isomorphisme strict $\rmR\Gamma_{\!\dR}(\wt{X}_K)\cong \rmR\Gamma_{\!\dR}(X_K)$. On définit le morphisme de Hyodo-Kato comme la composée
$$
\iota_{\HK}\colon \rmR\Gamma_{\!\HK}(X_K)\cong\rmR\Gamma_{\!\HK}(\wt{X}_K)\xrightarrow{\wt{\iota}_{0}}\rmR\Gamma_{\!\dR}(\wt{X}_K)\cong\rmR\Gamma_{\!\dR}(X_K),
$$
où le premier isomorphisme provient du lemme \ref{lemm:hksurconv}. On applique maintenant le produit tensoriel par $\cdot\ \otimes_{K_0}D$ à ce morphisme, qui, comme $D\otimes_{K_0}K=D_K$, définit une application 
$$
\iota_{\HK} \colon\rmR\Gamma_{\!\HK}(X_K)\otimes_{K_0}{D}\rightarrow\rmR\Gamma_{\!\dR}(X_K)\otimes_KD_K.
$$ 
Comme c'est le cas pour $\iota_{0}$, $\iota_{\HK}$ devient un quasi-isomorphisme strict lorsqu'on applique $\otimes_{K_0}K$ au terme de gauche. D'après le lemme \ref{lem:isotdr} et (\ref{eq:defhkisot}), cette flèche définit
$$
\iota_{\HK}\colon \rmR\Gamma_{\!\HK}\Harg{X_K}{D}\rightarrow\rmR\Gamma_{\!\dR}\Harg{X_K}{\CE},
$$
qui induit quasi-isomorphisme strict :
$$
\iota_{\HK}\colon \rmR\Gamma_{\!\HK}\Harg{X_K}{D}\otimes_{K_0}K\xrightarrow{\sim}\rmR\Gamma_{\!\dR}\Harg{X_K}{\CE}.
$$

\section{Cohomologies syntomiques isotriviales}
Soit $X_K$ un espace rigide Stein lisse sur $K$. On définit la cohomologie syntomique (géométrique) d'un $\Qp$-système local fortement isotrivial $\BV$ d'isocristal associé $D$ et de fibré associé $(\CE,\nabla,\Filb)$. Supposons que les pentes de $D$ sont toutes négatives. On a une application $\iota \colon \whBstp\hookrightarrow \Bdrp$ (\cf \cite[3.2.1]{codoniste} pour la définition de $\whBstp$ et de $\iota$) et on définit dans $\sD(\rmC_{\Qp})$ le complexe

$$
\rmR\Gamma_{\!\syn}\Harg{X_C}{\BV}\coloneqq \left [ \left [ \rmR\Gamma_{\!\HK}\Harg{X_K}{D}\wotimes_{K_0}^R\whBstp \right ]^{N=0,\varphi=1}\xrightarrow{\iota_{\HK}\otimes \iota} \intn{\rmR\Gamma_{\!\dR}\Harg{X_K}{\CE}\wotimes_K^R\Bdrp}/\Fil^{0}\right].
$$
Rappelons que $[-\xrightarrow{\iota_{\HK}\otimes \iota} -]\coloneqq \holim (-\xrightarrow{\iota_{\HK}\otimes \iota} -\leftarrow 0)$ désigne la fibre dans $\sD(\rmC_{\Qp})$ de l'application $\iota_{\HK}\otimes \iota$ et $\left [ \rmR\Gamma_{\!\HK}\Harg{X_K}{D}\wotimes_{K_0}^R\whBstp \right ]^{N=0,\varphi=p^i}$, pour $i\in \BN$ désigne l'espace propre dérivé de $N$ et $\varphi$ \ie la limite homotopique du diagramme
\begin{center}
\begin{tikzcd}
\rmR\Gamma_{\!\HK}\Harg{X_K}{D}\wotimes_{K_0}^R\whBstp \ar[r, "\varphi - p^i"]\ar[d, "N"] & \rmR\Gamma_{\!\HK}\Harg{X_K}{D}\wotimes_{K_0}^R\whBstp \ar[d, "N"]\\
\rmR\Gamma_{\!\HK}\Harg{X_K}{D}\wotimes_{K_0}^R\whBstp \ar[r, "p\varphi - p^i"] & \rmR\Gamma_{\!\HK}\Harg{X_K}{D}\wotimes_{K_0}^R\whBstp
\end{tikzcd}
\end{center}
Pour tout entier $i\geqslant 0$, posons
$$
\begin{gathered}
\HK(X_C,{D},i)\coloneqq \left [ \rmR\Gamma_{\!\HK}\Harg{X_K}{D}\wotimes_{K_0}^R\whBstp \right ]^{N=0,\varphi=p^{i}}, \\
{\DR}(X_C,\CE,i)\coloneqq\intn{\rmR\Gamma_{\!\dR}\Harg{X_K}{\CE}\wotimes_K^R\Bdrp}/\Fil^{i},
\end{gathered}
$$
de sorte que 
$$
\rmR\Gamma_{\!\syn}\Harg{X_C}{\BV(i)}=\left [ \HK(X_C,{D},i)\rightarrow {\DR}(X_C,\CE,i) \right ].
$$
Lorsque $i=1$, on omet simplement le $i$ pour alléger les notations et on note simplement $\HK(X_C,{D})\coloneqq\HK(X_C,{D},1)$ et ${\DR}(X_C,\CE)\coloneqq {\DR}(X_C,\CE,1)$.
\medskip
\begin{exem}\label{ex:h0syn}
Supposons que $X_K$ est un espace rigide Stein lisse sur $K$ et géométriquement irréductible. Alors
$$
\intn{\rmH_{\HK}^0\Harg{X_K}{D}\wotimes_{K_0}\Bstp}^{\varphi=1,N=0}=\intn{{D}\otimes_{K_0}\Bcrisp}^{\varphi=1}\cong X^{+}_{\st}({D}).
$$
De plus, $D_K=\rmH^0_{\dR}\Harg{X_K}{\CE}$, qui est muni de la filtration induite. Donc $\wt{\rmH}^0_{\syn}\Harg{X_C}{\BV}$ est le noyau du morphisme naturel
$$
X^{+}_{\st}({D})\rightarrow \intn{D_K\otimes_K\Bdrp}/\Fil^{0}.
$$
La filtration sur $D_K$ est admissible donc $\wt{\rmH}^0_{\syn}\Harg{X_C}{\BV}$ est de dimension finie sur $\Qp$ et $\Fil^0(D_K\otimes_K\Bdrp)\cap X_{\st}({D})$ est fermé dans $X^{+}_{\st}({D})$ ; ainsi $\wt{\rmH}^0_{\syn}\Harg{X_C}{\BV}$ est classique.
\end{exem}
\medskip
\Subsection{La partie Hyodo-Kato}
 En suivant la preuve de \cite[Lemma 3.28]{codoniste} et le lemme \ref{lemm:hksurconv}, on obtient :
\medskip
\begin{lemm}\label{lem:steinhk}
Supposons que $X$ est un schéma formel Stein semi-stable sur $\rmO_K$, alors pour tous $n\geqslant 0$ la cohomologie de $\HK(X_C,D)$ est classique et on a un isomorphisme topologique
$$
\rmH^n \HK(X_C,{D})\cong \intn{\rmH_{\HK}^n\Harg{X_K}{D}\wotimes_{K_0}\whBstp}^{N=0,\varphi=p}.
$$
En particulier, $\rmH^n \HK(X_C,{D})$ est un espace de Fréchet.
\end{lemm}
\medskip
\qed

\Subsection{La partie de Rham}
Si $(D_K,\Filb)$ est un $K$-espace de Fréchet filtré dont les sauts de la filtration sont négatifs, soit $X^+_{\dR}(D_K)\coloneqq D_K\wotimes_K \Bdrp$ et, pour $k\in \BZ$, posons 
$$
X^+_{\dR}(D_K)_k\coloneqq X^+_{\dR}(D_K)/\Fil^k\left ( D_K\wotimes_K \Bdrp \right ).
$$
Une conséquence du lemme \ref{lem:calcfilun} est le lemme suivant :
\medskip
\begin{lemm}\label{lem:filtdr}
Supposons que $X$ est un espace rigide Stein lisse sur $K$ et soit $(\CE,\nabla,\Filb)$ un fibré plat filtré fortement isotrivial dont les sauts de la filtration sont $\leqslant 0$. On a un quasi-isomorphisme strict
$$
{\DR}(X_C,\CE)\cong X_{\dR}^+(\CE(X_K))_{1}\rightarrow X_{\dR}^+(\Omega^1_{\CE}(X_K))_{0}\rightarrow \dots,
$$
où le complexe de droite est un complexe d'espaces de Fréchet.
\end{lemm}
\medskip
\qed

\subsubsection{Le cas des opers}
Supposons que $X_K$ soit une courbe rigide Stein lisse sur $K$ et que $\CE$ soit un oper de poids $(a,b)$ avec $a\geqslant 0$. L'intérêt des opers est que l'on peut calculer la cohomologie de ${\DR}(X_C,\CE)$ qui apparaît dans le diagramme fondamental.
\medskip
\begin{lemm}\label{lem:calch0}
On a une suite exacte stricte d'espaces de Fréchet
$$
0\rightarrow \Gr_b\CE(X_K)\wotimes_Kt^a\rmB_{b+1}\rightarrow\rmH^0{\DR}(X_C,\CE)\rightarrow \rmH^0_{\dR}\Harg{X_K}{\CE}\wotimes_K\rmB_a\rightarrow 0.
$$
De plus,
$$
\rmH^1{\DR}(X_C,\CE)\cong \rmH^1_{\dR}\Harg{X_K}{\CE}\wotimes_K\rmB_a.
$$
\end{lemm}
\medskip
\qed

\begin{proof}
Rappelons que d'après le lemme \ref{lem:calch0} $\rmH^0\coloneqq\rmH^0{\DR}\intn{X_C,\CE}$ est le noyau de la connexion
$$
(\CE(X_K)\wotimes_K \Bdrp)/\Fil^1\xrightarrow{ \nabla \otimes \id} (\Omega^1_{\CE}(X_K)\wotimes_K \Bdrp)/\Fil^0
$$
et $\rmH^1\coloneqq \rmH^1{\DR}(X_C,\CE)$ son conoyau. On commence par montrer que ce noyau est isomorphe au noyau de 
$$
\Gr_b\CE(X_K)\wotimes_K\rmB_{b+1}\xrightarrow{L_{\nabla}}\Gr_a\Omega_{\CE}(X_K)\wotimes_K\rmB_a.
$$
On a des décompositions $K$-linéaires du lemme \ref{lem:opdec}
$$
\begin{gathered}
\CE(X_K)\wotimes_K \Bdrp \cong \Gr_b\CE(X_K)\wotimes_K\Bdrp\oplus \Fil^{-b+1}\wotimes_K\Bdrp,\\
\Omega^1_{\CE}(X_K)\wotimes_K \Bdrp \cong \Gr_a\Omega_{\CE}(X_K)\wotimes \Bdrp\oplus \tfrac{\Omega_{\CE}(X_K)}{\Fil^{-a}}\wotimes_K \Bdrp ,
\end{gathered}
$$
et par la définition des opers (cf. définition \ref{def:oper}), la différentielle induit un isomorphisme
$$
\Fil^{-b+1}\wotimes_K\Bdrp\xrightarrow{\sim}  \tfrac{\CE(X_K)}{\Fil^{-a}}\wotimes_K \Bdrp.
$$
De plus, on a des décompositions $K$-linéaires
$$
\begin{gathered}
\Gr_b\CE(X_K)\wotimes_K\Bdrp\cap \Fil^1 = \Gr_b\CE(X_K)\wotimes_Kt^{b+1}\Bdrp,\\
\Gr_a\CE(X_K)\wotimes_K \Bdrp\cap \Fil^0 = \Gr_a\CE(X_K)\wotimes_K t^a\Bdrp,
\end{gathered}
$$
et comme, par définition, $\nabla$ induit des isomorphismes $\Fil^{-q}/\Fil^{-q+1}\xrightarrow{\sim}(\Fil^{-q-1}/\Fil^{-q})\otimes_{\CO(X_K)}\Omega^1(X_K)$ pour tout entier $q$ tel que $a\leqslant q \leqslant b-1$, on en déduit que $d$ induit un isomorphisme 
$$
(\Fil^{-b+1}\wotimes_K\Bdrp)\cap \Fil^1\xrightarrow{\sim}(\tfrac{\CE(X_K)}{\Fil^{-a}}\wotimes_K \Bdrp)\cap \Fil^0.
$$
Ainsi, comme annoncé, on a bien 
$$
\rmH^0 =\Ker\intn{\Gr_b\CE(X_K)\wotimes_K\rmB_{b+1}\xrightarrow{L_{\nabla}}\Gr_a\Omega_{\CE}(X_K)\wotimes_K\rmB_a},
$$
et de même $\rmH^1=\Coker \intn{\Gr_b\CE(X_K)\wotimes_K\rmB_{b+1}\xrightarrow{L_{\nabla}}\Gr_a\Omega_{\CE}(X_K)\wotimes_K\rmB_a}$.

En appliquant le lemme du serpent au diagramme 
\begin{center} 
\begin{tikzcd}
0\ar[r]&\Gr_b\CE(X_K)\wotimes_Kt^a\rmB_{b+1}\ar[r]\ar[d]&\Gr_b\CE(X_K)\wotimes_K\rmB_{b+1}\ar[r]\ar[d,"L_{\nabla}"]&\Gr_b\CE(X_K)\wotimes_K\rmB_a\ar[r]\ar[d,"L_{\nabla}"]& 0\\
&0\ar[r]&\Gr_a\Omega_{\CE}(X_K)\wotimes_K\rmB_a\ar[r,equal]&\Gr_a\Omega_{\CE}(X_K)\wotimes_K\rmB_a&
\end{tikzcd}
\end{center}
on obtient la suite exacte
$$
0\rightarrow \Gr_b\CE(X_K)\wotimes_Kt^a\rmB_{b+1}\rightarrow\rmH^0\rightarrow \rmH^0_{\dR}\Harg{X_K}{\CE}\wotimes_K\rmB_a\rightarrow 0,
$$
et l'isomorphisme 
$$
\rmH^1\cong \rmH^1_{\dR}\Harg{X_K}{\CE}\wotimes_K\rmB_a.
$$
\end{proof}
\Subsection{Le diagramme fondamental pour les opers}
Avant d'énoncer le théorème principal de cette section, on donne le lemme suivant qui est une généralisation du lemme \cite[Lemma 3.30]{codoniste} : 
\medskip
\begin{lemm}\label{lem:suitexho}
Soit $M$ un $(\varphi,N)$-module effectif (toutes les pentes sont $\geqslant 0$) fini sur $K_0$ et soient $a,b\in\BN$ des entiers tels que $b\geqslant a\geqslant 0$, alors
$$
0\rightarrow\intn{M\otimes_{K_0} \whBstp}^{\varphi=p^a,N=0}\xrightarrow{t^{b+1}} \intn{M\otimes_{K_0}\whBstp}^{\varphi=p^{a+b+1},N=0}\xrightarrow{} M\otimes_{K_0}\rmB_{b+1}
$$
est exacte. De plus, la dernière application est surjective si les pentes de $M$ sont $\leqslant a$.
\end{lemm}
\medskip
La première partie est obtenue par récurrence sur $b$ en utilisant un dévisage. La surjectivité se démontre comme dans la preuve de \cite[Lemma 3.30]{codoniste} en calculant la Dimension.
\medskip
\begin{rema}\label{rem:suitexho}
Comme dans  \cite{codoniste}, on a besoin d'une version du lemme précédent pour $M = \varprojlim M_i$ où $\{M_i\}_i$ est un système projectif de $(\varphi,N)$-modules effectifs finis sur $K_0$. Comme la limite projective est exacte à gauche, on obtient directement pour  $a,b\in\BN$ des entiers tel que $b\geqslant a\geqslant 0$, la suite exacte 
$$
0\rightarrow\intn{M\wotimes_{K_0} \whBstp}^{\varphi=p^a,N=0}\xrightarrow{t^{b+1}} \intn{M\wotimes_{K_0}\whBstp}^{\varphi=p^{a+b+1},N=0}\xrightarrow{} M\wotimes_{K_0}\rmB_{b+1}.
$$
De plus, on peut montrer que si les pentes des $M_i$ sont $\leqslant a$ alors la dernière flèche est surjective mais on ne l'utilisera pas. La surjectivité de la dernière flèche dans le cas $b=0$ est utilisé dans \cite[Proposition 3.36]{codoniste} et repose sur un argument donné dans la preuve de \cite[Lemme 3.28]{codoniste}. L'argument dans le cas où $b>0$ est le même. En particulier, pour $D$ un isocristal effectif, cette situation s'applique à $M=\rmH^j_{\HK}\Harg{X_K}{D}=\varprojlim_i \rmH^j_{\HK}\Harg{\wt{U}_{i,K}}{D}$ où $\{\wt{U}_{i,K}\}$ est un recouvrement Stein par des affinoides surconvergents d'un modèle sur convergent de $X_K$ (\cf le lemme \ref{lemm:hksurconv} et \cite[Lemme 5.3]{van}).
\end{rema}
\medskip
\medskip
\begin{prop}\label{prop:diagfondop}
Soit $X_K$ une courbe rigide Stein lisse sur $K$. Soit $\BV$ un $\Qp$-oper de poids $(a,b)$ fortement isotrivial d'isocristal associé ${D}$ et de fibré plat filtré associé $\CE=(\CE,\nabla,\Filb)$. Supposons que $a\geqslant 0$. On a une application naturelle entre suites exactes strictes d'espaces de Fréchet :
{\Small
\begin{center}
\begin{tikzcd}
t^aX_{\cri}^{1}(D[a])\ar[d]\ar[r]&\Gr_b\CE(X_K)\wotimes_K t^a\rmB_{b+1} \ar[r]\ar[d,equal]& \rmH_{\syn}^1\Harg{X_C}{\BV(1)}\ar[r]\ar[d, "\beta"]&t^aX_{\st}^{1}\intn{\rmH^1_{\HK}\Harg{X_K}{D}[a]}\ar[r]\ar[d, "\gamma"]&0\\
D_K\otimes_Kt^a\rmB_{b+1} \ar[r]& \Gr_b\CE(X_K)\wotimes_Kt^a\rmB_{b+1} \ar[r,"L_{\nabla}"]&\Gr_a\Omega_{\CE}(X_K)\wotimes_Kt^a\rmB_{b+1}\ar[r]&\rmH^1_{\dR}\Harg{X_K}{\CE}\wotimes_K t^a\rmB_{b+1}\ar[r]&0
\end{tikzcd}
\end{center}}
\noindent Ce diagramme est $\sG_K$-équivariant et on a noté :
$$
\begin{gathered}
X_{\st}^{1}\intn{\rmH^1_{\HK}\Harg{X_K}{D}[a]}\coloneqq  \intn{\rmH^1_{\HK}\Harg{X_K}{D}\otimes_{K_0}\Bstp}^{N=0,\varphi=p^{1-a}} \\
X_{\cri}^{1}(D[a])\coloneqq  \intn{D\otimes_{K_0}\Bcrisp}^{\varphi=p^{1-a}}.
\end{gathered}
$$
De plus, $\beta$ et $\gamma$ sont strictes, d'images fermées et $\Ker(\gamma) = t^{a+b+1}\intn{\rmH^1_{\HK}\Harg{X_K}{D}\otimes_{K_0}\Bstp}^{N=0,\varphi=p^{-a-b}}$. En particulier, si les pentes de $D$ sont $>-a-b$, alors $\Ker(\gamma)=0$.
\end{prop}
\medskip
\begin{proof}
Notons 
$$
\begin{gathered}
\rmR\Gamma_{\!\syn}\coloneqq \rmR\Gamma_{\!\syn}\Harg{X_C}{\BV(1)},\quad \rmR\Gamma_{\!\HK}\coloneqq \rmR\Gamma_{\!\HK}\Harg{X_K}{D}\\
\rmR\Gamma_{\!\Op}\coloneqq \rmR\Gamma_{\!\Op}\Harg{X_K}{\CE},\quad \rmR\Gamma_{\!\dR}\coloneqq \rmR\Gamma_{\!\dR}\Harg{X_K}{\CE},
\end{gathered}
$$
et pour $n\geqslant 0$ un entier, $\rmH^n_{\star}$, où $\star \in \{\syn, \HK, \Op, \dR\}$, le $n$-ième groupe de cohomologie du complexe correspondant.

En suivant le début de la preuve de \cite[Proposition 3.36]{codoniste}, notons qu'on peut remplacer $\Bstp$ par $\whBstp$ dans le diagramme, ce qui facilite les questions topologiques. On va commencer par définir les deux suites exactes du diagramme. La ligne du dessous est donnée par le petit complexe de de Rham auquel on applique $\cdot \wotimes_Kt^a\rmB_{b+1}$, soit
{\small
$$
0\rightarrow D_K\otimes_Kt^a\rmB_{b+1} \rightarrow \Gr_b\CE(X_K)\wotimes_Kt^a\rmB_{b+1}\xrightarrow{L_{\nabla} \otimes \id}\Gr_a\Omega_{\CE}(X_K)\wotimes_Kt^a\rmB_{b+1}\rightarrow \rmH^1_{\dR}\wotimes_K t^a\rmB_{b+1}\rightarrow 0.
$$
}
La suite reste exacte puisque on tensorise-complète une suite exacte de Fréchet par un Banach.

D'après la proposition \ref{prop:drtoop}, on a
$$
\rmR\Gamma_{\!\syn} = \left [ \left [ \rmR\Gamma_{\!\HK}\wotimes_{K_0}^R\whBstp \right ]^{N=0,\varphi=p}\xrightarrow{\iota_{\HK}\otimes \iota} \intn{\rmR\Gamma_{\!\Op}\wotimes_K^R\Bdrp}/\Fil^{1}\right].
$$
On obtient une suite exacte 
{
\begin{equation}\label{eq:suitexsyn0}
X_{\cri}^{1}(D)\rightarrow\rmH^0{\DR}(X_K,\CE)  \rightarrow  \rmH_{\syn}^1\rightarrow\intn{\rmH^1_{\HK}\wotimes_{K_0}\Bstp}^{N=0,\varphi=p} \xrightarrow{\iota_{\HK}\otimes\iota} \rmH^1{\DR}(X_C,\CE).
\end{equation}}
C'est presque la suite exacte du dessus dans le diagramme de la proposition, mais il faut corriger les termes extrémaux. Notons que $X_{\cris}^1(D)\cong\intn{D[a]\otimes_{K_0}\Bcrisp}^{\varphi=p^{a+1}}$.
\begin{itemize}
\itemb Pour la première flèche de (\ref{eq:suitexsyn0}), en comparant la suite exacte du lemme \ref{lem:suitexho} et la première partie du lemme \ref{lem:calch0}, on obtient le diagramme commutatif suivant : 
\begin{center}
\begin{tikzcd}
0 \ar[r]& t^aX^1_{\cris}(D[a])\ar[r]\ar[d]& X_{\cris}^1(D)\ar[r]\ar[d]& D_K\otimes_K\rmB_a\ar[r]\ar[d,equal]&0\\
0\ar[r]&\Gr_b\CE(X_K)\wotimes_K t^a\rmB_{b+1}\ar[r]&\rmH^0{\DR}(X_K,\CE)\ar[r]&\rmH^0_{\dR} \otimes_K\rmB_a\ar[r]& 0.
\end{tikzcd}
\end{center}
\itemb Pour la dernière flèche de (\ref{eq:suitexsyn0}), comme $ \rmH^1{\DR}(X_C,\CE)\cong \rmH^1_{\dR}\otimes_K \rmB_a$ d'après le lemme \ref{lem:calch0}, on utilise la remarque \ref{rem:suitexho} qui permet de calculer 
$$
\Ker\left [\intn{\rmH^1_{\HK}\wotimes_{K_0}\Bstp}^{N=0,\varphi=p} \xrightarrow{\iota_{\HK}\otimes\iota} \rmH^1{\DR}(X_C,\CE)\right ]\cong t^aX_{\st}^{1}\intn{\rmH^1_{\HK}[a]},
$$ 
\end{itemize}
Finalement, de ces deux calculs et de la suite exacte (\ref{eq:suitexsyn0}) on obtient la suite exacte suivante :
\begin{equation}\label{eq:suitexsyn1}
t^aX_{\cri}^{1}(D[a])\rightarrow\Gr_b\CE(X_K)\wotimes_K t^a\rmB_{b+1}  \rightarrow  \rmH_{\syn}^1 \rightarrow t^aX_{\st}^{1}\intn{\rmH^1_{\HK}[a]}\rightarrow 0.
\end{equation}
Dans le diagramme, la définition du carré de gauche et sa commutativité sont claire. On définit maintenant les applications $\beta$ et $\gamma$ puis on montre que le carré qu'elles forment dans le diagramme commute.
\begin{itemize}
\itemb L'application $\gamma$ est induite par la composée
$$
\rmR\Gamma_{\!\HK}\wotimes_{K_0}\Bstp\xrightarrow{\iota_{\HK}\otimes\iota}{\rmR\Gamma_{\!\Op}\wotimes_K\Bdrp}\xrightarrow{\mod t^{b+1} }\rmR\Gamma_{\!\Op}\wotimes_K\rmB_{b+1}.
$$
On déduit de la remarque \ref{rem:suitexho} le calcul de $\Ker(\gamma)$.
\itemb Le lemme \ref{lem:compfildr} nous donne l'isomorphisme
$$
\Fil^{1}\intn{\rmR\Gamma_{\!\Op}\wotimes_K^R\Bdrp}\cong \Gr_b\CE(X_K)\wotimes_Kt^{b+1}\Bdrp\rightarrow \Gr_a{\CE}(X_K)\wotimes_Kt^{a}\Bdrp,
$$
 et l'application $\beta$ est induite par la composée
{
$$
\rmR\Gamma_{\!\syn}\rightarrow\Fil^{1}\intn{\rmR\Gamma_{\!\Op}\wotimes_K^R\Bdrp}\xrightarrow{\mod t^{b+1}}\intn{0\rightarrow \Gr_a\Omega_{\CE}(X_K)\wotimes_K t^{a+1}\rmB_{b+1}\rightarrow 0}.
$$}
\end{itemize}
Finalement, le diagramme suivant, de morphismes entre triangles distingués, justifie que le carré de droite dans le diagramme commute :
{
\begin{center}
\begin{tikzcd}
\rmR\Gamma_{\!\syn}\ar[r]\ar[d,"\wt{\beta}"]&\left [ \rmR\Gamma_{\!\HK}\wotimes_{K_0}\Bstp \right ]^{\varphi=p,N=0}\ar[r,"\iota_{\HK}\otimes\iota"]\ar[d,"\iota_{\HK}\otimes\iota"]&\rmR\Gamma_{\!\Op}\wotimes_K\Bdrp /\Fil^{1}\ar[d,equal]\\
\Fil^{1}\intn{\rmR\Gamma_{\!\Op}\wotimes_K^R\Bdrp}\ar[r]\ar[d,"\mod t^{b+1}"]&\rmR\Gamma_{\!\Op}\wotimes_K\Bdrp\ar[r,"\mod \Fil^1"]\ar[d,"\mod t^{b+1}"]&\rmR\Gamma_{\!\Op}\wotimes_K\Bdrp /\Fil^{1}\ar[d]\\
\sigma_{\geqslant 1}\rmR\Gamma_{\!\Op}\wotimes_K\rmB_{b+1} \ar[r]&\rmR\Gamma_{\!\Op}\wotimes_K\rmB_{b+1}\ar[r]&\Gr_b\CE(X_K)\wotimes_K\rmB_{b+1}
\end{tikzcd}
\end{center}}
En effet, en degré $1$, on a $\beta \equiv \wt{\beta} \mod t^{b+1}$ et $\gamma\equiv \iota_{\HK}\otimes \iota \mod t^{b+1}$.
\end{proof}

\section{Théorèmes de comparaison}
Le but  de cette section est de démontrer le théorème de comparaison syntomique-proétale. L'énoncé est le suivant :
\medskip
\begin{theo}\label{thm:petsyn}
Soit $X_K$ une courbe rigide Stein lisse sur $K$. Soit $\BV$ un $\Qp$-système local isotrivial et $i\geqslant 0$ un entier et supposons que les poids de Hodge-Tate de $\BV$ soient tous positifs. Il existe, dans $\sD(\rmC_{\Qp})$, un morphisme
$$
\alpha_{\BV(i)}\colon\rmR\Gamma_{\!\syn}\Harg{X_C}{\BV(i)}\rightarrow \rmR\Gamma_{\!\pet}\Harg{X_C}{\BV(i)},
$$
qui est un quasi-isomorphisme strict après troncation par $\tau_{\leqslant i}$. En particulier, on a des isomorphismes topologiques,
$$
\rmH^j_{\syn}\Harg{X_C}{\BV(i)}\cong\rmH_{\pet}^j\Harg{X_C}{\BV(i)}
$$
pour tout entier $j$ tel que $0\leqslant j \leqslant i$.
\end{theo}
\medskip
\begin{rema}
Notons que dans le théorème \ref{thm:petsyn}, si $a\geqslant 0$ est le poids de Hodge-Tate minimal de $\BV$, alors on peut remplacer la troncation par $\tau_{\leqslant i+a}$.
\end{rema}
\medskip
Soit $D\coloneqq \BD(\BV)$ et $\CE=(\CE,\nabla,\Filb)$ le fibré plat filtré associés à $\BV$. Pour démontrer le théorème, on va appliquer $\rmR\Gamma_{\!\pet}\Harg{X_C}{\cdot}$ à la suite exacte fournit par la proposition \ref{prop:suitexfond} :
$$
0\rightarrow \BV\rightarrow \BX_{\cri}(D)\rightarrow \BX_{\dR}^-(\CE)\rightarrow 0,
$$
On obtient un triangle distingué
$$
\rmR\Gamma_{\!\pet}\Harg{X_C}{\BV}\rightarrow \rmR\Gamma_{\!\pet}\Harg{X_K}{\BX_{\cri}({D})}\rightarrow \rmR\Gamma_{\!\pet}\Harg{X_C}{\BX_{\dR}^-(\CE)}.
$$
La proposition suivante en calcul les termes :
\medskip
\begin{prop}\label{prop:compaff}
\
\begin{itemize}
\itemb Il existe un quasi-isomorphisme strict
$$
\rmR\Gamma_{\!\pet}\Harg{X_C}{\BX_{\dR}^-(\CE)}\cong\frac{\intn{\rmR\Gamma_{\!\dR}\Harg{X_K}{\CE}\wotimes^R_K\Bdr}}{\Fil^0}.
$$
\itemb Il existe un quasi-isomorphisme strict
$$
\rmR\Gamma_{\!\pet}\Harg{X_C}{\BX_{\cris}(D)}\cong\left [\rmR\Gamma_{\HK}\Harg{X_K}{D}\wotimes^R_{K_0}\Bst\right ]^{\varphi=1,N=0}.
$$
\end{itemize}
\end{prop}
\medskip
\begin{proof}
Notons $\rmR\Gamma_{\!\dR}\coloneqq \rmR\Gamma_{\!\dR}\Harg{X_K}{\CE}$ et $\rmR\Gamma_{\HK}\coloneqq \rmR\Gamma_{\HK}\Harg{X_K}{D}$. On commence par le premier point. 

Soit $\BV_{\dR}^{(+)}\coloneqq \BV\otimes_{\Qp}\rmB_{\dR}^{(+)}$. D'après le théorème \cite[Theorem 6.5]{bosc1}, en considérant les points des complexes solides, on a un quasi-isomorphisme 
$$
\rmR\Gamma_{\!\pet}\Harg{X_C}{\BV_{\dR}}\cong{\intn{\rmR\Gamma_{\!\dR}\wotimes^R_K\Bdr}},
$$
compatible avec la filtration, ce qui signifie en particulier qu'on a un quasi-isomorphisme 
$$
\rmR\Gamma_{\!\pet}\Harg{X_C}{\BV_{\dR}^+}\cong\Fil^0{\intn{\rmR\Gamma_{\!\dR}\wotimes^R_K\Bdr}}.
$$
Justifions que le second quasi-isomorphisme est strict, le premier s'en déduit en inversant $t$. Par le lemme \ref{lem:calcfilun}, le terme de droite est un complexe d'espaces de Fréchet. D'après le lemme de Poincaré (\cf \cite[Corollary 6.13]{schphdg}), le terme de gauche se représente par un complexe d'espaces de Fréchet. On conclut alors par le second point du lemme \cite[Lemme 2.1]{codoniste}. Ainsi, le triangle distingué obtenu en appliquant $\rmR\Gamma_{\!\pet}\Harg{X_C}{\ \cdot \ }$ à la suite exacte
$$
0\rightarrow \BV_{\dR}^+\rightarrow\BV_{\dR}\rightarrow \BX_{\dR}^-(\CE)\rightarrow 0
$$
fournit un quasi-isomorphisme strict
$$
\rmR\Gamma_{\!\pet}\Harg{X_C}{\BV_{\dR}/\BV_{\dR}^+}\cong\frac{\intn{\rmR\Gamma_{\!\dR}\otimes^R_K\Bdr}}{\Fil^0},
$$
ce qui démontre le premier point.
Pour le second point, on utilise \cite[Theorem 4.1 ]{bosc2} (voir aussi \cite[Proposition 3.11]{leb}) qui, après avoir inversé $t$ pour se débarrasser du foncteur de décalage et après avoir pris les points des complexes solides, donne un quasi-isomorphisme, (\cf \cite[Definition 2.23]{bosc2} pour la définition de $\BB$),
$$
\rmR\Gamma_{\!\pet}\Harg{X_C}{\BB[1/t]}\cong \intn{\rmR\Gamma_{\!\HK}(X_K)\wotimes^R_{K_0}\Bst}^{N=0}.
$$
On utilise ensuite la suite exacte de faisceaux proétales (\cf \cite[Corollary 2.26]{bosc2}),
$$
0\rightarrow (\BBcris)^{\varphi=1}\rightarrow \BB[1/t]\xrightarrow{\varphi-1}\BB[1/t]\rightarrow 0.
$$
Soit $V$ une représentation cristalline de $\sG_K$ telle que $\bD_{\cris}(V)=D$, en tensorisant la suite ci-dessus par $V$ on obtient
$$
0\rightarrow \BX_{\cris}(D)\rightarrow \BB[1/t]\otimes_{K_0}D\xrightarrow{\varphi-1}\BB[1/t]\otimes_{K_0}D\rightarrow 0.
$$
Ainsi, on obtient un triangle distingué
$$
\rmR\Gamma_{\!\pet}\Harg{X_C}{\BX_{\cris}(D)}\rightarrow \rmR\Gamma_{\!\pet}\Harg{X_C}{\BB[1/t]\otimes_{K_0}D}\rightarrow \rmR\Gamma_{\!\pet}\Harg{X_C}{\BB[1/t]\otimes_{K_0}D}
$$
ce qui permet de conclure qu'on a un quasi-isomorphisme entre les fibres
$$
\rmR\Gamma_{\!\pet}\Harg{X_C}{\BX_{\cris}(D)}\cong\left [\rmR\Gamma_{\HK}\otimes^R_{K_0}\Bst\right ]^{\varphi=1,N=0}.
$$
Ce quasi-isomorphisme est stricte pour les même raisons que le premier, les deux termes s'écrivent comme limite inductive d'espaces de Fréchet, où on utilise le lemme de Poincaré cristallin (\cf \cite[Corollary 2.17]{tantong}) pour le terme de gauche, le lemme \ref{lem:steinhk} pour le terme de droite, puis on applique le lemme \cite[Lemme 2.1]{codoniste}.
\end{proof}
Ainsi, on obtient un triangle distingué
$$
\rmR\Gamma_{\!\pet}\Harg{X_C}{\BV}\rightarrow\left [\rmR\Gamma_{\HK}\Harg{X_K}{D}\wotimes^R_{K_0}\Bst\right ]^{\varphi=1,N=0}\rightarrow\frac{\intn{\rmR\Gamma_{\!\dR}\Harg{X_K}{\CE}\wotimes^R_K\Bdr}}{\Fil^0},
$$
et comme on travail dans $\sD(\rmC_{\Qp})$, le terme de gauche est bien la fibre de la flèche de droite. On obtient le théorème en remplaçant $\Bdr$ par $\Bdrp$ et $\Bst$ par $\Bstp$ quitte à tordre par une puissance du caractère cyclotomique.
\qed
\medskip
\begin{rema}
Le même argument permet de démontrer le théorème \ref{thm:petsyn} dans le cas où $X$ est propre, mais l'argument est plus direct puisqu'on peut appliquer les théorèmes de \cite{schphdg} à la place de la proposition \ref{prop:compaff}.
\end{rema}

\cleardoublepage
\part{Cohomologie isotriviale de la tour de Drinfeld}
\section{Le demi-plan $p$-adique et sa tour de revêtements}

\Subsection{Le demi-plan $p$-adique}
On note $C\coloneqq \wh{\bar{\BQ}}_p$ le complété d'une clôture algébrique de $\Qp$ et $G\coloneqq \GL_2(\Qp)$. On rappelle que le \emph{demi-plan $p$-adique} est l'espace rigide analytique, ouvert admissible de~$\BP^1_{\Qp}$, défini par
$$
\BH_{\Qp}\coloneqq \BP^1_{\Qp}\setminus \BP^1(\Qp).
$$
On note $\BH_{C}\coloneqq \BH_{\Qp}\otimes_{\Qp}C$ l'extension des scalaires à $C$. On note aussi $\pi\colon \BH_{\Qp} \hookrightarrow \BP^1_{\Qp}$ le plongement naturel. L'espace rigide analytique $\BH_{\Qp}$ est un espace Stein, puisqu'en effet, si l'on note pour $n\in \BN$ 
$$
\BH_{n} = \BP^1_{\Qp}\setminus\bigcup_{x\in \BP^1(\BZ/p^n)} \rmB(x,n),
$$
où $\rmB(x,n)$ est la boule ouverte centrée en $x$ et de rayon $p^{-n}$, \ie  si $\tilde x=a_0 +\dots +a_np^n$ a pour réduction  $x\in \BZ/p^n$, les $C$-points sont $\rmB(x,n)(C) = \{z\in C\mid v_{p}(\tilde x-z)\geqslant n\}$ et si $\tilde x$ a pour réduction $x^{-1}\in (\BZ/p^n)$  alors $\rmB(x,n)(C) = \{z\in C\mid v_{p}(\tilde x^{-1}-z^{-1})\geqslant n\}$. On a alors $\BH_{\Qp}=\varinjlim_n \BH_{n}$. Ainsi, les espaces des sections globales $\CO\coloneqq\CO(\BH_{\Qp})=\varprojlim_n\CO(\BH_{n})$ et $\Omega^1\coloneqq\Omega^1(\BH_{\Qp})$ sont des espaces de Fréchet. On note $z\in \CO$ la coordonnée naturelle de $\BP^1_{\Qp}$. L'action naturelle de $G$ sur $\BP^1_{\Qp}$ est donnée sur $z$ par 
$$
g=\begin{pmatrix} a&b\\ c& d \end{pmatrix}\in G,\quad g\cdot z = \frac{dz-b}{a-cz}.
$$
Cette action se restreint en une action sur $\BH_{\Qp}$ et munit $\CO$ d'une action linéaire de $G$.
Soient $k,m\in \BZ$ des entiers, on définit $\CO\{k,m\}$ comme $\Qp$-représentation de $G$ de sorte qu'en tant que $\Qp$-espace vectoriel topologique $\CO\{k,m\}=\CO$ muni de l'action de $G$ donnée par
$$
f\in \CO\{k,m\}(\BH_{\Qp}),\quad (g\star f)(z)=j(z,g)^{-k}\det(g)^{m}f(g\cdot z),
$$
où $j$ est le facteur d'automorphie donné pour $g=\begin{pmatrix} a&b\\ c& d \end{pmatrix}$ par $j(z,g) \coloneqq a-cz$. On a par exemple, $\Omega^1\cong \CO\{2,1\}$ 

\Subsection{L'espace de Drinfeld}

\subsubsection{$\rmOD$-modules formels spéciaux}\label{subsub:sfdrecap}
Soit $\rmD$ l'algèbre des quaternions non scindé sur $\Qp$, \ie  l'unique corps gauche tel que $\inv_{\Qp}\rmD=\frac 1 2$. On note $\rmOD\subset \rmD$ l'unique ordre maximal et on fixe $\varpi_{D}\in \rmO_{\rmD}$ tel que $\varpi_{D}^2=p$. Si l'on choisit un plongement $\Qpd\hookrightarrow \rmD$, où $\Qpd/\Qp$ est l'unique extension quadratique non ramifiée, on a $\rmD=\Qpd[\varpi_{D}]$ et $\rmOD=\Zpd[\varpi_{D}]$. On note $\czG\coloneqq \rmD^{\times}$ le groupe des inversibles de $D$.

Soit $S$ un schéma formel sur $\Spf (\Zp)$. Un \emph{$\rmOD$-module $p$-divisible} sur $S$ est un groupe $p$-divisible formel $\CG$ sur $S$ muni d'une action de $\rmOD$, \ie  muni d'un morphisme $\iota\colon\rmOD\rightarrow \End_S(\CG)$. En particulier, $\Lie(\CG)$ est un $\Zpd\otimes_{\Zp}\CO_S$-module et on dit que $\CG$ est un \emph{$\rmOD$-module formel spécial} si $\Lie(\CG)$ est un $\Zpd\otimes_{\Zp}\CO_S$-module localement libre de rang $1$.

En suivant \cite[3.58]{razi}, on explicite maintenant le module de Dieudonné d'un $\rmOD$-module formel sur $\Fpbar$, qui est unique à isogénie près d'après \cite[Lemma 3.60]{razi}. Rappelons qu'on note $\Zpbr\coloneqq\rmW(\Fpbar)$, l'anneau des entiers de $\Qpbr\coloneqq \wh{\Qp^{\nr}}$, le complété de l'extension maximale non-ramifiée de $\Qp$, qui est muni d'un automorphisme de Frobenius noté $\sigma$. Posons le $\Zpbr$-module
$$
\brv{\bM}^+\coloneqq\rmOD\otimes_{\Zp}\Zpbr,\text{ muni de } \bV=\varpi_{D}\otimes\sigma^{-1}.
$$
Ces données définissent un groupe $p$-divisible qui est naturellement un $\rmOD$-module formel spécial sur $\Fpbar$, que l'on note $\bG$. L'isocristal associé est $\brv{\bM}\coloneqq \brv{\bM}^+\otimes_{\Zp}\Qp=\rmD\otimes_{\Qp}\Qpbr$, muni du Frobenius $\bF=\varpi_{D}\otimes \sigma$. Cet isocristal est muni d'une action de $G$ qui s'identifie au groupe des automorphismes $\rmOD$-invariants (\cf \cite[Lemma 3.60]{razi}).


\subsubsection{L'espace de Drinfeld}
Soit $\Nilp_{\Zpbr}$ la catégorie des $\Zpbr$-algèbres telles que $p$ est Zariski-localement nilpotent. On définit un foncteur 
$$
R\mapsto\left \{ (X,\rho) \mid \begin{gathered} X \text{ un $\rmOD$-module formel spécial sur } R\\ \rho\colon \bG\otimes_{\Fpbar} R/p \dashrightarrow X\otimes_{R}R/p \text{ une quasi-isogénie $\rmOD$-linéaire} \end{gathered}\right \}.
$$
D'après \cite[Proposition 3.63]{razi}, ce foncteur est représentable par un schéma formel $p$-adique sur $\Zpbr$, que l'on note $\sM_{\Dr}$. Alors, $\sM_{\Dr}$ est muni d'une action naturelle de $G$ par précomposition. Cette action est donnée pour $g\in G$ sur les $R$-points par $(X,\rho)\mapsto (X,\rho\circ g)$. On peut décomposer $\sM_{\Dr}$ suivant la hauteur de la quasi-isogénie 
$$
\sM_{\Dr}=\bigsqcup_{n\in\BZ}\sM_{\Dr}[2n],
$$
où les composantes sont isomorphes et définies sur $\Zpbr$. L'action de $G$ sur les composantes connexes est donnée pour $g\in G$ par $n\mapsto n+v_{p}(\det(g))$. On a une donnée de descente à $\BZ_p$ qui est définie par la multiplication $p\colon \sM_{\Dr}[2n]\rightarrow \sM_{\Dr}[2n+4]$ vu comme élément du centre de $G$. Le quotient $\presp{\sM_{\Dr}}\coloneqq\sM_{\Dr}/{p}^{\BZ}$ admet un modèle sur $\Zp$. On note $\rmM^0_{\Dr}$ la fibre générique de $\sM_{\Dr}[0]$ et $\brv{\rmM}^0_{\Dr}$ la fibre générique de $\sM_{\Dr}$ puis on note $\rmM^0_{C}\coloneqq \rmM^0_{\Dr}\otimes_{\Qp} C$ et $\brv{\rmM}^0_{C}\coloneqq \brv{\rmM}^0_{\Dr}\otimes_{\Qp} C$ leur extension des scalaires à $C$. Notons que $\brv{\rmM}^0_{\Dr}\cong \bigsqcup_{n\in \BZ}\rmM_{\Dr}^0$ ce qui permet d'exprimer le quotient $\presp{\rmM^0_{\Dr}}\coloneqq\brv{\rmM}^0_{\Dr}/p^{\BZ}$ à partir de deux copies de $\rmM^0_{\Dr}$. De plus, $\presp{\rmM^0_{\Dr}}$ admet un modèle sur $\Qp$ que l'on note $\presp{\rmM^0_{\Qp}}$ et $\presp{\rmM}^0_{C}\coloneqq \presp{\rmM^0_{\Qp}}\otimes_{\Qp} C$ son extension des scalaires à $C$.

\subsubsection{La famille universelle}\label{subsub:unimif}
L'espace $\sM_{\Dr}$ est muni d'une famille universelle $\CG_{\Dr}\rightarrow \sM_{\Dr}$ qui est un $\rmOD$-module formel spécial sur $\sM_{\Dr}$. Le module de Tate de ce groupe $p$-divisible universel définit sur $\brv{\rmM}^0_{\Dr}$ un $\Zp$-système local étale noté $\BV_{\Dr}^+\coloneqq \BT_p(\CG_{\Dr})$ (\cf \ref{subsubsec:grppd}). D'après le lemme \ref{lem:pdivisoc}, il définit un $\Qp$-système local étale, noté $\BV_{\Dr}\coloneqq\BV_p(\CG_{\Dr})$ qui est isotrivial de $\varphi$-module associé $\brv{\bM}[-1]$ et de poids de Hodge-Tate $0$ et $1$ (chacun de multiplicité $2$).  

on a un morphisme des périodes $\pi \colon \rmM_{C}^0\rightarrow \BP^1_C$ qui est un plongement d'image l'ouvert admissible $\BH_{C}\subset \BP^1_C$ d'après le théorème de Drinfeld (\cf \cite{drin}, \cite[Theorem 3.72]{razi}). Ainsi, on a des isomorphismes $\rmM_{C}^0\cong \BH_C$ et $\presp{\rmM}^0_{C}\cong \BH_C\sqcup \BH_C$. Ce morphisme est caractérisé par la propriété 
\begin{equation}\label{eq:propunivperiodmor}
\Zpd\otimes_{\Zp}\pi^*\CO_{\BP^1}(1)\cong \Lie\ \sG_{\Dr}
\end{equation}
où $\CO_{\BP^1}(1)$ est le fibré en droite ample usuel de la droite projective et $\Lie\ \sG_{\Dr}$ est le $\Zpd\otimes_{\Zp}\CO$-module inversible donné par l'algèbre de Lie du $\rmOD$-module formel spécial universel.

\subsection{La tour de revêtements}
\subsubsection{Définition}
Soit $n\geqslant 1$ un entier, la famille universelle définit un revêtement étale $\CG_{\Dr}[\varpi_{\rmD}^n]_C\rightarrow \rmM_{C}^0$. Ce revêtement est étale puisque toute algèbre de Hopf de rang finie en caractéristique~$0$ est étale. On définit le \emph{$n$-ième étage de la tour de Drinfeld} par 
$$
\rmM_{C}^n\coloneqq \CG_{\Dr}[\varpi_{\rmD}^n]_C\setminus \CG_{\Dr}[\varpi_{\rmD}^{n-1}]_{C}\rightarrow \rmM_{C}^0.
$$
En terme de problème de modules : posons $V_{\Dr}^+\coloneqq \rmOD$, l'espace $\rmM_{C}^n$ classifie alors les isomorphismes $\rmODt$-équivariants entre $V_{\Dr}^+/\varpi_{\rmD}^n$ et $\BV_{\Dr}^+/\varpi_{\rmD}^n$. D'après \cite{scwe}, il existe un espace perfectoïde  $\wh{\rmM^{\infty}_{C}}$ tel que  $\wh{\rmM^{\infty}_{C}}\sim \varprojlim_{n}\rmM^n_{C}$ et $\wh{\brv{\rmM}^{\infty}_{C}}$ tel que  $\wh{\brv{\rmM}^{\infty}_{C}}\sim \varprojlim_{n}\brv{\rmM}^n_{C}$. Le revêtement $\rmM_{C}^n\rightarrow \rmM_{C}^0$ a pour groupe de Galois $\rmODt/({1+\varpi_{\rmD}^n\rmOD})$ et on obtient un diagramme
\begin{center}
\begin{tikzcd}
\wh{\brv{\rmM}^{\infty}_{C}}\ar[r]\ar[d]\ar[dr,"\czG"]&\wh{\rmM^{\infty}_{C}}\ar[d,"\rmODt"]\ar[dr,"\pi"]\\
\brv{\rmM}_{C}^0\ar[r,"\varpi_{\rmD}^{\BZ}"]&\rmM_{C}^0\ar[r,hook]&\BP^1_C
\end{tikzcd}
\end{center}
Le $\Qp$-système local étale isotrivial $\BV_{\Dr}$ se définit sur $\rmM_{C}^n$ de la même manière que précédemment. Sur $\wh{\rmM_{C}^{\infty}}$, le système local $\BV_{\Dr}$ devient trivial. On note simplement $\brv{\rmM}^{\infty}_{C}$ et $\rmM^{\infty}_{C}$ les systèmes projectifs que l'on distinguera des espaces perfectoïdes complétés correspondant. On appellera abusivement ces systèmes projectifs des \emph{tours (de revêtement)} et c'est eux dont on veut calculer la cohomologie. La tour de revêtements $\brv{\rmM}^{\infty}_{C}$ est muni d'une action de $G\times \czG$.

Comme précédemment, on définit le quotient $\presp{\rmM}^n_{C}\coloneqq \brv{\rmM}_{C}^{n}/p^{\BZ}$ qui admet un modèle sur $\Qp$ que l'on note $\presp{\rmM}^n_{\Qp}$ et qui définit une tour de revêtements $\presp{\rmM^{\infty}_{\Qp}}$ de $\presp{\rmM}^0_{\Qp}$. Ainsi, on obtient une action de $\sW_{\Qp}$, ce qui défini finalement sur le système projectif $\presp{\rmM^{\infty}_{C}}$ une action du produit 
$$
\BG\coloneqq \sW_{\Qp}\times G\times \czG.
$$
On peut définir une action de $\sW_{\Qp}$ sur $\brv{\rmM}^{\infty}_{C}$ qui donne l'action décrite sur son quotient $\presp{\rmM^n_{C}}$ et définit ainsi une action de $\BG$ tout entier sur $\brv{\rmM}^{\infty}_{C}$ : il suffit de remarquer que le carré du Frobenius décrit la donné de descente effective (\cf \cite[Theorem 3.72]{razi} et \cite[Theorem 3.49]{razi}).
On note $\CO^n\coloneqq \CO(\presp{\rmM_{\Qp}^n})$ qui est une $\Qp$-représentation de $G\times \czG$. De même que pour le demi-plan, on définit pour $k\geqslant 0$ un entier et $m\in\BZ$ la représentation $\CO^n\{k,m\}$ : en tant que $\Qp[\czG]$-module topologique on pose $\CO^n\{k,m\}=\CO^n$, et on définit l'action de $g\in G$ pour $f\in\CO^n$ par 
$$
g = 
\begin{pmatrix}
 a & b \\ c & d
\end{pmatrix}\in G,\quad 
(g\star f)(z) = j(z,g)^{-k}\det(g)^{m}f(g\cdot z),
$$
où l'on note $\pi \colon \presp{\rmM_{\Qp}^n}\rightarrow \BH_{\Qp}$ la projection naturelle et $j(z,g) = (a-c\pi(z))\in\CO$ le facteur d'automorphie. L'action considérée a un sens puisque $\CO^n$ est naturellement un $\CO$-module.
\subsubsection{Composantes connexes}\label{subsubsec:compco}
Décrivons les composantes connexes de $\brv{\rmM}_{C}^{\infty}$. Rappelons que, d'après Strauch (\cf  \cite{str}) on a une surjection $\rmM_{C}^{\infty}\rightarrow \Zp^{\times}$ donnant les composantes connexes de la tour. Plus précisément, si l'on choisit un générateur de $\Zp(1)$, on obtient un isomorphisme
$$
\pi_0(\rmM_{C}^{\infty})\cong\Isom(\Zp,\Zp(1))\cong \Zp^{\times}.
$$
Ainsi, on en déduit un isomorphisme $\pi_0(\brv{\rmM}_{C}^{\infty})\cong \Qpt$ et une surjection $\brv{\rmM}_{C}^{\infty}\rightarrow \Qp^{\times}$. Comme $\BG$ opère sur $\brv{\rmM}_{C}^{\infty}$ cette action induit une action sur $\Qpt$. Si l'on note, comme dans l'introduction, $\nu \colon \BG\rightarrow \Qpt$ le morphisme donné par\footnote{Rappelons que le morphisme de réciprocité est normalisé de sorte à ce que le Frobenius arithmétique soit envoyé sur $p$.}
$$
(\sigma, g, \cz{g}) \mapsto \rec (\sigma)\otimes \det (g)^{-1}\otimes \nrd( \cz{g}),
$$
alors $\BG$ opère sur $\Qpt$ par multiplication par $\nu$. De plus, on introduit les groupes 
$$
\BG^{+}\coloneqq \nu^{-1}(\Zpt),\quad \BG^{\circ}\coloneqq \ker \nu.
$$
Alors, $\BG^{+}$ stabilise $\rmM_{C}^{\infty}$ et $\BG^{\circ}$ stabilise les composantes connexes.

Le lemme suivant résume ces observations, déduites de \cite[Theorem 4.4]{str} : 
\medskip
\begin{lemm}\label{lem:deccar}
Soit $L/\Qp$ une extension finie. En tant que $L$-représentation de $\BG$,
$$
\rmH^0_{\et}\Harg{\brv{\rmM}_{C}^{\infty}}{L}=\Ind_{\BG^+}^{\BG}\rmH^0_{\et}\Harg{\rmM_{C}^{\infty}}{L}.
$$
De plus,
$$
\rmH^0_{\et}\Harg{\rmM_{C}^{\infty}}{L}=\bigoplus_{\chi\colon \tim{\BZ}_p\rightarrow L^{\times}} (\chi\circ \nu)
$$
où la somme directe porte sur l'ensemble des caractères lisses de $\tim{\BZ}_p$ à valeurs dans $L^{\times}$.
\end{lemm}
\medskip
\begin{rema}\label{rem:concomp}
\begin{enumerate}
\item Dans cette remarque, à partir du résultat pour $\rmH^0_{\et}\coloneqq\rmH^0_{\et}\Harg{\rmM_{C}^{\infty}}{L}$,
on va expliquer que tous les caractères lisses de $\Qpt$ apparaissent en sous-objet de $\brv{\rmH}^0_{\et}\coloneqq\rmH^0_{\et}\Harg{\brv{\rmM}_{C}^{\infty}}{L}$.
 Notons que comme $\BG/\BG^+\cong p^{\BZ}$ via $\nu$, l'induction $\Ind_{\BG^+}^{\BG}\rmH^0_{\et}$ s'identifie aux fonctions sur $\BZ$ à valeurs dans $\rmH^0_{\et}$ et même, puisque ce dernier n'est constitué que de caractères, on a un isomorphisme 
$$
\brv{\rmH}^0_{\et}\cong \prod_{i\in \BZ}\rmH^0_{\et}
$$
où l'action de $\BG$ décale les facteurs. On en déduit que l'inclusion diagonale $\rmH^0_{\et}\rightarrow \brv{\rmH}^0_{\et}$ est stable par $\BG$. On peut tordre cette représentation par les caractères de $p^{\BZ}$ pour finalement obtenir une inclusion
$$
\bigoplus_{\chi\colon \Qpt\rightarrow L^{\times}} (\chi \circ \nu)\incl \brv{\rmH}^0_{\et},
$$
où la somme à gauche porte sur l'ensemble des caractères lisses $\chi \colon \Qpt \rightarrow L^{\times}$.
\item Dans la description précédente, l'inclusion $\rmH^0_{\et}\incl\brv{\rmH}^0_{\et}$ stable par $\BG$ n'est pas l'inclusion naturelle donné géométriquement. Le même phénomène apparait pour l'inclusion du caractère trivial $\boldsymbol{1} \circ \nu
\incl \rmH^0_{\et}$ qui n'est pas donné par la projection sur une composante connexe ; cette dernière n'est pas stable par $\BG^+$.
\item Comme on travaillera surtout avec $\presp{\rmM}_{C}^{\infty}$, notons que l'on a 
$$
\presp{\rmH}_{\et}^0\coloneqq\rmH^0_{\et}\Harg{\presp{\rmM}_{C}^{\infty}}{L} = \rmH^0_{\et}\oplus \rmH^0_{\et}\otimes (\chi_2\circ \nu)
$$
où $\chi_2\colon \Qpt \rightarrow L^{\times}$ est l'unique caractère quadratique non-ramifié, \ie $\chi_2(p)=-1$. Comme pour le point précédent, l'inclusion $ \rmH^0_{\et} \incl \presp{\rmH}^0_{\et}$ de la décomposition n'est pas donné par géométriquement par une projection.
\end{enumerate}
\end{rema}
Le lemme \ref{lem:deccar} ainsi que le second point de la remarque nous permet de décomposer les groupes de cohomologie étale suivant les caractères qui apparaissent dans le $\brv{\rmH}_{\et}^0$, comme on le fera au paragraphe \ref{secsec:uncomp}.
\subsubsection{Le $L$-système local universel}\label{subsub:univit}
Soit $\nu_1\colon G\times \czG\rightarrow \Qp^{\times}$ la restriction de $\nu$ à $G\times \czG$. On note  $\BV\coloneqq \BV_{\Dr}\otimes_{\Qp} L\cdot \lvert \nu_1 \rvert_p^{1/2}$ qui définit un $L$-système local isotrivial sur la tour $\presp{\rmM}_{\Qp}^{\infty}$ et on définit
$$
\Sym \BV\coloneqq \bigoplus_{k\geqslant 0} \BV_k,\quad \BV_k\coloneqq\Sym_{L}^k\BV.
$$
On note $\BD$ l'isocristal associé à $\BV$ qui est un $L$-espace vectoriel. Le formalisme Tannakien fait de $\BD$ une $L$-représentation de $G\times \czG$. De plus, si on pose 
$$
\Sym \BD\coloneqq\bigoplus_{k\geqslant 0} \BD_k,\quad \BD_{k}\coloneqq \Sym_L^k\BD,
$$
alors, $\BD_k$ est l'isocristal associé à $\BV_k$ et on peut considérer $\Sym \BD$ comme l'isocristal associé à $\Sym \BV$. D'après le paragraphe \ref{subsub:unimif}, on a bien $\BD_{k}\otimes_{\Qp}\Qpbr=\Sym_{\Zpbr}^k\brv{\bM}[-k]\otimes_{\Qp}L\cdot \lvert \nu_1 \rvert^{k/2}$. Notons que $\Sym \BV$ définit aussi un $L$-système local isotrivial sur $\brv{\rmM}_{C}^{\infty}$ et sur ${\rmM}_{C}^{\infty}$ mais l'isocristal associé est alors $\Sym \BD \otimes_{\Qp}\Qpbr$.

\section{Représentations de $G$ et $\czG$}
Rappelons que l'on note $G\coloneqq \GL_2(\Qp)$, $\czG\coloneqq\rmD^{\times}$ et $L/\Qp$ une extension finie assez grande (\ie contenant les extensions quadratiques de $\Qp$). On note finalement $B\subset G$ le sous-groupe de Borel des matrices triangulaires supérieurs et $T\subset B$ le sous-groupe des matrices diagonales ; on a un quotient naturel $B\rightarrow T$.
\subsection{Représentations (localement) algébriques}

Rappelons que pour $H$ un groupe $p$-adique et $V$ une $L$-représentation de $H$, un vecteur $v\in V$ est dit (localement) \emph{algébrique} si l'application orbitale $h\mapsto h\cdot v$ est une fonction (localement) algébrique sur $H$.  Si, comme dans les cas qui nous intéressent, $H$ est l'ensemble des $\Qp$-points d'un groupe algébrique, ces représentations algébriques proviennent directement des représentations du groupe algébrique sous-jacent. Les représentations que l'on va introduire ne sont localement algébriques mais se sont simplement des tordues de représentations algébriques par un caractère lisse.

On définit les \emph{poids réguliers} comme l'ensemble 
$$
P\coloneqq \{(\lambda_1,\lambda_2)\in \BZ^2\mid \lambda_2>\lambda_1\}.
$$
On pose de plus $P_+\subset P$ l'ensemble des \emph{poids réguliers positifs}, \ie les $\lambda=(\lambda_1,\lambda_2)\in P$ tels que $\lambda_1\geqslant 0$. Pour $\lambda\in P$, posons $w(\lambda)\coloneqq \lambda_2-\lambda_1\in \BN$ et $\lvert\lambda \rvert = \lambda_2+\lambda_1$. Pour $k\geqslant 1$ un entier, $K/\Qp$ une extension finie de $\Qp$, on note simplement $\Sym^k_{K}\coloneqq \Sym^k_{K}K^2$, la puissance symétrique $k$-ième de la représentation standard de $\GL_2(K)$.
\subsubsection{Représentations algébriques irréductibles de $G$}\label{subsubsec:raig}
Soit $\lambda\in P$. On définit une $L$-représentation de $G$ par
$$
W^{*}_{\lambda}\coloneqq \Sym_{L}^{w(\lambda)-1}\otimes_{L}{\det}^{\lambda_1}\otimes_L \lvert \det \rvert_p^{\frac{\lvert \lambda \rvert-1}{2}}.
$$
C'est une $L$-représentation localement algébrique irréductible de $G$. Rappelons que $\Sym_{L}^1=L e^*_0\oplus L e^*_1$ et pour $g\in G$ on a 
\begin{equation}
g=
\begin{pmatrix}
 a & b \\ c  & d \\
\end{pmatrix},\quad
g\cdot e_0^*=ae_0^*+ce_1^*,\quad g\cdot e_1^*=be_0^*+de_1^*.
\end{equation}
Ainsi, une base de $W^*_{\lambda}$ est donnée par $\{(e_0^*)^i(e_1^*)^{j}\}_{i+j=w(\lambda)-1}$. On note $W_{\lambda}$ le dual de $W^*_{\lambda}$ qui a pour base $\{(e_0)^i(e_1)^{j}\}_{i+j=w(\lambda)-1}$ et qui est construit à partir de $\left ( \Sym_{L}^1\right )^*=\Qp e_0\oplus \Qp e_1$. Pour $g\in G$ on a 
$$
g=
\begin{pmatrix}
 a & b \\ c  & d \\
\end{pmatrix},\quad 
g\cdot e_0=\frac 1 {\det(g)}(de_0-be_1),\quad g\cdot e_1=\frac 1 {\det(g)}(-ce_0+ae_1),
$$
On remarque que $W^*_{\lambda}\cong W_{\lambda^*}$ où $\lambda^*\coloneqq (1-\lambda_2,1-\lambda_1)$ ; en particulier $w(\lambda) = w(\lambda^*)$.
\subsubsection{Représentations algébriques irréductibles de $\czG$}\label{subsubsec:raiczg}
On a un isomorphisme 
$$
\rmD\otimes_{\Qp}L\cong M_2(L),
$$
ce qui fournit un plongement $\czG\incl \GL_2(L)$ et permet de définir $\Sym^1_L$, la standard, comme $L$-représentation de $\czG$. Remarquons que sous cet isomorphisme la norme réduite s'identifie au déterminant. Ainsi, pour $\lambda\in P$, on définit une $L$-représentation localement algébrique irréductible de $\czG$ par
$$
\cz{W}_{\lambda}\coloneqq\Sym_{L}^{w(\lambda)-1}\otimes_{L}{\nrd}^{\lambda_1}\otimes_L \lvert \nrd \rvert_p^{\frac{\lvert \lambda \rvert-1}{2}},
$$
et on note $\cz{W}^*_{\lambda}$ son dual (on a inversé les conventions par rapport à $G$). On a toujours $\cz{W}^{*}_{\lambda}\cong \cz{W}_{\lambda^*}$.
\subsection{L'isocristal de Drinfeld}
Rappelons qu'au numéro \ref{subsub:univit}, on a introduit $\Sym \BD$ comme l'isocristal associé au $L$-système local $\Sym \BV$ qui a naturellement une structure de $L$-représentation localement algébrique de $G\times \czG$. Le but de ce numéro est de décrire cette représentation. On commence par factoriser $\BD$, puis on décompose $\Sym\BD$ à l'aide de foncteurs de Schur.

\subsubsection{Factorisation de $\BD$}
\medskip
\begin{lemm}\label{lem:dec}
Soit $L/{\Qp}$ une extension finie contenant une extension quadratique de ${\Qp}$ et $\rho=(0,2)\in P_+$. En tant que $L$-représentation de $G\times \cz G$ on a 
$$
\BD\cong W_{\rho}\otimes_L \cz{W}_{\rho}.
$$
\end{lemm}
\begin{proof}
On donne une preuve conceptuelle de ce résultat qui peut aussi se démontrer à la main avec une base adaptée (\cf \cite{van}). Via l'isomorphisme $\rmD\otimes_{\Qp} L\cong \rmM_2(L)$ on obtient deux actions (l'une à gauche et l'autre à droite) de $\GL_2(L)$ qui commutent. Ces actions se restreignent en une représentation de $G\times \czG$ via le plongement de ces groupes dans $\GL_2(L)$. Maintenant, si $V$ est un $L$-espace vectoriel de dimension finie, $\End_L(V)\cong V\otimes V^*$ et l'action de $\GL(V)\times \GL(V)$ est à gauche sur le premier facteur $V$ et à droite sur le second facteur $V^*$. En ce sens, l'isomorphisme $\End_L(V)\cong V\otimes V^*$ est $\GL(V)\times \GL(V)$-équivariant. Ceci nous donne le lemme en prenant $V=L^{2}$ et en tordant par $\lvert \nu_1 \rvert_p^{1/2}=\lvert {\det}^{-1}\otimes \nrd \rvert_p^{1/2}$.
\end{proof}
\begin{rema}\label{rem:phdgint}
Ce résultat peut aussi se démontrer à la mains en utilisant une base adapté. Donnons une dernière interprétation de ce résultat en terme de théorie de Hodge $p$-adique. On reprend les notations du numéro \ref{subsub:sfdrecap}. Rappelons qu'à isomorphisme près, il existe un unique groupe formel de hauteur $2$ et de dimension $1$ sur $\Fpbar$ ; notons le $\bH$. Soit $\brv{\bN}^+$ le module de Dieudonné de $\bH$ et $\brv{\bN}$ l'isocristal associé. Alors $\brv{\bN}^+\cong\rmOD\otimes_{\Zpd}\Zpbr$ muni du Frobenius $\bF=\varpi_{\rmD}\otimes\sigma$. Ainsi,
$$
\brv{\bM}^+=\rmOD\otimes_{\Zp}\Zpbr=\rmOD\otimes_{\Zpd}(\Zpd\otimes_{\Zp}\Zpbr)\cong \rmOD\otimes_{\Zpd}(\Zpbr\times \Zpbr)\cong \brv{\bN}^+\oplus \brv{\bN}^+.
$$
De ce petit calcul on récupère plusieurs conséquences :
\begin{itemize}
\itemb On a un isomorphisme $\bG\cong \bH\times \bH$ en tant que groupe $p$-divisible.
\itemb En inversant $p$, on peut écrire $\brv{\bM}\cong (\Sym^1_{\Qp})^* \otimes_{\Qp}\brv{\bN}$, ce qui muni $\brv{\bM}$ d'une action de $G$. Le calcul montre que cette action coïncide avec l'action par automorphisme.
\itemb La description de $\brv{\bN}^+$ permet de retrouver que les automorphismes de $\bH$ sont $\rmOD$. On obtient que $\rmOD$ agit diagonalement sur $\brv{\bM}^+\cong \brv{\bN}^+\oplus \brv{\bN}^+$ et on recupère que l'action par $\czG$ sur les composantes $\brv{\bN}$ est la standard.
\end{itemize}
En conclusion, $\brv{\bM}\cong (\Sym^1_{\Qp})^* \otimes_{\Qp}\brv{\bN}$ comme $\Qpbr$-représentation de $G\times\czG$ et on a montré que
$$
\brv{\bM}\otimes_{\Qp}L\cdot\lvert \nu_1 \rvert_p^{1/2} \cong (W_{\rho}\otimes_L \cz{W}_{\rho})\otimes_{\Qp}\Qpbr.
$$
On renvoit à \cite[Théorème 8.1.5]{fafo} qui résume la théorie de Hodge $p$-adique des $\rmOD$-modules formels spéciaux.
\end{rema}
\subsubsection{Décomposition de l'isocristal}
On veut maintenant décomposer $\Sym \BD$. Commençons par quelques rappels sur les foncteurs de Schur. Soit $K$ un corps de caractéristique $0$, pour tout entier $k\geqslant 0$ et toute partition $\mu\dashv k$ on a un \emph{foncteur de Schur} $\BS_{\mu}\colon \Vect_{K}\rightarrow \Vect_{K}$ qui permet de construire les représentations irréductibles du groupe symétrique $\fkS_k$ et les représentations du groupe général linéaire. Les partitions de $k$ sont représentés par des diagrammes de Young. Pour plus de détails sur ces constructions et le résultat qui va suivre, on renvoie à \cite[Lecture~6]{fuha}. Soient $V$ et $W$ deux $K$-espaces vectoriels. Alors d'après \cite[Exercice 6.11]{fuha}, on a 
$$
\Sym_{K}^k(V\otimes_{K} W) = \bigoplus_{\mu\dashv k}\BS_{\mu}V\otimes_{K} \BS_{\mu}W,
$$
où la somme porte sur l'ensemble des tableaux de Young $\mu$ tels que le nombre de lignes de $\mu$ est inférieur à $\dim_{K} V$ et à $\dim_{K} W$, \ie  on considère les partitions en au plus $\min\{\dim_{K}V,\dim_{K} W\}$ entiers. De plus, pour $V$ un ${K}$-espace vectoriel de dimension $2$ et $\mu=(\mu_1,\mu_2)$ on a 
$$
\BS_{\mu}V = \Sym^{\mu_1-\mu_2}V\otimes_{K}{\det}^{\mu_2}.
$$
Combiné au lemme \ref{lem:dec} on en déduit le résultat suivant :
\medskip
\begin{coro}\label{cor:schurep}
En tant que $L$-représentation de $G\times \czG$, on a 
$$
\Sym \BD\cong \bigoplus_{\lambda\in P_+}W_{\lambda}\otimes_{L}\cz{W}_{\lambda}.
$$
De plus, pour $k\in \BN$, $\BD_k\cong \bigoplus_{\lvert \lambda\rvert = k+1}W_{\lambda}\otimes_{L}\cz{W}_{\lambda}$.
\end{coro}

\Subsection{La série speciale et la Steinberg}\label{sub:steinberg}
\subsubsection{Définitions et dévisages}\label{subsubsec:spean}
Rappelons qu'une $L$-représentation de Steinberg de $G$ est construite à partir des fonctions d'une certaine régularité sur $\BP^1(\Qp)$  à valeur dans $L$ dont on prend le quotient par le sous espace des fonctions constantes. 

On définit la \emph{Steinberg lisse} sur $L$ par 
$$
\St^{\infty}\coloneqq\sC^{\infty}\intn{\BP^1(\Qp),L}/L.
$$
C'est une $L$-représentation lisse de $G$ de type $LB$ (limite inductive de Banach) qui est irréductible. 

De même, on définit la \emph{Steinberg localement analytique} sur $L$ par
$$
\St^{\lan}\coloneqq\sC^{\lan}\intn{\BP^1(\Qp),L}/L.
$$ 
C'est une $L$-représentation localement analytique de $G$ de type $LB$. Ce n'est pas une représentation irréductible puisque $\St^{\infty}\subset \St^{\lan}$. Il s'avère que le quotient est irréductible et qu'il s'interprète en termes de fonctions sur le demi-plan $p$-adique. Plus précisément, le quotient s'identifie au dual de $(\CO/\Qp)\otimes_{\Qp}L$.  On va étendre ces constructions en poids supérieurs. 

Pour $\lambda\in P$, on définit la \emph{Steinberg localement algébrique de poids $\lambda$} par
$$
\St^{\lalg}_{\lambda}\coloneqq \St^{\infty}\otimes_L W^*_{\lambda}.
$$
C'est une $L$-représentation localement algébrique de $G$ de type $LB$ qui est irréductible.
 
Pour $\lambda\in P$, on définit finalement la \emph{Steinberg localement analytique de poids $\lambda$} par
$$
\St^{\lan}_{\lambda}\coloneqq \St^{\lan}\otimes_L W^*_{\lambda}.
$$
C'est une $L$-représentation de $G$ de type $LB$. Ces représentations admettent aussi des descriptions en termes d'induite localement analytiques (\cite[Remarque 2.2]{colan}). Pour tout $i\in\BN$, notons $x^i$ le caractère algébrique $\Qp^{\times} \rightarrow L^{\times}$, $a\mapsto a^i$. Posons
$$
B_{\lambda}\coloneqq \Ind_B^G(x^{\lambda_1}\otimes x^{\lambda_2-1})\otimes_L \lvert \det \rvert_p^{\frac{\lvert \lambda\rvert -1}{2}}, \quad B^{\lambda}\coloneqq \Ind_B^G( x^{\lambda_2-1}\otimes x^{\lambda_1})\otimes_L \lvert \det \rvert_p^{\frac{\lvert \lambda\rvert -1}{2}}.
$$
Alors $B^{\lambda}$ est irréductible et $W^*_{\lambda}\subset B_{\lambda}$, qui n'est donc pas irréductible. On obtient
$$
\St^{\lan}_{\lambda}= B_{\lambda}/W^*_{\lambda}.
$$
On note $\Ext^1_{L[\bar G]}$ le groupe des extensions dans la catégorie des $L$-représentations localement analytiques de caractère central fixé\footnote{Si on note $\zeta_{\lambda}\colon \Qpt \rightarrow L^{\times}$ le caractère central de $W_{\lambda}^*$ et $\St_{\lambda}^{\lan}$, un petit calcul montre que $\zeta(a)=\omega(a)^{\lvert \lambda \rvert-1}$ où $\omega\colon \Qpt \rightarrow L^{\times}$ est le caractère cyclotomique défini par $\omega(a)\coloneqq a\lvert a \rvert$.}. D'après \cite[Proposition 2.6]{colan} et \cite[Théorème 2.7]{colan} on obtient le lemme suivant :
\medskip
\begin{lemm}\label{lem:extserispeci}
\
\begin{enumerate}
\item L'action par dérivation, donnée par l'élément $u^{+}\coloneqq\begin{pmatrix}0&1\\0&0\end{pmatrix} \in\fkg$ de l'algèbre de Lie de $G$, induit une suite exacte 
$$
0\rightarrow \St^{\lalg}_{\lambda}\rightarrow \St^{\lan}_{\lambda}\xrightarrow{(u^+)^{w(\lambda)}}B^{\lambda}\rightarrow 0.
$$
En particulier, $\St^{\lan}_{\lambda}$ a deux composantes de Jordan-Hölder.
\item On a un isomorphisme naturel $\Hom(\BQ^{\times}_p,L)\xrightarrow{\sim} \Ext^1_{L[\bar G]}\intnn{W_{\lambda}^*}{\St^{\lan}_{\lambda}}$. 
\end{enumerate}
\end{lemm}
\medskip
Le membre de gauche de cet isomorphisme est de dimension $2$ sur $L$, engendré par la valuation $p$-adique $v_p$ et le logarithme $\log$ normalisé par $\log(p)=0$. La base $(v_p,\log)$ donne une identification $\BP(\Hom(\BQ^{\times}_p,L))\cong \BP^1(L)$ ce qui permet de définir pour $\CL\in L\subset \BP^1(L)$ une représentation localement analytique de type $LB$ que l'on note $\Sigma_{\CL}^{\lambda}$ et qui est par définition contenue dans une suite exacte de $L$-représentations de $G$ non scindées
$$
0\rightarrow \St^{\lan}_{\lambda} \rightarrow \Sigma_{\CL}^{\lambda}\rightarrow W_{\lambda}^*\rightarrow 0.
$$
En particulier, $\Sigma_{\CL}^{\lambda}$ a trois composantes de Jordan-Hölder. Rappelons de plus que $\Sigma_{\CL}^{\lambda}$ est la représentation localement analytique construite par Breuil dans \cite{bre} où il démontre les résultats suivants :
\medskip
\begin{lemm}
On a,
\begin{itemize}
\itemb $\End_L\intn{\Sigma_{\CL}^{\lambda}}=L$,
\itemb si $(\lambda,\CL,\chi)\neq (\lambda',\CL',\chi')$ alors 
$$
\Sigma_{\CL}^{\lambda}\otimes (\chi\circ \det) \not \cong \Sigma_{\CL'}^{\lambda'}\otimes (\chi'\circ \det),
$$
où $\chi,\chi'\colon \Qpt\rightarrow L^{\times}$ sont des caractères lisses,
\itemb les vecteurs localement algébriques de $\Sigma_{\CL}^{\lambda}$ sont $\St_{\lambda}^{\lalg}$.
\end{itemize}
\end{lemm}
\medskip
\subsubsection{Dualité de Morita}\label{subsub:morita}
La dualité de Morita donne un isomorphisme explicite entre les formes différentielles sur le demi-plan $p$-adique et les distributions localement analytiques sur l'espace projectif. Plus précisément, on a un isomorphisme topologique $G$-équivariant $\intn{\St^{\lan}}'\cong \Omega^1\otimes_{\Qp}L\cong \CO\{2,1\}$, donné explicitement par \emph{le noyau de Poisson}
$$
\mu\in \intn{\St^{\lan}}'\mapsto \int_{\BP^1({\Qp})}\frac{dz}{z-x}\mu(x),
$$
où l'intégrale à droite signifie l'accouplement entre la distribution $\mu$ et la fonction $x\in\BP^1({\Qp})\mapsto \frac{dz}{z-x}\in\Omega^1$ à valeurs dans les différentielles sur $\BH_{\Qp}$. Breuil a étendu cette dualité à la série spéciale $p$-adique (\cf \cite[Théorème 3.2.3]{bre}). On aura besoin d'une version tordue de la dualité de Morita, qui apparait chez Breuil (\cf \cite[Théorème 3.1.2]{bre}) mais aussi chez Dasgupta-Teitelbaum (\cf \cite[Theorem 2.2.1]{date}). Nous énonçons le théorème en utilisant le noyau de Poisson, pour lequel on renvoit à \cite[Proposition 2.2.6]{date}. Soit $\lambda\in P$, notons
$$
\omega_{\lambda}\coloneqq \CO\{1+w(\lambda),1-\lambda_1\}\otimes_L \lvert \det \rvert^{\frac{1-\lvert \lambda \rvert}{2}}.
$$
\medskip
\begin{prop}\label{prop:morita}
Soit $\lambda \in P_+$.  Alors on a un isomorphisme topologique entre $L$-représentations de $G$, $\intn{\St_{\lambda}^{\lan}}'\cong \omega_{\lambda}$ donné par le noyau de Poisson
$$
\mu\in \intn{\St_{\lambda}^{\lan}}'\mapsto \int_{\BP^1({\Qp})}\frac{1}{z-x}\mu(x).
$$
\end{prop}
\medskip
\Subsection{Préliminaires techniques}
\subsubsection{Cohomologie de $\czG^+(n)$ à valeur dans une induite}
On note $\czG^+\coloneqq \rmODt$, $\czG^1\coloneqq \ker (\nrd)$, $\czG(n)\coloneqq1+\varpi_{\rmD}^n\rmOD$ et $\czG^1(n)\coloneqq \czG(n)\cap \czG^1$. Le but de ce paragraphe est de démontrer la proposition suivante, concernant la cohomologie continue d'une induite continue :
\medskip
\begin{prop}\label{prop:anulcog}
Soient $n\geqslant 0$ et $k\geqslant 0$ deux entiers. Alors 
$$
\rmH^1\Harg{\czG(n)}{\Ind_{\czG^1}^{\czG^+}\Sym_{\Zp}^k\rmOD}
$$
est un $\Zp$-module de torsion $p$-primaire et de type fini. En particulier, il est tué par une puissance de $p$.
\end{prop}
La preuve repose sur deux lemmes issues de résultats de Lazard (\cf \cite{laz}) que l'on utilisera après avoir appliqué le lemme de Shapiro. Le premier lemme permet de montrer que la cohomologie est de $p$-torsion.
\medskip
\begin{lemm}\label{lem:cohalg}
Soit $H\subset \czG^1$ un sous-groupe ouvert compact. Soit $V$ une $\Qp$-représentation algébrique de $H$ de dimension finie, alors
$$
\rmH^1\Harg{H}{V}=0.
$$
\end{lemm}
\begin{proof}
Premièrement, remarquons que l'algèbre de Lie de $H$ est la même que celle de $\czG^1$, qui est $\rmD^1\subset \rmD$, la sous-algèbre de Lie des éléments dont la trace réduite est nulle. De plus, $H$ est un groupe de Lie $p$-adique et $V$ est a fortiori une représentation localement analytique. Ainsi, d'après un théorème de Lazard, \cite[Chapitre V Théorème 2.4.9]{laz}, la cohomologie de groupe est calculée par la cohomologie de l'algèbre de Lie. Plus précisément,
$$
\rmH^1\Harg{H}{V}\cong\rmH^1\Harg{\rmD^1}{V}.
$$
Donc, il suffit de montrer que $\rmH^1\Harg{\rmD^1}{V}=0$. Or $\rmD^1$ est une algèbre de Lie semi-simple et $V$ est une représentation de dimension finie. On utilise maintenant le théorème de Whitehead (\cf  \cite[III. Theorem 13]{jac}) pour conclure que $\rmH^1\Harg{\rmD^1}{V}=0$.
\end{proof}
Le second lemme permet de montrer que le groupe de cohomologie en question est de type fini.
\begin{lemm}\label{lem:tfgrpco}
Soient $n\geqslant 0$ et $k\geqslant 0$ deux entiers. Alors,
$$
\rmH^1\Harg{\czG^1(n)}{\Sym_{\Zp}^k\rmOD},
$$
est un $\Zp$-module de type fini.
\end{lemm}
\medskip
\begin{proof}
Notons $S\coloneqq \Sym_{\BF_p}^k\rmOD$, $\rmH^1 \coloneqq \rmH^1\Harg{\czG^1(n)}{S}$ et $\bar{S}\coloneqq S/\varpi_{\rmD}\cong \Sym_{\BF_p}^k\BF_{p^2}$. Alors $\rmH^1\coloneqq \varprojlim_m \rmH^1/{\varpi_{\rmD}}^m$ (\cf par exemple \cite[Corollary 2.7.6]{nescwi}) et
$$
\rmH^1/{\varpi_{\rmD}} =  \rmH^1\Harg{\czG^1(n)}{\bar S}.
$$ 
On va montrer que $\rmH^1/{\varpi_{\rmD}}$ est fini. 

Supposons que $n\geqslant 1$, alors l'action de $\czG^1(n)$ est triviale sur $\bar S$. Donc, pour montrer que $\rmH^1/{\varpi_{\rmD}}$ est de dimension finie, il suffit de justifier que $\rmH^1\Harg{\czG^1(n)}{\BF_p}$ est de dimension finie. Mais d'après un résultat de Lazard (\cf \cite[Chapitre V, Théorème 2.5.8]{laz}), comme $\czG^1(n)$ est un groupe analytique $p$-adique compact et sans-torsion, c'est un groupe de Poincaré ; en particulier $\rmH^1\Harg{\czG^1(n)}{\BF_p}$ est fini et donc $\rmH^1/{\varpi_{\rmD}}$ est fini dès que $n\geqslant 1$. 

Si $n=0$, comme $\czG^1/\czG^1(1)$ est un groupe d'ordre $p+1$, donc premier à $p$, on a $\rmH^i\Harg{\czG^1/\czG^1(1)}{\bar S}=0$ pour $i\geqslant 1$ et la suite d'inflation restriction donne un isomorphisme
$$
\rmH^1\Harg{\czG^1}{\bar S}\cong \rmH^1\Harg{\czG^1(1)}{\bar S}^{\czG},
$$
et on vient de montrer que le terme de droite est fini. Ainsi $\rmH^1/{\varpi_{\rmD}}$ est fini pour tout entier $n\geqslant 0$.

Finalement, d'après le lemme de Nakayama, $\rmH^1/{\varpi_{\rmD}}^m$ est un $\BZ/p^m$-module, et donc un $\Zp$-module, de type fini engendré par $\dim_{\BF_p} \rmH^1\Harg{\czG^1(n)}{\bar S}$ générateurs. On en déduit que $\rmH^1\coloneqq \varprojlim_m \rmH^1/{\varpi_{\rmD}}^m$ est de type fini.
\end{proof}
On démontre maintenant la proposition \ref{prop:anulcog}.
\begin{proof}
Notons $S\coloneqq \Sym_{\Zp}^k\rmOD$, $I(1)\coloneqq \Ind_{\czG^1}^{\czG^+}S$ et $I(n)\coloneqq \Ind_{\czG^1(n)}^{\czG(n)}S$. On commence par deux observations :
\begin{itemize}
\itemb On a $I(1)\cong \sC^0\intnn{\Zp^{\times}}{\Zp}\otimes_{\Zp}S$ puisque $S$ est un $\Zp$-module libre de rang fini. De même, $I(n)\cong \sC^0\intnn{1+p^n\Zp}{\Zp}\otimes_{\Zp}S$.
\itemb Par le lemme de Shapiro,
$$
\rmH^1\Harg{\czG(n)}{I(n)}\cong \rmH^1\Harg{\czG^1(n)}{S}.
$$
\end{itemize}
Or, par le lemme \ref{lem:cohalg}, $\rmH^1\Harg{\czG^1(n)}{S}\otimes_{\Zp}\Qp=0$ ce qui signifie que $\rmH^1\Harg{\czG^1(n)}{S}$ est de torsion $p$-primaire et le lemme \ref{lem:tfgrpco} nous assure qu'il est de type fini. Ainsi $\rmH^1\Harg{\czG(n)}{I(n)}$ est de torsion $p$-primaire et de type fini. Or,
$$
I(1)\cong \bigoplus_{a\in (\BZ/p^n)^{\times}}\sC^0\intnn{a+p^n\Zp}{\Zp}\otimes_{\Zp}S\cong \bigoplus_{a\in (\BZ/p^n)^{\times}}I(n)
$$
en tant que représentation de $\czG(n)$. Ainsi $\rmH^1\Harg{\czG(n)}{I(1)}$ est de torsion $p$-primaire et de type fini.
\end{proof}
\subsubsection{Annulation d'un groupe d'extension localement analytique}\label{subsub:anulocext}
Commençons par motiver ce que l'on va faire, quitte à sortir légèrement du cadre de ce paragraphe. On introduit la $L$-représentation~$\Pi_{\CL}^{\lambda}$ de $G$ en \ref{subsec:seriespe}. Les vecteurs localement analytiques de $\Pi_{\CL}^{\lambda}$ ne sont pas exactement ${\Sigma}_{\CL}^{\lambda}$~; ils ont une composante de Jordan-Hölder supplémentaire dont on va montrer ici qu'elle n'interviendra pas dans nos futurs calculs. Plus précisément, la $L$-representation localement analytique $\Pilan_{\CL}^{\lambda}$ de $G$ a $4$ composantes de Jordan-Hölder et on a une suite exacte non scindée
$$
0\rightarrow {\Sigma}_{\CL}^{\lambda}\rightarrow \Pilan_{\CL}^{\lambda}\rightarrow {\wt{B}}^{\lambda}\rightarrow 0,
$$
où $\wt{B}^{\lambda}\coloneqq \Ind_B^G(\omega x^{\lambda_2-1}\otimes \omega^{-1}x^{\lambda_1})\otimes_L\lvert \det \rvert_p^{\frac{\lvert \lambda \rvert -1}{2}}$, avec $\omega(x)=x\lvert x \rvert$ le caractère cyclotomique, qui est une $L$-représentation localement analytique de $G$. On va montrer le résultat suivant :
\medskip
\begin{lemm}\label{lem:anuprinc}
Soit $\lambda\in P_+$. Alors 
$$
\Ext^1_{L[\bar G]}\intnn{{W}_{\lambda}^*}{{\wt{B}}^{\lambda}}=0.
$$
\end{lemm}
\medskip
Avant de démontrer ce lemme on a besoin d'un peu de préparation. Rappelons que qu'on note $T\subset G$ le sous-groupe des matrices diagonales que l'on voit comme un quotient $B\rightarrow T$ du Borel des matrices triangulaires supérieures.
\medskip
\begin{lemm}\label{lem:anucar}
Soit $\delta\colon T\rightarrow L^{\times}$ un caractère localement analytique mais non algébrique (en particulier non trivial), que l'on voit comme un caractère de $B$ via $B\rightarrow T$. Alors, pour tout entier $i\geqslant 0$, 
$$
\rmH^i\Harg{B}{\delta}=0.
$$
\end{lemm}
\medskip
\begin{proof}
On commence par justifier qu'il suffit de montrer que pour tout caractère localement analytique non trivial $\delta \colon T \rightarrow L^{\times}$, on a $\rmH^i\Harg{T}{\delta}=0$. En effet, la suite spectrale d'Hochschild-Serre s'écrit
$$
E_2^{i,j} \coloneqq \rmH^{i}\Harg{T}{\rmH^j\Harg{N}{\delta}}\implies \rmH^{i+j}\Harg{B}{\delta}
$$
et en tant que représentation de $T$, on a $\rmH^j\Harg{N}{\delta}=\rmH^j\Harg{N}{\Qp}\otimes_{\Qp}\delta$. Mais, en tant que représentation de $T$,  $\rmH^j\Harg{N}{\Qp}$ est un caractère algébrique donc, par hypothèse,  $\rmH^j\Harg{N}{\delta}$ est un caractère non algébrique. Ainsi, il suffit de montrer que l'on a $\rmH^i\Harg{T}{\delta}=0$ pour tout entier $i\geqslant 0$ comme annoncé. 

Par la formule de Künneth, en écrivant $\delta=\delta_1\otimes \delta_2$ on obtient 
$$
\rmH^i\Harg{T}{\delta}=\bigoplus_{k_1+k_2=i}\rmH^{k_1}\Harg{\Qp^{\times}}{\delta_1}\otimes_L\rmH^{k_2}\Harg{\Qp^{\times}}{\delta_2},
$$
et on se ramène donc à montrer que pour tout caractère unitaire $\delta_0\colon \Qp^{\times}\rightarrow L^{\times}$ et pour tout entier $i\geqslant 0$ on a 
$$
\rmH^i\Harg{\Qp^{\times}}{\delta_0}=0.
$$
Pour démontrer ce résultat, on décompose $\Qp^{\times}=p^{\BZ}\cdot\mu_{p-1}\cdot(1+p\BZ_p)\cong \BZ\times \mu_{p-1}\times \BZ_p$ et la décomposition de Künneth donne
$$
\rmH^i\Harg{\Qp^{\times}}{\delta_0}\cong \bigoplus_{k_1+k_2+k_3=i}\rmH^{k_1}\Harg{\BZ}{\delta_0}\otimes_L\rmH^{k_2}\Harg{\mu_{p-1}}{\delta_0}\otimes_L\rmH^{k_3}\Harg{\BZ_p}{\delta_0}.
$$
On analyse les termes un à un pour justifier que dans chaque terme de la somme au moins l'un des groupes de cohomologie s'annule.
\begin{itemize}
\itemb On a $\rmH^{k_1}\Harg{\BZ}{\delta_0}=0$ pour $k_1\geqslant 2$ ou pour tout entier $k_1\geqslant 0$ tel que $\delta_0(p)\neq 1$.
\itemb On a $\rmH^{k_2}\Harg{\mu_{p-1}}{\delta_0}=0$ pour $k_2\geqslant 1$ puisque les représentations de groupes finis en caractéristique zéro sont semi-simples, puis $\rmH^{0}\Harg{\mu_{p-1}}{\delta_0}=0$ si $\restr{\delta_0}{\mu_{p-1}}\neq 1$.
\itemb On a $\rmH^{k_3}\Harg{\BZ_p}{\delta_0}=0$ pour $k_3\geqslant 2$, ou pour tout entier $k_3\geqslant 0$ tel que $\restr{\delta_0}{(1+p\BZ_p)}\neq 1$.
\end{itemize}
Puisque $\delta_0\neq 1$, l'une des restrictions considérées n'est pas triviale, donc l'un des trois groupes de cohomologie est toujours nul ce qui permet de conclure que $\rmH^i\Harg{\Qp^{\times}}{\delta_0}=0$ pour tout entier $i\geqslant 0$.
\end{proof}
On peut maintenant passer à la démonstration du lemme \ref{lem:anuprinc}.
\begin{proof}
D'après le Lemme de Shapiro on a 
$$
\Ext^1_{L[\bar G]}\intnn{{W}_{\lambda}^*}{{\wt{B}}^{\lambda}}\cong \rmH^1\Harg{B}{W_{\lambda}\otimes_L(\omega x^{\lambda_2-1}\otimes \omega^{-1}x^{\lambda_1})\otimes_L \lvert \det \rvert_p^{\frac{\lvert \lambda \rvert -1}{2}}},
$$
où la cohomologie à droite est la cohomologie localement analytique de groupes. Or, en tant que représentation de $B$, $\Sym^{w(\lambda)-1}$ est une extension successive de caractères algébriques et donc $W_{\lambda}\otimes_L(\omega x^{\lambda_2-1}\otimes \omega^{-1}x^{\lambda_1})\otimes_L \lvert \det \rvert_p^{\frac{\lvert \lambda \rvert -1}{2}}$ est une extension successive de caractères localement algébriques dont la partie lisse n'est pas triviale. Ainsi, il suffit de montrer que pour tout entier $i\geqslant 0$ et pour tout caractère localement algébrique non trivial $\delta\colon \Qp^{\times}\rightarrow L^{\times}$ on a 
$$
\rmH^i\Harg{B}{\delta}=0.
$$
Or, c'est le contenu du lemme \ref{lem:anucar}, ce qui conclut la preuve du lemme.
\end{proof}

\section{Calculs de cohomologie}
Le but de cette section est de décomposer les différentes cohomologies isotriviales de la tour de Drinfeld suivant les poids pour se ramener à des opers dont on sait maintenant calculer la cohomologie (\cf la première partie). On traitera aussi des torsions par des caractères lisses qui apparaissent au travers des composantes connexes.
\Subsection{Décomposition des cohomologies}\label{subsec:decompcoh}
\subsubsection{Rappels des définitions}
Faisons quelques rappels tout en introduisant de nouvelles notations. On a la tour d'espaces rigides $\brv{\rmM}_C^{\infty}$ muni d'une action de $\BG=\sW_{\Qp}\times G \times \czG$. Son quotient par $p\in G$, noté $\presp{\rmM}_C^{\infty}$, est la somme de deux copies de $\rmM_C^{\infty}$ et ce dernier est uniquement muni d'une action de $\BG^+\coloneqq \nu^{-1}(\Zpt)$. Sur ces différentes tours, on définit le $L$-système local $\BV\coloneqq \BV_{\Dr}\otimes_{\Qp}L\cdot\lvert \nu \rvert_p^{1/2}$ où $\BV_{\Dr}$ est le module de Tate rationnel du $\rmOD$-module formel spécial universel. Le module de Tate $\BV_{\Dr}^+$ définit un réseau $\BV^+$ et on pose
$$
\Sym \BV\coloneqq \bigoplus_{k\geqslant 0} \Sym_L^k\BV,\quad \Sym \BV^+\coloneqq \bigoplus_{k\geqslant 0} \Sym_{\rmO_L}^k\BV^+.
$$
On définit la cohomologie étale de $\Sym \BV(1)$ comme étant la somme directe des cohomologies étales des $\Sym_L^k\BV(1)$. Soit
$$
\brv{\rmH}^1_{\et}=\rmH^1_{\et}\Harg{\brv{\rmM}_{C}^{\infty}}{\Sym \BV(1)}\coloneqq \bigoplus_{k\geqslant 0}\varinjlim_n\rmH^1_{\et}\Harg{\brv{\rmM}_{C}^{n}}{\Sym_{\rmO_L}^k \BV^+(1)}\otimes_{\rmO_L}L.
$$
De façon similaire,
$$
\rmH^1_{\et}\coloneqq \rmH^1_{\et}\Harg{\rmM_{C}^{\infty}}{\Sym \BV(1)},\quad \pbrv{\rmH}^1_{\et}=\presp{\rmH}^1_{\et}\coloneqq \rmH^1_{\et}\Harg{\presp{\rmM}_{C}^{\infty}}{\Sym \BV(1)}.
$$
Ici, $\brv{\rmH}^1_{\et}$ et $\presp{\rmH}^1_{\et}$ (resp. $\rmH^1_{\et}$) sont des $L$-représentations de $\BG$ (resp. $\BG^+$). De même, les cohomologies proétales sont définies par 
$$
\brv{\rmH}^1_{\pet}=\rmH^1_{\et}\Harg{\brv{\rmM}_{C}^{\infty}}{\Sym \BV(1)}\coloneqq\bigoplus_{k\geqslant 0}\varinjlim_n\rmH^1_{\pet}\Harg{\brv{\rmM}_{C}^{n}}{\Sym_L^k \BV(1)},
$$
et de façon similaire 
$$
\rmH^1_{\pet}\coloneqq \rmH^1_{\pet}\Harg{\rmM_{C}^{\infty}}{\Sym \BV(1)},\quad \pbrv{\rmH}^1_{\pet}=\presp{\rmH}^1_{\pet}\coloneqq \rmH^1_{\pet}\Harg{\presp{\rmM}_{C}^{\infty}}{\Sym \BV(1)}.
$$
De même, $\brv{\rmH}^1_{\pet}$ $\presp{\rmH}^1_{\pet}$ (resp. $\rmH^1_{\pet}$) sont des $L$-représentations de $\BG$ (resp. $\BG^+$). Pour les cohomologies de de Rham et de Hyodo-Kato on va surtout utiliser $\presp{\rmM}_{\Qp}^{\infty}$, comme on a besoin d'un modèle sur une extension finie de $\Qp$ pour appliquer ce qui à été développé dans la première partie. On associe à $\BV$ un fibré plat filtré $\CE=(\CE,\nabla,\Fil)$ et comme $\BV$ est fortement isotrivial de fibré associé $\BD$ on obtient que $\CE\coloneqq \BD\otimes_{\Qp}\CO^{\infty}$ (on note temporairement $\CO^n$ le faisceau des fonctions sur $\presp{\rmM}_{\Qp}^{n}$) et $\nabla = \id\otimes d$. Le fibré associé à $\Sym^k_L\BV$ est donné par la puissance symmétrique $k$-ième de $\CE$. Ainsi, on définit
$$
\Sym \CE \coloneqq \bigoplus_{k\geqslant 0}\Sym_{\CO^{\infty}}^k\CE,\quad \Sym \BD \coloneqq \bigoplus_{k\geqslant 0}\Sym_L^k \BD,
$$
et
$$
\presp{\rmH}_{\dR}^1\coloneqq \bigoplus_{k\geqslant 0}\varinjlim_n\rmH_{\dR}^1\Harg{\presp{\rmM}_{\Qp}^{n}}{\Sym_{\CO^n}^k\CE},\quad \presp{\rmH}_{\HK}^1\coloneqq \bigoplus_{k\geqslant 0}\varinjlim_n\rmH_{\HK}^1\Harg{\presp{\rmM}_{\Qp}^{n}}{\Sym_{L}^k\BD}.
$$
se sont des $L$-représentations de $\BG$ (où $\sW_{\Qp}$ agit trivialement) et $\presp{\rmH}_{\HK}^1\cong\presp{\rmH}_{\dR}^1$ par isomorphisme de Hyodo-Kato. De plus, $\presp{\rmH}_{\dR}^1$ est muni d'une filtration $L[\czG]$-stable et $\presp{\rmH}_{\HK}^1$ est muni d'un endomorphisme $L$-linéaire $\varphi$ et d'une monodromie, qui s'avère être triviale (c'est un $1$-morphisme non trivial vers $\rmH^0_{\HK}$ comme on le verra en \ref{subsub:nontrivmono}). On définit aussi
$$
\begin{gathered}
\pbrv{\rmH}^1_{\dR}\coloneqq  \bigoplus_{k\geqslant 0}\varinjlim_n\rmH_{\dR}\Harg{\presp{\rmM}_{C}^{n}}{\Sym_{\CO^n}\CE}, \quad  \pbrv{\rmH}_{\HK}^1\coloneqq \bigoplus_{k\geqslant 0}\varinjlim_n\rmH_{\HK}^1\Harg{\presp{\rmM}_{C}^{n}}{\Sym_{L}^k\BD},
\end{gathered}
$$
où le fibré plat est définit de la même façon. Alors $\pbrv{\rmH}^1_{\dR}\cong\presp{\rmH}^1_{\dR}\otimes_{\Qp}C$ qui est donc un un $C$-espace vectoriel muni d'une action de $\BG$ et d'une filtration $L[\sW_{\Qp}\times \czG]$-stable ; ${\pbrv{\rmH}}^1_{\HK}\cong \presp{\rmH}_{\HK}\otimes_{\Qp}\Qpbr$ est un $L\otimes_{\Qp}\Qpbr$-espace vectoriel, muni d'une action de $\BG$ et d'un endomorphisme $\sigma$-linéaire. Notons que $\sW_{\Qp}$ n'agit plus trivialement et l'isomorphisme de Hyodo-Kato s'écrit $\brv{\rmH}^1_{\HK}\otimes_{\Qpbr}C\cong\brv{\rmH}^1_{\dR}$. Les définitions de $\brv{\rmH}^1_{\dR}$ et $\brv{\rmH}^1_{\HK}$ se font à partir de $\brv{\rmM}_{C}^{\infty}$ puis $\rmH^1_{\dR}$ et $\rmH^1_{\HK}$ se font à partir de $\rmM_{C}^{\infty}$ ; elles sont très similaires et ne seront pas détaillés.

\subsubsection{Torsion par des caractères lisses et composantes connexes}\label{secsec:uncomp}
On va utiliser la décomposition de $\brv{\rmH}^0_{\et}$ (\cf lemme \ref{lem:deccar}) pour se débarasser des caractères et se restreindre au calcul des $\presp{\rmH}^1_{\star}$ (voire $\rmH^1_{\star}$) pour le reste de l'article. 
\medskip
\begin{lemm}\label{lem:uncomp}
Pour $\star=\{\et,\pet,\dR,\HK\}$ on a  :
$$
\brv{\rmH}^{1}_{\star}=\Ind_{\BG^+}^{\BG}\rmH^1_{\star},
$$
où l'induite est continue, et 
$$
\rmH^{1}_{\star}=\bigoplus_{\chi\colon \Zpt\rightarrow L}\rmH^1_{\star}[\chi],
$$
où la somme porte sur les caractères lisses $\chi \colon \Zpt\rightarrow L^{\times}$. De plus, $\rmH^1_{\star}[\chi]\cong \rmH^1_{\star}[\boldsymbol{1}]\otimes (\chi\circ \nu)$ où $\boldsymbol{1}$ désigne le caractère trivial.
\end{lemm}
\medskip
\begin{rema}\label{rema:sumchar}
\begin{enumerate}
\item Pour $\chi\colon \Qpt \rightarrow L^{\times}$ un caractère lisse on pose $\rmH^1_{\star}[\chi]\coloneqq \rmH^1_{\star}[\boldsymbol{1}]\otimes (\chi\circ \nu)$. En utilisant le premier point de la remarque \ref{rem:concomp} on obtient un plongement $\BG$-équivariant. 
$$
\bigoplus_{\chi\colon \Qpt\rightarrow L}\rmH^1_{\star}[\chi]\incl \brv{\rmH}^{1}_{\star}.
$$
Donc en tant que $L$-représentation de $\BG$, le socle de $\brv{\rmH}^{1}_{\star}$ est contenu dans le membre de gauche.
\item On déduit de la remarque précédente qu'on a
$$
\pbrv{\rmH}^1_{\star}\cong \rmH^1_{\star}\otimes(\boldsymbol{1}\oplus \chi_2)\circ \nu.
$$
où $\chi_2\colon \Qpt\rightarrow L^{\times}$ est le caractère non-ramifié d'ordre $2$.
\end{enumerate}
\end{rema}
\medskip
\begin{lemm}\label{lem:fkcompcon}
Soit $\star \in \{\et,\pet,\dR,\HK\}$.
\begin{enumerate}
\item On a une inclusion $\pbrv{\rmH}^1_{\star}\incl \brv{\rmH}^1_{\star}$ qui identifie $\presp{\rmH}^1_{\star}$ aux vecteurs fixés par $p\in \BG$ vu comme élément de $G$ ou $\czG$. En d'autres termes, si $W$ est une $L$-représentation de $\BG$ tel que $p\in \BG$ agit trivialement, alors 
$$
\Hom_{L[\BG]}\intnn{W}{\brv{\rmH}^1_{\star}} = \Hom_{L[\BG]}\intnn{W}{\pbrv{\rmH}^1_{\star}}.
$$
\item Soit $\chi \colon \Qpt \rightarrow L^{\times}$ un caractère lisse, et $W$ une représentation de $\BG$, alors
$$
\Hom_{L[\BG]}\intnn{W}{\brv{\rmH}^1_{\star}} \cong \Hom_{L[\BG]}\intnn{W\otimes (\chi\circ \nu)}{\brv{\rmH}^1_{\star}}.
$$
\end{enumerate}
\end{lemm}
\medskip
\begin{proof}
La second point du lemme est une conséquence immédiate du premier point de la remarque \ref{rema:sumchar} puisqu'elle implique que $\brv{\rmH}^1_{\star}\otimes \chi=\brv{\rmH}^1_{\star}$. On démontre la première partie du lemme. Pour éviter toute confusion, notons $[p]$ l'endomorphisme de $\brv{\rmH}^1_{\star}$ induit par $p\in\BG$. On écrit la suite spectrale de Hochschild-Serre pour le recouvrement $\brv{\rmM}_C^{\infty}\rightarrow \presp{\rmM}_C^{\infty}$, qui donne, en bas degrés, une suite exacte
$$
0\rightarrow  \rmH^1\Harg{p^{\BZ}}{\brv{\rmH}^0_{\star}}\rightarrow \pbrv{\rmH}^1_{\star}\rightarrow\left ( \brv{\rmH}^1_{\star}\right )^{[p]=1}\rightarrow\rmH^2\Harg{p^{\BZ}}{\brv{\rmH}^0_{\star}}.
$$ 
Donc, il suffit de montrer que $ \rmH^1\Harg{p^{\BZ}}{\brv{\rmH}^0_{\star}}=0$ pour tout entier $i\geqslant 0$. Or, $\brv{\rmH}^0_{\star}=\brv{\rmH}^0_{\et}$ et comme $\brv{\rmH}^0_{\et}=\Ind_{\BG^+}^{\BG}\rmH^0_{\et}$ d'après le lemme \ref{lem:deccar}, c'est un $L[p^{\BZ}]$-module acyclique. Donc 
$$
\rmH^i\Harg{p^{\BZ}}{\brv{\rmH}^0_{\star}}=0,
$$
ce qui prouve la première partie. 
\end{proof}
Le lemme précédent permet, en tordant par un caractère, de ramener le calcul des multiplicités dans $\brv{\rmH}^1_{\star}$ au calcul des multiplicités dans $\presp{\rmH}^1_{\star}$ ; on met cet énoncé sous forme d'un corollaire :
\medskip
\begin{coro}\label{cor:fkcompcon}
Soit $\star \in \{\et,\pet,\dR,\HK\}$ et soit $W$ une $L$-représentation absolument irréductible de $\BG$, alors il existe un caractère lisse $\chi \colon \Qpt \rightarrow L^{\times}$ tel que
$$
\Hom_{L[\BG]}\intnn{W}{\brv{\rmH}^1_{\star}} \cong \Hom_{L[\BG]}\intnn{W\otimes (\chi\circ \nu)}{\pbrv{\rmH}^1_{\star}}
$$
\end{coro}
\subsubsection{Décomposition du système local}
Le but de ce paragraphe est de décomposer le système local suivant les différents poids. On commence par observer que le corollaire \ref{cor:schurep} donne une décomposition de $\Sym\BD$ et donc du fibré $\Sym\CE$,
$$
\Sym\BD = \bigoplus_{\lambda\in P_+}\BD_{\lambda}\otimes_{L}\cz{W}_{\lambda},\quad \Sym \CE=\bigoplus_{\lambda\in P_+}\CE_{\lambda}\otimes_{L}\cz{W}_{\lambda},
$$
où 
\begin{itemize}
\itemb $\BD_{\lambda}$ est un $\varphi$-module sur $\Qp$, mais aussi un  $L[G]$-module isomorphe à $W_{\lambda}$ et
\itemb $\CE_{\lambda}$ est un fibré plat filtré $G$-équivariant, tel que $\CE_{\lambda}=\CO^{\infty}\otimes_{\Qp}W_{\lambda}$ en tant que fibré $G$-équivariant, muni de la connexion $\nabla=d\otimes \id$ et de sa filtration de Hodge induite par la puissance symétrique.
\end{itemize}
Le théorème de reconstruction des systèmes locaux isotriviaux (\cf  la proposition \ref{prop:reconstr}) nous donne directement : 
\medskip
\begin{prop}\label{prop:calcdec}
Pour tout $\lambda\in P_+$, il existe un $L$-système local fortement isotrivial $\BV_{\lambda}$, $G$-équivariant et tel que 
$$
\Sym \BV=\bigoplus_{\lambda\in P_+}\BV_{\lambda}\otimes_{L}\cz{W}_{\lambda}.
$$
De plus, en notant $\rho\in P_+$ le poids $\rho=(0,2)$ ce système local est donné explicitement par 
$$
\BV_{\lambda}=\Sym_{L}^{w(\lambda)-1}\BV_{\rho}\otimes_{L}{\det}^{\lambda_1}.
$$
Finalement $\BD_{\lambda}=\BD(\BV_{\lambda})$, l'action de $\varphi$ sur $\BD_{\lambda}$ est telle que $\varphi^2=p^{1-\lvert \lambda\rvert}$ et $\BD_{\lambda}=\BD_{(0,w(\lambda))}[-\lambda_1]$ en tant que $\varphi$-module.
\end{prop}
\medskip
\begin{proof}
En utilisant la décomposition selon les poids, il suffit de montrer $\BV\cong \BV_{\rho}\otimes_L \cz{W}_{\rho}$ en tant que $L$-système local. Rappelons que $\BV$ est muni d'une action de $\rmOD$ dont la trivialisation $\czG^1$-équivariante fournit la tour de revêtements étales $\rmM^{\infty}_C\rightarrow\rmM^0_C$. Pour $\bar x$ un point géométrique de $\rmM^n_C$, cette tour définit le morphisme $\pi_1(\rmM^n_C,\bar x)\rightarrow \czG^1(n)$ et la fibre de $\BV$ en $\bar x$ est décrite par 
$$
\BV(\bar x) \cong D\otimes_{\Qp}L\cong \cz{W}_{\rho}\oplus \cz{W}_{\rho}.
$$
Donc $\BV\cong \BV_{\rho}\otimes_L U$ où $U$ est un $L$-espace vectoriel correspondant à un système local trivial et $\BV_{\rho}(\bar x)=\cz{W}_{\rho}$. Il reste à montrer que $\BD(\BV_{\rho})\cong W_{\rho}$ ou, de manière équivalente, que $U\cong \cz{W}_{\rho}$.

Supposons le contraire, \ie $\BD(\BV_{\rho})\cong \cz{W}_{\rho}$. Alors le fibré plat filtré associé à $\BV_{\rho}$ est $\cz{W}_{\rho}\otimes_L\CO^{n}\otimes_{\Qp}C$ et comme sa filtration est $G$-équivariante, celle-ci est constante. C'est une contradiction puisque $\BV_{\rho}$ n'est pas un système local trivial.
\end{proof}
\begin{rema}
La proposition \ref{prop:calcdec} (ainsi que sa preuve) peut porter à confusion puisqu'on pourrait rétorquer que \og l'action de $\czG^1(n)$ sur $\BV_{\rho}$ ne peut pas être triviale puisque c'est l'action du $\pi_1$ et que le système local n'est pas trivial \fg. Or l'action du $\pi_1$ sur le système local est bien triviale (par définition) mais ne l'est pas sur les fibres. Le système local $\BV$ est muni d'une action de $\rmOD$ qui permet d'expliciter l'action du $\pi_1$ sur les fibres géométrique de $\BV$ mais ces deux actions sont différentes. La différence apparait puisque le $\pi_1$ agit sur les trivialisation du système local alors que $\rmOD$ agit sur le système local.

On pourrait aussi rétorquer que \og comme le Frobenius sur $\BD(\BV)$ est défini à partir de l'action de $\rmOD$ et comme l'action de $\rmOD$ est triviale sur $\BD(\BV_{\rho})$, c'est un isocristal trivial \fg. L'erreur dans cet énoncé provient une fois de plus de la différence entre l'action du $\pi_1$ et l'action de $\rmOD$ : c'est l'action du $\pi_1$ qui définit la structure d'isocristal sur $\BD(\BV_{\rho})$.
\end{rema}
\medskip
Cette décomposition permet de factoriser la partie algébrique de l'action de $\czG$ et de réduire le calcul de la cohomologie du système local . On note pour $\lambda\in P_+$ 
\begin{equation}\label{eq:notanul}
\begin{gathered}
 \brv{\rmH}^{\lambda}_{\et}\coloneqq \rmH^1_{\et}\Harg{\brv{\rmM}_{C}^{\infty}}{\BV_{\lambda}(1)},\quad  \presp{\rmH}^{\lambda}_{\et}\coloneqq  \rmH^1_{\et}\Harg{\presp{\rmM}_{C}^{\infty}}{\BV_{\lambda}(1)}\\
\brv{\rmH}^{\lambda}_{\pet}\coloneqq \rmH^1_{\pet}\Harg{\brv{\rmM}_{C}^{\infty}}{\BV_{\lambda}(1)},\quad \presp{\rmH}^{\lambda}_{\pet}\coloneqq \rmH^1_{\pet}\Harg{\presp{\rmM}_{C}^{\infty}}{\BV_{\lambda}(1)}\\
\brv{\rmH}^{\lambda}_{\dR}\coloneqq \rmH^1_{\dR}\Harg{\brv{\rmM}_{C}^{\infty}}{\CE_{\lambda}},\quad \presp{\rmH}^{\lambda}_{\dR}\coloneqq \rmH^1_{\dR}\Harg{\presp{\rmM}_{\Qp}^{\infty}}{\CE_{\lambda}}\\
\brv{\rmH}^{\lambda}_{\HK}\coloneqq \rmH^1_{\HK}\Harg{\brv{\rmM}_{C}^{\infty}}{\BD_{\lambda}},\quad \presp{\rmH}^{\lambda}_{\HK}\coloneqq \rmH^1_{\HK}\Harg{\presp{\rmM}_{\Qp}^{\infty}}{\BD_{\lambda}}
\end{gathered}
\end{equation}
On a préféré la notation $\rmH^{\lambda}_{\star}$ à $\rmH^{1,\lambda}_{\star}$ en espérant qu'elle n'induise pas de confusion. On pourrait aussi définir les $\pbrv{\rmH}^{\lambda}_{\star}$ . Une conséquence direct de la proposition \ref{prop:calcdec} sont les décompositions suivantes :
\medskip
\begin{coro}\label{coro:decomppoid}
Pour $\star \in \{\et,\pet,\dR,\HK\}$ on a des décompostion $\BG$-équivariantes
$$
\brv{\rmH}^1_{\star}=\bigoplus_{\lambda\in P_+}\brv{\rmH}^{\lambda}_{\star}\otimes_{L} \cz{W}_{\lambda}, \quad \presp{\rmH}^1_{\star}=\bigoplus_{\lambda\in P_+}\presp{\rmH}^{\lambda}_{\star}\otimes_{L} \cz{W}_{\lambda}.
$$
\end{coro}
\medskip

\Subsection{L'oper de Drinfeld}
\subsubsection{Filtration de $\CO^n\otimes W_{\lambda}$}\label{par:filt}
Rappelons qu'on note $\rho=(0,2)$. On commence par supposer que $\lambda=(0,k+1)$ pour $k\in \BN$ mais il est facile d'en déduire les formules générales en tordant par le déterminant ; c'est ce que l'on fera pour énoncer les résultats. On suit de près \cite[Paragraphe 5]{scst} mais on décale la filtration pour obtenir celle de $\CE_{\lambda}$. On définit l'élément $f = (\pi(z)e_1-e_0)$ qui vit dans $\CO^n\otimes_{\Qp} W_{\rho}$ et qui vérifie, pour $g\in G$, la relation $g\cdot f = j(g,z)^{-1}f$ ; c'est une forme modulaire de poids $-1$. Ainsi, on définit sur $\CO^n\otimes_{\Qp} W_{\lambda}$ une filtration décroissante de $\CO^n$-modules qui est $G$-équivariante et $\czG$-invariante, $\CO^n\otimes_{\Qp} W_{\lambda}=F^{-k}\supsetneq F^{-k+1} \supsetneq \dots \supsetneq F^{0}\supsetneq 0$, définie pour $q\in \BZ$ tel que $0\leqslant q \leqslant k+1$ par
$$
F^{q-k} = \bigoplus_{j=0}^{k-q}\CO^nf^{q}e_1^{k-q-j}e_0^{j},
$$
 $F^{q-k}=0$ si $q>k+1$ et $F^{q-k}=F^{-k}$ si $q<0$. De plus, les $e_0^je_1^{k-j}$ forment une $\CO^n$-base de $\CO^n\otimes_{\Qp} W_{\lambda}$ et de la formule $e_0e_1^j=\pi(z)e_1^{j+1}-fe_1^j$, on déduit par récurrence que pour $j\in\BN$ tel que $0\leqslant j \leqslant q+1$,
$$
f^{q}e_1^{k-q-j}e_0^j\in \CO^nf^{q-\lambda_1}e_1^{k-q}+F^{q-k+1}.
$$
Ainsi, les $f^qe_1^{k-q}$ avec $0\leqslant q \leqslant k$ forment une $\CO^n$-base de $\CO^n\otimes_{\Qp} W_{\lambda}$ et donc on définit un isomorphisme de $\CO^n$-modules $\Theta_q \colon \CO^n\rightarrow F^{q-k}/F^{q-k+1}$ par $\Theta_q(h) \coloneqq hf^qe_1^{k-q}$. Un petit calcul montre qu'il donne un isomorphisme $G$-équivariant $\Theta_q \colon \CO^n\{2q-k,q-k\}\xrightarrow{\sim} F^{q-k}/F^{q-k+1}$. En effet, soit $h\in\CO^n$, on a pour $g = \begin{pmatrix}a & b \\ c & d \end{pmatrix}\in G$
$$
g\cdot [hf^qe_1^{k-q}]= j(g,z)^{-q}\det(g)^{q-k}g\cdot hf^q\underbrace{(-ce_0+ae_1)^{k-q}}_{(j(g,z)e_1+cf)^{k-q}} \in j(g,z)^{k-2q}\det(g)^{q-k}g\cdot hf^qe_1^{k-q}+F^{q-k+1}.
$$
Soit $\Omega^n\coloneqq \Omega^1(\presp{\rmM_{\Qp}^n})$, rappelons que $\Omega^n\cong \CO^n\{2,1\}$ d'après le numéro \ref{subsub:morita}. Naturellement, on a une filtration $G$-équivariante $\Omega^n\otimes_{\Qp} W_{\lambda}=\Omega^n\otimes_{\CO^n}F^{-k}\supsetneq \Omega^n\otimes_{\CO^n}F^{-k+1} \supsetneq \dots\supsetneq  \Omega^n\otimes_{\CO^n}F^{1}=0$ et les quotients successifs sont décrits par $\Theta'_q\colon \CO^n\{2(q+1)-k,q+1-k\}\xrightarrow{\sim}\Omega^n\otimes_{\CO^n}F^{q-k}/\Omega^n\otimes_{\CO^n}F^{q-k+1}$. De plus, la dérivée donne
$$
(d\otimes \id) (hf^qe_1^{k-q}) = dz\otimes q f^{q-1}e_1^{k-q+1}+dh\otimes f^qe_1^{k-q}.
$$
Ainsi,
$$
(d\otimes \id) (hf^qe_1^{k-q})\in dz\otimes q f^{q-1}e_1^{k-q+1}+\Omega^n\otimes_{\CO^n}F^{q-k}.
$$
On en déduit que $d\otimes \id \colon F^{q-k}/F^{q-k+1}\rightarrow \Omega^n\otimes_{\CO^n}F^{q-k-1}/\Omega^n\otimes_{\CO^n}F^{q-k}$ s'identifie à la multiplication par $q$ si $1\leqslant q \leqslant k$. Finalement, on a montré le lemme suivant : 
\medskip
\begin{lemm}
Si $\lambda=(0,k+1)$, le fibré plat filtré $(\CO^n\otimes_{\Qp}W_{\lambda}, d\otimes id, F^{\bullet})$ est un oper de poids $(0,k)$.
\end{lemm}
\medskip
\subsubsection{L'oper $\CE_{\lambda}$}\label{subsub:operdedrin}
On tord maintenant tous les calculs précédents par ${\det}^{-\lambda_1}$. On a $\CE_{\lambda}=\CO^n\otimes W_{\lambda}$ et la filtration que l'on a construite coïncide naturellement avec la filtration de Hodge définit par $\BV_{\lambda}$ sur $\CE_{\lambda}$ d'après la formule (\ref{eq:propunivperiodmor}). En particulier, cette filtration est $G\times \czG$-équivariante et on en a calculé les quotients. On obtient la proposition suivante :
\medskip
\begin{prop}\label{prop:equifil}
Soit $\lambda \in P_+$, la filtration de Hodge $\CE_{\lambda}=F^{-\lambda_2+1}\supsetneq F^{-\lambda_2+2}\supsetneq \dots \supsetneq F^{-\lambda_1}\supsetneq 0$ est telle que 
$$
F^q/F^{q+1}\cong \CO^n\{2q+\lvert \lambda \rvert - 1,q\}\otimes_{\Qp}L\cdot \lvert \nu_1 \rvert_p^{\frac{1-\lvert \lambda \rvert}{2}}.
$$
En particulier, $\BV_{\lambda}$ est un $L$-oper de poids $(\lambda_1,\lambda_2-1)$
\end{prop}
\medskip
On peut ainsi définir \emph{petit complexe de de Rham}\footnote{Schneider-Stuhler l'appellent le complexe de de Rham réduit.} (\cf le numéro \ref{subsub:lepetitcom}) qui prend ici la forme:
$$
\rmR\Gamma_{\!\Op}\Harg{\presp{\rmM}_{\Qp}^n}{\CE_{\lambda}}\coloneqq \CO^n\{1-w(\lambda),1-\lambda_2\}\otimes_{\Qp}L\cdot\lvert \nu_1 \rvert_p^{\frac{1-\lvert \lambda \rvert}{2}}\xrightarrow{\partial^{w(\lambda)}}\CO^n\{1+w(\lambda),1-\lambda_1\}\otimes_{\Qp}L\cdot \lvert \nu_1 \rvert_p^{\frac{1-\lvert \lambda \rvert}{2}}
$$
où $\partial=u^+\in \fkg$ opère sur $h\in\CO^n$ par $u^+(h)=h'$, où $h'$ est défini par $d(h) = h'dz$. La proposition \ref{prop:drtoop} nous assure que ce complexe est quasi-isomorphe au complexe de de Rham. Plus précisément on a un quasi-isomorphisme $\rmR\Gamma_{\!\dR}\Harg{\presp{\rmM}_{\Qp}^n}{\CE_{\lambda}}\cong \rmR\Gamma_{\!\Op}\Harg{\presp{\rmM}_{\Qp}^n}{\CE_{\lambda}}$ qui est $G\times \czG$-équivariant.
\subsubsection{Diagramme fondamental pour l'espace de Drinfeld}\label{subsubsec:diagfond}
Pour réduire la taille du diagramme, on introduit quelques notations supplémentaires. Pour $\lambda \in P_+$ on note 
\begin{equation}\label{eq:notab}
\begin{gathered}
\CO_{\lambda}^{\infty}\coloneqq \varinjlim_n\CO_{\lambda}^n\quad \CO_{\lambda}^n\coloneqq \CO^n\{1-w(\lambda),1-\lambda_2\}\otimes_{\Qp}L\cdot\lvert \nu \rvert_p^{\frac{1-\lvert \lambda \rvert}{2}},\\ 
\omega_{\lambda}^{\infty}\coloneqq \varinjlim_n \omega_{\lambda}^n\quad  \omega_{\lambda}^n\coloneqq \CO^n\{1+w(\lambda),1-\lambda_1\}\otimes_{\Qp}L\cdot\lvert \nu \rvert_p^{\frac{1-\lvert \lambda \rvert}{2}},\\
{W}_{\lambda}^{\pi_0^{\infty}} \coloneqq \varinjlim_n {W}_{\lambda}^{\pi_0^{n}}\quad \pi_0^{n}\coloneqq \pi_0(\presp{\rmM}^n_{C}).
\end{gathered}
\end{equation}
On note de plus
$$
\rmB_{\lambda}\coloneqq t^{\lambda_1}\Bdrp/t^{\lambda_2}\Bdrp,\quad \rmU_{\lambda}^0\coloneqq t^{\lambda_1}\intn{\Bcrisp}^{\varphi^2=p^{w(\lambda)+1}},\quad \rmU_{\lambda}^1\coloneqq t^{\lambda_1}\intn{\Bcrisp}^{\varphi^2=p^{w(\lambda)-1}}.
$$
On applique maintenant le diagramme fondamental et le théorème de comparaison, soit la proposition \ref{prop:diagfondop} et le théorème \ref{thm:petsyn} à la tour d'espaces Stein $\presp{\rmM}_{\Qp}^{\infty}$ ; on utilise aussi les notations de (\ref{eq:notanul}).
\medskip
\begin{theo}\label{theo:fondrin}
Soit $\lambda\in P_+$. On a un diagramme commutatif, $\BG$-équivariant d'espaces de Fréchet, dont les lignes sont exactes
 {
\begin{center} 
\begin{tikzcd}[column sep=small]
&\rmU_{\lambda}^0\otimes_{\Qpd}{W}_{\lambda}^{\pi_0^{\infty}} \ar[r]\ar[d]&\rmB_{\lambda}\wotimes_{\Qp}\CO_{\lambda}^{\infty}\ar[r]\ar[d,equal]& \presp{\rmH}^{\lambda}_{\pet}\ar[r]\ar[d]&t^{\lambda_1}X^1_{\st}\intn{\presp{\rmH}^{\lambda}_{\HK}[\lambda_1]}\ar[r]\ar[d,hook] &0\\
0\ar[r]&\rmB_{\lambda}\otimes_{\Qp}{W}_{\lambda}^{\pi_0^{\infty}}\ar[r]&\rmB_{\lambda}\wotimes_{\Qp}\CO_{\lambda}^{\infty}\ar[r] & \rmB_{\lambda}\wotimes_{\Qp}\omega_{\lambda}^{\infty}\ar[r]& \rmB_{\lambda}\wotimes_{\Qp}\presp{\rmH}^{\lambda}_{\dR}\ar[r]& 0.
\end{tikzcd}
\end{center}}
De plus, les flèches verticales sont d'images fermés et la première flèche verticale (resp. horizontale) est injective si et seulement si $w(\lambda)\neq1$.
\end{theo}
\medskip
\begin{rema}\label{rem:inclsymp}
Si $w(\lambda)\neq 1$, l'inclusion $\rmU_{\lambda}^0\otimes_{\Qpd}L\rightarrow \rmB_{\lambda}\otimes_{\Qp}L$ peut surprendre mais elle a une interprétation en terme de théorie de Hodge $p$-adique. Donnons là pour $\lambda = \rho \coloneqq (0,2)$. Reprenons les notations de la remarque \ref{rem:phdgint}, alors 
$$
{\Bcrisp}^{\varphi^2=p^3} =(\brv{\bN}\otimes_{\Qpbr}{\Bcrisp})^{\varphi=p^2}
$$
De plus, on définit $\bN_{\dR}\cong (\brv{\bN}\otimes_{\Qpbr}C)^{\GQp}$ qui est un $\Qp$-espace vectoriel de dimension $2$. Ainsi, on a inclusion ${\Bcrisp}^{\varphi^2=p^3}\incl \bN_{\dR}\otimes_{\Qp} \Bdrp$, mais comme ${\Bcrisp}^{\varphi^2=p^3}\cap(\bN_{\dR}\otimes_{\Qp} t^2\Bdrp) =0$, on obtient une inclusion ${\Bcrisp}^{\varphi^2=p^3}\incl \bN_{\dR}\otimes_{\Qp} \Bdrp/t^2$. Finalement, on peut identifier $\bN_{\dR}\cong \Qpd$ pour obtenir une inclusion
$$
{\Bcrisp}^{\varphi^2=p^3}\otimes_{\Qpd}L\incl \Bdrp/t^2\otimes_{\Qp} L.
$$
\end{rema}

\section{Cohomologie étale isotriviale de $\presp{\rmM}_{C}^{\infty}$ et vecteurs bornés}
Dans cette sectio on montre que le sous-espace des vecteurs $G$-bornés de la cohomologie proétale isotriviale de la tour de Drinfeld coïncide avec la cohomologie étale isotrivial. Rappelons qu'un vecteur $v\in \Pi$ d'une $L$-représentation $\Pi$ de $G$ est dit \emph{$G$-borné} si l'ensemble $\{g\cdot v\}_{g\in G}$ est borné. On dit qu'un sous-ensemble $X$ d'un $L$-espace vectoriel topologique est \emph{borné} si pour toute suite $(x_n)_{n\in \BN}$ d'éléments de $X$ on a $p^nx_n\xrightarrow{n\rightarrow+\infty}0$. Dans un espace Fréchet, où la topologie est définie par une famille de valuations $(v_k)_{k\in\BN}$, ceci équivaut à l'existence, pour tout $k\in\BN$ d'un $N_k\in\BZ$ tel que $v_k(x)\geqslant N_k$ pour tout $x\in X$. On note $\Pi^{G\text{-}\rmb}\subset \Pi$ le sous-espace des vecteurs $G$-bornés de $\Pi$, qui définit une sous-$L$-représentation de $\Pi$.
\medskip
\begin{theo}\label{thm:proetoet}
L'application $\rmH^1_{\et}\Harg{\presp{\rmM}_{C}^{\infty}}{\Sym\BV(1)}\rightarrow \rmH^1_{\pet}\Harg{\presp{\rmM}_{C}^{\infty}}{\Sym\BV(1)}$ est injective et induit un isomorphisme
$$
\rmH^1_{\pet}\Harg{\presp{\rmM}_{C}^{\infty}}{\Sym\BV(1)}^{G\text{-}{{\rmb}}}\cong\rmH^1_{\et}\Harg{\presp{\rmM}_{C}^{\infty}}{\Sym\BV(1)}.
$$
\end{theo}
Rappelons que pour $k\in \BN$ on note $\BV_k\coloneqq \Sym_L^k\BV$ qui admet un réseau $\BV_k^+\coloneqq \Sym_L^k\BV^+$. La preuve se fait en deux temps. On peut alors se restreindre à démontrer le théorème \ref{thm:proetoet} en niveau $n\in \BN$ pour les $\BV_k$. Dans un premier temps, on montre que $\rmH^1_{\et}\Harg{\presp{\rmM}_{C}^{n}}{\BV^+_{k}(1)}$ est de torsion bornée, donc qu'il définit un réseau invariant dans $\rmH^1_{\et}\Harg{\presp{\rmM}_{C}^{n}}{\BV_{k}(1)}$. Ceci montre une inclusion et on conclut alors en expliquant que l'argument de \cite[Proposition 2.12]{codoni} s'adapte sans difficultés. 
\Subsection{Annulations}
Dans ce numéro, on montre l'annulation de différents groupes de cohomologie proétale. Plus précisément, on montre le résultat suivant :
\medskip
\begin{theo}\label{thm:anulco}
Soit $\lambda\in P_+$. Si $i\geqslant 2$ alors 
$$
\rmH^i_{\pet}\Harg{\presp{\rmM}_{C}^{\infty}}{\BV_{\lambda}(1)}=0.
$$
De plus,
$$
\rmH^0_{\pet}\Harg{\presp{\rmM}_{C}^{\infty}}{\BV_{\lambda}(1)}=\rmH^0_{\et}\Harg{\presp{\rmM}_{C}^{\infty}}{\BV_{\lambda}(1)}= 
\begin{cases}
0 & \text{ si $w(\lambda)\neq 1$,}\\
\Qp(\lambda_1+1)^{\pi_0^{\infty}} & \text{ si $w(\lambda)=1$.}
\end{cases}
$$
\end{theo}
On va montrer ce théorème pour la cohomologie syntomique et utilise le théorème de comparaison proétale-syntomique (\cf le théorème \ref{thm:petsyn}) pour conclure.
\subsubsection{Annulation de $\rmH^0_{\et}$}
On commence par montrer la seconde partie du théorème \ref{thm:anulco}.
\medskip
\begin{prop}\label{prop:anulco0}
Soit $n\geqslant 0$ et $\lambda\in P_+$. On a 
$$
\rmH^0_{\et}\Harg{\presp{\rmM}_{C}^{n}}{\BV_{\lambda}(1)}=\rmH^0_{\pet}\Harg{\presp{\rmM}_{C}^{n}}{\BV_{\lambda}(1)}=
\begin{cases}
0 & \text{ si $w(\lambda)\neq1$,}\\
\Qp(\lambda_1+1)^{\pi_0^n} & \text{ si $w(\lambda)=1$.}
\end{cases}
$$
\end{prop}
\medskip
\begin{proof}
Si $w(\lambda)=1$, alors $\BV_{\lambda}=\Qp(\lambda_1)$ et le résultat est claire donc on suppose que $w(\lambda)\neq 1$. On sait, par définition, que $\rmH^0_{\syn}\Harg{\presp{\rmM}_{C}^{n}}{\BV_{\lambda}(1)}$ est le noyau de l'application 
$$
\left ( \rmH^0_{\HK}\Harg{\presp{\rmM}_{C}^{n}}{\BD_{\lambda}}\wotimes_{\brv{\BQ}_p}\Bstp\right ) ^{\varphi=p}\rightarrow \rmH^0\DR\left ( \presp{\rmM}_{C}^{n},\CE_{\lambda}\right ).
$$
Or, d'après le théorème \ref{theo:fondrin}, le noyau de cette application est le même que celui de 
$$
\rmU_{\lambda}^0\otimes_{\Qpd}{W}_{\lambda}^{\pi_0^n} \rightarrow\rmB_{\lambda}\wotimes_{\Qp} \CO_{\lambda}^n,
$$
qui est injective (\cf la remarque \ref{rem:inclsymp}) puisque $w(\lambda)\neq 1$. D'après le théorème \ref{thm:petsyn}
$$
\rmH^0_{\pet}\Harg{\presp{\rmM}_{C}^{n}}{\BV_{\lambda}(1)}\cong \rmH^0_{\syn}\Harg{\presp{\rmM}_{C}^{n}}{\BV_{\lambda}(1)}= 0.
$$
\end{proof}
\begin{rema}
On peut proposer une autre preuve de ce résultat en utilisant le groupe fondamental étale. Pour un point géométrique $x\colon \Spa(C,C^+)\rightarrow Y$, on sait que $\BV(x)$ correspond à la la $L$-représentation $\Sym_L^{w(\lambda)-1}$ de $\pi_1^{\et}\Harg{\presp{\rmM}_{C}^{n}}{x}\rightarrow \czG^1(n)\subset \czG$. Cette représentation est triviale si et seulement si $w(\lambda)=1$ et n'admet des invariants sous $\czG^1(n)$ que dans ce cas là.
\end{rema}
\subsubsection{Annulation de $\rmH^2_{\pet}$}
\medskip
\begin{prop}\label{prop:anulco}
Soit $n\geqslant 0$ un entier,
$$
\rmH^2_{\pet}\Harg{\presp{\rmM}_{C}^{n}}{\Sym\BV(1)}=0.
$$
\end{prop}
\medskip
\begin{proof}
Soit $\lambda\in P_+$. Il suffit de montrer que $\rmH^2_{\pet}\Harg{\presp{\rmM}_{C}^{n}}{\BV_{\lambda}(2)}$ est nul. En vertu du théorème \ref{thm:petsyn}), il suffit de montrer que 
$$
\rmH^2_{\syn}\Harg{\presp{\rmM}_{C}^{n}}{\BV_{\lambda}(2)}=0.
$$
La définition de la cohomologie syntomique nous donne une suite exacte longue
$$
\rmH^1\HK(\presp{\rmM}_{C}^{n},\BD_{\lambda})\rightarrow\rmH^1\DR(\presp{\rmM}_{C}^{n},\CE_{\lambda})\rightarrow \rmH^2_{\syn}\Harg{\presp{\rmM}_{C}^{n}}{\BV_{\lambda}(1)}\rightarrow \rmH^2\HK(\presp{\rmM}_{C}^{n},{{\BD}_{\lambda}}).
$$
Premièrement, comme $\presp{\rmM}_{\Qp}^{n}$ est une courbe Stein, on a $\rmH^2_{\HK}(\presp{\rmM}_{\Qp}^{n})=0$ et donc 
$$
\rmH^2\HK(\presp{\rmM}_{C}^{n},\BD_{\lambda})=0.
$$
Ainsi, il suffit de montrer que l'application $\rmH^1\HK(\presp{\rmM}_{\Qp}^{n},\BD_{\lambda})\rightarrow\rmH^2\DR(\presp{\rmM}_{C}^{n},\CE_{\lambda})$ est surjective. Or, le corollaire \ref{lem:calch0} et le lemme \ref{lem:steinhk} nous donnent
$$
\begin{gathered}
\rmH^1\DR(\presp{\rmM}_{\Qp}^{n},\CE_{\lambda}) \cong\rmH^1_{\dR}\Harg{\rmM_{\Qp}^n}{\CE_{\lambda}}\wotimes_{\Qp}\rmB_{\lambda_1}\\
\rmH^1\HK(\presp{\rmM}_{C}^{n},\BD_{\lambda}) \cong \intn{\rmH^1_{\HK}\Harg{\presp{\rmM}_{C}^{n}}{\BD_{\lambda}}\wotimes_{\Qpbr}\Bstp}^{\varphi=p,N=0}.
\end{gathered}
$$
Puisque les pentes de $\rmH^1_{\HK}\Harg{\presp{\rmM}_{C}^{n}}{\BD_{\lambda}}$ sont $\leqslant -\lambda_1+1$, la remarque \ref{rem:suitexho} nous assure que
$$
 \intn{\rmH^1_{\HK}\Harg{\presp{\rmM}_{C}^{n}}{\BD_{\lambda}}\wotimes_{\Qpbr}\Bstp}^{\varphi=p,N=0}\rightarrow \rmH^1_{\dR}\Harg{\rmM_{\Qp}^n}{\CE_{\lambda}}\wotimes_{\Qp}\rmB_{\lambda_1}
$$
est surjective, ce qui permet de conclure la preuve.
\end{proof}
\begin{rema}
Comme pour la cohomologie à coefficients triviaux, il n'est pas clair du tout que $\rmH^2_{\et}\Harg{\presp{\rmM}_{C}^{n}}{\BV_{\lambda}(1)}=0$.
\end{rema}
\medskip
Le théorème de comparaison syntomique-proétale assure que, pour $i\geqslant 2$, on a 
$$
\rmH^i_{\pet}\Harg{\presp{\rmM}_{C}^{n}}{\BV_{\lambda}(i)}\cong\rmH^i_{\syn}\Harg{\presp{\rmM}_{C}^{n}}{\BV_{\lambda}(i)}=0.
$$
En effet, les termes qui apparaissent dans la cohomologie syntomique sont nuls en degrés $\geqslant i$. Ainsi, quitte à tordre par une puissance du caractère cyclotomique, les groupes de cohomologie proétales supérieurs sont nuls\footnote{On peut aussi raisonner en utilisant que les affinoïdes sont de dimension cohomologique $1$ et donc que les espaces Stein sont de dimension cohomologique au plus $2$.}, ce qui permet d'énoncer le corollaire suivant, qui synthétise les résultats précédents :
\medskip
\begin{coro}\label{cor:anulco}
Soit $\lambda\in P_+$ tel que $w(\lambda)>1$. Alors $\rmH^i_{\pet}\Harg{\presp{\rmM}_{C}^{\infty}}{\BV_{\lambda}(1)}\neq 0$ si et seulement si $i=1$.
\end{coro}
\medskip
\Subsection{Torsion dans la cohomologie étale}
Dans ce numéro, on montre que la cohomologie étale du système local est de torsion bornée. Soit $n\geqslant 0$ et $k\geqslant 0$ des entier, rappelons que $\BV_k^+\coloneqq \Sym^k_{\rmO_L}\BV^+$. Il s'agit de montrer que $\rmH^1_{\et}\Harg{\rmM_{C}^{n}}{\BV_{k}^+}$ est de torsion bornée\footnote{Ceci est équivalent à ce que $\rmH^1_{\et}\Harg{\presp{\rmM}_{C}^{n}}{\BV_{k}^+(1)}$ soit de torsion borné puisque cet cet espace est la somme de deux copies de $\rmH^1_{\et}\Harg{\rmM_{C}^{n}}{\BV_{k}^+}$ par la $3$ de la remarque \ref{rem:concomp}. La raison de ce changement dans ce numéro est d'alléger la preuve de proposition \ref{prop:pprimbor}.}. Pour cela on va utiliser le résultat pour les coefficients triviaux, que l'on rappelle. 
\medskip
\begin{lemm}\label{lem:cohtor}
Soit $Y_C$ un espace adique sur $C$. Alors $\rmH_{\et}^1\Harg{Y_C}{\Zp(1)}$ est un $\Zp$-module sans torsion.
\end{lemm}
\medskip
\begin{proof}
On peut supposer que $Y_C$ est connexe puisque 
$$
\rmH_{\et}^1\Harg{Y_C}{\Zp(1)}=\prod_{X\in \pi_0(Y_C)}\rmH_{\et}^1\Harg{X}{\Zp(1)}.
$$
La preuve utilise la suite exacte de Kummer. Pour $n\geqslant 1$ un entier, la suite exacte longue associée à la suite exacte de faisceaux étales
$$
0 \rightarrow \BZ/p^n(1)\rightarrow \XBGm{Y_C}\xrightarrow{[p^n]}\XBGm{Y_C}\rightarrow 0
$$
fournit
$$
0\rightarrow (\CO/C)^{\times}\otimes_{\BZ}\BZ/p^n\rightarrow \rmH_{\et}^1\Harg{Y_C}{\BZ/p^n(1)}\rightarrow \Pic[p^n]\rightarrow 0,
$$
où $(\CO/C)^{\times}\coloneqq \CO(Y_C)^{\times}/C^{\times}$ est le groupe des fonctions inversibles modulo les constantes et $\Pic \coloneqq \rmH^1_{\et}\Harg{Y_C}{\XBGm{Y_C}}$ est le groupe de Picard (où plutôt ses $C$-points) de $Y_C$. En passant à la limite, comme le système de gauche satisfait Mittag-Leffler, on obtient une suite exacte courte
\begin{equation}\label{eq:kumexaseq}
0\rightarrow (\CO/C)^{\times}\wotimes_{\BZ}\Zp \rightarrow \rmH_{\et}^1\Harg{Y_C}{\Zp(1)}\rightarrow T_p\intn{\Pic}\rightarrow 0,
\end{equation}
où $(\CO/C)^{\times}\wotimes_{\BZ}\Zp$ est le complété $p$-adique du groupe $(\CO/C)^{\times}$ et $T_p\intn{\Pic}=\varprojlim_n\Pic[p^n]$ est le module de Tate de $\Pic$. Or, $(\CO/C)^{\times}$ est sans torsion et le complété $p$-adique d'un module sans torsion est sans torsion donc $(\CO/C)^{\times}\wotimes_{\BZ}\Zp$ est sans torsion. De plus, si $M$ est un groupe abélien alors $T_pM$ est sans torsion et donc $T_p\intn{\Pic}$ est sans torsion. Finalement, la suite exacte (\ref{eq:kumexaseq}) montre que $\rmH_{\et}^1\Harg{Y_C}{\Zp(1)}$ est sans torsion.
\end{proof}
Pour utiliser ce résultat on doit trivialiser le système local. Pour cela, on utilise le revêtement proétale 
$$
\wh{\rmM_{C}^{\infty}}\rightarrow \rmM_{C}^{n},
$$
dont le groupe de Galois est $\czG(n)=(1+\varpi_{\rmD}^n\rmOD)$ auquel on applique la suite spectrale de Hochschild-Serre.
\medskip
\begin{prop}\label{prop:pprimbor}
Soit $n\geqslant 0$ un entier. Le noyau de l'application naturelle 
$$
\rmH^1_{\et}\Harg{\presp{\rmM}_{C}^{n}}{\BV_{k}^+}\rightarrow \rmH^1_{\et}\Harg{\presp{\rmM}_{C}^{n}}{\BV_{k}},
$$
est de torsion $p$-primaire bornée, et l'image de l'application définit un réseau.
\end{prop}
\medskip
\begin{proof}
D'après la note de bas de page (\thefootnote), on peut remplacer $\presp{\rmM}_{C}^{n}$ par $\rmM_{C}^{n}$ dans l'énoncé. Par Hochschild-Serre appliqué au revêtement $\wh{\rmM_{C}^{\infty}}\rightarrow \rmM_{C}^n$ on a 
$$
\rmR\Gamma\Harg{\czG(n)}{\rmR\Gamma_{\!\et}\Harg{\wh{\rmM_{C}^{\infty}}}{\BV_{k}^+}}=\rmR\Gamma_{\!\et}\Harg{\rmM_{C}^n}{\BV_{k}^+}.
$$
Ceci donne une suite spectrale dont la suite des premiers termes est
$$
0\rightarrow \rmH^1\Harg{\czG(n)}{\BV_{k}^+(\wh{\rmM_{C}^{\infty}})}\rightarrow \rmH^1_{\et}\Harg{\rmM^n_{C}}{\BV_{k}^+}\rightarrow \rmH^1_{\et}\Harg{\wh{\rmM_{C}^{\infty}}}{\BV_{k}^+}^{\czG(n)}.
$$
Rappelons que le revêtement considéré trivialise le système local. Ainsi $\rmH^1_{\et}\Harg{\wh{\rmM_{C}^{\infty}}}{\BV_{k}^+}$ est sans torsion d'après le lemme \ref{lem:cohtor} et donc a fortiori $\rmH^1_{\et}\Harg{\wh{\rmM_{C}^{\infty}}}{\BV_{k}^+}^{\czG(n)}$ est sans torsion. Montrons maintenant que $\rmH^1\Harg{\czG(n)}{\BV_{k}^+(\rmM_{C}^{\infty})}$ est de torsion $p$-primaire et de type fini.

Rappelons (\cf  \ref{subsubsec:compco}) qu'on a une application $\wh{\rmM_{C}^{\infty}}\rightarrow \Zp^{\times}$, donnant les composantes connexes de la tour, comme l'action de $\czG$ sur $\rmM_{C}^{\infty}$ induit sur $\Zp^{\times}$ l'action de la norme réduite et donc~$\czG^1$ fixe les fibres de l'application $\rmM_{C}^{\infty}\rightarrow \Zp^{\times}$. De plus, on sait qu'en un point de $\rmM_{C}^{\infty}$, le germe du système local, trivial sur cet espace, est simplement $\Sym^k_{\Zp}\rmOD\otimes_{\Zp}\rmO_L$. On en déduit que $\BV_{k}^+(\wh{\rmM_{C}^{\infty}})=\Ind_{\czG^1}^{\czG^+}\intn{\Sym^k_{\Zp}\rmOD}\otimes_{\Zp}\rmO_L$ en tant que représentation de $\czG(n)\subset \rmODt$. Or, d'après la proposition \ref{prop:anulcog}, $\rmH^1\Harg{\czG(n)}{\Ind_{\czG^1}^{\czG}\intn{\Sym^k_{\Zp}\rmOD}}\otimes_{\Zp}\rmO_L$ est de torsion $p$-primaire bornée. Ceci achève la preuve. 
\end{proof}
\begin{rema}
 On pourrait se demander pourquoi le groupe $\czG^+$ intervient et non $\czG^1$, le noyau de la norme réduite puisque c'est le quotient naturel de $\pi_1^{\et}(\rmM_C^0,\bar x)$ donné par la tour. La raison est que le quotient $\pi_1^{\et}(\rmM_{C}^0,\bar x )\rightarrow \czG^1$ décrit une composante connexe de la tour ; c'est la situation géométrique qui ne voit pas les différentes composantes connexes. Plus précisément, on a un quotient $\pi_1^{\et}(\presp{\rmM}_{\Qp}^0,\bar x )\rightarrow \czG^+$ qui décrit la situation arithmétique qui intervient ici. 
\end{rema}
\Subsection{Vecteurs bornés de la cohomologie proétale isotriviale}
\subsubsection{Fin de la preuve}
On termine la démonstration du théorème \ref{thm:proetoet}. Montrons que tout vecteur $G$-borné de $\presp{\rmH}^1_{\pet}\coloneqq \rmH^1_{\pet}\Harg{\presp{\rmM}_{C}^{\infty}}{\Sym \BV(1)}$ appartient à  $\rmH^1_{\et}\Harg{\presp{\rmM}_{C}^{\infty}}{\Sym \BV(1)}$. Soit $n\geqslant 0$ un entier, la preuve se fait exactement comme celle de \cite[Proposition 2.12]{codoni} (\cf \cite{van} pour plus de détails): on choisit $\Gamma\subset \PGL_2(\Qp)$ cocompact puis à l'aide d'un recouvrement adapté et de la suite exacte de Čech on construit un affinoïde $Y\subset \presp{\rmM}_C^{n}$ tel que
{ 
\begin{equation}\label{eq:defrezo}
\rmH^1_{\et}\Harg{\presp{\rmM}^n_{C}}{\Sym{\BV}(1)}=\varinjlim_{m}\left \{v\in \presp{\rmH}^1_{\pet}\mid \forall \gamma \in \Gamma,\ p^m\Res_{Y}(\gamma\cdot v)\in \rmH^1_{\et}\Harg{Y}{\Sym{\BV^+}(1)}\right \},
\end{equation}}
où $\Res$ est le morphisme en cohomologie de restriction à un affinoïde. En dehors de la structure de l'espace, le point clé de cet égalité est que la cohomologie étale est séparée ; ce qui est exactement ce qu'on a montré dans la proposition \ref{prop:pprimbor}.
\qed
\medskip
\begin{rema}
Notons que $\rmH^1_{\et}\Harg{\presp{\rmM}_{C}^{\infty}}{\Sym{\BV^+}(1)}$ n'est pas un réseau stable par $G$. Dans la preuve, on a défini un réseau dépendant d'un sous-groupe cocompact $\Gamma\subset \PGL_2(\Qp)$, d'un affinoïde $Y\subset \presp{\rmM}_C^{n}$ et d'un entier $m\in \BN$ par
$$
\left \{v\in \presp{\rmH}^1_{\pet}\mid \forall \gamma \in \Gamma,\ p^m\Res_{Y}(\gamma\cdot v)\in \rmH^1_{\et}\Harg{Y}{\Sym{\BV^+}}\right \}.
$$
Cette définition rappelle celle donnée par Breuil des réseaux pour la série spéciale dans \cite{bre}. Il est cependant possible de construire un réseau différement en remarquant que pour $\lambda\in P_+$
$$
\presp{\rmH}^{\lambda}_{\et}=\Hom_{\czG(n)}\intnn{W_{\lambda}}{\rmH^1_{\et}\Harg{\widehat{\presp{\rmM}^{\infty}_C}}{\Qp(1)}}.
$$
On choisit un réseau $W_{\lambda}^+\subset W_{\lambda}$ stable sous $\czG$ puis on définit un réseau $G$-stable de $\presp{\rmH}^{\lambda}_{\et}$ comme l'image de $\Hom_{\czG(n)}\intnn{W^+_{\lambda}}{\rmH^1_{\et}\Harg{\widehat{\presp{\rmM}^{\infty}_C}}{\Zp(1)}}$ par l'application naturelle, qui est simplement cet espace modulo sa torsion. Il est possible de montrer que l'on obtient des réseaux comesurables par ces deux définitions.
\end{rema}
\medskip
\subsubsection{Reformulation}
On reformule le théorème principal en termes des multiplicités des représentations de $G$. La preuve du corollaire suivant est la même que celle de \cite[Lemme 5.9]{codonifac}.
\medskip
\begin{coro}\label{cor:complocan}
Soient $\Pi$ une $L$-représentation de Banach unitaire de $G$. Alors les injections  $\rmH^1_{\et}\Harg{\presp{\rmM_{C}^{\infty}}}{\Sym\BV(1)}\incl \rmH^1_{\pet}\Harg{\presp{\rmM_{C}^{\infty}}}{\Sym\BV(1)}$ et $\Pi'\incl (\Pi^{\lan})'$ induisent un isomorphisme naturel
$$
\Hom_G\intnn{\Pi'}{\rmH^1_{\et}\Harg{\presp{\rmM_{C}^{\infty}}}{\Sym\BV(1)}}\cong \Hom_G\intnn{(\Pi^{\lan})'}{\rmH^1_{\pet}\Harg{\presp{\rmM_{C}^{\infty}}}{\Sym\BV(1)}}.
$$
\end{coro}
\Subsection{Résultat de finitude}
Dans ce paragraphe on démontre le résultat de finitude suivant :
\medskip
\begin{prop}\label{prop:longfinui}
Soient $k\geqslant 0$ et $n\geqslant 0$ des entiers et soit $K$ une extension finie de $\Qp$. L'espace de cohomologie étale suivant, 
$$
\rmH^1_{\et}\Harg{\presp{\rmM^n_{K}}}{\BV_{k}^+/\varpi_{D}},
$$
est le dual stéréotypique d'un $\rmO_L[G]$-module lisse de longueur finie.
\end{prop}
\medskip
\begin{proof}
À coefficients dans $k_L$, c'est le théorème $4.1$ de \cite{codonifac}. De plus, ce résultat nous dit que si $\BL$ est un $\BF_p$-système local étale constant, \ie  $\BL\cong \ud{M}$ où $M$ est un $\BF_p$-espace vectoriel de dimension finie, alors pour tout entier $n\geqslant 0$, l'espace de cohomologie étale
$$
\rmH^1_{\et}\Harg{\presp{\rmM}_{K}^n}{\BL}
$$
est le dual d'un $\Zp[G]$-module de longueur finie. On fixe maintenant $k\geqslant 0$ un entier. On sait que par définition, $\BV_{k}^+/\varpi_{D}\cong \Sym^k_{\BF_p}\CG[\varpi_{D}]\otimes_{\Zp}\rmO_L$ est un $\BF_p$-système local trivial sur $\presp{\rmM}_{K}^1$ et donc a fortiori sur $\presp{\rmM}_{K}^n$ dès que $n\geqslant 1$. 

Ainsi, on a montré le résultat dès que $n\geqslant 1$ et il reste à le démontrer pour $n=0$, ce que l'on va faire en appliquant la suite spectrale de Hochschild-Serre au revêtement $\presp{\rmM}^1_{K}\rightarrow\presp{\rmM}^0_{K}$. C'est un revêtement galoisien de groupe $\rmODt/(1+\varpi_{D}\rmOD)\cong \BF_{p^2}^{\times}$ et la suite spectrale s'écrit
$$
E_2^{i,j} \coloneqq \rmH^i\Harg{\BF_{p^2}^{\times}}{\rmH^j_{\et}\Harg{\presp{\rmM}_{K}^1}{\BV_{k}^+/\varpi_{D}}}\implies \rmH^{i+j}_{\et}\Harg{\presp{\rmM}_{K}^0}{\BV_{k}^+/\varpi_{D}}.
$$
Or, comme $\BF_{p^2}^{\times}$ est un groupe fini d'ordre premier à $p$, la cohomologie des $\BF_p$-représentations de $\BF_{p^2}^{\times}$ s'annule en degré $\geqslant 1$. Ainsi, on obtient un isomorphisme 
$$
\rmH^{1}_{\et}\Harg{\presp{\rmM}_{K}^0}{\BV_{k}^+/\varpi_{D}}\xrightarrow{\sim}\rmH^{1}_{\et}\Harg{\presp{\rmM}_{K}^1}{\BV_{k}^+/\varpi_{D}}^{\BF_{p^2}^{\times}}\subset\rmH^{1}_{\et}\Harg{\presp{\rmM}_{K}^1}{\BV_{k}^+/\varpi_{D}}.
$$
On a déjà montré que le terme de droite était le dual d'une représentation lisse de $G$, admissible et de longueur finie donc il en est de même pour le terme de gauche, puisque comme $\BF_{p^2}^{\times}$ agit continument, ce dernier est un sous-espace fermé.


\end{proof}
\begin{rema}
On peut déduire de ce théorème que pour tout entier $l\geqslant 0$, l'espace de cohomologie étale 
$$
\rmH^1_{\et}\Harg{\presp{\rmM}^n_{K}}{\BV_{k}^+/\varpi_{D}^l}
$$
est le dual stéréotypique d'un $\BZ_p[G]$-module lisse de longueur finie.
\end{rema}

\section{Le cas cuspidal}
On conserve les notations des sections précédentes.
\Subsection{Représentations spéciales et cuspidales}\label{subsec:rapcusp}
Commençons par énoncé en terme de $L$-$(\varphi,N,\GQp)$-modules les définitions de \emph{spéciale} et \emph{cuspidale}. Dans cette section, on s'intéresse surtout au cas cuspidal et on détaillera le cas spécial dans la prochaine (\cf \ref{sec:caspe}).
\medskip
\begin{defi}\label{def:mcuspispe}
Soit $M$ un $L$-$(\varphi,N,\GQp)$-module, libre sur $L\otimes_{\Qp}\Qp^{\nr}$ de rang $2$. Alors, on dit que $M$ est 
\begin{itemize}
\itemb \emph{cuspidale} si la représentation de Weil-Deligne associée $\WD_M$ est absolument irréductible,
\itemb \emph{spéciale} si $N\neq 0$.
\end{itemize}
\end{defi}
\medskip
Remarquons que $\WD_M$ est absolument indécomposable si et seulement si $M$ est spécial ou cuspidal.

Si $M$ est cuspidal, alors les pentes de $M$ sont toutes égales à un même nombre rationnel que l'on appelle \emph{la pente} de $M$. Si $M$ est spécial, alors les pentes diffèrent de $1$ et on appellera \emph{la pente} la moyenne de ses pentes. 

Pour $M$ un $L$-$(\varphi,N,\GQp)$-module, notons $M_{\dR}\coloneqq (M\otimes_{\Qp^{\nr}} C)^{\GQp}$, qui est un $L$-espace vectoriel de dimension $2$,  et pour $k\in \BN$,
\begin{equation}
X^+_{\st}(M)=(\Bstp\otimes_{\Qp^{\nr}}M)^{N=0,\varphi=1},\quad X^k_{\st}(M)=(\Bstp\otimes_{\Qp^{\nr}}M)^{N=0,\varphi=p^k},
\end{equation}
qui définissent des $L$-espaces de Banach-Colmez. On rappelle finalement que $M[k]$ est obtenu en multipliant le Frobenius de $M$ par $p^k$ et qu'ainsi $ X^{j+k}_{\st}(M[k])= X^j_{\st}(M)$ pour $j\in \BN$ assez grand.
\subsubsection{Représentations de $\GQp$}
Soit $V$ une $L$-représentation de $\GQp$ de Rham de dimension $2$ à poids de Hodge-Tate $\lambda\in P_+$. Dans ce cas, $V$ est \emph{cuspidale} (resp. \emph{spéciale}) si et seulement si $M\coloneqq \bD_{\pst}(V)[1]$ est cuspidal (resp. spécial) au sens de la définition \ref{def:mcuspispe}. Alors, $M$ est de pente $1-\lvert \lambda\rvert /2$ et d'après \cite{cofo} il existe $0\subsetneq\CL\subsetneq \bD_{\dR}(V)=M_{\dR}$ tel que
$$
V\cong 
V_{M,\CL}^{\lambda}\coloneqq \Ker\left (X_{\st}^{1}(M)\rightarrow \frac{M_{\dR}\otimes_{\Qp}\Bdrp}{\CL\otimes_{\Qp} t^{\lambda_1}\Bdrp+M_{\dR}\otimes_{\Qp} t^{\lambda_2}\Bdrp}\right ).
$$
Ainsi, $V$ est déterminé par le triplet $(M,\CL,\lambda)$ et si $(M,\CL,\lambda)\neq (M',\CL',\lambda')$ alors $V_{M,\CL}^{\lambda}\not \cong V_{M',\CL'}^{\lambda'}$. De plus, $V_{M,\CL}^{\lambda}$ est irréductible si et seulement si $M$ est cuspidal ou $M$ est spécial et $w(\lambda)\neq 1$.
\subsubsection{Correspondances localement algébriques}
Soit $M$ un $L$-$(\varphi,N,\GQp)$-module libre de rang $2$, absolument indécomposable, auquel on associe $\WD_M$ puis, par la correspondance de Langlands locale, $\LL_M\coloneqq\LL(\WD_M)$ une $L$-représentation lisse irréductible de $G$. Si $M$ est cuspidale, alors $\LL(\WD_M)$ est cuspidale au sens usuel. Pour $\lambda\in P_+$, posons (\cf \ref{subsubsec:raig}) 
$$
\LL_M^{\lambda}\coloneqq \LL_M\otimes_L \Sym_L^{w(\lambda)-1}\otimes_L {\det}^{\lambda_1},
$$
qui définit une $L$-représentation localement algébrique unitaire irréductible de $G$. De même, en notant $\JL_M\coloneqq \JL(\LL_M)$ la $L$-représentation lisse irréductible de $\czG$ associée à $\LL_M$ par la correspondance de Jacquet-Langlands, on définit une $L$-représentation localement algébrique, unitaire, irréductible de $\czG$ par (\cf \ref{subsubsec:raiczg})
$$
\JL_M^{\lambda}\coloneqq \JL_M\otimes_L\Sym_L^{w(\lambda)-1}\otimes_L \nrd^{\lambda_1}.
$$
Ces constructions permettent de définir des correspondances localement algébriques en associant à une $L$-représentation spéciale ou cuspidale de $\GQp$ de la forme $V=V_{M,\CL}^{\lambda}$ les $L$-représentations $\JL_M^{\lambda}$ et $\LL_M^{\lambda}$. Ces dernières perdent l'information du paramètre de la filtration. Remarquons que, dans ce cas, $M$ et $\lambda$ ne sont pas indépendants puisque $M$ est de pente $1-\lvert \lambda \rvert/2$.

On introduit la notion suivante, qui généralise \og $p$-compatible \fg dans \cite{codoni}, pour traiter les torsions par des caractères (\cf \ref{rema:tchaocar}) : 
\medskip
\begin{defi}
 Soit $k\geqslant$ un entier, on dit qu'un $L$-$(\varphi,N,\GQp)$-module libre de rang $2$, absolument indécomposable, est \emph{$p$-compatible de poids $k$} si $p\in\czG$ agit trivialement sur $\JL_M^{\lambda}$ pour $\lambda\in P_+$ tel que $\vert \lambda \rvert =k+1$. On note $\Phi N^{p}_{k}$ l'ensemble des $L$-$(\varphi,N,\GQp)$-modules $p$-compatible de poids $k$.
\end{defi}
Un élément de $\Phi N^{p}_{k}$ est en particulier de pente $-(k-1) /2$ et réciproquement, tout $L$-$(\varphi,N,\GQp)$-module absolument indécomposable est le tordu par un caractère lisse d'un élément de $\Phi N^{p}_{k}$.

Si $M$ est un $L$-$(\varphi,N,\GQp)$-module libre de rang $2$, absolument indécomposable on note $Z\mapsto Z[M]\coloneqq\Hom_{L[\czG]}\intnn{\JL_M}{Z}$ comme foncteur sur les $L[\czG]$-modules.
\subsubsection{Représentations de $G$}\label{subsubsec:repcuspg}
Soit $\Pi$ une $L$-représentation de Banach unitaire admissible et absolument indécomposable de $G$, alors, on dit que $\Pi$ est \emph{cuspidale} (resp. \emph{spéciale}) si, ses vecteurs localement algébriques ne sont pas triviaux et sont de la forme
$$
\Pi^{\lalg}\cong \LL_M^{\lambda},
$$
où $\lambda \in P_+$ et $M$ est un $L$-$(\varphi,N,\GQp)$-module cuspidal (resp. spéciale) de pente $1-\lvert \lambda \rvert /2$.

La correspondance de Langlands $p$-adique (\cf \cite{colp}, \cite{codopa}) définit un $L$-Banach cuspidal ou spécial (sous entendu unitaire admissible absolument indécomposable) de $G$ à partir d'un triplet $(M,\CL,\lambda)$ avec $M$ cuspidal (resp. spécial) et de pente $1-\lvert \lambda \rvert /2$ par
$$
\Pi_{M,\CL}^{\lambda}\coloneqq \bPi(V_{M,\CL}^{\lambda}).
$$
La compatibilité entre la correspondance de Langlands locale et $p$-adique (\cf \cite{colp}) assure que les vecteurs localement algébriques de $\Pi_{M,\CL}^{\lambda}$ sont donnés par
\begin{equation}\label{eq:compatlocp}
 \prescript{\lalg}{}{\Pi}_{M,\CL}^{\lambda}\cong \LL_M^{\lambda},
\end{equation}
ce qui fait de $\Pi_{M,\CL}^{\lambda}$ un complété unitaire admissible de $\LL_M^{\lambda}$. De plus, $\Pi_{M,\CL}^{\lambda}$ est irréductible sauf si $M$ est spécial et $w(\lambda)=1$. 

Enfin, on note $\Pilan_{M,\CL}^{\lambda}\subset \Pi_{M,\CL}^{\lambda}$ le sous-espace des vecteurs localement analytiques, qui est un sous-espace de type $LB$, donc son dual $\intn{\Pilan_{M,\CL}^{\lambda}}'$ est un espace de Fréchet. Alors, d'après \cite{codo}, $\Pi_{M,\CL}^{\lambda}$ est le complété unitaire universel de $\Pilan_{M,\CL}^{\lambda}$ et $(\Pi_{M,\CL}^{\lambda})'$ est le sous-espace des vecteurs $G$-bornés de $(\Pilan_{M,\CL}^{\lambda})'$.

\Subsection{Cohomologie de Hyodo-Kato et de de Rham de la tour}

On calcule la cohomologie de Hyodo-Kato et de de Rham de la tour de Drinfeld à coefficients isotriviaux. Ces cohomologies s'expriment simplement en termes des cohomologies à coefficients constants et on obtient la proposition suivante :
\medskip
\begin{prop}\label{prop:drhkdr}
Soit $\lambda\in P_+$, $M$ un $L$-$(\varphi,N,\GQp)$-module cuspidal. On a un carré commutatif $\sW_{\Qp}\times G$-équivariant de $L$-espaces de Fréchet
\begin{center}
\begin{tikzcd}
\Hom_{L[\czG]}\intnn{\JL_M^{\lambda}}{\rmH^1_{\HK}\Harg{\brv{\rmM}_{C}^{\infty}}{\Sym\BD}}\ar[r,"\sim"]\ar[d,"\iota_{\HK}"]& (\Qpbr\otimes_{\Qp^{nr}}M)\wotimes_L{\LL_M^{\lambda}}'\ar[d]\\
\Hom_{L[\czG]}\intnn{\JL_M^{\lambda}}{\rmH^1_{\dR}\Harg{\brv{\rmM}_{C}^{\infty}}{\Sym \CE}}\ar[r,"\sim"]& (C\otimes_{\Qp}M_{\dR})\wotimes_L{\LL_M^{\lambda}}'
\end{tikzcd}
\end{center}
\end{prop}
\medskip
Avant de commencer la preuve de cette proposition, commençons par une remarque pour justifier que l'on peut remplacer $\brv{\rmM}_C^{\infty}$ par $\presp{\rmM}_C^{\infty}$ dans la preuve. Cette remarque est déjà présente dans \cite[5.1]{codoni}.
\begin{rema}\label{rema:tchaocar}
Soit $\star \in \{\dR,\HK\}$. Notons que pour $M\in \Phi N^{p}_{\lvert \lambda \rvert}$,  en vertu du lemme \ref{lem:fkcompcon} et avec les mêmes notations, on a un isomorphisme de $L$-représentations de $\sW_{\Qp}\times G$ :
$$
\Hom_{L[\czG]}\Harg{\JL_M^{\lambda}}{\brv{\rmH}_{\star}}\cong\Hom_{L[\czG]}\Harg{\JL_M^{\lambda}}{\pbrv{\rmH}_{\star}}.
$$
Pour $M$ un $L$-$(\varphi,N,\GQp)$-module cuspidal, il existe un caractère $\chi \colon \Qpt \rightarrow L$ non-ramifié, vu comme caractère du groupe de Weil, tel que $M\otimes \chi\in \Phi N_{\lvert \lambda \rvert}^p$. La compatibilité à la torsions dans la correspondance de Langlands locale et le lemme \ref{lem:fkcompcon} assurent alors que
$$
\Hom_{L[\czG]}\Harg{\JL_M^{\lambda}}{\brv{\rmH}_{\star}}\cong\Hom_{L[\czG]}\Harg{\JL_M^{\lambda}\otimes (\chi\circ \nrd)}{\pbrv{\rmH}_{\star}}\otimes (\chi\circ \nu_3).
$$
où $\nu_3$ est la restriction de $\nu$ à $\sW_{\Qp}\times G$.
\end{rema}
\medskip
\begin{proof}
Notons que ce résultat est l'équivalent à coefficient de \cite[Théorème 0.4]{codoni}, que l'on va utiliser. On montre l'isomorphisme pour la cohomologie de Hyodo-Kato, l'argument pour la cohomologie de de Rham est exactement le même. Comme au numéro \ref{subsec:decompcoh}, notons
$$
\pbrv{\rmH}_{\HK}^{1}\coloneqq \rmH^1_{\HK}\Harg{\brv{\rmM}_{C}^{\infty}}{\Sym\BD}.
$$
Soit $\lambda\in P_+$ et $M\in \Phi N^{p}_{\lvert \lambda \rvert}$, on veut calculer 
$$
\rmH_M^{\lambda}\coloneqq\Hom_{L[\czG]}\intnn{\JL_M^{\lambda}}{\pbrv{\rmH}_{\HK}^{1}}
$$
D'après le corollaire \ref{coro:decomppoid},
$$
\pbrv{\rmH}_{\HK}^{1}=\bigoplus_{\lambda'\in P_+}\pbrv{\rmH}_{\HK}^{\lambda'}\otimes_L \cz{W}_{\lambda'},\quad \pbrv{\rmH}_{\HK}^{\lambda'}\coloneqq \rmH^1_{\HK}\Harg{\presp{\rmM}_{C}^{\infty}}{\BD_{\lambda}}.
$$
Notons $s\coloneqq \frac{\lvert \lambda \rvert-1}{2}$. Pour $\lambda'\in P_+$, on a
$$
\Hom_{L[\czG]}\intnn{\JL_M^{\lambda}}{\pbrv{\rmH}_{\HK}^{\lambda'}\otimes \cz{W}_{\lambda'}}=\big (\underbrace{\pbrv{\rmH}_{\HK}^{\lambda'}\otimes_L\JL_M'\otimes \lvert \nrd\rvert_p^{s}}_{\text{lisse}}\otimes_L \underbrace{\cz{W}_{\lambda}^*\otimes_L\cz{W}_{\lambda'}}_{\text{loc. algébrique}}\big )^{\czG}
$$
Notons que l'action de $\czG$ sur $\pbrv{\rmH}_{\HK}^{\lambda'}$ est bien lisse puisque c'est une colimite de $L$-représentations de $\czG$ sur lesquels l'action de $\czG$ se factorise par un quotient fini. Or, le terme de terme de droite est en particulier contenu dans les éléments tués par $\cz{\fkg}$ l'algèbre de Lie de $\czG$. Comme la partie lisse est annulée par $\cz{\fkg}$, on calcule à l'aide du lemme de Schur :
$$
(\cz{W}_{\lambda}^*\otimes_L\cz{W}_{\lambda'})^{\cz{\fkg}}=\Hom_{\cz{\fkg}}\intnn{\cz{W}_{\lambda}}{\cz{W}_{\lambda'}}=
\begin{cases}
L & \text{ si } \lambda = \lambda'\\
0  & \text{ sinon}.
\end{cases}
$$
Comme $\JL_M'\otimes \lvert \nrd\rvert_p^{s}=\JL_{M[s]}'$, on en déduit que $\rmH_M^{\lambda}=\Hom_{\czG}\intnn{\JL_{M[s]}}{\pbrv{\rmH}_{\HK}^{\lambda}}$. Si $\lambda_0\coloneqq (0,1)$, remarquons que 
$$
\pbrv{\rmH}_{\HK}^{\lambda}=\pbrv{\rmH}_{\HK}^{\lambda_0}[-s]\otimes_L W_{\lambda}
$$
par la proposition \ref{prop:calcdec}. Mais par \cite[Théorème 0.4]{codoni}, comme $M[s]$ est de pente $1-\tfrac{\lvert \lambda \rvert}{2}+s=\tfrac 1 2$ on obtient 
$$
\Hom_{\czG}\intnn{\JL_{M[s]}}{\pbrv{\rmH}_{\HK}^{\lambda_0}}=(\Qpbr\otimes_{\Qp^{\nr}}M[s])\wotimes_L\LL_{M[s]}.
$$
Finalement,
$$
\pbrv{\rmH}_{\HK}^{\lambda}=(\Qpbr\otimes_{\Qp^{\nr}}M)\wotimes_L\LL_{M[s]}'\otimes W_{\lambda}=(\Qpbr\otimes_{\Qp^{\nr}}M)\wotimes_L{\LL_{M}^{\lambda}}',
$$
puisque $\LL_{M[s]}\otimes_L W_{\lambda}^*=\LL_{M}\otimes_L \lvert \det \rvert_p^{-s}\otimes_L W_{\lambda}^*=\LL_M^{\lambda}$.
\end{proof}
Une autre formulation du résultat précédent dans le cas de Rham est la décomposition de la cohomologie de de Rham à support compact de la tour ; on omet la preuve qui est {\it mutatis mutandis} la même que celle de \cite[Théorème 5.8]{codoni}. 
\medskip
\begin{coro}
On a un isomorphisme de $L$-représentation localement algébrique de $\BG$
$$
\rmH^1_{\dR,c}\Harg{\presp{\rmM_{C}^{\infty}}}{\Sym \CE}=\bigoplus_{\lambda\in P_+}\bigoplus_{M\in\Phi N_{\lvert \lambda \rvert}^{p}}C\otimes_{\Qp}M\otimes_L{\LL_M^{\lambda}}^{\vee}\otimes_L\JL_M^{\lambda},
$$
où ${\LL_M^{\lambda}}^{\vee}$ est le $L$-dual localement algébrique de ${\LL_M^{\lambda}}$.
\end{coro}
\medskip

\Subsection{La conjecture de Breuil-Strauch}
Le but de ce numéro est d'expliquer comment étendre la conjecture de Breuil-Strauch en poids supérieur en suivant \cite{colw} à partir de \cite{dolb}. On commence par quelques rappels sur le changement de poids.
\subsubsection{Changement de poids}
Soit $M$ un $L$-$(\varphi,N,\GQp )$-module de rang $2$. Dans \cite{colw}, Colmez associe à $M$ une représentation localement analytique de $G$ notée $\Pi(M)$ et décrit comment on peut en \og changer les poids \fg, \ie  définir des représentations que l'on notera $\wt{\Pi}_M^{\lambda}$ et ${\Pi}_M^{\lambda}$. On rappelle brièvement la construction de $\Pi(M,k)$ à partir de $\Pi(M)$. Dans \cite{colw}, Colmez définit un isomorphisme d'espaces vectoriels topologiques
$$
\partial \colon \Pi(M)\rightarrow \Pi(M),
$$
tel que pour tout couple $(a,c)\in \Qp^2\setminus\{(0,0)\}$, l'application $(a-c\partial)\colon \Pi(M)\rightarrow \Pi(M)$ est un isomorphisme. Ceci permet, pour tout $k\in\BZ$, de définir une nouvelle action linéaire de $G$ sur $\Pi(M)$ en posant 
$$
\begin{pmatrix}
 a&b\\
 c&d
\end{pmatrix}
\cdot_k v = (a-c\partial)^{k}
\left ( 
\begin{pmatrix}
 a&b\\
 c&d
\end{pmatrix}
\cdot v
\right ).
$$
On note $\Pi(M,k)$ la $L$-représentation de $G$ ainsi obtenue. Finalement, pour $\lambda\in P_+$, on pose 
$$
\wt{\Pi}_M^{\lambda}\coloneqq \Pi(M,-w(\lambda))\otimes_L {\det}^{\lambda_2},\quad \Pi_M^{\lambda}\coloneqq \Pi(M,w(\lambda))\otimes_L {\det}^{\lambda_1}.
$$
On rappelle la proposition suivante\footnote{Les conventions sont les duales des conventions des références. On suit plutôt la convention de \cite{codoni}.} (\cf  \cite[Théorème 0.6, Théorème 0.3, Corollaire 0.4]{colw}) : 
\medskip
\begin{prop}\label{prop:weighchange}
Soit $\lambda \in P_+$ et soit $M$ un $L$-$(\varphi,N,\GQp )$-module cuspidal de pente $1-\lvert \lambda \rvert /2$. Alors, $\prescript{\lalg}{}{\Pi}_{M}^{\lambda}\cong M_{\dR}^*\otimes_L \LL_M^{\lambda}$ et on a un diagramme commutatif, à lignes exactes, de $L$-représentations localement analytiques
\begin{center}
\begin{tikzcd}
0\ar[r]&\CL^{\bot}\otimes_L \LL_M^{\lambda}\ar[r]\ar[d,hook]&\Pi_M^{\lambda}\ar[r]\ar[d,equal]&\Pilan_{M,\CL}^{\lambda} \ar[r]\ar[d,two heads]& 0\\
0\ar[r]&\prescript{\lalg}{}{\Pi}_{M}^{\lambda}\ar[r]&\Pi_M^{\lambda}\ar[r]&\wt{\Pi}_M^{\lambda}\ar[r]&0.
\end{tikzcd}
\end{center}
De plus, On a $\prescript{\lalg}{}{\Pi}_{M,\CL}^{\lambda}\cong (M_{\dR}^*/\CL^{\bot})\otimes_L \LL_M^{\lambda}$ et $\Pilan_{M,\CL}^{\lambda}$ est de longueur $2$.
\end{prop}
\medskip
Rappelons (\cf \ref{subsub:operdedrin}) que l'on a défini $\rmR\Gamma_{\!\Op}\Harg{\presp{\rmM}_{\Qp}^{\infty}}{\CE_{\lambda}}$ qui est quasi-isomorphe à $\rmR\Gamma_{\!\dR}\Harg{\presp{\rmM}_{\Qp}^{\infty}}{\CE_{\lambda}}$. Ainsi, on obtient une suite exacte courte
\begin{equation}\label{eq:extdR}
0\rightarrow \CO_{\lambda}^{\infty}/W_{\lambda}\rightarrow \omega^{\infty}_{\lambda}\rightarrow\rmH^1_{\dR}\Harg{\presp{\rmM}_{\Qp}^{\infty}}{\CE_{\lambda}}\rightarrow 0.
\end{equation}
On a le corollaire suivant (\cf \cite[Conjecture 0.10]{colw} qui est en réalité un théorème par \cite[Remarque 0.11]{colw} ; c'est une conséquence directe du \cite[Théorème 1.4]{dolb}, de la définition des $\Pi(M,k)$ et de \cite[Théorème 1.7]{dolb} pour \og $\partial = z $\fg) :
\medskip
\begin{coro}\label{cor:witch}
Soit $\lambda \in P_+$ et soit $M$ un $L$-$(\varphi,N,\GQp )$-module cuspidal de pente $1-\lvert \lambda \rvert /2$. On a des isomorphismes topologiques de $L$-représentations coadmissibles de $G$
$$
\begin{gathered}
\Hom_{L[\czG]}\intnn{\JL_M}{\CO^{\infty}_{\lambda}}\cong \intn{\wt{\Pi}_M^{\lambda}}', \quad \Hom_{L[\czG]}\intnn{\JL_M}{\omega^{\infty}_{\lambda}}\cong \intn{\Pi_M^{\lambda}}'\\
 \Hom_{L[\czG]}\intnn{\JL_M}{\rmH^1_{\dR}\Harg{\presp{\rmM}_{\Qp}^{\infty}}{\CE_{\lambda}}}\cong M_{\dR}\otimes_L {\LL_M^{\lambda}}'.
\end{gathered}
$$
En appliquant $\Hom_{L[\czG]}\intnn{\JL_M}{\ \cdot \ }$ à (\ref{eq:extdR}), on obtient le dual de la première suite exacte du bas dans le diagramme de la proposition \ref{prop:weighchange}, soit
$$
0\rightarrow \intn{\wt{\Pi}_M^{\lambda}}'\rightarrow \intn{{\Pi}_M^{\lambda}}'\rightarrow M_{\dR}\otimes_L {\LL_M^{\lambda}}'\rightarrow 0.
$$
\end{coro}
\medskip

\medskip
\subsubsection{L'invariant $\CL$ et les différentielles}
Notons qu'on a une décomposition
$$
\omega_{\lambda}^{\infty}=\bigoplus_{M\in \Phi N_{\lvert \lambda \rvert}^p}\omega_{\lambda}^{\infty}[M]\otimes_L\JL_M,\quad\omega_{\lambda}^{\infty}[M]\coloneqq \Hom_{L[\czG]}\intnn{\JL_M}{\omega_{\lambda}^{\infty}}.
$$

\begin{theo}\label{thm:cuspautmult}
Soit $\lambda\in P_+$ et $M\in \Phi N_{\lvert \lambda \rvert}^p$ cuspidal. Soit $\Pi$ une $L$-représentation de Banach unitaire admissible et absolument irréductible de $G$. Alors 
$$
\Hom_G\intnn{\Pi'}{\omega_{\lambda}^{\infty}[M]}=
\begin{cases}
 \CL& \text{si $\Pi=\Pi_{M,\CL}^{\lambda}$,}\\
 0& \text{si $\Pi$ n'est pas de cette forme.}\\
\end{cases}
$$
\end{theo}
\begin{proof}
On peut supposer que $\Pi=\bPi(V)$ où $V$ est une $L$-représentation galoisienne absolument irréductible de dimension $2$. 
D'après le corollaire \ref{cor:witch}, on a un isomorphisme de $L$-représentations de $G$ 
$$
\Hom_{L[\czG]}\intnn{\JL_M}{\rmH^1_{\dR}\Harg{\presp{\rmM}^{\infty}_{\Qp}}{\CE_{\lambda}}}\cong M_{\dR}\otimes_L {\LL_M^{\lambda}}'.
$$
et pour chaque $L$-droite $\CL\subset M_{\dR}$, une suite exacte 
\begin{equation}\label{eq:suitexx}
0\rightarrow  (\Pilan_{M,\CL}^{\lambda} )'\rightarrow \omega_{\lambda}^{\infty}[M]\rightarrow (M_{\dR}/\CL) \otimes_L {\LL_M^{\lambda}}'\rightarrow 0.
\end{equation}
Or, d'après \cite{codo} on a $(\Pilan')^{G-{\rm b}}\cong \Pi'$ et on en déduit que 
$$
\Hom_G\intnn{\Pi'}{(\Pilan_{M,\CL}^{\lambda})'}\cong \Hom_G\intnn{\Pi'}{(\Pi_{M,\CL}^{\lambda})'}\cong \Hom_G\intnn{\Pi_{M,\CL}^{\lambda}}{\Pi}.
$$
L'injectivité de la correspondance $V\mapsto \bPi(V)$ pour les représentations galoisiennes irréductibles de dimension $2$ et le lemme de Schur assurent que
\begin{equation}\label{eq:calcen1}
\Hom_G\intnn{\Pi'}{(\Pilan_{M,\CL}^{\lambda} )'}\cong
\begin{cases}
L & \text{ si } V\cong V_{M,\CL}^{\lambda},\\
0 & \text{ sinon. }
\end{cases}
\end{equation}
Comme $\LL_M^{\lambda}$ est irréductible, (\ref{eq:compatlocp}) implique que
\begin{equation}\label{eq:calcen2}
\Hom_G\intnn{\Pi'}{{\LL_M^{\lambda}}'}\cong\Hom_G\intnn{{\LL_M^{\lambda}}}{\Pi}\cong
\begin{cases}
L & \text{ si $V$ est de type }(M,\lambda), \\
0 & \text{ sinon. }
\end{cases}
\end{equation}
De ces deux calculs, on peut déduire le théorème.
\begin{itemize}
\itemb Si $V$ n'est pas de type $M$ ou de poids $\lambda$, alors on obtient
$$
\Hom_G\intnn{\Pi'}{\omega_{\lambda}^{\infty}[M]}=0.
$$
\itemb Supposons maintenant que $V$ est de type $M$ et de poids $\lambda$, plus précisément supposons que $V=V_{M,\CL_1}^{\lambda}$  pour $\CL_1\subset M_{\dR}$. On applique le foncteur $\Hom_G(\Pi', \ \cdot \ )$ à la suite exacte (\ref{eq:suitexx}) pour $\CL\neq \CL_1$, ce qui nous assure d'après (\ref{eq:calcen1}) qu'on a une inclusion
$$
\Hom_G\intnn{\Pi'}{\omega_{\lambda}^{\infty}[M]}\incl \Hom_G\intn{\Pi',(M_{\dR}/\CL)\otimes_L {\LL_M^{\lambda}}'}.
$$
Donc, d'après (\ref{eq:calcen2}), $\Hom_G\intnn{\Pi'}{\omega_{\lambda}^{\infty}[M]}$ est de $L$-dimension $\leqslant 1$. On applique maintenant le même foncteur à la suite exacte (\ref{eq:suitexx}) pour $\CL=\CL_1$. D'après (\ref{eq:calcen1}), on en déduit que 
$$
\Hom_G\intnn{\Pi'}{\omega_{\lambda}^{\infty}[M]}\neq 0,
$$
et donc que $\Hom_G\intnn{\Pi'}{\omega_{\lambda}^{\infty}[M]}$ est de $L$-dimension $1$. Il reste à justifier que $\Hom_G\intnn{\Pi'}{\omega_{\lambda}^{\infty}[M]}\cong \CL_1$, mais on a obtenu la suite exacte
$$
0 \rightarrow\underbrace{\Hom_G\intnn{\Pi'}{\left (\Pilan_{M,\CL}^{\lambda}\right )'}}_{\cong L}\xrightarrow{\sim}\Hom_G\intnn{\Pi'}{\omega_{\lambda}^{\infty}[M]}\xrightarrow{0}\underbrace{\Hom_G\intn{\Pi',(M_{\dR}/\CL)\otimes_L {\LL_M^{\lambda}}'}}_{\cong M_{\dR}/\CL_1}.
$$
Comme la dernière flèche est induite par l'inclusion naturelle $\Hom_G\intnn{\Pi'}{\omega_{\lambda}^{\infty}[M]}\incl \Hom_G\intn{\Pi',M_{\dR}\otimes_L {\LL_M^{\lambda}}'}$, on en déduit une inclusion $\Hom_G\intnn{\Pi'}{\omega_{\lambda}^{\infty}[M]}\incl \CL_1$ et donc finalement, on a bien montré que
$$
\Hom_G\intnn{\Pi'}{\omega_{\lambda}^{\infty}[M]}= \CL_1.
$$
\end{itemize}
\end{proof}
\Subsection{Entrelacements cuspidaux}
Dans ce numéro on démontre le théorème suivant :
\medskip
\begin{theo}\label{theo:multcusptot}
Soit $\lambda \in P_+$, $M$ un $L$-$(\varphi,N,\GQp)$-module cuspidal de pente $1-\lvert \lambda \rvert /2$ et $0\subsetneq \CL \subsetneq M_{\dR}$. 
\begin{itemize}
\itemb On a un isomorphisme topologique de $L$-représentations de $G\times \czG$ 
\begin{equation}\label{eq:cuspgal}
\Hom_{L[\sW_{\Qp}]}\intnn{V_{M,\CL}^{\lambda}}{\rmH^1_{\et}\Harg{\brv{\rmM}^{\infty}_{C}}{\Sym\BV(1)}}\cong (\Pi_{M,\CL}^{\lambda})'\otimes_L \JL_M^{\lambda}.
\end{equation}
\itemb On a un ismorphisme de $L$-représentations de $\sW_{\Qp}\times \czG$
\begin{equation}\label{eq:cuspaut}
\Hom_{L[G]}\intnn{(\Pi_{M,\CL}^{\lambda})'}{\rmH^1_{\et}\Harg{\brv{\rmM}^{\infty}_{C}}{\Sym\BV(1)}}\cong V_{M,\CL}^{\lambda}\otimes_L \JL_M^{\lambda}.
\end{equation}
\end{itemize}
\end{theo}
Pour $\star\in \{\et,\pet\}$ notons
$$
\rmH^{\lambda}_{\star}[M]\coloneqq \Hom_{L[\czG]}\intnn{\JL_M}{\rmH^1_{\star}\Harg{\presp{\rmM^{\infty}_{C}}}{\BV_{\lambda}(1)}}.
$$
On va maintenant calculer les multiplicités dans $\rmH^{\lambda}_{\pet}[M]$ et utiliser le théorème \ref{thm:proetoet} et son corollaire \ref{cor:complocan} pour en déduire les multiplicités dans $\rmH^{\lambda}_{\et}[M]$. La remarque suivante justifie qu'il suffit de calculer la multiplicité dans $\rmH^{\lambda}_{\et}[M]$ pour obtenir le théorème.
\medskip
\begin{rema}
On utilise les notations du numéro \ref{subsec:decompcoh}. Premièrement, notons qu'on a 
$$
\presp{\rmH}^{1}_{\star}=\bigoplus_{\lambda \in P_+}\bigoplus_{M\in \Phi N^{p}_{\lvert \lambda \rvert}}\presp{\rmH}^{\lambda}_{\star}[M]\otimes_L\JL_M^{\lambda}
$$
Ainsi pour $W$ une $L$-représentation de $\BG$ absolument irréductible il existe $\lambda\in P_+$ et $M\in \Phi N^{p}_{\lvert \lambda \rvert}$ tel que
$$
\Hom_{L[\BG]}\intnn{W}{\presp{\rmH}^{1}_{\star}}\cong\Hom_{L[\BG]}\intnn{W}{\presp{\rmH}^{\lambda}_{\star}[M]\otimes_L \JL_M^{\lambda}}.
$$
Finalement, d'après \ref{lem:fkcompcon} il existe un caractère lisse $\chi \colon \Qpt \rightarrow L^{\times}$ tel que
$$
\Hom_{L[\BG]}\intnn{W}{\brv{\rmH}^{1}_{\star}}=\Hom_{L[\BG]}\intnn{W\otimes \chi}{\presp{\rmH}^{1}_{\star}}.
$$
On choisit alors $W=V_{M,\CL}^{\lambda}\otimes_L(\Pi_{M,\CL}^{\lambda})'\otimes_L \JL_M^{\lambda}$.
\end{rema}

\subsubsection{Multiplicité des $V_{M,\CL}^{\lambda}$}

\medskip
\begin{coro}\label{cor:cusp}
Soit $\lambda \in P_+$ et $M\in \Phi N_{\lvert \lambda \rvert}^p$. Le diagramme suivant $\GQp\times G$-équivariant de $L$-espaces de Fréchet est commutatif :
{
\begin{center}
\begin{tikzcd}
0\ar[r]&\rmB_{\lambda}\wotimes_{\Qp}\CO_{\lambda}^{\infty}[M]\ar[r]\ar[d,equal]&\presp{\rmH}^{\lambda}_{\pet}[M]\ar[r]\ar[d,hook]&t^{\lambda_1}X_{\st}^{1}(M[\lambda_1])\wotimes_L{\LL_M^{\lambda}}'\ar[r]\ar[d,hook]&0\\
0\ar[r]&\rmB_{\lambda}\wotimes_{\Qp}\CO_{\lambda}^{\infty}[M]\ar[r]&\rmB_{\lambda}\wotimes_{\Qp}\omega_{\lambda}^{\infty}[M]\ar[r]&(\rmB_{\lambda}\otimes_{\Qp}M_{\dR})\wotimes_L{\LL_M^{\lambda}}'\ar[r]&0
\end{tikzcd}
\end{center}}
De plus, les lignes sont exactes strictes et les flèches verticales sont d'images fermés.
\end{coro}
\begin{proof}
On obtient le diagramme en appliquant $Z\mapsto Z[M]\coloneqq \Hom_{L[\czG]}\intnn{\JL_M}{Z}$ au diagramme du théorème \ref{theo:fondrin}. Comme ce foncteur est exact, l'exactitude des suites est conservée et de plus :
\begin{itemize}
\itemb Comme l'action de $\czG$ sur les composantes connexes est donnée par la norme réduite on a $\Hom_{L[\czG]}\intnn{\JL_M}{{W}_{\lambda}^{\pi_0^{\infty}}}=0$ pour tout $n\in \BN$, ce qui montre que le foncteur tue les deux termes tout à gauche du diagramme dans le théorème \ref{theo:fondrin}.
\itemb Les deux termes à droite sont donnés par la proposition \ref{prop:drhkdr}.
\end{itemize}
\end{proof}

Posons
$$
H_{M,\CL}^{\lambda}\coloneqq\Hom_{L[\GQp]}(V_{M,\CL}^{\lambda},t^{\lambda_1}X_{\st}^{1}(M[\lambda_1])),\quad H_{C,\CL}^{\lambda}\coloneqq\Hom_{L[\GQp]}(V_{M,\CL}^{\lambda},L\otimes_{\Qp}\rmB_{\lambda}).
$$
On a naturellement une inclusion $t^{\lambda_1}X_{\st}^{1}(M[\lambda_1])\rightarrow  M_{\dR}\otimes_{\Qp}\rmB_{\lambda}$ qui induit une inclusion $H_{M,\CL}^{\lambda}\incl H_{M,\CL}^{\lambda}\otimes_L M_{\dR}$.
\medskip
\begin{lemm}\label{lem:mulgalcusp}
Les $L$-espaces vectoriels $H_{M,\CL}^{\lambda}$ et $H_{C,\CL}^{\lambda}$ sont de dimensions $1$ et l'inclusion naturelle composé à la trivialisation, $\iota\colon H_{M,\CL}^{\lambda}\incl H_{C,\CL}^{\lambda}\otimes_L M_{\dR}\cong M_{\dR}$ identifie $\iota(H_{M,\CL}^{\lambda})=\CL\subset M_{\dR}$.
\end{lemm}
\medskip
\begin{proof}
On note $V=V_{M,\CL}^{\lambda}$ et quitte a tordre, on peut supposer que $\lambda_1=0$ ; on note donc $w\coloneqq w(\lambda)=\lambda_2$. On choisit une base de $M_{\dR}$ compatible à la filtration, \ie $M_{\dR}=Lf_0\oplus Lf_{w}$ de sorte que $\CL=Lf_{w}$. Alors, en passant au dual, on obtient $M_{\dR}^*=Lf_0^*\oplus Lf_{w}^*$ et $\CL^{\bot}=Lf_0^*$. On calcule $H_{C,\CL}^{\lambda}=(V\otimes_{\Qp} \rmB_{\lambda})^{\GQp}$. Or,
$$
V\otimes_{\Qp} \rmB_{\lambda}=\Fil^0(M_{\dR}^*\otimes_{\Qp}\Bdr)/\Fil^w=f_0^*\rmB_w\oplus f_w^*t^{-w}\rmB_w,
$$
et donc $H_{C,\CL}^{\lambda}=f_0^*=\CL^{\bot}$. 

Pour $H_{M,\CL}^{\lambda}$ on utilise une généralisation immédiate de \cite[Proposition 2.5]{codoni} combiné à \cite[Remarque 2.6]{codoni}  pour obtenir
$$
H_{M,\CL}^{\lambda} = \Hom_{L[\sWD_{\Qp}]}\intnn{M^*[-1]}{\bD_{\pst}(V)^*}=\Hom_{L[\sWD_{\Qp}]}\intnn{M^*}{M^*}=L
$$
Finalement, l'inclusion $H_{M,\CL}^{\lambda}\incl M_{\dR}$ est obtenue en appliquant $\otimes_{L}\CL$ à $H_{M,\CL}^{\lambda}\incl H_{C,\CL}^{\lambda}\otimes_L M_{\dR}$ et ainsi $\iota(H_{M,\CL}^{\lambda})=\CL\subset M_{\dR}$.
\end{proof}
On déduit de ce lemme qu'en appliquant $\Hom_{L[\GQp]}(V_{M,\CL}^{\lambda}, \cdot )$ au diagramme, la flèche verticale de droite devient l'inclusion $\CL \otimes_L{\LL_M^{\lambda}}'\incl M_{\dR} \otimes_L{\LL_M^{\lambda}}'$. On en déduit le lemme suivant :
\medskip
\begin{prop}
Le diagramme suivant $G$-équivariant de $L$-espaces de Fréchet est commutatif :
{
\begin{center}
\begin{tikzcd}
0\ar[r]& \CO_{\lambda}^{\infty}[M]\ar[r]\ar[d,equal]&\Hom_{L[\GQp]}(V_{M,\CL}^{\lambda},\rmH^{\lambda}_{\pet}[M])\ar[r]\ar[d,hook]&\CL \otimes_L{\LL_M^{\lambda}}'\ar[d,hook]\ar[r]&0\\
0\ar[r]& \CO_{\lambda}^{\infty}[M]\ar[r]&\omega_{\lambda}^{\infty}[M]\ar[r]&M_{\dR}\otimes_L {\LL_M^{\lambda}}'\ar[r]& 0
\end{tikzcd}
\end{center}}
Dans ce diagramme, les flèches verticales sont injectives et d'image fermée.
\end{prop}
\medskip
\begin{proof}
Notons $V=V_{M,\CL}^{\lambda}$. D'après le lemme \ref{lem:mulgalcusp}, il suffit de montrer que $\Hom_{L[\GQp]}(V,\ \cdot\ )\cong (\ \cdot\  \otimes_{L}V^*)^{\GQp}$ conserve l'exactitude des lignes du diagramme \ref{cor:cusp}. L'application naturelle
$$
\rmH^1\Harg{\GQp}{\rmB_{\lambda}\wotimes_{\Qp}\CO_{\lambda}^{\infty}[M]\otimes_LV^{*}}\rightarrow \rmH^1\Harg{\GQp}{\rmB_{\lambda}\wotimes_{\Qp}\omega_{\lambda}^{\infty}[M]\otimes_LV^{*}}
$$
est injective puisqu'on a\footnote{À l'aide des presque-C-représentation on peut même montrer que $\rmH^1\Harg{\GQp}{\rmB_{\lambda}\otimes_{\Qp} V^{*}}$ est un $L$-espace vectoriel de dimension $1$.}
$$
\begin{gathered}
\rmH^1\Harg{\GQp}{\rmB_{\lambda}\wotimes_{\Qp}\CO_{\lambda}^{\infty}[M]\otimes_LV^{*}}\cong\rmH^1\Harg{\GQp}{\rmB_{\lambda}\otimes_{\Qp} V^{*}}\otimes_{L}\CO_{\lambda}^{\infty}[M],\\
\rmH^1\Harg{\GQp}{\rmB_{\lambda}\wotimes_{\Qp}\omega_{\lambda}^{\infty}[M]\otimes_LV^{*}}\cong \rmH^1\Harg{\GQp}{\rmB_{\lambda}\otimes_{\Qp}V^{*}}\otimes_{L}\omega_{\lambda}^{\infty}[M],
\end{gathered}
$$
car l'application $\CO_{\lambda}^{\infty}[M]\rightarrow \omega_{\lambda}^{\infty}[M]$ est injective. Ceci justifie l'exactitude de la ligne du bas. Par la commutativité du diagramme \ref{cor:cusp} on en déduit que l'application
$$
\rmH^1\Harg{\GQp}{\rmB_{\lambda}\wotimes_{\Qp}\CO_{\lambda}^{\infty}[M]\otimes_LV^{*}}\rightarrow \rmH^1\Harg{\GQp}{\rmH^{\lambda}_{\pet}[M]\otimes_LV^{*}},
$$
est injective.
\end{proof}
 Or, d'après la proposition \ref{cor:witch} en considérant l'image inverse de $\CL\otimes_L {\LL_M^{\lambda}}'$ dans $\omega_{\lambda}^{\infty}[M]$ on obtient la suite exacte
$$
0\rightarrow \CO_{\lambda}^{\infty}[M]\rightarrow \intn{\Pilan^{\lambda}_{M,\CL}}'\rightarrow \CL \otimes_L {\LL_M^{\lambda}}'\rightarrow 0.
$$
\begin{coro}\label{prop:invlcusp}
On a un isomorphisme topologique entre $L$-représentations $G$
$$
\intn{\Pilan^{\lambda}_{M,\CL}}'\cong \Hom_{L[\GQp]}\intnn{V_{M,\CL}^{\lambda}}{\rmH^{\lambda}_{\pet}[M]}.
$$
\end{coro}
Finalement, comme on a montré au théorème \ref{thm:proetoet} que $\rmH^{\lambda}_{\et}[M]$ est le le sous-espace des vecteurs $G$-bornés de $\rmH^{\lambda}_{\et}[M]$ et que, comme on l'a rappelé au numéro \ref{subsubsec:repcuspg}, les vecteurs $G$-bornés de $\intn{\Pilan^{\lambda}_{M,\CL}}'$ sont $\intn{\Pi^{\lambda}_{M,\CL}}'$, ceci démontre l'isomorphisme (\ref{eq:cuspgal}) du théorème \ref{theo:multcusptot}.
\Subsubsection{Multiplicité des $\Pi_{M,\CL}^{\lambda}$}

\medskip
\begin{prop}
Soit $\lambda \in P_+$, $M\in \Phi N_{\lvert \lambda \rvert}^p$ cuspidal et soit $\Pi$ une $L$-représentation de Banach unitaire, admissible, absolument irréductible de $G$. Alors
{
$$
\Hom_{L[G]}\intnn{(\Pi^{\lan})'}{\rmH^{\lambda}_{\pet}[M]}=
\begin{cases}
 V_{M,\CL}^{\lambda}\otimes_L \JL_M^{\lambda} & \text{si $\Pi=\Pi_{M,\CL}^{\lambda}$,}\\
 0& \text{sinon.}\\
\end{cases}
$$}
\end{prop}
\medskip
\begin{proof}
 Pour $M$ cuspidal, on a montré l'existence d'un diagramme $G$-équivariant de $L$-espaces de Fréchet, 
 \begin{center}
\begin{tikzcd}
0\ar[r]&\rmB_{\lambda}\wotimes_{\Qp}\CO_{\lambda}^{\infty}[M]\ar[r]\ar[d,equal]&\rmH^{\lambda}_{\pet}[M]\ar[r]\ar[d,hook]&t^
{\lambda_1}X_{\st}^{1}(M[\lambda_1])\wotimes_L{\LL_M^{\lambda}}'\ar[r]\ar[d,hook]&0\\
0\ar[r]&\rmB_{\lambda}\wotimes_{\Qp}\CO_{\lambda}^{\infty}[M]\ar[r]&\rmB_{\lambda}\wotimes_{\Qp}\omega_{\lambda}^{\infty}[M]\ar[r]&(\rmB_{\lambda}\otimes_{\Qp}M_{\dR})\wotimes_L{\LL_M^{\lambda}}'\ar[r]&0
\end{tikzcd}
\end{center}
Si on note 
$$
Q_M^{\lambda}\coloneqq\frac {(\rmB_{\lambda}\otimes_{\Qp}M_{\dR})}{t^{\lambda_1}X_{\st}^{1}(M[\lambda_1])},
$$
on en déduit la suite exacte 
$$
0 \rightarrow \rmH^{\lambda}_{\pet}[M] \rightarrow \rmB_{\lambda}\wotimes_{\Qp}\omega_{\lambda}^{\infty}[M]\rightarrow Q_M^{\lambda}\wotimes_L{\LL_M^{\lambda}}' \rightarrow 0.
$$
On applique $\Hom_G\intnn{(\Pi^{\lan})'}{\ \cdot \ }$ à cette suite exacte pour obtenir 
$$
\begin{gathered}
\Hom_G\intnn{(\Pi^{\lan})'}{\rmH^{\lambda}_{\pet}[M]} = \\
\Ker\left ( \Hom_G\intnn{(\Pi^{\lan})'}{\rmB_{\lambda}\wotimes_{\Qp}\omega_{\lambda}^{\infty}[M]}\rightarrow \Hom_G\left ((\Pi^{\lan})',Q_M^{\lambda}\wotimes_L{\LL_M^{\lambda}}'\right )\right )  =\\
\Ker\left ( \Hom_G\intnn{(\Pi^{\lan})'}{\omega_{\lambda}^{\infty}[M]}\wotimes_{\Qp}\rmB_{\lambda}\rightarrow \Hom_G\intnn{(\Pi^{\lan})'}{{\LL_M^{\lambda}}'}\wotimes_LQ_M^{\lambda}\right )  .
\end{gathered}
$$
Or, d'après le théorème \ref{thm:cuspautmult}
$$
\Hom_G\intnn{(\Pi^{\lan})'}{\omega_{\lambda}^{\infty}[M]} =
\begin{cases}
 \CL & \text{si } \Pi=\Pi_{M,\CL}^{\lambda},\\
 0& \text{sinon.}
\end{cases}
$$
Ainsi, on peut supposer que $\Pi=\Pi_{M,\CL}^{\lambda}$. Dans ce cas, $\Hom_G\intnn{(\Pi^{\lan})'}{{\LL_M^{\lambda}}'}$ est de dimension~$1$, ce qui donne 
$$
\begin{gathered}
 \Hom_G\intnn{(\Pi^{\lan})'}{\rmH^{\lambda}_{\pet}[M]}=\Ker\left (\CL\otimes_{\Qp}\rmB_{\lambda}\rightarrow Q_M^{\lambda}\right )\\
 =t^{\lambda_1}X_{\st}^{1}(M[\lambda_1])\cap (\CL\otimes_{\Qp}\rmB_{\lambda})= V_{M,\CL}^{\lambda}.
\end{gathered}
$$
\end{proof}
Cette proposition couplée au corollaire \ref{cor:complocan} nous donne la seconde partie du théorème \ref{theo:multcusptot} soit :
\medskip
\begin{coro}\label{cor:multautcuspet}
Soit $\lambda \in P_+$, $M\in \Phi N_{\lvert \lambda \rvert}^p$ cuspidal et $\Pi$ une $L$-représentation de Banach unitaire, admissible, absolument irréductible de $G$. Alors
{
$$
\Hom_{L[G]}\intnn{\Pi'}{\rmH^{\lambda}_{\et}[M]}=
\begin{cases}
 V_{M,\CL}^{\lambda}\otimes_L \JL_M^{\lambda} & \text{si $\Pi=\Pi_{M,\CL}^{\lambda}$,}\\
 0& \text{sinon.}\\
\end{cases}
$$}
\end{coro}
\medskip

\section{Le cas spécial}\label{sec:caspe}
On conserve les notations des sections précédentes.
\Subsection{La série spéciale $p$-adique}\label{subsec:seriespe}
Dans ce numéro on fait quelques rappels sur la série spéciale pour compléter ce que l'on a expliqué au numéro \ref{subsec:rapcusp}.

Soit $k\in \BZ$ un entier, on définit le $L$-$(\varphi,N)$-module spécial $\Sp_L(k)$ comme le $L$-espace vectoriel $\Sp_L(k)=Le_0\oplus L e_1$ muni des endomorphismes
$$
\begin{cases}
\varphi(e_1)=p^{-\frac{k-1}{2}}e_1\\
\varphi(e_0)=p^{-\frac{k+1}{2}}e_0,
\end{cases}
\quad
\begin{cases}
Ne_1=e_0\\
Ne_0=0.
\end{cases}
$$
Sa pente est $-\tfrac k 2$ et $\Sp_L(k)[i]=\Sp_L(k-2i)$. Alors, tout $L$-$(\varphi,N,\GQp)$-module spécial, au sens de la définition \ref{def:mcuspispe}, est de la forme $(\Sp_L(k)\otimes_{\Qp}\Qp^{\nr})\otimes_L \chi$ pour $\chi \colon \Qpt \rightarrow L^{\times}$ un caractère (\cf la preuve de \cite[Proposition 1.3]{sch}). Soit $\lambda\in P$ et fixons $M$ de la forme précédente avec $k=\lvert \lambda \rvert -2$
\begin{itemize}
\itemb $\LL_M^{\lambda}= \St^{\lalg}_{\lambda}\otimes_L (\chi\circ \det)$ où l'on rappelle que $\St^{\lalg}_{\lambda}=\St^{\infty}\otimes W_{\lambda}^*$ est la Steinberg localement algébrique de poids $\lambda$ sur $L$ (\cf \ref{subsubsec:spean}),
\itemb $\JL_M^{\lambda}\coloneqq \cz{W}_{\lambda}\otimes_L(\chi\circ \nrd)$.
\end{itemize}

Soit $\lambda\in P_+$ et $\CL\in L$, alors on définit $D_{\CL}^{\lambda}$ le $L$-$(\varphi,N)$-module filtré de $L$-$(\varphi,N)$-module sous-jacent $\Sp_L(\lvert \lambda \rvert -2)$ et de filtration 
$$
\Fil^iD_{\CL}^{\lambda}\coloneqq 
\begin{cases}
D_{\CL}^{\lambda}&\text{ si } i\leqslant -\lambda_2+1,\\
\sL=L(e_1-\CL e_0) &\text{ si } -\lambda_2+2\leqslant i\leqslant -\lambda_1+1,\\
0&\text{ si }  -\lambda_1+2\leqslant i.
\end{cases}
$$
Alors, en prenant $M=D_{\CL}^{\lambda}\otimes_{\Qp}\Qp^{\nr}$ et $\sL=L(e_0-\CL e_1)\subset M_{\dR}$ définit par $\CL\in L$ comme ci-dessus, on a $M\in\Phi N_{\lvert \lambda \rvert}^p$ et on note simplement $V_{\CL}^{\lambda}=V_{M,\sL}^{\lambda}$ et $\Pi_{\CL}^{\lambda}=\Pi_{M,\sL}^{\lambda}$.

Coté Galoisien, toute $L$-représentation semi-stable non cristabéline de dimension $2$ (\ie spéciale) à poids de Hodge-Tate $\lambda\in P_+$ est de la forme $V_{\CL}^{\lambda}\otimes \chi$ pour $\chi \colon \GQp\rightarrow L^{\times}$ un caractère lisse. La représentation $V^{\lambda}_{\CL}$ est irréductible si $w(\lambda)\neq 1$ et si $w(\lambda)=1$ on a une suite exacte courte.
$$
0\rightarrow L(\lambda_1+1)\rightarrow V^{\lambda}_{\CL}\rightarrow L(\lambda_1)\rightarrow 0.
$$

Coté automorphe, toute représentation spéciale est de la forme $\Pi_{\CL}^{\lambda}\otimes (\chi\circ \det)$ pour $\chi \colon \Qpt \rightarrow L^{\times}$. La représentation $\Pi_{\CL}^{\lambda}$ est une $L$-représentation irréductible si $w(\lambda)\neq 1$ et si $w(\lambda)=1$ on a une suite exacte courte
$$
0\rightarrow \St^0\otimes_L (\det)^{\lambda_1}\rightarrow \Pi_{\CL}^{\lambda}\rightarrow (\det)^{\lambda_1}\rightarrow 0,
$$
où $\St^0$ est la Steinberg continue à coefficients dans $L$. En tant que $L$-représentation $\Pilan_{\CL}^{\lambda}$ à quatres composantes de Jordan-Hölder que l'on a déjà rencontré (\cf \ref{sub:steinberg} et \ref{subsub:anulocext} pour les énoncés et notations), et que l'on résume dans le diagramme\footnote{Dans cette notation, chaque terme du diagramme est une composante irréductible, le premier est le socle, le second est le socle du quotient par le socle, etc. En particulier, la composante la plus à droite est le cosocle. } suivant :
\newcommand\extline{\mathrel{\relbar\joinrel\relbar}}

{\large
\begin{equation}\label{eq:jhcompspe}
 \Pilan_{\CL}^{\lambda} \ \colon\  \overbrace{\underbrace{\St_{\lambda}^{\lalg}\extline B^{\lambda}}_{\St_{\lambda}^{\lan}}\extline W_{\lambda}^{*}}^{\Sigma_{\CL}^{\lambda}}\extline \wt{B}^{\lambda}.
\end{equation}
}
Rappelons en particulier, d'après \cite{codopa}, que le complété unitaire universel de ${\Sigma}_{\CL}^{\lambda}$ et $\Pi_{\CL}^{\lambda}$.
\Subsection{Cohomologie isotriviale du demi-plan $p$-adique}
\Subsubsection{Cohomologie de de Rham et de Hyodo-Kato}
On va maintenant décrire les cohomologies de de Rham et de Hyodo-Kato en niveau zéro. La difficulté est que côté automorphe, on doit calculer un entrelacement dérivé. Cette difficulté apparaît déjà dans les travaux de Schraen (\cf \cite{sch}.), dont on va expliquer et reformuler légèrement le résultat. Notons $\sD(G,L)$ l'algèbre des distributions localement analytiques sur $G$ à valeurs dans $L$. On travaillera dans la catégorie dérivée des modules sur $\sD(G,L)$ de caractère central fixé\footnote{Comme expliqué au numéro \ref{subsubsec:spean}, on ne précise pas le caractère central mais il n'y aura pas d'ambiguité : en poids $\lambda$ ce caractère central est $\omega^{1-\lvert \lambda \rvert}$ avec les $L$-représentations de Fréchet de $G$ qui apparaissent (comme dans la proposition \ref{prop:schra}) et $\omega^{\lvert \lambda \rvert-1}$ avec leurs duaux localement analytiques (comme au lemme \ref{lem:schr}).} et on note $\bHom\ $ les homomorphismes dans cette catégorie : c'est en particulier le $\rmH^0$ du bifoncteur dérivé $\rmR\Hom$. De même, on note $\Ext^1_{L[\bar G]}$ le groupe des extensions dans la catégorie des $\sD(G,L)$-modules à caractère central fixé.
\medskip
\begin{prop}\label{prop:schraen}
L'isomorphisme de Hyodo-Kato,
$$
\rmR\Gamma_{\!\dR}\Harg{\BH_{\Qp}}{\CE_{\lambda}}\cong\rmR\Gamma_{\!\HK}\Harg{\BH_{\Qp}}{\BD_{\lambda}},
$$
munit le complexe de de Rham d'une structure de $(\varphi,N)$-module filtré et on a un isomorphisme
$$
\bHom\ \intnn{\intn{\Pilan_{\CL}^{\lambda}}'[-1]}{\rmR\Gamma_{\!\dR}\Harg{\BH_{\Qp}}{\CE_{\lambda}}}\cong D_{\CL}^{\lambda}.
$$
\end{prop}
\medskip
\subsubsection{Rappels du résultat de Schraen}\label{subsubsec:schraen}
Dans ce paragraphe, on explique le résultat principal de \cite{sch}, mais en tordant par une puissance du déterminant ce qui ne change rien à l'argument.
\medskip
\begin{prop}\label{prop:schra}
On a un isomorphisme de $(\varphi,N)$-modules filtrés
$$
\bHom\ \intnn{\intn{\Sigma_{\CL}^{\lambda}}'[-1]}{\rmR\Gamma_{\!\dR}(\BH_{\Qp})\otimes_{\Qp}{W}_{\lambda}}\cong D_{\CL}^{\lambda}.
$$
\end{prop}
\medskip
Dans cet énoncé, le membre de gauche est muni d'un Frobenius, d'une monodromie et d'une filtration \og à la main \fg ; on va les définir en expliquant la preuve. On commence par rappeler plusieurs résultats préliminaires démontrés dans \cite{sch} :
\medskip
\begin{lemm}\label{lem:schr}
Soit $\lambda\in P_+$, alors 
\begin{itemize}
\itemb $\Ext^1_{L[\bar G]}\intnn{W^*_{\lambda}}{\Sigma_{\CL}^{\lambda}}$ est un $L$-espace vectoriel de dimension $1$,
\itemb $\Hom_{L[ G]}\intnn{\St^{\lalg}_{\lambda}}{\Sigma_{\CL}^{\lambda}}$ et $\Hom_{L[ G]}\intnn{\St^{\lan}_{\lambda}}{\Sigma_{\CL}^{\lambda}}$ sont des $L$-espaces vectoriels de dimension $1$,
\itemb l'application naturelle $\Ext^1_{L[\bar G]}\intnn{\boldsymbol{1}}{\St^{\infty}}\rightarrow \Ext^1_{L[\bar G]}\intnn{W_{\lambda}^*}{\St^{\lalg}_{\lambda}}$, où les extensions sont considérées dans la catégorie des $L$-représentations localement algébriques de $G$, est un isomorphisme, 
\itemb l'image de l'application naturelle $\Ext^1_{L[\bar G]}\intnn{W_{\lambda}^*}{\St^{\lalg}_{\lambda}}\rightarrow\Ext^1_{L[\bar G]}\intnn{W_{\lambda}^*}{\St^{\lan}_{\lambda}}\cong \Hom\intnn{\Qp^{\times}}{L}$ est la $L$-droite engendrée par $v_p$.
\end{itemize}
\end{lemm}
\begin{proof}
Le premier point est le contenu de \cite[Corollaire 4.11]{sch}.

Le second point provient de la description des composantes de Jordan-Hölder de $\Sigma_{\CL}^{\lambda}$ (\cf (\ref{eq:jhcompspe}).

Le troisième point est le contenu de \cite[Proposition 4.14]{sch}.

Le dernier point est le contenu de \cite[Corollaire 4.16]{sch}.
\end{proof}
On explique maintenant la preuve de la proposition \ref{prop:schra} en suivant la démonstration de \cite[Théorème 5.3]{sch}.
\begin{proof}
Le complexe de de Rham de $\BH_{\Qp}$ se scinde et s'écrit $\rmR\Gamma_{\!\dR}(\BH_{\Qp})\cong \rmH^0_{\dR}(\BH_{\Qp})\oplus \rmH^1_{\dR}(\BH_{\Qp})[-1]$. Or $\rmH^0_{\dR}(\BH_{\Qp})\cong \Qp$ et $\rmH^1_{\dR}(\BH_{\Qp})\cong (\St^{\infty}_{\Qp})'$. Ainsi, il est naturel de définir le complexe de $L$-représentations localement algébriques
\begin{equation}\label{eq:hlambda}
\CH_{\lambda}\coloneqq W_{\lambda}\oplus \St^{\lalg}_{\lambda}[-1],
\end{equation}
qui est quasi-isomorphe à $\rmR\Gamma_{\!\dR}(\BH_{\Qp})\otimes_{\Qp}W_{\lambda}$. Ainsi, on obtient que 
\begin{equation}\label{eq:dedcdeun}
D\coloneqq \bHom\ \intnn{\intn{\Sigma_{\CL}^{\lambda}}'[-1]}{\CH_{\lambda}} = D_0 \oplus D_1,
\end{equation}
où 
\begin{itemize}
\itemb $D_0= \Ext^1_{L[\bar G]}\intnn{\intn{\Sigma_{\CL}^{\lambda}}'}{W_{\lambda}}\cong  \Ext^1_{L[\bar G]}\intnn{W_{\lambda}^*}{\Sigma_{\CL}^{\lambda}}$ qui est un $L$-espace vectoriel de dimension $1$ d'après le lemme \ref{lem:schr},
\itemb $D_1=\Hom_{L[\bar G]}\intnn{\intn{\Sigma_{\CL}^{\lambda}}'}{\intn{\St^{\lalg}_{\lambda}}'}\cong \Hom_{L[\bar G]}\intnn{\St^{\lalg}_{\lambda}}{\Sigma_{\CL}^{\lambda}}$ qui est un $L$-espace vectoriel de dimension $1$ d'après le lemme \ref{lem:schr}.
\end{itemize}
Schraen munit alors $D$ d'une structure de $(\varphi,N)$-module : en posant $\varphi\coloneqq p^{-\frac{\lvert \lambda \rvert-1}{2}}$ sur $D_0$ et~$\varphi\coloneqq p^{-\frac{\lvert \lambda \rvert -3}{2}}$ sur $D_1$ puis en définissant un opérateur $N\colon D_1\rightarrow D_0$ que l'on va expliciter. Soit $f\in D_1$ non nul que l'on voit comme une application $L$-linéaire non nulle $f \colon\St^{\lalg}_{\lambda}\rightarrow \Sigma_{\CL}^{\lambda}$. Alors $N(f)$ est l'image de $v_p$ par l'application induite par $f$
$$
\Ext^1_{L[\bar G]}\intnn{W^*_{\lambda}}{\St_{\lambda}^{\lalg}}\xrightarrow{f_*}\Ext^1_{L[\bar G]}\intnn{W^*_{\lambda}}{\Sigma_{\CL}^{\lambda}}.
$$
D'après le dernier point du lemme \ref{lem:schr} $N(f)\neq 0$, ce qui finit la démonstration que $D\cong \Sp_L(\lvert \lambda \rvert)$.
On définit une filtration sur $\CH_{\lambda}$ à partir de l'isomorphisme avec $\rmR\Gamma_{\!\dR}(\BH_{\Qp})\otimes_{\Qp}W_{\lambda}$ par 
$$
\Fil^i\CH_{\lambda}=
\begin{cases}
\CH_{\lambda}& \text{ si } i\leqslant -\lambda_2+1,\\
0\rightarrow \omega_{\lambda} & \text{ si } -\lambda_2+2\leqslant i \leqslant -\lambda_1+1,\\
0 & \text{ si } -\lambda_1+2\leqslant i.
\end{cases}
$$
Ainsi, pour $i$ un entier tel que $-\lambda_2+2\leqslant i \leqslant -\lambda_1+1$, on obtient par la proposition \ref{prop:morita}
$$
\bF\coloneqq\bHom\intnn{\intn{\Sigma_{\CL}^{\lambda}}'[-1]}{\Fil^i\CH_{\lambda}}\cong \Hom_{L[\bar G]}\intnn{\intn{\Sigma_{\CL}^{\lambda}}'}{\intn{\St^{\lan}_{\lambda}}'}.
$$
D'après le lemme \ref{lem:schr}, $\bF$ est un $L$-espace vectoriel de dimension $1$. On a les deux projections 
\begin{center}
\begin{tikzcd}
&\ar[dl,"p_0"]\bF=\Hom_{L[\bar G]}\intnn{\St^{\lan}_{\lambda}}{\Sigma_{\CL}^{\lambda}}\ar[dr,"p_1"]&\\
D_0=\Ext^1_{L[\bar G]}\intnn{W_{\lambda}^*}{\Sigma_{\CL}^{\lambda}}&&D_1=\Hom_{L[G]}\intnn{\St_{\lambda}^{\lalg}}{\Sigma_{\CL}^{\lambda}}.
\end{tikzcd}
\end{center}
On choisit $v\in \bF$ l'unique élément tel que $v_1\coloneqq p_1(v)$ correspond à l'application naturelle $\St^{\lalg}_{\lambda}\rightarrow \Sigma_{\CL}^{\lambda}$. On a 
\begin{itemize}
\itemb $N(v_1) = v_p$, naturellement obtenu par l'application $\Ext^1_{L[\bar G]}\intnn{W_{\lambda}}{\Sigma_{\CL}^{\lambda}}\rightarrow \Ext^1_{L[\bar G]}\intnn{W_{\lambda}^*}{\St_{\lambda}^{\lan}}$
\itemb $v_0\coloneqq p_0(v) = -\log_0$, obtenu par l'application naturelle $\Ext^1_{L[\bar G]}\intnn{W_{\lambda}}{\St^{\lan}_{\lambda}}\rightarrow \Ext^1_{L[\bar G]}\intnn{W_{\lambda}}{\Sigma_{\CL}^{\lambda}}$.
\end{itemize}
Or, dans $\Ext^1_{L[\bar G]}\intnn{W_{\lambda}}{\Sigma_{\CL}^{\lambda}}$, on a par définition $\log_0=\CL v_p$. Ainsi $v_0=-\CL v_1$ et donc $v=v_1-\CL N(v_1)$ dans $D$, ce qui conclut que la filtration est bien celle de $D_{\CL}^{\lambda}$.
\end{proof}
\subsubsection{Preuve de la proposition \ref{prop:schraen}}\label{subsub:nontrivmono}
Notons qu'une première différence avec le résultat de Schraen est une torsion du Frobenius. Pour résumer, il faut justifier deux choses pour obtenir la proposition \ref{prop:schraen} à partir de la proposition \ref{prop:schra}. On doit montrer :
\begin{enumerate}
\item que l'on obtient le $(\varphi,N)$-module $\Sp_L(\lvert \lambda \rvert -2 )$ lorsqu'on calcule l'entrelacement avec la cohomologie de Hyodo-Kato isotriviale de $\BH_{\Qp}$ à coefficient dans $\BD_{\lambda}$,
\item que l'on obtient le même résultat en remplaçant $\Sigma^{\lambda}_{\CL}$ par $\Pilan_{\CL}^{\lambda}$.
\end{enumerate}
Pour le premier point on veut montrer que 
$$
\bHom\ \intnn{\intn{{\Sigma}_{\CL}^{\lambda}}'[-1]}{\rmR\Gamma_{\!\HK}\Harg{\BH_{\Qp}}{\BD_{\lambda}}}\cong \Sp_L(\lvert \lambda \rvert -2).
$$
Notons $D$ le terme de gauche. On va décomposer le calcul en deux étapes en utilisant la preuve de \ref{prop:schra}.
\begin{itemize}
\itemb De l'isomorphisme de Hyodo-Kato, on obtient un isomorphisme de complexe de $L$-représentations localement algébriques
$$
\rmR\Gamma_{\!\HK}\Harg{\BH_{\Qp}}{\BD_{\lambda}} \cong \CH_{\lambda}.
$$
Montrons que les Frobenius coïncident. Rappelons que d'après \ref{prop:calcdec} $\varphi^2=p^{1-\lvert \lambda \rvert}$ sur $\BD_{\lambda}$. D'après (\ref{eq:dedcdeun}) on obtient la décomposition $D=D_0\oplus D_1$ avec $D_0$ et $D_1$ de dimension $1$ sur $L$. On a : 
\begin{itemize}
\item[$\diamond$] $D_0\coloneqq \Ext^1_{L[\bar{G}]}\intnn{({\Sigma}_{\CL}^{\lambda})'}{\rmH^0_{\HK}\Harg{\BH_{\Qp}}{\BD_{\lambda}}}$ et comme $\varphi=1$ sur $\rmH^0_{\HK}(\BH_{\Qp})$ on obtient que $\varphi^2=p^{1-\lvert \lambda \rvert}$ sur $D_0$. Comme ce dernier est de dimension $1$, on obtient $\varphi=p^{\frac{1-\lvert \lambda \rvert}{2}}$.
\item[$\diamond$] $D_1\coloneqq \Hom_{L[G]}\intnn{({\Sigma}_{\CL}^{\lambda})'}{\rmH^1_{\HK}\Harg{\BH_{\Qp}}{\BD_{\lambda}}}$ et comme $\varphi=p$ sur $\rmH^1_{\HK}(\BH_{\Qp})$ on obtient que $\varphi^2=p^{3-\lvert \lambda \rvert}$ sur $D_1$. Comme ce dernier est de dimension $1$, on obtient $\varphi=p^{\frac{3-\lvert \lambda \rvert}{2}}$.
\end{itemize}
Ceci montre que les $\varphi$-modules coïncident bien.
\itemb Reste a montrer que les opérateurs de monodromie coïncident. On va commencer par montrer que le morphisme
$$
N\colon \rmR\Gamma_{\!\HK}\Harg{\BH_{\Qp}}{\BD_{\lambda}}\rightarrow \rmR\Gamma_{\!\HK}\Harg{\BH_{\Qp}}{\BD_{\lambda}},
$$
n'est pas nul. Il suffit de le montrer à coefficient triviaux \ie montrer que $N\colon \rmR\Gamma_{\!\HK}\rightarrow \rmR\Gamma_{\!\HK}$ n'est pas nul, où l'on note $\rmR\Gamma_{\!\HK}\coloneqq \rmR\Gamma_{\!\HK}(\BH_{\Qp})$ et $\rmH_{\HK}^{i}\coloneqq \rmH_{\HK}^{i}(\BH_{\Qp})$ pour simplifier. Soit $\Gamma\subset \PSL_2(\Qp)$ un sous-groupe de Schottky cocompact et considérons la courbe rigide $X_{\Gamma}\coloneqq \BH_{\Qp}/\Gamma$, naturellement muni d'un modèle à réduction semi-stable. On sait d'après \cite{des} que $N_{\Gamma}\colon \rmR\Gamma_{\!\HK}(X_{\Gamma})\rightarrow \rmR\Gamma_{\!\HK}(X_{\Gamma})$ n'est pas nul.
De plus, on a un quasi-isomorphisme
$$
\rmR\Gamma\intnn{\Gamma}{ \rmR\Gamma_{\!\HK}}\cong \rmR\Gamma_{\!\HK}(X_{\Gamma}).
$$
Mais la fonctorialité de la cohomologie de Hyodo-Kato nous dit que l'opérateur de monodromie $N_{\Gamma}\colon  \rmR\Gamma_{\!\HK}(X_{\Gamma})\rightarrow  \rmR\Gamma_{\!\HK}(X_{\Gamma})$ est induit par l'opérateur de monodromie $N$ ; en particulier si $N$ était nul, $N_{\Gamma}$ serait nul, ce qui n'est pas le cas. De plus, $\rmR\Gamma_{\!\HK}(\BH_{\Qp})\cong \rmH^0_{\HK}\oplus  \rmH^1_{\HK}[-1]$ et puisque $N\varphi= p\varphi N$ on en déduit que $N$ se restreint en $N\colon \rmH^1_{\HK}[-1]\rightarrow \rmH^0_{\HK}$ qui n'est donc pas nul. Puisque $\Ext^1_{\Qp[\bar G]}\intnn{  \rmH^1_{\HK}}{  \rmH_{\HK}^0}\cong \Ext^1_{\Qp[\bar G]}\intnn{\Qp}{\St_{\Qp}^{\infty}}$ on remarque d'après le lemme \ref{lem:schr} que cette classe est donnée par la valuation $p$-adique $v_p$, ce qui est en accord avec la définition de Schraen.
\end{itemize}
Ceci fini de prouver le premier point.

Pour le second point, on utilise la suite exacte 
$$
0\rightarrow {\Sigma}_{\CL}^{\lambda}\rightarrow \Pilan_{\CL}^{\lambda}\rightarrow {\wt{B}}^{\lambda}\rightarrow 0.
$$
Ainsi, il suffit de montrer que 
$$
\bHom\ \intnn{\intn{{\wt{B}}_{\lambda}}'[-1]}{\rmR\Gamma_{\!\dR}\Harg{\BH_{\Qp}}{\CE_{\lambda}}}\cong 0.
$$
Mais puisque le complexe de de Rham est scindé, on écrit cet entrelacement dérivé comme la somme directe 
$$
\Hom_{L[G]}\intnn{{\St}^{\lalg}_{\lambda}}{{{\wt{B}}^{\lambda}}}\oplus \Ext^1_{L[\bar G]}\intnn{{W}^*_{\lambda}}{{{\wt{B}}^{\lambda}}},
$$
où l'on rappelle que les entrelacements et extensions sont considérés dans la catégorie des $L$-représentations localement analytiques à caractère central fixé. Le terme de droite, le groupe des extensions, est nul d'après le lemme \ref{lem:anuprinc}. Il est clair que le terme de gauche est nul : se sont deux représentations irréductibles non isomorphes.

Ceci achève la démonstration de la proposition \ref{prop:schraen}.

\qed
\subsubsection{Le diagramme fondamental}
On rappelle la forme que prend le diagramme dans ce cas. Les notations sont presque les mêmes qu'au numéro \ref{subsubsec:diagfond}, la différence est que l'on travaille avec $\BH_{\Qp}$ plutôt que $\presp{\rmM}_{\Qp}^{\infty}$. On note en particulier 
$$
\begin{gathered}
\CO_{\lambda}\coloneqq \CO\{1-w(\lambda),1-\lambda_1\}\otimes_{\Qp}L\cdot \lvert \det \rvert_p^{\frac{1-\lvert \lambda \rvert}{2}},\\
\omega_{\lambda}\coloneqq\CO\{1+w(\lambda),1-\lambda_2\}\otimes_{\Qp}L\cdot \lvert \det \rvert_p^{\frac{1-\lvert \lambda \rvert}{2}},\\
\rmH^{\lambda}_{\pet}\coloneqq \rmH^1_{\pet}\Harg{\BH_C}{\BV_{\lambda}(1)}.
\end{gathered}
$$
Rappelons qu'on a l'isomorphisme de Morita $\omega_{\lambda}\cong \intn{{\St}^{\lan}_{\lambda}}'$ (\cf \ref{subsub:morita}). On note de plus (\cf la remarque \ref{rem:inclsymp} pour les inclusions qui définissent les $Q_{\lambda}^i$).
$$
\begin{gathered}
\rmB_{\lambda}\coloneqq t^{\lambda_1}\Bdrp/t^{\lambda_2}\Bdrp,\quad \rmU_{\lambda}^0\coloneqq t^{\lambda_1}\intn{\Bcrisp}^{\varphi^2=p^{w(\lambda)+1}},\quad \rmU_{\lambda}^1(L)\coloneqq t^{\lambda_1}\intn{\Bcrisp}^{\varphi^2=p^{w(\lambda)-1}}, \\
Q_{\lambda}^0\coloneqq \frac{\rmB_{\lambda}\otimes_{\Qp}L}{\rmU_{\lambda}^0\otimes_{\Qpd}L},\quad Q_{\lambda}^1\coloneqq \frac{\rmB_{\lambda}\otimes_{\Qp}L}{\rmU_{\lambda}^1\otimes_{\Qpd}L}.
\end{gathered}
$$

 La proposition est la suivante, elle découle directement du théorème \ref{theo:fondrin}.
\medskip
\begin{prop}\label{prop:diagspe}
Soit $\lambda \in P_+$ tel que $w(\lambda)>1$. On a un diagramme commutatif, $\GQp\times G$-équivariant d'espaces de Fréchet dont les lignes sont exactes 
{
\begin{center} 
\begin{tikzcd}
0\ar[r]&\rmU_{\lambda}^0\otimes_{\Qpd}{W}_{\lambda} \ar[d,hook]\ar[r]&\rmB_{\lambda}\wotimes_{\Qp}\CO_{\lambda}\ar[r]\ar[d,equal]& \rmH^{\lambda}_{\pet}\ar[r]\ar[d]&\rmU_{\lambda}^1\wotimes_{\Qpd} \intn{{\St}^{\lalg}_{\lambda}}'\ar[r]\ar[d,hook] &0\\
 0\ar[r]&\rmB_{\lambda}\wotimes_{\Qp}{W}_{\lambda}\ar[r]&\rmB_{\lambda}\wotimes_{\Qp}\CO_{\lambda}\ar[r] &\rmB_{\lambda} \wotimes_{\Qp}\omega_{\lambda}\ar[r]& \rmB_{\lambda}\wotimes_{\Qp}\intn{{\St}^{\lalg}_{\lambda}}'\ar[r]& 0
\end{tikzcd}
\end{center}}
Les flèches verticales sont d'images fermés.
\end{prop}
Si on prend le quotient par les deux termes de gauche dans le diagramme et que l'on applique le lemme du serpent, on obtient le corollaire suivant :
\medskip
\begin{coro}\label{cor:suitexcool}
On a une suite exacte stricte $\GQp\times G$-équivariante d'espaces de Fréchet
$$
0\rightarrow Q_{\lambda}^0\otimes_L W_{\lambda}\rightarrow  \rmH^{\lambda}_{\pet}\rightarrow \rmB_{\lambda}\wotimes_{\Qp}\omega_{\lambda}\rightarrow Q_{\lambda}^1\wotimes \intn{{\St}^{\lalg}_{\lambda}}'\rightarrow 0.
$$
\end{coro}
\medskip
\begin{rema}\label{rema:morita}
Par l'isomorphisme de Morita (\cf \ref{prop:morita}), la suite exacte courte obtenue à partir du petit complexe de de Rham
$$
0\rightarrow \CO_{\lambda}/W_{\lambda}\xrightarrow{\partial^{w(\lambda)}} \omega_{\lambda}\rightarrow (\St_{\lambda}^{\lalg})'\rightarrow 0,
$$
est le dual de la suite exacte du lemme \ref{lem:extserispeci}
$$
0\rightarrow \St^{\lalg}_{\lambda}\rightarrow \St^{\lan}_{\lambda}\xrightarrow{(u^+)^{w(\lambda)}}B^{\lambda}\rightarrow 0.
$$
\end{rema}
\Subsection{Entrelacements automorphes}
Le but de ce numéro est de démontrer le théorème suivant : 
\medskip
\begin{theo}\label{thm:muautspe}
Soit $\lambda\in P_+$ tel que $w(\lambda)\neq 1$. Soit $\Pi$ une $L$-représentation de Banach unitaire, admissible et absolument irréductible de $G$. Alors, en tant que $L$-représentation de $\GQp$ on a 
$$
\Hom_{L[G]}\intnn{(\Pi^{\lan})'}{\rmH^1_{\pet}\Harg{\BH_C}{\Sym \BV(1)}}\cong
\begin{cases}
 V^{\lambda}_{\CL} & \text{si $\Pi=\Pi_{\CL}^{\lambda}$,}\\
 0& \text{si $\Pi$ n'est pas spéciale.}\\
\end{cases}
$$
\end{theo}
Ce qui combiné au corollaire \ref{cor:complocan} donne :
\medskip
\begin{coro}\label{cor:muautspe}
Soit $\lambda\in P_+$ tel que $w(\lambda)\neq 1$. Soit $\Pi$ une $L$-représentation de Banach unitaire, admissible et absolument irréductible de $G$. Alors, en tant que $L$-représentation de $\GQp$ on a
$$
\Hom_{L[G]}\intnn{\Pi'}{\rmH^1_{\et}\Harg{\BH_C}{\Sym\BV(1)}}\cong
\begin{cases}
 V^{\lambda}_{\CL} & \text{si $\Pi=\Pi_{\CL}^{\lambda}$,}\\
 0& \text{si $\Pi$ n'est pas spéciale.}\\
\end{cases}
$$
\end{coro}
Rappelons que pour $\star\in \{\et,\pet\}$ on a
$$
\rmH^1_{\star}\Harg{\BH_C}{\Sym\BV(1)}=\bigoplus_{\lambda\in P_+} \rmH^1_{\star}\Harg{\BH_C}{\BV_{\lambda}(1)},
$$
Ainsi, il suffit de calculer les entrelacement avec $\rmH^1_{\star}\Harg{\BH_C}{\BV_{\lambda}(1)}$, comme on le fera.
\subsubsection{Réduction de l'entrelacement dérivé}\label{par:redder}
On commence par un lemme pour montrer qu'on peut se ramener au calcul dans la catégorie dérivée.
\medskip
\begin{lemm}\label{lem:dertoproet}
Soit $\lambda \in P_+$ tel que $w(\lambda)\neq 1$ et soit $\Pi$ une $L$-représentation de Banach unitaire, admissible, absolument irréductible de $G$. Alors 
{
$$
\Hom_{L[G]}\intnn{(\Pi^{\lan})'}{\rmH^1_{\pet}\Harg{\BH_C}{\BV_{\lambda}(1)}}=\bHom\ \intnn{(\Pi^{\lan})'[-1]}{\rmR\Gamma_{\!\pet}\Harg{\BH_C}{\BV_{\lambda}(1)}},
$$}

\end{lemm}
\medskip
\begin{rema}
Dans le cas où $w(\lambda)=1$, le $\rmH^0$ n'est pas nul et ce résultat est faux. Dans le cas des coefficients non-triviaux, toute l'information du complexe proétale est contenue dans le premier groupe de cohomologie.
\end{rema}
\medskip
\begin{proof}
Notons $\rmH^i_{\pet}\coloneqq \rmH^i_{\pet}\Harg{\BH_C}{\BV_{\lambda}(1)}$ et $\rmR\Gamma_{\!\pet}\coloneqq\rmR\Gamma_{\!\pet}\Harg{\BH_C}{\BV_{\lambda}(1)}$. On a une suite spectrale qui calcule le terme de droite 
$$
E^{i+1,j}_2\coloneqq \Ext^{i+1}_{L[\bar G]}\intnn{(\Pi^{\lan})'}{\rmH^j_{\pet}}\implies \rmH^{i+j}\rmR\Hom_{L[\bar G]}\intnn{(\Pi^{\lan})'[-1]}{\rmR\Gamma_{\!\pet}}.
$$
Or, puisque $\rmH^j_{\pet}=0$ si $j\neq 0$ d'après le corollaire \ref{cor:anulco}, cette suite spectrale est concentrée sur la colonne $j=1$ et dégénère dès la seconde page. On en déduit le résultat, puisque le seul terme qui contribue au $\rmH^0$ à droite est $\Hom_{L[G]}\intnn{(\Pi^{\lan})'}{\rmH^1_{\pet}}$.
\end{proof}
\subsubsection{Le cas $\Pi=\Pi_{\CL}^{\lambda}$}
À partir de la définition de la cohomologie syntomique et de l'isomorphisme de comparaison, on déduit de la proposition \ref{prop:schraen} la proposition suivante. En combinant cette proposition au lemme \ref{lem:dertoproet}, on obtient la première partie du théorème \ref{thm:muautspe}.
\medskip
\begin{theo}\label{prop:muautder}
Soit $\lambda\in P_+$ et $\CL\in L$, alors 
$$
\bHom\ \intnn{\intn{\Pilan_{\CL}^{\lambda}}'[-1]}{\rmR\Gamma_{\!\pet}\Harg{\BH_C}{\BV_{\lambda}(1)}}=V_{\CL}^{\lambda}.
$$
\end{theo}
\medskip
\begin{proof}
Notons $\Pi={\Pi}_{\CL}^{\lambda}$ et $\rmR\Gamma_{\star}$ le complexe de cohomologie isotriviale de $\BH_C$ à coefficients dans $\BV_{\lambda}(1)$ pour $\star\in\{\pet,\syn,\HK,\dR \}$. Par définition on a un triangle distingué
$$
\rmR\Gamma_{\syn}\rightarrow \left (\rmR\Gamma_{\!\HK}\wotimes_{\Qp}\Bstp\right )^{\varphi=p}\rightarrow \intn{\rmR\Gamma_{\!\dR}\wotimes_{\Qp}\Bdrp}/\Fil^1.
$$
On applique $\rmR\Hom\intnn{\big (\Pi^{\lan}\big )'}{\cdot}$. D'après la proposition, \ref{prop:schraen} on a 
$$
\begin{gathered}
\bHom\ \intnn{\intn{\Pi^{\lan}}'[-1]}{\rmR\Gamma_{\!\HK}}=D_{\CL}^{\lambda},\\
\bHom\ \intnn{\intn{\Pi^{\lan}}'[-1]}{\rmR\Gamma_{\!\dR}}=D_{\CL}^{\lambda}.
\end{gathered}
$$
On en déduit la suite exacte 
$$
0\rightarrow \bHom\ \intnn{\big (\Pi^{\lan}\big )'[-1]}{\rmR\Gamma_{\syn}}\rightarrow X_{\st}^1 ( D_{\CL}^{\lambda} )\rightarrow ( D_{\CL}^{\lambda}\otimes_{\Qp}\Bdrp)/\Fil^1.
$$
Or, la dernière flèche provient de l'inclusion naturelle $X_{\st}^1\left (D_{\CL}^{\lambda}\right )\incl ( D_{\CL}^{\lambda}\otimes_{\Qp}\Bdrp)$, donc la définition de $V_{\CL}^{\lambda}$ assure que
$$
\bHom\ \intnn{\intn{\Pi^{\lan}}'[-1]}{\rmR\Gamma_{\syn}}\cong V_{\CL}^{\lambda}.
$$
Mais la suite spectrale $E_2^{i,j}$ qui calcule le terme de gauche est nulle en dehors des colonnes $j=0,1$ en vertu du corollaire \ref{cor:anulco}.
Finalement, le théorème \ref{thm:petsyn} nous donne un isomorphisme
$$
\bHom\ \intnn{\intn{\Pi^{\lan}}'[-1]}{\rmR\Gamma_{\syn}}\cong \bHom\ \intnn{\intn{\Pi^{\lan}}'[-1]}{\rmR\Gamma_{\!\pet}},
$$
ce qui achève la preuve.
\end{proof}
\subsubsection{Exclusivité}
On montre maintenant la seconde partie du théorème \ref{thm:muautspe}. 
\medskip
\begin{prop}
Soit $\lambda\in P_+$ tel que $w(\lambda)\neq 1$. Soit $\Pi$ une $L$-représentation de Banach unitaire admissible absolument irréductible de $G$. Si $\Pi$ n'est pas de la forme $\Pi_{\CL}^{\lambda}$, alors
$$
\Hom_{L[G]}\intnn{(\Pi^{\lan})'}{\rmH^1_{\pet}\Harg{\BH_C}{\BV_{\lambda}(1)}}=0.
$$
\end{prop}
\medskip
\begin{proof}
On utilise la suite exacte du corollaire \ref{cor:suitexcool}, issue du diagramme fondamental, pour obtenir une suite exacte courte :
$$
0\rightarrow Q_{\lambda}^0\otimes_L W_{\lambda}\rightarrow  \rmH^{\lambda}_{\pet}\rightarrow K\rightarrow 0
$$
où $K\coloneqq \Ker(\rmB_{\lambda}\otimes_{\Qp}\omega_{\lambda}\rightarrow Q_{\lambda}^1\wotimes_L \intn{{\St}^{\lalg}_{\lambda}}')$. De plus, la dualité de Morita assure que $\omega_{\lambda}\cong\intn{{\St}^{\lan}_{\lambda}}'$ et comme $\Pi^{\lan}\not \cong W_{\lambda}$ puisque $\Pi$ est supposé unitaire on obtient des plongements fermés
\begin{equation}\label{eq:injj}
\Hom_{L[G]}\intnn{(\Pi^{\lan})'}{\rmH^{\lambda}_{\pet}}\incl \Hom_{L[G]}\intnn{(\Pi^{\lan})'}{K}\incl \Hom_{L[G]}\intnn{{\St}_{\lambda}^{\lan}}{\Pi^{\lan}}\wotimes_{\Qp}\rmB_{\lambda}.
\end{equation}
Montrons que si $\Pi^{\lalg} = 0$, alors $\Hom_{L[G]}\intnn{{\St}_{\lambda}^{\lan}}{\Pi^{\lan}} = 0$. Puisque d'après le lemme \ref{lem:extserispeci} on a la suite exacte 
$$
0\rightarrow{\St}^{\lalg}_{\lambda}\rightarrow {\St}^{\lan}_{\lambda}\xrightarrow{(u^+)^{w(\lambda)}}B^{\lambda}\rightarrow 0,
$$
et que $\Hom_{L[G]}\intnn{{\St}^{\lalg}_{\lambda}}{\Pi^{\lan}} =\Hom_{L[G]}\intnn{{\St}^{\lalg}_{\lambda}}{\Pi^{\lalg}}=0$, on en déduit que 
$$
\Hom_{L[G]}\intnn{{\St}_{\lambda}^{\lan}}{\Pi^{\lan}}\cong \Hom_{L[G]}\intnn{B^{\lambda}}{\Pi^{\lan}}.
$$
Il faut justifier que $\Hom_{L[G]}\intnn{B^{\lambda}}{\Pi^{\lan}}=0$. Soit $B^{\lambda}\rightarrow \Pi^{\lan}$ une application $L$-linéaire, continue et $G$-équivariante. Par la propriété universelle du complété unitaire universel, la composée $B^{\lambda}\rightarrow \Pi^{\lan}\incl \Pi$ se factorise à travers le complété unitaire universel de $B^{\lambda}$. Or, ce complété est nul d'après le critère donné par \cite[Lemma 2.1]{emep} et donc la flèche initiale $B^{\lambda}\rightarrow \Pi^{\lan}$ est nulle.

Supposons que $\Hom_{L[G]}\intnn{(\Pi^{\lan})'}{\rmH^{\lambda}_{\pet}}\neq 0$. On a montré que $\Pi^{\lalg}\neq 0$, ce qui signifie que le socle de $\Pi^{\lan}$ est contenu dans $\Pi^{\lalg}$. On déduit de l'injection (\ref{eq:injj}) que \hbox{$\Hom_{L[G]}\intnn{{\St}_{\lambda}^{\lan}}{\Pi^{\lan}}\neq 0$}. Soit $f\colon {\St}_{\lambda}^{\lan}\rightarrow{\Pi^{\lan}}$ une application non nulle. Le socle de ${\St}_{\lambda}^{\lan}$ est ${\St}_{\lambda}^{\lalg}$ mais comme $f\neq 0$ et $\Hom_{L[G]}\intnn{B^{\lambda}}{\Pi^{\lan}}=0$, $f$ se restreint en une application non nulle ${\St}_{\lambda}^{\lalg}\rightarrow \Pi^{\lan}$ qui est injective puisque ${\St}_{\lambda}^{\lalg}$ est irréductible. Cette application induit, en composant avec l'injection $\Pi^{\lan}\incl \Pi$, une injection ${\St}_{\lambda}^{\lalg}\incl \Pi$. Mais d'après  \cite[Theorem 1.3]{codopa}, les complétés unitaires admissibles absolument irréductibles de ${\St}_{\lambda}^{\lalg}$ sont les $\Pi_{\CL}^{\lambda}$ et on en déduit qu'il existe un $\CL\in L$ tel que  $\Pi_{\CL}^{\lambda}\cong \Pi$.

\end{proof}
\subsubsection{Le cas réductible}
Dans le cas où $w(\lambda)=1$ les représentations $V_{\CL}^{\lambda}$ et $\Pi_{\CL}^{\lambda}$ sont réductibles et le théorème \ref{thm:muautspe} et son corollaire tombent en défaut sous cette forme puisque comme $\rmH^0_{\pet}\Harg{\BH_C}{\Qp(\lambda_1)}\neq 0$. Néanmoins, à partir du théorème \ref{prop:muautder}, on peut toutefois calculer la multiplicité d'une représentation indécomposable de $G$ dans le complexe de cohomologie étale et obtenir le résultat attendu. On n'obtient qu'une exclusivité partielle que l'on énonce à l'aide de la théorie des blocs de Paskunas (\cf  \cite{pas}). 
\medskip
\begin{prop}\label{prop:multred}
Soit $\lambda\in P_+$. Soit $\Pi$ une $L$-représentation unitaire, admissible et indécomposable de $G$, alors
$$
\bHom\ \intnn{\Pi'[-1]}{\rmR\Gamma_{\!\et}\Harg{\BH_C}{\BV_{\lambda}(1)}}=
\begin{cases}
  V^{\lambda}_{\CL} & \text{ si $\Pi$ est de la forme $\Pi_{\CL}^{\lambda}$,}\\
  0 & \text{ si $\Pi$ n'est pas dans le bloc de la Steinberg.}
\end{cases}
$$
\end{prop}
\begin{proof}
On a démontré ce résultat dans le cas où $w(\lambda)\neq 1$ donc on peut supposer $w(\lambda)=1$ et même $\lambda_1=0$, quitte à tordre par une puissance du caractère cyclotomique. Notons $\rmH^i_{\et}=\rmH^i_{\et}\Harg{\BH_C}{\Qp(1)}$, alors la suite spectrale
$$
E^{i+1,j}_2\coloneqq \Ext^{i+1}_{L[\bar G]}\intnn{\Pi'}{\rmH^j_{\et}}\implies \rmH^{i+j}\rmR\Hom\intnn{\Pi'[-1]}{\rmR\Gamma_{\!\et}}
$$
nous donne la suite exacte courte 
$$
0\rightarrow \Ext^{1}_{L[\bar G]}\intnn{\Pi'}{\rmH^0_{\et}}\rightarrow \bHom\ \intnn{\Pi'[-1]}{\rmR\Gamma_{\!\et}}\rightarrow\Hom_{L[G]}\intnn{\Pi'}{\rmH^1_{\et}}\rightarrow 0,
$$
où $\bH_{\et}\coloneqq  \bHom\ \intnn{\Pi'[-1]}{\rmR\Gamma_{\!\et}}$.
Puisque $\rmH^0_{\et}\cong L(1)$ et $\rmH^1_{\et}\cong \intn{\St^0}'$, la suite exacte ci-dessus s'écrit 
$$
0\rightarrow \rmH^1\Harg{G}{\Pi}(1)\rightarrow \bH_{\et}\rightarrow\Hom_{L[G]}\intnn{\St^0_L}{\Pi}\rightarrow 0.
$$
\begin{itemize}
\itemb Si $\Pi$ n'est pas dans le bloc de la Steinberg, alors $\Hom_{L[G]}\intnn{\St^0_L}{\Pi}=0$ et comme la représentation triviale est aussi dans ce bloc, on obtient $\rmH^1\Harg{G}{\Pi}=0$. Ainsi $\bH_{\et}=0$ dans ce cas. 
\itemb Supposons maintenant que $\Pi=\Pi_{\CL}^{[0,1]}$ pour $\CL\in L$. On obtient la même suite exacte que précédemment pour la cohomologie proétale :
$$
0\rightarrow \rmH^1\Harg{G}{\Pi^{\lan}}(1)\rightarrow \bH_{\pet}\rightarrow\Hom_{L[G]}\intnn{\intn{\Pi^{\lan}}'}{\rmH^1_{\pet}}\rightarrow 0.
$$
Montrons que $\Hom_{L[G]}\intnn{\intn{\Pi^{\lan}}'}{\rmH^1_{\pet}}\cong \Hom_{L[G]}\intnn{\St^{\infty}}{\Pi^{\lan}}$. On a une flèche naturelle, qui provient du diagramme de la proposition \ref{prop:diagspe}
$$
\Hom_{L[G]}\intnn{\intn{\Pi^{\lan}}'}{\rmH^1_{\pet}}\rightarrow \Hom_{L[G]}\intnn{\St^{\infty}}{\Pi^{\lan}}.
$$
Pour montrer que cette flèche est un isomorphisme on considère le diagramme commutatif suivant :
\begin{center} 
\begin{tikzcd}
\Hom_{L[G]}\intnn{\Pi'}{\rmH^1_{\et}}\ar[d]\ar[r] &\Hom_{L[G]}\intnn{\St^{0}}{\Pi}\ar[d]\\
\Hom_{L[G]}\intnn{\intn{\Pi^{\lan}}'}{\rmH^1_{\pet}}\ar[r] &\Hom_{L[G]}\intnn{\St^{\infty}}{\Pi^{\lan}}
\end{tikzcd}
\end{center}
La flèche horizontale en haut est un isomorphisme, puisque $\rmH^1_{\et}\cong \intn{\St^{0}}'$. Comme $\St^0$ est le complété unitaire universel de $\St^{\infty}$ et $\Pi$ celui de $\Pi^{\lan}$ d'après \cite{codo}, la flèche verticale de droite est un isomorphisme. Finalement, d'après le corollaire \ref{cor:complocan}, la flèche verticale de gauche est un isomorphisme. On en déduit que la flèche horizontale en bas est un isomorphisme, comme on voulait. 

D'après le théorème \ref{prop:muautder}, on sait que $\bH_{\pet}\cong V_{\CL}^{[0,1]}$ et on obtient un diagramme commutatif, $\GQp$-équivariant à lignes exactes
{
\begin{center} 
\begin{tikzcd}
0 \ar[r]& \rmH^1\Harg{G}{\Pi}(1)\ar[d,hook]\ar[r]&\bH_{\et}\ar[d]\ar[r]&\Hom_{L[G]}\intnn{\St^0_L}{\Pi}\ar[d,equal]\ar[r]& 0\\
0 \ar[r]& \rmH^1\Harg{G}{\Pi^{\lan}}(1)\ar[r]&V_{\CL}^{[0,1]}\ar[r]&\Hom_{L[G]}\intnn{\St^{\infty}_L}{\Pi^{\lan}}\ar[r]& 0
\end{tikzcd}
\end{center}}
Pour justifier que la flèche verticale de gauche est bien injective, il faut montrer qu'une extension de $L$ par $\Pi$, qui se scinde lorsqu'on prend les vecteurs localement analytiques, est scindée. Supposons que $\bE$ est une $L$-représentation de Banach unitaire de $G$, extension de $L$ par $\Pi$ :
$$
0\rightarrow \Pi \rightarrow \bE\rightarrow L\rightarrow 0.
$$
Supposons que cette extension se scinde lorsqu'on prend les vecteurs localement analytiques, \ie  $\bE^{\lan}\cong\Pi^{\lan}\oplus L$. On prend maintenant le complété unitaire universel pour obtenir $\bE\cong\Pi\oplus L$ (\cf  \cite{codo}), c'est-à-dire que l'extension initiale était scindée.

De plus, notons que $\Hom_{L[G]}\intnn{\St^{\infty}_L}{\Pi}$ est un $L$-espace vectoriel de dimension $1$ et donc $\rmH^1\Harg{G}{\Pi^{\lan}}$ est un $L$-espace vectoriel de dimension $1$. Mais à partir du diagramme, on obtient de plus que $\rmH^1\Harg{G}{\Pi}$ est un $L$-espace vectoriel de dimension $1$ puisque s'il était nul, on obtiendrait une injection $\GQp$-équivariante $L\incl V_{\CL}^{[0,1]}$ ce qui est impossible. Les flèches verticales dans le diagramme ci-dessus sont donc des isomorphismes et on a bien démontré que $\bH_{\et}\cong V_{\CL}^{[0,1]}$.
\end{itemize}
\end{proof}

\Subsection{Entrelacement des $V_{\CL}^{\lambda}$}
Le but de ce numéro est de montrer le théorème suivant :
\medskip
\begin{theo}\label{thm:mugalspe}
Soit $\lambda \in P_+$ tel que $w(\lambda)\neq1$,
$$
\Hom_{L[\GQp ]}\intnn{V_{\CL}^{\lambda}}{\rmH^1_{\pet}\Harg{\BH_C}{\BV_{\lambda}(1)}}\cong \intn{{\Sigma}_{\CL}^{\lambda}}'\oplus {W}_{\lambda}.
$$
\end{theo}
\medskip
une conséquence de ce théorème et du théorème \ref{thm:proetoet} est la multiplicité étale : en effet, ${W}_{\lambda}$ n'a pas de vecteurs $G$-bornés comme $w(\lambda)>1$ et le complété unitaire universel de ${\Sigma}_{\CL}^{\lambda}$ et $\Pi_{\CL}^{\lambda}$ comme on l'a rappelé au numéro \ref{subsec:seriespe}.
\medskip
\begin{coro}\label{thm:mugalspeet}
Soit $\lambda \in P_+$ tel que $w(\lambda)\neq1$,
$$
\Hom_{L[\GQp ]}\intnn{V_{\CL}^{\lambda}}{\rmH^1_{\et}\Harg{\BH_{\Qp}}{\Sym\BV(1)}}\cong (\Pi_{\CL}^{\lambda})'.
$$
\end{coro}
\medskip
On supposera toujours que $w(\lambda)\neq1$.

\subsubsection{Les $L$-presque-$C$-représentations}
On va utiliser $\CC_L$ la catégorie des $L$-presque-$C$-représentations, construite à partir de la catégorie des presque-$C$-représentations (cf. \cite{focp}), noté $\CC$, munis d'une structure de $L$-espace vectoriel. Comme pour les presques-$C$-représentations, $\CC_L$ est muni d'une $L$-hauteur $h_L$ et d'une dimension principale $d$. Si $X$ et $Y$ sont des $L$-presque-$C$-représentations, les groupes $\Ext^n_{\CC_L}(X,Y)$ sont des $L$-espaces vectoriels de dimension finie, nuls si $n\geqslant 3$ et pour $i=0,1,2$, on a une dualité parfaite
$$
\Ext_{\CC_L}^i(X,Y)\times \Ext_{\CC_L}^{2-i}(Y,X(1))\rightarrow L.
$$ 
où $X(1)$ est le tordu de $X$ par le caractère cyclotomique. De plus, on a une formule d'Euler-Poincaré qui s'écrit $\chi_L(X,Y) = -h_L(X)h_L(Y)$ où
$$
\chi_L(X,Y)\coloneqq \sum_{i=0}^2(-1)^i\dim_{L}\Ext_{\CC_L}^i(X,Y).
$$
Tous ces résultats se déduisent directement des résultats correspondant dans $\CC$.

On peut réinterpréter les presque-$C$-représentations (et surement aussi les $L$-presque-$C$-représentations ce qui ne sera pas nécessaires) en termes de fibrés sur la courbe de Fargues-Fontaine$\CX_{\FF}=\CX_{\FF,C}$. Notons $\CM$ la catégorie des $\CO_{\FF}\coloneqq \CO_{\CX_{\FF}}$-modules cohérents $\GQp$-équivariants sur la courbe. Cette catégorie a une théorie des pentes de Harder-Narasimhan et on peut considérer la sous catégorie $\CM^+\subset \CM$ des fibrés effectifs, \ie ceux qui sont de pentes positives (\cf \cite{fafo}). On rappelle le théorème $A$ de \cite{fove} sous la forme d'une proposition
\medskip
\begin{prop}\label{prop:fove}
Il existe un foncteur $\CF\mapsto \rmH^0\Harg{\CX_{\FF}}{\sF}$ de \emph{sections globales}, à valeurs dans~$\CC$, qui induit une équivalence de catégories $\CM^+\cong\CC^+$.
\end{prop}
\medskip

Soit $d=[L\coloneqq \Qp]$, alors par hypothèse, les presque-$C$-représentations qui apparaissent dans le diagramme sont les suivantes :
\begin{itemize}
\itemb $\rmB_{\lambda}(L)\coloneqq \rmB_{\lambda}\otimes_{\Qp}L$ de dimension principale $d\cdot w(\lambda)$ et de $L$-hauteur $0$,
\itemb $\rmU_{\lambda}^i(L)\coloneqq \rmU_{\lambda}^0\otimes_{\Qpd }L$, pour $i=0,1$, de $L$-hauteur $1$ et de dimension principale $\frac{d}{2}\cdot w(\lambda)$ ; rappelons qu'on a la suite exacte
$$
0\rightarrow \Qpd t^{\lambda_1+1}t_2^{w(\lambda)-1}\rightarrow \rmU_{\lambda}^0\rightarrow \rmB_{\lambda}\rightarrow 0,
$$
\itemb $Q_{\lambda}^i=\rmB_{\lambda}(L)/\rmU_{\lambda}^i(L)$ pour $i=0,1$ est de $L$-hauteur $-1$ et de dimension principale $\frac{d}{2}\cdot w(\lambda)$
\end{itemize}
On commence par un lemme préliminaire :
\medskip
\begin{lemm}\label{lem:pcr}
Soit $V$ une $L$-représentation de $\GQp$ de dimension finie et $U\in\{\rmU_{\lambda}^0(L),\rmB_{\lambda}(L),\rmB_{\lambda}(L)/\rmU_{\lambda}^0(L)\}$. Alors
$$
\Ext^2_{\CC_{L}}\intnn{V}{U}=0.
$$
\end{lemm}
\medskip
\begin{proof}
On pose $w\coloneqq w(\lambda)$, il suffit de montrer le résultat pour $\lambda_1=0$, quitte à tordre par $t^{\lambda_1}$ (ainsi $w=\lambda_2$). En utilisant la dualité de Poincaré, il suffit de montrer que $\Hom_{\CC_{L}}\intnn{U}{V}=0$. Commençons par le montrer pour $\rmU_{\lambda}^0(L)$ ce qui revient à le montrer dans $\CC$ pour $U=\rmU_{\lambda}^0$. Or $U$ est une presque-$C$-représentation effective et on utilise la proposition \ref{prop:fove}, ce qui donne 
$$
\rmH^0\Harg{\CX_{\FF}}{\CO\intn{\tfrac {w+1} 2}}= U,\ \rmH^0\Harg{\CX_{\FF}}{\CO\otimes_{\Qp} V}= V.
$$
Comme $\frac {w+1} 2>0$, on a
$$
\Hom_{\CX_{\FF}}\intnn{\CO\intn{\tfrac {w+1} 2}}{\CO\otimes_{\Qp} V}=0.
$$
Ainsi, d'après la proposition \ref{prop:fove}, a fortiori $\Hom_{\CC}\intnn{U}{V}=0$ et de même \hbox{$\Hom_{\CC_{L}}\intnn{\rmU_{\lambda}^0(L)}{V}=0$}. Pour $U=\rmB_{\lambda}$, comme on a une surjection $\rmU_{\lambda}^0\rightarrow \rmB_{\lambda}$ on en déduit une injection 
$$
\Hom_{\CC}\intnn{\rmB_{\lambda}}{V}\incl\Hom_{\CC}\intnn{\rmU_{\lambda}^0}{V}=0,
$$
et donc $\Hom_{\CC}\intnn{\rmB_{\lambda}}{V}=0$ et de même $\Hom_{\CC_L}\intnn{\rmB_{\lambda}(L)}{V}=0$. Finalement, la surjection $\rmB_{\lambda}(L)\rightarrow Q_{\lambda}^0$ donne l'injection 
$$
\Hom_{\CC_L}\intnn{Q_{\lambda}^0}{V}\incl\Hom_{\CC_L}\intnn{\rmB_{\lambda}(L)}{V}=0,
$$
ce qui démontre que $\Hom_{\CC_L}\intnn{Q_{\lambda}^0}{V}=0$.
\end{proof}
\begin{lemm}\label{lem:galcalc}
Soit $V=V_{\CL}^{\lambda}$. Alors on a 
\begin{enumerate}
\item $\Hom_{L[\GQp ]}\intnn{V}{\rmU_{\lambda}^i(L)}$ est un $L$-espace vectoriel de dimension $1$ pour $i=0,1$,
\item $\Hom_{L[\GQp ]}\intnn{V}{\rmB_{\lambda}(L)}$ et $\Ext^1_{L[\GQp ]}\intnn{V}{\rmB_{\lambda}(L)}$ sont des $L$-espaces vectoriels de dimension~$1$.
\item $\Ext^1_{L[\GQp ]}\intnn{V}{\rmU_{\lambda}^0(L)}$ est un $L$-espace vectoriel de dimension $3$,
\item $\Ext^1_{L[\GQp ]}\intnn{V}{Q_{\lambda}^0}=0$ et $\Hom_{L[\GQp ]}\intnn{V}{Q_{\lambda}^0}$ est un $L$-espace vectoriel de dimension $2$. 
\end{enumerate}
\end{lemm}
\begin{proof}
On omettra les indices de $\Hom$ et $\Ext$, qui sont tous dans la catégorie $\CC_L$. Comme précédemment, posons $w\coloneqq w(\lambda)$ et supposons que $\lambda_1=0$ (ainsi $w=\lambda_2$); le cas $\lambda_1>0$ s'en déduit directement en tordant par un caractère. 
\begin{itemize}
\itemb Rappelons que $\bD_{\st}(V^*)=\Sp_L(w)^*=Le_0^{*}\oplus Le_1^{*}$ avec $\varphi(e_0^{*})=p^{(w+1)/2}e_0^*$, $\varphi(e_1^{*})=p^{(w-1)/2}e_1^*$, $Ne_0^*=e_1^*$ et $Ne_1^*=0$. De plus, $\Bstp=\Bcrisp[u]$ avec $Nu=-1$ et $\varphi(u)=p$. On commence par remarquer que, comme $V$ est semi-stable, on a 
$$
V^*\otimes_{\Qp }\Bcris\cong \intn{\Sp_L(w)^*\otimes_{\Qp^{\nr}}\Bstp}^{N=0}=e_1^*\Bcris(L)\oplus (e_0^*+ue_1^*)\Bcris(L),
$$
ce qui donne 
$$
\begin{gathered}
V^*\otimes_{\Qp}\intn{\Bcrisp}^{\varphi^2=p^{w+1}}\cong \intn{\Bcris(L)}^{\varphi^2=p^{2}}\oplus \intn{\Bcris(L)}^{\varphi^2=1},\\
V^*\otimes_{\Qp}\intn{\Bcrisp}^{\varphi^2=p^{w-1}}\cong \intn{\Bcris(L)}^{\varphi^2=1}\oplus \intn{\Bcris(L)}^{\varphi^2=p^{-2}}.
\end{gathered}
$$
Ceci nous donne le premier point puisque $\intn{\Bcris(L)}^{\varphi^2=p^{-2}}$ et $\intn{\Bcris(L)}^{\varphi^2=p^2}$ n'ont pas d'invariants sous $\GQp$ et les invariants sous $\GQp$ de $\intn{\Bcrisp(L)}^{\varphi^2=1}$ sont $L$.

\itemb Passons au point $2$. Comme $V$ est de Rham, on a 
$$
V^*\otimes_{L}\rmB_{\lambda}(L)\cong\Fil^0\intn{\bD_{\dR}(V^*)\otimes_{L}\Bdr(L)}/\Fil^w\intn{\bD_{\dR}(V^*)\otimes_L\Bdr(L)}.
$$
On choisit une base compatible avec la filtration $\bD_{\dR}(V^*)=Lf_{w}\oplus L f_0$, ce qui donne
$$
\begin{gathered}
\Fil^0\intn{\bD_{\dR}(V^*)\otimes_{\Qp}\Bdr}=t^{-w}f_w\Bdrp(L)\oplus f_0\Bdrp(L),\\
\Fil^w\intn{\bD_{\dR}(V^*)\otimes_{\Qp}\Bdr}=f_w\Bdrp(L)\oplus t^{w}f_0\Bdrp(L).
\end{gathered}
$$
Donc $V^*\otimes_{\Qp}\rmB_w \cong f_w t^{-w}\rmB_w\oplus f_0\rmB_w$, dont la première composante n'a pas d'invariant sous $\GQp$ et la seconde a exactement $L$ comme invariants sous $\GQp$ ; ainsi $\Hom\intnn{V}{\rmB_{\lambda}(L)}=L$. Pour la seconde partie du second point on remarque que $\chi_L(V,\rmB_{\lambda})=0$ mais comme par le lemme~\ref{lem:pcr} $\Ext^2\intnn{V}{\rmB_{\lambda}(L)}=0$, ceci donne $\dim_L\Ext^1\intnn{V}{\rmB_{\lambda}(L)}=\dim_L\Hom\intnn{V}{\rmB_{\lambda}(L)}=1$, ce qui finit la preuve du second point.

\itemb Pour le troisième point, la formule d'Euler-Poincaré donne $\chi_L(V,\rmU_{\lambda}^0(L))=-2$. Le lemme \ref{lem:pcr} donne $\Ext^2\intnn{V}{\rmU_{\lambda}^0(L)}=0$. Donc $\dim_L\Ext^1\intnn{V}{\rmU_{\lambda}^0(L)}=2+\dim_L\Hom\intnn{V}{\rmU_{\lambda}^0(L)}=3$ en vertu du premier point.

\itemb Le dernier point est plus délicat, commençons par montrer que les deux énoncés sont équivalents. La formule d'Euler-Poincaré nous assure que 
$$
\chi_L(V,Q_{\lambda}^0)=2
$$
et, par le lemme \ref{lem:pcr}, $\Ext^2\intnn{V}{Q_{\lambda}^0}=0$, ce qui donne
$$
\dim_L\Hom\intnn{V}{Q_{\lambda}^0}=2+\dim_L\Ext^1\intnn{V}{Q_{\lambda}^0}.
$$
Il reste à montrer $\Ext^1\intnn{V}{Q_{\lambda}^0}=0$. On applique $\Hom\intnn{V}{\ \cdot \ }$ à la suite exacte courte
$$
0\rightarrow \rmU_{\lambda}^0(L)\rightarrow \rmB_{\lambda}(L)\rightarrow Q_{\lambda}^0\rightarrow 0,
$$
pour obtenir la suite exacte 
$$
\Ext^1\intnn{V}{\rmU_{\lambda}^0(L)}\rightarrow \Ext^1\intnn{V}{\rmB_{\lambda}(L)}\rightarrow  \Ext^1\intnn{V}{Q_{\lambda}^0}\rightarrow 0,
$$
où le zéro à droite provient de l'annulation de $\Ext^2\intnn{V}{\rmU_{\lambda}^0(L)}$ donnée par le lemme \ref{lem:pcr}. Il suffit de montrer que l'application
$$
\Ext^1\intnn{V}{\rmU_{\lambda}^0(L)}\rightarrow \Ext^1\intnn{V}{\rmB_{\lambda}(L)}
$$
est surjective. On applique maintenant $\Hom\intnn{V}{\ \cdot \ }$ à la suite exacte fondamentale (à laquelle on a appliqué $\cdot \ \otimes_{\Qpd}L$)
$$
0\rightarrow L\cdot t_2^{w-1}t \rightarrow \rmU_{\lambda}^0(L)\rightarrow \rmB_{\lambda}\otimes_{\Qpd}L\rightarrow 0,
$$
pour obtenir la suite exacte
$$
\Ext^1\intnn{V}{\rmU_{\lambda}^0(L)}\rightarrow \Ext^1\intnn{V}{\rmB_{\lambda}\otimes_{\Qpd}L}\rightarrow \Ext^2\intnn{V}{L\cdot t_2^{w-1}t}.
$$
Mais par la dualité $\Ext^2\intnn{V}{L\cdot t_2^{w-1}t}\cong  \Hom\intnn{L\cdot t_2^{w-1}t}{V(1)}=0$ puisque les deux représentations sont irréductibles et ne sont pas isomorphes. Ainsi $\Ext^1\intnn{V}{\rmU_{\lambda}^0(L)}\rightarrow \Ext^1\intnn{V}{\rmB_{\lambda}\otimes_{\Qpd}L}$ est surjective. Il suffit de montrer que l'application naturelle 
$$ 
g\colon \Ext^1\intnn{V}{\rmB_{\lambda}(L)}\rightarrow \Ext^1\intnn{V}{\rmB_{\lambda}\otimes_{\Qpd}L}
$$
est un isomorphisme ; elle est induite par la surjection naturelle $p\colon \rmB_{\lambda}\otimes_{\Qp}L=\rmB_{\lambda}(L)\rightarrow \rmB_{\lambda}\otimes_{\Qpd}L$. Il existe $\alpha'\in \Bdrp$ et $\alpha\in L$ tel que le noyau de $p$ soit engendré par $e\coloneqq \alpha'\otimes 1 - 1\otimes \alpha$, \ie  $\Ker p = e\cdot \rmB_{\lambda}(L)$. Notons que $\GQp$ agit non-trivialement sur $e$ au travers de son quotient non-ramifié $\GQp\rightarrow \BZ/2\BZ$ (soit au travers du caractère non-ramifié $\chi_2$). Ainsi, par la suite exacte longue que l'on obtient en appliquant $\Hom\intnn{V}{\cdot}$ à
$$
0\rightarrow \Ker p \rightarrow \rmB_{\lambda}(L)\rightarrow \rmB_{\lambda}\otimes_{\Qpd}L\rightarrow 0,
$$
il suffit de montrer que $\Ext^1\intnn{V}{\Ker p}=0$ et $\Ext^2\intnn{V}{\Ker p} = 0$. Or, d'après le lemme \ref{lem:pcr}, $\Ext^2\intnn{V\otimes \chi_2}{\rmB_{\lambda}(L)} = 0$ et comme $\chi_L(V,\Ker p)=0$, il reste à montrer que $\Hom\intnn{V}{\Ker p}=0$ pour conclure, ce que l'on fait par la même méthode que pour le second point. En gardant les mêmes notations, on obtient 
$$
V^*\otimes_{\Qp}\Ker p \cong e f_w t^{-w}\rmB_{\lambda}(L)\oplus e f_0\rmB_{\lambda}(L),
$$
qui n'a pas d'invariant sous $\GQp$ puisqu'il agit non-trivialement sur $e$ au travers d'un quotient fini. Ainsi, $\Hom\intnn{V}{\Ker p}=0$ et on obtient que $g$ est un isomorphisme, ce qui achève la preuve du lemme.
\end{itemize}
\end{proof}

\subsubsection{Preuve du théorème \ref{thm:mugalspe}}
On se donne $\lambda\in P_+$ et $\CL\in L$. On commence par appliquer $\Hom_{L[\GQp ]}\intnn{V_{\CL}^{\lambda}}{\ \cdot\ }$ à la suite exacte du corollaire \ref{cor:suitexcool}. Posons
$$
E_2\coloneqq \Hom_{L[\GQp ]}\intnn{V_{\CL}^{\lambda}}{Q_{\lambda}^0},
$$
qui est un $L$-espace vectoriel de dimension $2$ d'après le $4$ du lemme \ref{lem:galcalc}. On obtient
\medskip
\begin{prop}\label{prop:mugalspe}
On a une suite exacte stricte de $L$-représentations de Fréchet de $G$
\begin{equation}\label{eq:extcoh}
0\rightarrow  E_2\otimes_L{W}_{\lambda}\rightarrow  \Hom_{L[\GQp ]}\intnn{V_{\CL}^{\lambda}}{\rmH^{\lambda}_{\pet}} \rightarrow \intn{{\St}^{\lan}_{\lambda}}'\rightarrow 0.
\end{equation}
\end{prop}
\begin{proof}
Notons $V=V_{\CL}^{\lambda}$ et simplement $\Hom=\Hom_{L[\GQp]}$ et $\Ext^1=\Ext^1_{L[\GQp]}$ comme précédemment. À partir du corollaire \ref{cor:suitexcool} on obtient la suite exacte 
\begin{equation}\label{eq:exactseqspes}
0\rightarrow Q_{\lambda}^0\otimes_L W_{\lambda}\rightarrow  \rmH^{\lambda}_{\pet}\rightarrow K\rightarrow 0
\end{equation}
avec $K\coloneqq \Ker(\rmB_{\lambda}(L)\otimes_{L}\omega_{\lambda}\rightarrow Q_{\lambda}^1\wotimes_L \intn{{\St}^{\lalg}_{\lambda}}')$. On va montrer que $\Hom_{L[\GQp ]}\intnn{V}{K}\cong \omega_{\lambda}$. Comme $K\incl \rmB_{\lambda}(L)\otimes_{L}\omega_{\lambda}$, on obtient une inclusion $\Hom_{L[\GQp ]}\intnn{V}{K}\incl \omega_{\lambda}$ puisque $\Hom_{L[\GQp ]}\intnn{V}{\rmB_{\lambda}(L)}\cong L$ d'après le lemme \ref{lem:galcalc}. On dévisse $K$ ; le lemme du serpent appliqué au diagramme commutatif suivant
\begin{center}
\begin{tikzcd}
0\ar[r]& \rmB_{\lambda}\wotimes_{\Qp} \CO_{\lambda}/W_{\lambda}\ar[r]\ar[d, hook] & \rmB_{\lambda}\wotimes_{\Qp} \omega_{\lambda} \ar[r]\ar[d,equal] &\rmB_{\lambda}\otimes_{\Qp}(\St_{\lambda}^{\lalg})'\ar[r]\ar[d,two heads]& 0\\
0\ar[r] & K\ar[r] & \rmB_{\lambda}\wotimes_{\Qp} \omega_{\lambda} \ar[r] &Q_{\lambda}^0\otimes_{L}(\St_{\lambda}^{\lalg})'\ar[r]& 0.
\end{tikzcd}
\end{center}
donne la suite exacte 
$$
0\rightarrow \rmB_{\lambda}(L)\wotimes_{L} \CO_{\lambda}/W_{\lambda} \rightarrow K \rightarrow \rmU_{\lambda}^1\otimes_{L}(\St_{\lambda}^{\lalg})'\rightarrow 0.
$$
On lui applique $\Hom(V,\cdot)$ pour obtenir à l'aide du lemme \ref{lem:galcalc} la suite exacte
\begin{equation}\label{eq:bbbb}
0\rightarrow \CO_{\lambda}/W_{\lambda}\rightarrow \Hom_{L[\sG_{\Qp}]}(V,K)\rightarrow (\St_{\lambda}^{\lalg})'\xrightarrow{0}\CO_{\lambda}/W_{\lambda},
\end{equation}
où la dernière flèche est nulle puisque $(\St_{\lambda}^{\lalg})'$ et $\CO_{\lambda}/W_{\lambda}$ sont deux $L$-représentations irréductibles de $G$ non-isomorphes. Or $\Hom_{L[\sG_{\Qp}]}(V,K)\incl \omega_{\lambda}$ et comme elles ont les mêmes composantes d'après la suite exacte (\ref{eq:bbbb}) (\cf la remarque \ref{rema:morita}) cette inclusion est un isomorphisme. Finalement, on applique $\Hom(V,\ \cdot \ )$ à \ref{eq:exactseqspes} pour obtenir la suite exacte
$$
0\rightarrow E_2\otimes_L W_{\lambda}\rightarrow  \Hom(V,\rmH^{\lambda}_{\pet})\rightarrow (\St_{\lambda}^{\lan})'\rightarrow 0
$$
puisque $\Ext^1(V,Q_{\lambda}^0)=0$ d'après le quatrième point du lemme \ref{lem:galcalc}, ce qui conclut la preuve de la proposition.
\end{proof}
On peut maintenant démontrer le théorème \ref{thm:mugalspe}. Notons $\rmH_{\CL}\coloneqq \Hom_{L[\GQp ]}\intnn{V_{\CL}^{\lambda}}{\rmH^{\lambda}_{\pet}}$. Il suffit d'analyser l'extension que l'on obtient dans la proposition \ref{prop:mugalspe}. Pour commencer, on applique $\Hom_{L[G]}\intnn{\intn{{\St}^{\lan}_{\lambda}}'}{\ \cdot\ }$ à (\ref{eq:extcoh}). Rappelons (\cf lemme \ref{lem:extserispeci}) que $\Ext^1_{L[\bar G]}\intnn{\intn{{\St}^{\lan}_{\lambda}}'}{{W}_{\lambda}}\cong \Ext^1_{L[\bar G]}\intnn{{W}^*_{\lambda}}{{\St}^{\lan}_{\lambda}}\cong \Hom\intnn{\Qp^{\times}}{L}$ qui est un $L$-espace vectoriel de dimension $2$. Ainsi, l'extension (\ref{eq:extcoh}) nous donne un sous-espace  $L\cdot e \subsetneq \BP\intn{\Hom\intnn{\Qp^{\times}}{L}}$ correspondant à $\Hom_{L[\GQp ]}\intnn{V_{\CL}^{\lambda}}{\rmH^{\lambda}_{\pet}}$ ; on veut montrer qu'il n'est pas nul et que par l'identification $\BP\intn{\Hom\intnn{\Qp^{\times}}{L}}\cong \BP^1(L)$ du paragraphe \ref{subsubsec:spean}, il correspond à $\CL\in L$.

Soit $\CL_1\in L\subset\BP^1(L)$ et considérons l'extension non-scindée associée à $\CL_1$ :
\begin{equation}\label{eq:extspe}
0\rightarrow {W}_{\lambda}\rightarrow \intn{{\Sigma}_{\CL_1}^{\lambda}}'\rightarrow \intn{{\St}^{\lan}_{\lambda}}'\rightarrow 0.
\end{equation}
Alors, d'après le théorème \ref{thm:muautspe}, et comme les actions de $\GQp $ et $G$ commutent sur $\rmH^{\lambda}_{\pet}$, on en déduit que
{
$$
\Hom_{L[G]}\intnn{\intn{{\Sigma}_{\CL_1}^{\lambda}}'}{\rmH_{\CL}}\cong \Hom_{L[\GQp ]}\intnn{V_{\CL}^{\lambda}}{\Hom_{L[G]}\intnn{\intn{{\Sigma}_{\CL_1}^{\lambda}}'}{\rmH^{\lambda}_{\pet}}} \cong
\begin{cases}
 L & \text{ si }\CL_1=\CL,\\
 0 & \text{ si }\CL_1\neq \CL. 
\end{cases}
$$}
En particulier, l'extension (\ref{eq:extcoh}) ne peut pas être triviale et il reste a montrer que ce n'est pas \og l'extension universelle \fg. Plus concrètement, on veut montrer $H_{\CL}\cong W_{\lambda}\oplus \intn{{\Sigma}_{\CL}^{\lambda}}'$. On vient de montrer que $\intn{{\Sigma}_{\CL}^{\lambda}}'\incl H_{\CL}$ et en analysant les composantes de Jordan-Hölder, donnée par la suite exacte (\ref{eq:extcoh}), on obtient la suite exacte courte 
\begin{equation}\label{eq:jspextriv}
0\rightarrow \intn{{\Sigma}_{\CL}^{\lambda}}'\rightarrow H_{\CL}\rightarrow W_{\lambda}\rightarrow 0.
\end{equation}
Il suffit de montrer que cette suite exacte est scindé ; pour cela, on lui applique $\Hom_G\intnn{W_{\lambda}}{\cdot}$ et on obtient
\begin{equation}\label{eq:pqextriv}
0\rightarrow {\Hom_G\intnn{W_{\lambda}}{{\Sigma}_{\CL}^{\lambda}}'}\rightarrow{\Hom_G\intnn{W_{\lambda}}{H_{\CL}}}\xrightarrow{f}{\Hom_G\intnn{W_{\lambda}}{W_{\lambda}}\xrightarrow{\delta}\Ext^1_G\intnn{W_{\lambda}}{{\Sigma}_{\CL}^{\lambda}}'}
\end{equation}
On va montrer que $f$ est surjective, ce qui implique $\delta=0$ comme on le veut. On calcule les trois premiers termes de la suite exacte précédente :
\begin{itemize}
\itemb ${\Hom_G\intnn{W_{\lambda}}{{\Sigma}_{\CL}^{\lambda}}'}=L$ d'après le lemme de Schur,
\itemb $\Hom_G\intnn{W_{\lambda}}{H_{\CL}}=E_2\cong L^2$ en appliquant $\Hom_G\intnn{W_{\lambda}}{\cdot}$ à la suite exacte (\ref{eq:extcoh}) et en utilisant que $\Hom_G\intnn{W_{\lambda}}{\intn{{\St}^{\lan}_{\lambda}}'}=0$ (\cf (\ref{eq:jhcompspe})).
\itemb $\Hom_G\intnn{W_{\lambda}}{W_{\lambda}}=L$ par le lemme de Schur.
\end{itemize}
Ainsi, la suite exacte (\ref{eq:pqextriv}) devient 
$$
0\rightarrow L \rightarrow L^2 \xrightarrow{f} L,
$$
ce qui montre bien que $f$ est surjective, donc que la suite exacte courte (\ref{eq:jspextriv}) est scindé. On obtient donc 
$$
H_{\CL}\cong W_{\lambda}\oplus \intn{{\Sigma}_{\CL}^{\lambda}}'
$$
comme on voulait.

\qed

\section{Fin de la preuve}
Le but de cette ultime section est de finir la démonstration du théorème principal :
\medskip
\begin{theo}\label{thm:chad}
Soit $V$ une $L$-représentation absolument irréductible de $\GQp $ de dimension~$\geqslant 2$.
\begin{enumerate}
\item Si $V$ est spéciale ou cuspidale, de dimension $2$, 
$$
\Hom_{L[\sW_{\Qp}]}\intnn{V}{\rmH^1_{\et}\Harg{\brv{\rmM}_{C}^{\infty}}{\Sym\BV(1)}} \cong \bPi(V)'\otimes_{L}\JL^{\lalg}(V).
$$
\item Dans tous les autres cas,
$$
\Hom_{L[\sW_{\Qp}]}\intnn{V}{\rmH^1_{\et}\Harg{\brv{\rmM}_{C}^{\infty}}{\Sym\BV(1)}} = 0.
$$
\end{enumerate}
 \end{theo}
 \medskip
Nos efforts précédents ont aboutis au corollaire suivant : 
\medskip
\begin{coro}\label{cor:excluspecusp}
Soit $\Pi$ une $L$-représentation de Banach, unitaire, admissible et absolument irréductible de $G$, alors
{
$$
\Hom_{L[G]}\intnn{\Pi'}{\rmH^1_{\et}\Harg{\brv{\rmM}^{\infty}_{C}}{\Sym\BV(1)}}\cong
\begin{cases}
 V_{M,\CL}^{\lambda}\otimes_L \JL_M^{\lambda} & \text{si $\Pi=\Pi_{M,\CL}^{\lambda}$ spéciale ou cuspidale,}\\
 0& \text{si $\Pi$ n'est pas du type ci-dessus.}
\end{cases}
$$}
\end{coro}
\medskip
\begin{proof}
On utilise les notations du numéro \ref{subsec:decompcoh}. D'après le corollaire \ref{cor:fkcompcon} il existe $\chi\colon \Qpt\rightarrow L^{\times}$ un caractère lisse tel que
$$
\Hom_{L[G]}\intnn{\Pi'}{\brv{\rmH}^1_{\et}} = \Hom_{L[G]}\intnn{\Pi'}{\presp{\rmH}^1_{\et}\otimes (\chi\circ \nu)},
$$
et quitte a tordre par $\chi$ on peut donc supposer que $\chi=1$. Comme l'action de $\czG$ est lisse sur $\presp{\rmH}^1_{\et}$ et se factorise par un quotient compact, on a 
\begin{equation}\label{eq:decompml}
\presp{\rmH}^1_{\et}=\bigoplus_{\lambda\in P_+}\bigoplus_{M\in \Phi N_{\lvert \lambda \rvert}^p}\presp{\rmH}^{\lambda}_{\et}[M]\otimes_L \JL_M^{\lambda}
\end{equation}
où $\presp{\rmH}^{\lambda}_{\et}[M]\coloneqq \Hom_{L[\czG]}\intnn{\JL_M^{\lambda}}{\presp{\rmH}^{1}_{\et}}$. Ainsi, il existe $\lambda\in P_+$ et $M\in \Phi N_{\lvert \lambda \rvert}^p$ tel que
$$
\Hom_{L[G]}\intnn{\Pi'}{\brv{\rmH}^1_{\et}}=\Hom_{L[G]}\intnn{\Pi'}{\rmH^{\lambda}_{\et}[M]}\otimes_L\JL_M^{\lambda}
$$
Supposons que cet espace n'est pas nul. 
\begin{itemize}
\itemb Si $M$ est cuspidal, alors d'après le théorème \ref{cor:multautcuspet}, $\Pi=\Pi_{M,\CL}^{\lambda}$ pour $\CL\subset M_{\dR}$ et 
$$
\Hom_{L[G]}\intnn{\Pi'}{\brv{\rmH}^1_{\et}}= V_{M,\CL}^{\lambda}\otimes_L\JL_M^{\lambda}.
$$
\itemb Si $M$ est spécial, on suppose $w(\lambda)\neq 1$ et comme $\chi=1$ on a $M= \Sp_L(\lvert \lambda \rvert)$. Donc $\rmH^{\lambda}_{\et}[M]=(\rmH^{\lambda}_{\et})^{\czG}\cong\rmH^1_{\et}\Harg{\BH_C}{\BV_{\lambda}(1)}$ (à partir de Hochschild-Serre, comme $\rmH^0_{\et}\Harg{\brv{\rmM}_C^{\infty}}{\BV_{\lambda}(1)}=0$ puisque $w(\lambda)>1$) . Alors, d'après le théorème \ref{cor:muautspe} on a $\Pi=\Pi_{\CL}^{\lambda}$ pour $\CL\in L$ et 
$$
\Hom_{L[G]}\intnn{\Pi'}{\brv{\rmH}^1_{\et}}= V_{\CL}^{\lambda}\otimes_L\cz{W}_{\lambda}.
$$
\end{itemize}

\end{proof}
On a déjà la première partie du théorème \ref{thm:chad} que l'on résume aussi sous la forme d'un corollaire ; il se déduit à partir des théorèmes \ref{theo:multcusptot} et \ref{thm:mugalspeet}. Comme la preuve est exactement la même que pour le corollaire \ref{cor:excluspecusp} on se permet de l'omettre.
\medskip
\begin{coro}\label{cor:fin}
Soit $M$ un $L$-$(\varphi,N,\GQp)$-module absolument irréductible de pente $1-\lvert \lambda \rvert/2$ et $\CL$ est une $L$-droite de $M_{\dR}$ tel que si $M$ est spécial alors $w(\lambda)\neq1$ (\ie $V_{M,\CL}^{\lambda}$ est absolument irréductible), on a
 $$
 \Hom_{L[\GQp ]}\intnn{V_{M,\CL}^{\lambda}}{\rmH^{1}_{\et}\Harg{\brv{\rmM}_C^{\infty}}{\Sym \BV(1)}}\cong {\Pi_{M,\CL}^{\lambda}}'\otimes \JL_M^{\lambda}.
 $$
En particulier, cet entrelacement n'est pas nul.
\end{coro}
\medskip
On déduit de ces théorèmes que la cohomologie étale $p$-adique à coefficients isotriviaux encode la correspondance de Langlands locale $p$-adique pour les représentations spéciales et cuspidales de dimension $2$. Pour finir la démonstration du théorème \ref{thm:chad} Il reste à montrer qu'à part des caractères lisses, le socle de $\rmH^1_{\et}\Harg{\brv{\rmM}_{C}^{\infty}}{\Sym\BV(1)}$ ne contient que des représentations spéciales ou cuspidales de dimension $2$ de poids $\lambda\in P_+$. On le fait en suivant l'argument de \cite[5.10]{codonifac}.
\medskip
\begin{prop}\label{prop:chad}
 Soit $V$ une $L$-représentation absolument irréductible de $\GQp$ de dimension $\geqslant 2$. Si $\Hom_{L[\sW_{\Qp} ]}\intnn{V}{\rmH^1_{\et}\Harg{\brv{\rmM}_{C}^{\infty}}{\Sym \BV(1)}}\neq 0$ alors $V$ est de dimension $2$, spéciale ou cuspidale.
\end{prop}
\medskip
\begin{proof}
Comme précédemment, quitte à tordre par un caractère lisse, on peut remplacer $\brv{\rmM}_{C}^{\infty}$ par $\presp{\rmM}_{C}^{\infty}$. Montrons que si $V$ est une $L$-représentation absolument irréductible de $\GQp$ qui n'est ni spéciale, ni cuspidale ou bien de dimension $\geqslant 3$, alors 
$$
\Hom_{\GQp }\intnn{V}{\rmH^1_{\et}\Harg{\presp{\rmM}_{C}^{\infty}}{\Sym\BV(1)}} = 0.
$$
Supposons le contraire, on raisonne par contradiction. Posons
$$
\Pi(V)\coloneqq\Hom_{\GQp }\intnn{V}{\rmH^1_{\et}\Harg{\presp{\rmM}_{C}^{\infty}}{\Sym\BV(1)}}',
$$
qui définit une $L$-représentation de Banach contenant un réseau (défini à partir d'un réseau dans $V$ et des réseaux $\BV_k^+$) que l'on note $\Pi^+(V)$ et dont la réduction modulo l'idéal maximal $\fkm_L$ de $\rmO_L$ est de longueur finie par le théorème de finitude \ref{prop:longfinui} : $\Pi(V)$ est donc admissible. Ainsi, $\Pi(V)$ est une $L$-représentation de Banach unitaire admissible et de longueur finie. On peut donc supposer que $\Pi(V)$ a comme quotient une $L$-représentation de Banach unitaire, admissible et absolument irréductible que l'on note $\Pi$. On en déduit une injection $V\incl\Hom_{G}\intnn{\Pi'}{\rmH^1_{\et}\Harg{\presp{\rmM}_{C}^n}{\Sym\BV(1)}} $ ce qui est une contradiction avec le corollaire \ref{cor:excluspecusp}. 
\end{proof}

\end{document}